\documentclass[12pt]{amsart}
\usepackage[utf8]{inputenc}
\usepackage[letterpaper,total={6.5in,9in},top=1in, left=1in, right=1in, bottom=1in]{geometry}
\usepackage{fancyhdr}
\usepackage{libertine}
\usepackage[libertine]{newtxmath}
\usepackage{color,amsmath,mathtools,graphicx}
\usepackage{bbold}
\usepackage{epsfig,fancyhdr}
\usepackage{dsfont} 
\usepackage{graphicx,amsmath,mathtools,float}
\usepackage{epsfig,color,fancyhdr,setspace}
\usepackage{dsfont}
\usepackage{epstopdf}
\usepackage{import}
\usepackage{lineno}
\usepackage{stackrel,dsfont}
\usepackage[allcolors = blue,colorlinks]{hyperref}
\usepackage[abs]{overpic}
\usepackage[center]{caption}
\usepackage{subcaption}
\usepackage{tikz-cd}
\usepackage{enumitem}
\counterwithin{figure}{section}
\usepackage[utf8]{inputenc}
\usepackage[T1]{fontenc}
\usepackage{ragged2e}
\usepackage{extarrows}
\usepackage{rotating}
\allowdisplaybreaks
\newtheorem{theorem}{Theorem}[section]
\newtheorem{lemma}[theorem]{Lemma}
\newtheorem{definition}[theorem]{Definition}
\newtheorem{example}[theorem]{Example}
\newtheorem{proposition}[theorem]{Proposition}
\newtheorem{remark}[theorem]{Remark}
\newtheorem{corollary}[theorem]{Corollary}
\newtheorem{conjecture}[theorem]{Conjecture}

\newtheorem{innercustomgeneric}{\customgenericname}
\providecommand{\customgenericname}{}
\newcommand{\newcustomtheorem}[2]{%
  \newenvironment{#1}[1]
  {%
   \renewcommand\customgenericname{#2}%
   \renewcommand\theinnercustomgeneric{##1}%
   \innercustomgeneric
  }
  {\endinnercustomgeneric}
}

\newcustomtheorem{customthm}{Theorem}
\newcustomtheorem{customlemma}{Lemma}

\title{On the Kauffman bracket skein module of 
$\mathbf {(S^1 \times S^2) \ \# \ (S^1 \times S^2)}$}

\author{Rhea Palak Bakshi}
\address{Institute for Theoretical Studies, ETH Zürich, Switzerland and \newline \indent Department of Mathematics, University of California, Santa Barbara, USA}
\email{\rm rheapalak.bakshi@eth-its.ethz.ch | rheapalak@math.ucsb.edu}

\author{Seongjeong Kim} 
\address{Jilin University, Changchun, China}
\email{\rm kimseongjeong@jlu.edu.cn}

\author{Shangjun Shi} 
\address{East China Normal University, Shanghai, China}
\email{\rm 51255500022@stu.ecnu.edu.cn}

\author{Xiao Wang} 
\address{Jilin University, Changchun, China}
\email{\rm wangxiaotop@jlu.edu.cn}


\begin{document}

\begin{abstract}

Determining the structure of the Kauffman bracket skein module of all $3$-manifolds over the ring of Laurent polynomials $\mathbb Z[A^{\pm 1}]$ is a big open problem in skein theory. Very little is known about the skein module of non-prime manifolds over this ring. In this paper, we compute the Kauffman bracket skein module of the $3$-manifold $(S^1 \times S^2) \ \# \ (S^1 \times S^2)$ over the ring $\mathbb Z[A^{\pm 1}]$. We do this by analysing the submodule of handle sliding relations, for which we provide a suitable basis. Along the way we compute the Kauffman bracket skein module of $(S^1 \times S^2) \ \# \ (S^1 \times D^2)$. We also show that the skein module of $(S^1 \times S^2) \ \# \ (S^1 \times S^2)$ does not split into the sum of free and torsion submodules. Furthermore, we illustrate two families of torsion elements in this skein module. 

\end{abstract}

\maketitle
\tableofcontents

\section{Introduction}

Skein modules are $3$-manifold invariants that generalise the skein theory of polynomial link invariants in $S^3$ to any arbitrary $3$-manifold. They were introduced independently by Przytycki \cite{smof3} and Turaev \cite{turaevsolidtorus} in the late 1980's, and have since  become indispensable in bridging the fields of quantum topology, knot theory, algebraic geometry, hyperbolic geometry, and physics. The Kauffman bracket skein module, which serves as a generalisation of the Kauffman bracket polynomial \cite{sm&j} to arbitrary $3$-manifolds, is conceivably the best understood skein module of all. In this paper, we determine the Kauffman bracket skein module of $(S^1 \times S^2) \ \# \ (S^1 \times S^2)$. One motivation for our work comes from constructing traces, such as the Yang-Mills measure, on the Kauffman bracket skein module of a thickened surface. In \cite{bfkyangmills}, the Yang-Mills measure is defined away from roots of unity using the Kauffman bracket skein module, henceforth known simply as the skein module, of $\#_k (S^1 \times S^2)$ over the field $\mathbb C$. To construct other traces on the skein module of a surface at roots of unity, it is imperative to know the structure of the skein module of $\#_k (S^1 \times S^2)$ over $\mathbb Z[A^{\pm 1}]$. Knowing the traces on a skein module will aid us in the construction of the $3$-manifold invariants that may be defined as traces. For example, the Yang-Mills measure may be used to define the Turaev-Viro invariant at roots of unity. We expect that our computation of the skein module of $(S^1 \times S^2) \ \# \ (S^1 \times S^2)$ will prove to be effective in this regard. \\

From the perspective of algebraic geometry, it is known that, modulo the nilradical, the Kauffman bracket skein module of an oriented $3$-manifold over $\mathbb C[A^{\pm 1}]$ with $A = -1$ has an algebra structure that is isomorphic to the coordinate ring of the $SL(2,\mathbb C)$ character variety of the fundamental group of that manifold  \cite{sl2cbullock, saofsurfaces}. Furthermore, if the underlying algebraic set $X(M)$ of the $SL(2,\mathbb C)$ character variety of the fundamental group of a closed oriented $3$-manifold $M$ is infinite, then the skein module of $M$ is wild, that is, it is not tame \cite{dks}.\footnote{A $\mathbb Z[A^{\pm 1}]$-module is said to be tame if it is a direct sum of cyclic $\mathbb Z[A^{\pm 1}]$-modules and it does not contain $\mathbb Z[A^{\pm 1}]/(\phi_{2N})$ as a submodule, for at least one odd $N$, where $\phi_{2N}$ is the $2N$-th cyclotomic polynomial.} For example, the underlying algebraic set $X(M)$ of the $SL(2,\mathbb C)$ character variety of $\pi_1(S^1 \times S^2)$ is infinite, and hence, the Kauffman bracket skein module of $S^1 \times S^2$ is not tame as has been proved by Hoste and Przytycki in \cite{s1s2}. In fact,  until now, $S^1 \times S^2$ is the only closed $3$-manifold with infinite $X(M)$ whose skein module has been computed. The manifold $(S^1\times S^2) \ \# \ (S^1\times S^2)$ is the next example of the skein module of a closed $3$-manifold with this property (see \cite{gmsl2cfreegroup}).  \\

Furthermore, the resolution of Witten's finiteness conjecture for Kauffman bracket skein modules in \cite{wittenresolved} implies that over $\mathbb Q(A)$, the Kauffman bracket skein module of any closed oriented $3$-manifold is always finite dimensional. However, over $\mathbb Z[A^{\pm 1}]$, the structure of the skein module is more complicated. For example, the skein module of $S^1 \times S^2$ is infinitely generated over $\mathbb Z[A^{\pm 1}]$ \cite{s1s2}. Recently, the first author \cite{rheasolomarche} disproved a conjecture posited by Marché (see \cite{basissurfacetimess1}), which stated that the skein module of any closed oriented $3$-manifold can be decomposed into free and torsion modules. The counterexample to this conjecture was given by the skein module of the connected sum of two copies of the real projective space (see \cite{rp3rp3}). This emphasises the fact that save for a handful of manifolds, the structure of the skein module is not as well understood  over $\mathbb Z[A^{\pm 1}]$ as it is over $\mathbb Q(A)$. To better understand the structure of the skein module of oriented $3$-manifolds over $\mathbb Z[A^{\pm 1}]$, we study the skein module of the connected sums of $3$-manifolds. Thus, with these motivations, we compute the skein module of $(S^1 \times S^2) \ \# \ (S^1 \times S^2)$. \\ 

The paper is organised as follows. In Section \ref{basicss1s2s1s2}, we define absolute and relative Kauffman bracket skein modules and discuss some of their properties. We include a description of the module using generators and relations. 
We then compute the skein module of $(S^1 \times S^2) \ \# \ (S^1 \times S^2)$ in Section \ref{sectionkbsms1s2s1s2}. Our technique employs the relative Kauffman bracket skein module in determining the complete set of handle sliding relations. We include all our calculations towards this computation in the \hyperref[appendixs1s2s1s2]{Appendix}. In Section \ref{props1s2s1s2}, we show that the skein module of $(S^1 \times S^2) \ \# \ (S^1 \times S^2)$ does not split into the direct sum of free and torsion submodules. Furthermore, in Section \ref{section: torsionelements}, we present some families of torsion elements in the  skein module of $(S^1 \times S^2) \ \# \ (S^1 \times S^2)$. Finally, in Section \ref{fds1s2s1s2}, we discuss future directions. 

\subsection*{Acknowledgements}

The first author acknowledges the support of Dr. Max Rössler, the Walter Haefner Foundation, and the ETH Zürich Foundation. The second author is partially supported by the National Natural Science Foundation of China (Grant No. 12201239). The fourth author is partially supported by the National Natural Science Foundation of China (Grant No. 11901229) and by the National Natural Science Foundation of China (Grant No. 22341304). The second and fourth authors are also together partially supported by the National Natural Science Foundation of China (Grant No. 12371029). The authors thank Charles Frohman, Józef H. Przytycki, Adam Sikora, and Renaud Detcherry for helpful and stimulating discussions. The authors also thank the referee for helpful comments.

\section{Basic definitions and properties}\label{basicss1s2s1s2}

We begin with introducing the Kauffman bracket skein module and the relative Kauffman bracket skein module.

\begin{definition}\label{kbsmdef}

Let $M$ be an oriented $3$-manifold, $R$ a commutative ring with unity, and $A \in R$ a fixed invertible element. Consider  the set of ambient isotopy classes of unoriented framed links (including the empty link $\varnothing$) in $M$, which we denote by $\mathcal{L}^{\mathit{fr}}$,  and the free $R$-module with basis $\mathcal{L}^{\mathit{fr}}$, denoted by $R\mathcal{L}^{\mathit{fr}}$.
Let $S_{2, \infty}^{\mathit{sub}}$ be the submodule of $R\mathcal{L}^{\mathit{fr}}$ generated by the following expressions: 

\begin{enumerate}

    \item the Kauffman bracket skein expression: $L_+ - AL_0 - A^{-1}L_{\infty}$ and
    
    \item the trivial component expression: $L \sqcup {\pmb \bigcirc}  + (A^2 + A^{-2})L$,
    
\end{enumerate}
\noindent
where $\pmb{\bigcirc}$ denotes the trivial framed knot in $M$ and the skein triple $(L_+$, $L_0$, $L_{\infty})$ denotes three framed links in $M$, which are identical except in a small $3$-ball in $M$ where they differ as illustrated in Figure \ref{skeintriple}.

\begin{figure}[ht]
    \centering
\[  \begin{minipage}{1.3 in} \includegraphics[width=\textwidth]{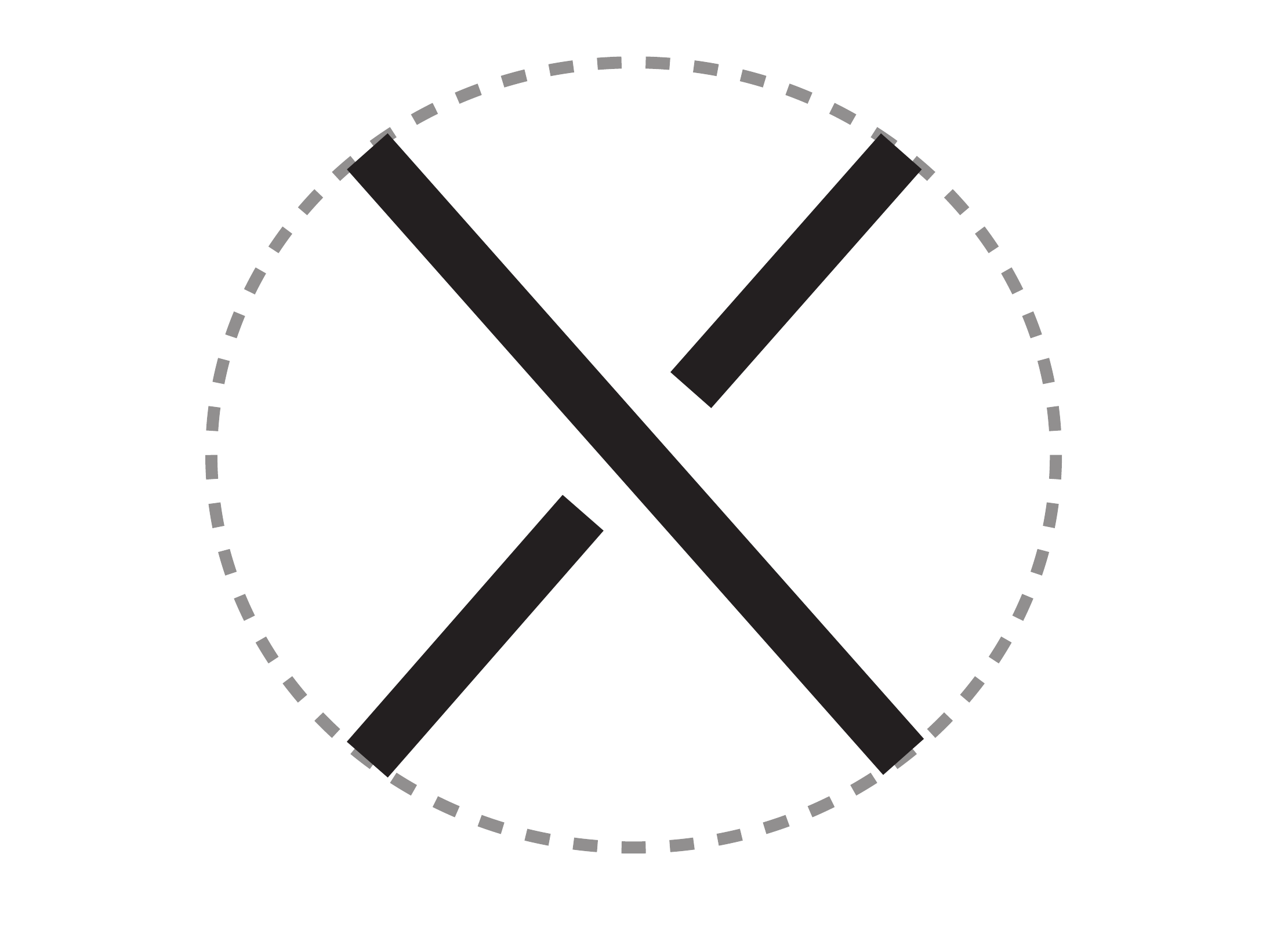} \vspace{-15pt} \[L_+\] \end{minipage} 
               \qquad
        \begin{minipage}{1.3 in}\includegraphics[width=\textwidth]{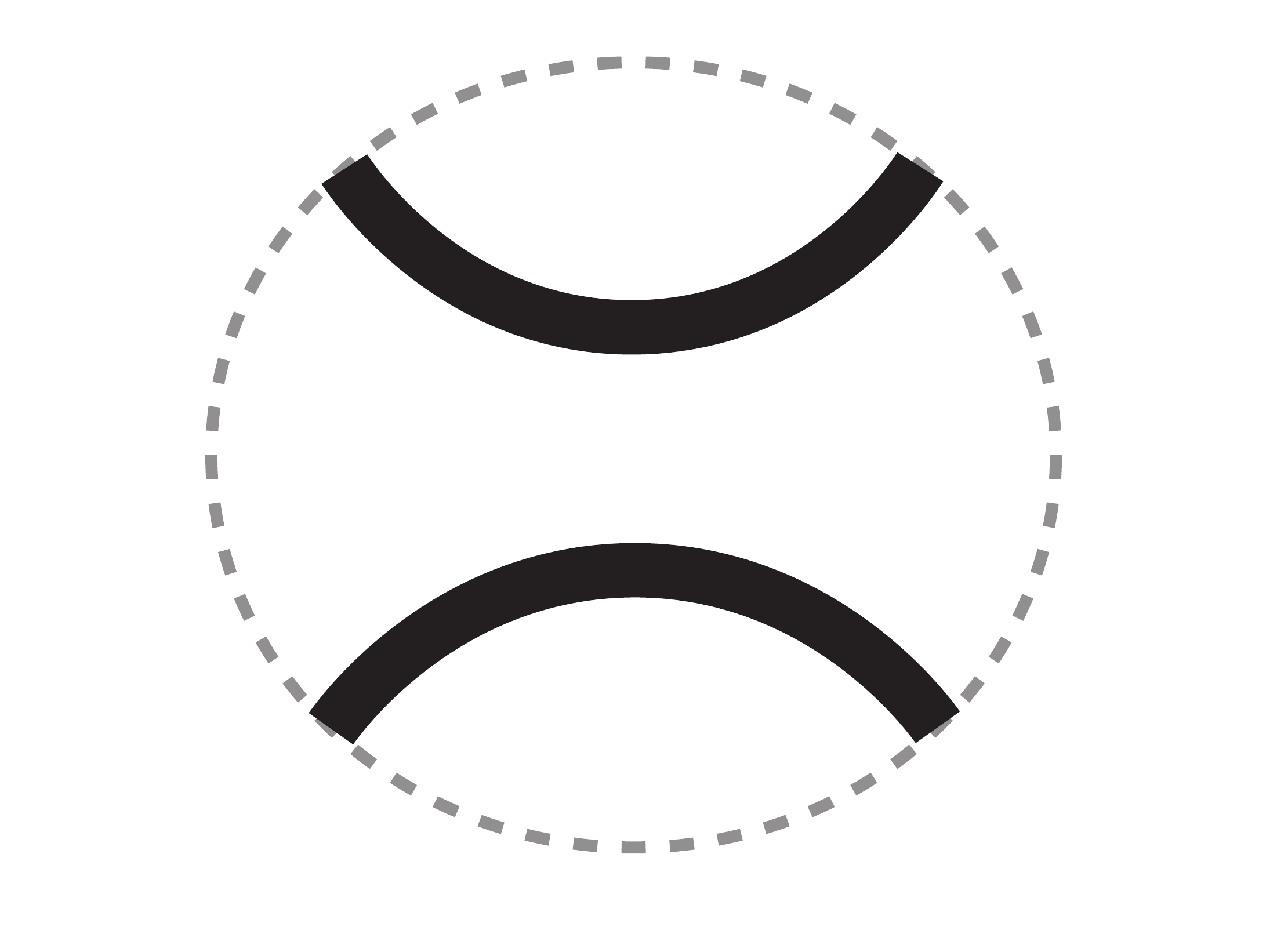} \vspace{-15pt} \[L_0\] \end{minipage}
         \qquad
        \begin{minipage}{1.3 in}\includegraphics[width=\textwidth]{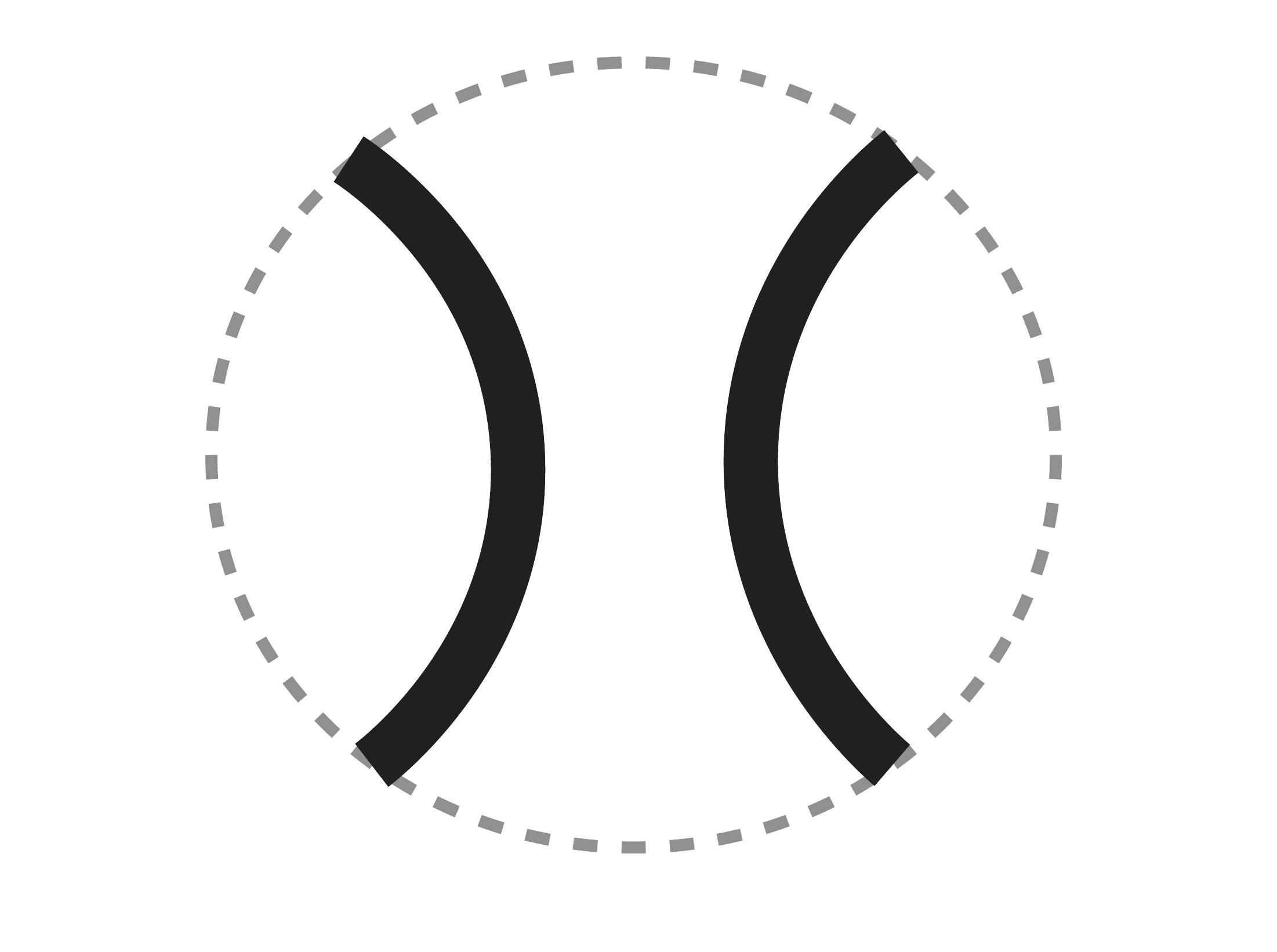} \vspace{-15pt} \[ L_\infty\]\end{minipage} 
        \]
\caption{Skein triple for the Kauffman bracket skein module.}
          \label{skeintriple}
        \end{figure}

The {\bf Kauffman bracket skein module} of $M$ is defined as the quotient: $$\mathcal{S}_{2,\infty}(M;R,A) = \frac{R\mathcal{L}^{\mathit{fr}}}{S_{2, \infty}^{sub}}.$$

\end{definition}

Computations of the Kauffman bracket skein module for various $3$-manifolds have been carried out over several rings, such as $\mathbb{Z}[A^{\pm 1}]$, $\mathbb Q(A)$, or a ring $R$ in which $A^k - 1$ is invertible for all $k$. In our paper, we work over $\mathbb{Z}[A^{\pm 1}]$ and use the notation $\mathcal{S}_{2,\infty}(M)$ in this case. The existence of the Kauffman bracket polynomial can be interpreted in the language of skein modules as follows. We can also define a relative version of the Kauffman bracket skein module for oriented $3$-manifolds that have framed (or marked) points on their boundaries (see \cite{smof3, fundamentals}).  

    

\begin{definition}

Let $(M,\partial M)$ be an oriented $3$-manifold, $\{x_i\}_{1}^{2n}$ be a set of $2n$ oriented framed\footnote{A framed point in $\partial M$ is an interval in $\partial M$. Thus, a relative framed link intersects
$\partial M$ at framed points.} points on $\partial M$, and $R$ be a commutative ring with unity with a fixed invertible element $A$. Let $\mathcal{L}^{\mathit{fr}}(2n)$ be the set of all relative framed links in $(M, \partial M)$ considered up to ambient isotopy keeping $\partial M$ fixed, such that $L \cap \partial M = \partial L = \{x_i\}_1^{2n}$. Consider the submodule $S_{2,\infty}^{\mathit{sub}}(2n)$ of the free $R$-module $R\mathcal{L}^{\mathit{fr}}(2n)$ generated by the Kauffman bracket skein expressions. Then the {\bf relative Kauffman bracket skein module}, henceforth known as the relative skein module, of $M$ is the quotient: $$\mathcal{S}_{2,\infty}(M, \{x_i\}_1^{2n}; R, A) = \frac{R\mathcal{L}^{\mathit{fr}}(2n)}{ S_{2,\infty}^{\mathit{sub}}(2n)}.$$ 

\end{definition}

We will use the notation $\mathcal{S}_{2,\infty}(M,\{x_i\}_1^{2n})$ when $R = \mathbb{Z}[A^{\pm 1}]$. The following results about the skein module and its relative version of the product of an oriented surface with the unit interval are pertinent to our work.

\begin{theorem}\label{ftimesi}\cite{smof3,fundamentals}
 
Let $\Sigma$ be an oriented surface in which each link is equipped with blackboard framing and let $I$ denote the unit interval $[0,1]$. Then $\mathcal{S}_{2,\infty}(\Sigma \times I;R,A)$ is a free $R$-module whose basis consists of the empty link $\varnothing$ and simple closed multicurves in $\Sigma$ that have no trivial components. This applies in particular to handlebodies, since $H_{n} = \Sigma_{0,n+1} \times I$, where $H_n$ is a handlebody of genus $n$ and $\Sigma_{g,b}$ denotes a genus $g$ surface with $b$ boundary components.
 
 \end{theorem}

 The following example discusses the skein module of the thickened annulus.

 \begin{example}\label{solidtoruskbsm}

$\mathcal S_{2,\infty} (\Sigma_{0,2} \times I; R, A)$ is free and infinitely generated by the curves $\{x^i\}_{i \geq 0}$, where $x$ denotes the homotopically nontrivial simple closed curve on the annulus and $x^0$ denotes the empty link $\varnothing$. Note that, $\mathcal S_{2,\infty} (S^1 \times D^2; R, A) \cong \mathcal S_{2,\infty} (\Sigma_{0,2} \times I; R, A)$. 
     
 \end{example}

 A result similar to Theorem \ref{ftimesi} also holds for relative skein modules.

 \begin{theorem}\label{rkbsmsib}\cite{fundamentals}

Let $\Sigma$ be an oriented surface, where $\partial \Sigma \neq \emptyset$, and let $\{x_i\}_1^{2n}$ be $2n$ oriented framed points centred at $\partial \Sigma \times \{\frac{1}{2}\}$. 
Then $\mathcal S_{2, \infty}(\Sigma \times I, \{x_i\}_1^{2n}; R, A)$ is a free $R$-module whose basis is composed of relative links\footnote{Relative links in $\Sigma$ are families of properly embedded arcs and closed curves in $\Sigma \times \{\frac{1}{2}\}$.} in $\Sigma \times \{\frac{1}{2}\}$ without trivial components.

\end{theorem}

The skein module of a surface times an interval may be equipped with an algebra structure for which the multiplication operation is defined as follows.

\begin{definition}

 Consider two framed links $L_1$ and $L_2$ in $\Sigma \times I$. Define their product $\cdot$ by placing $L_1$ over $L_2$ in $\Sigma \times I$, that is, $L_1 \cdot L_2 = L_1 \sqcup L_2$ such that $L_{1} \subset \Sigma\times (\frac{1}{2}, 1)$ and $L_2 \subset \Sigma\times(0,\frac{1}{2})$. The empty link $\varnothing$ serves as the multiplicative identity. This multiplication endows the skein module of a thickened surface $\Sigma \times I$ with a natural algebra structure. The Kauffman bracket skein module equipped with this algebra structure is called the Kauffman bracket skein algebra.
    
\end{definition}

We denote the Kauffman bracket skein algebra, henceforth known simply as the skein algebra, by $\mathcal{S}^{\mathit{alg}}(\Sigma; R, A)$. This new notation emphasises the fact that the skein algebra depends on the surface and its product structure. For brevity, we use the notation $\mathcal{S}^{\mathit{alg}}(\Sigma)$ when $R = \mathbb{Z}[A^{\pm 1}]$.  

\begin{remark}\label{relativetimesframed}

 $\mathcal S_{2, \infty}(\Sigma \times I, \{x_i\}_1^{2n}; R, A)$ is a bimodule over the algebra $\mathcal{S}^{\mathit{alg}}(\Sigma; R, A)$, which contains the ring $R$. Let $L_1$ be a relative framed link in $\Sigma \times I$ and $L_2$ be a framed link in $\Sigma \times I$. Then, $L_1 \cdot L_2$ is defined by placing $L_1$ above $L_2$, that is, $L_1 \subset \Sigma \times (\frac{1}{3},1)$ and $L_2 \subset \Sigma \times (0,\frac{1}{3})$. Similarly, $L_2 \cdot L_1$ is defined by placing $L_2$ over $L_1$, that is, $L_2 \subset \Sigma \times (\frac{2}{3},1)$ and $L_1 \subset \Sigma \times (0,\frac{2}{3})$. 

\end{remark}

We now state some properties of the skein module required for proving our main results. The following theorem determines how the Kauffman bracket skein module behaves under handle addition, thereby giving its presentation in terms of generators and relations. 









    \begin{theorem}\cite{fundamentals,kbsmlens}\label{propkbsm2}

    \begin{enumerate}

    \item  If $N$ is obtained from $M$ by adding a $3$-handle to $M$ and $i : M \hookrightarrow N$ is the associated embedding, then the induced homomorphism 
$i_* : \mathcal{S}_{2,\infty}(M;R,A) \longrightarrow \mathcal{S}_{2,\infty}(N;R,A)$ is an isomorphism.\label{3hand} \\

\item (Handle Sliding Lemma)\label{handleslidinglemma} Let $(M,\partial M)$ be a $3$-manifold with boundary and $\gamma$ be a
simple closed curve on $\partial M$. Additionally, let $N = M_{\gamma}$ be the $3$-manifold
obtained from $M$ by adding a $2$-handle along $\gamma$ and $i : M \hookrightarrow N$ be the associated embedding. Then the induced homomorphism
$i_* : \mathcal{S}_{2,\infty}(M;R,A) \longrightarrow \mathcal{S}_{2,\infty}(N;R,A)$ is an epimorphism.\label{2hand} Furthermore, the kernel of $i_*$ is generated by the relations yielded by $2$-handle slidings. In particular, if $\mathcal{L}^{\mathit{fr}}_{\mathit{gen}}$
is a set of framed links in $M$ that generates $\mathcal{S}_{2,\infty}(M;R,A)$,
then $\mathcal{S}_{2,\infty}(N;R,A) \cong \mathcal{S}_{2,\infty}(M;R,A)/\mathcal J$, where $\mathcal J$ is the submodule of
$\mathcal{S}_{2,\infty}(M;R,A)$ generated by the expressions $L - sl_{\gamma}(L)$.  Here $L \in \mathcal{L}^{\mathit{fr}}_{\mathit{gen}}$ and $sl_{\gamma}(L)$ is obtained from $L$ by sliding it along $\gamma$. 

\end{enumerate}

\end{theorem}

The following result forms the basis of our computation. 

\begin{theorem}\cite{fundamentals}

     Let $M$ be a compact oriented $3$-manifold. Then $M$ is obtained from a genus $m$ handlebody $H_m$ by adding
$2$- and $3$-handles to it and the generators
of $\mathcal{S}_{2,\infty}(M;R,A)$ are generators of $\mathcal{S}_{2,\infty}(H_m;R,A)$, while the relations of $\mathcal{S}_{2,\infty}(M;R,A)$
are yielded by $2$-handle slidings. In particular, let $\beta$ and $\eta$ be disjoint simple closed curves in $\partial H_m$. Glue two $2$-handles to $H_m$, one each along the curves $\beta$ and $\eta$, and denote the resultant $3$-manifold by $M$. If $\mathcal J_1$ is the submodule of $\mathcal{S}_{2,\infty}(H_m;R,A)$ generated by handle slides along $\beta$ and $\mathcal J_2$ is the submodule of $\mathcal{S}_{2,\infty}(H_m;R,A)$ generated by handle slides along $\eta$, then $\mathcal{S}_{2,\infty}(M;R,A) \cong \mathcal{S}_{2,\infty}(H_m;R,A)/(\mathcal J_1 + \mathcal J_2)$.
\label{kbsmheegaard}

\end{theorem}





Hence, the problem of computing the skein module of any compact oriented $3$-manifold is reduced to that of determining all the $2$-handle sliding relations. In the next section we compute the Kauffman bracket skein module of $(S^1 \times S^2) \ \# \ (S^1 \ \times \ S^2)$.





    


    \section{The Kauffman bracket skein module of $(S^1 \times S^2) \ \# \ (S^1 \ \times \ S^2)$}\label{sectionkbsms1s2s1s2}

  Let $\beta$ and $\eta$ be two simple closed curves in the boundary of the genus two handlebody, $H_2$, as illustrated in Figure \ref{s1s2connsumh1fig}. Glue a $2$-handle along each of these curves and then cap off the holes with two $3$-handles. The resultant $3$-manifold is $(S^1 \times S^2) \ \# \ (S^1 \times S^2)$. 
  From Theorem \ref{propkbsm2}, it follows that the natural embedding $i: H_2 \hookrightarrow (S^1 \times S^2) \ \# \ (S^1 \times S^2)$ yields the epimorphism $i_*: \mathcal S_{2,\infty}(H_2) \longrightarrow \mathcal S_{2,\infty} ((S^1 \times S^2) \ \# \ (S^1 \times S^2))$ of skein modules. Let $\mathcal J_1$ and $\mathcal J_2$ be the submodules of $\mathcal S_{2,\infty}(H_2)$ generated by the handle sliding relations obtained from $2$-handle sliding along $\beta$ and $\eta$, respectively. From Theorem \ref{kbsmheegaard} we get that $\mathcal S_{2,\infty}((S^1 \times S^2) \ \# \ (S^1 \times S^2)) \cong \mathcal S_{2,\infty}(H_2)/ (\mathcal J_1 + \mathcal J_2)$. Thus, our main problem reduces to determining the submodules $\mathcal J_1$ and $\mathcal J_2$. \\

  

     \begin{figure}[ht]
    \centering
  $\vcenter{\hbox{\begin{overpic}[unit=1mm, scale = .2]{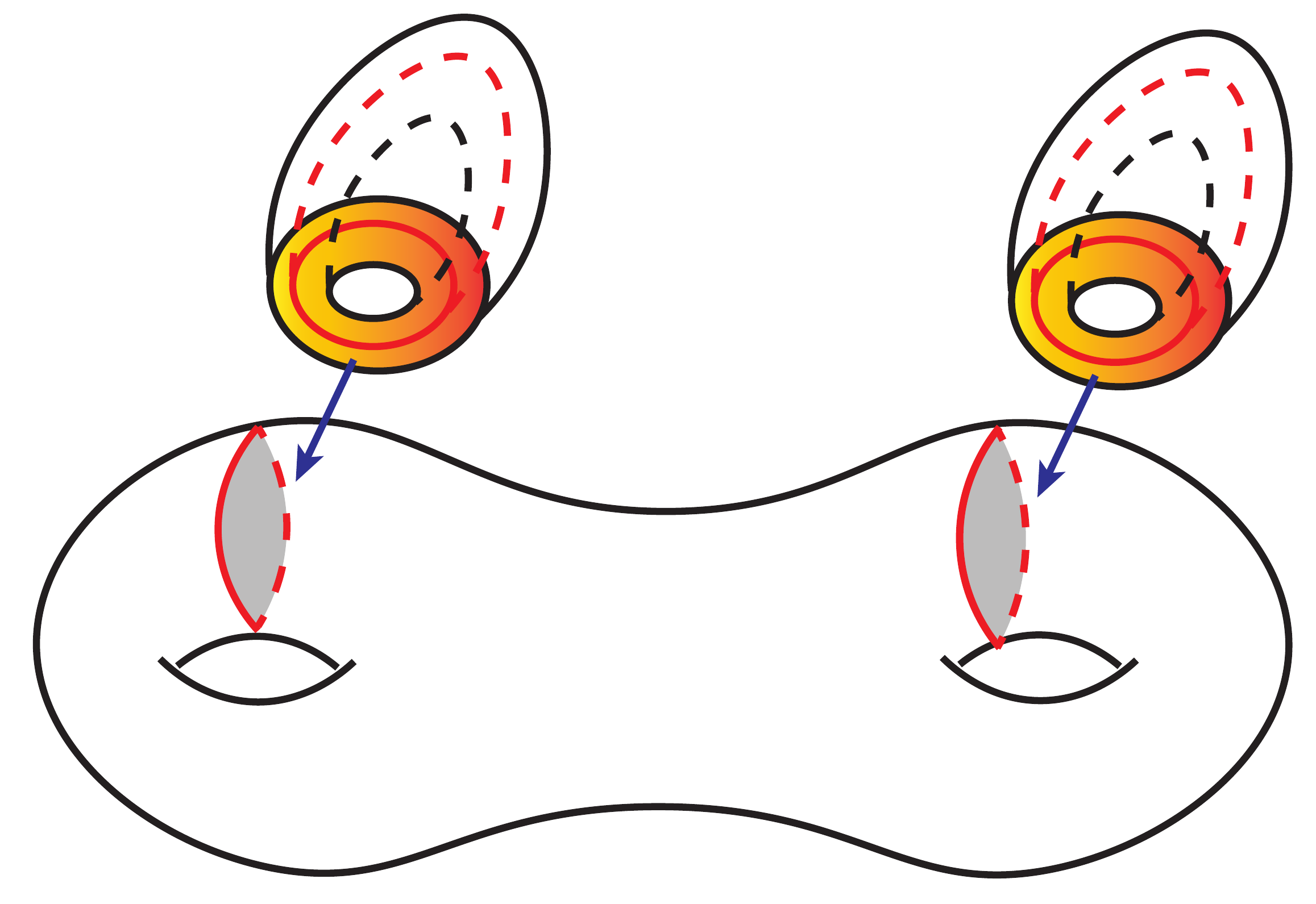}
 \put(10,22){\footnotesize{$\beta$}} 
  \put(22,29){\tiny{\text{glue}}}
  \put(34,43){\footnotesize{$2$\text{-handle}}}
  \put(-40,15){$(S^1 \times S^2) \ \# \ (S^1 \times S^2) \cong $} 
  \put(29,-1){$H_2 \cong \Sigma_{0,3} \times I$}
   \put(54.5,22){\footnotesize{$\eta$}} 
   \put(66,28.5){\tiny{\text{glue}}}
  \put(78,43){\footnotesize{$2$\text{-handle}}}
  \end{overpic} }}  $
    \caption{The 3-manifold $(S^1 \times S^2) \ \# \ (S^1 \times S^2) $ is obtained by gluing a $2$-handle each to $\partial H_2$ along the curves $\beta$ and $\eta$.}
    \label{s1s2connsumh1fig}
\end{figure}

Note that, $H_2 \cong \Sigma_{0,3} \times I$. 
Since all framed links in $(S^1 \times S^2) \ \# \ (S^1 \times S^2)$ have representatives in $H_2$, framed links in $(S^1 \times S^2) \ \# \ (S^1 \times S^2)$ may be presented by link diagrams on $\Sigma_{0,3}$; see Theorem \ref{ftimesi}. We illustrate the curves $\beta$ and $\eta$ as line segments in $\Sigma_{0,3}$; 
 see Figure \ref{s1s2connh1gamma}. Thus, we first describe the skein module and algebra of $\Sigma_{0,3} \times I$.

\subsection{The skein module of $\Sigma_{0,3} \times I$}  From Theorem \ref{ftimesi} we obtain the following result about the skein module of $\Sigma_{0,3} \times I$.

\begin{theorem}\cite{smof3, fundamentals}

$\mathcal S_{2,\infty}(\Sigma_{0,3} \times I)$ is a free and infinitely generated $\mathbb Z[A^{\pm 1}]$-module whose standard basis consists of monomials of the form $\{a_1^ia_2^ja_3^k\}_{i,j,k \geq 0}$, where $a_1, a_2$, and $a_3$ represent the homotopically nontrivial curves on $\Sigma_{0,3}$ as illustrated in Figure \ref{pantscurves}. The empty link is represented by $a_1^0a_2^0a_3^0$.
    
\end{theorem}

Note that the set $\{S_{i}(a_{1})S_{j}(a_{2})S_{k}(a_{3})\}_{i,j,k\geq 0}$ also forms a basis for $\mathcal S_{2,\infty}(\Sigma_{0,3} \times I)$. Here, $S_q$ denotes the Chebyshev polynomials of the second kind, which satisfy the recurrence relation $S_{q+1}(x) = xS_q(x) - S_{q-1}(x)$, with the initial conditions $S_{0}(x) = 1$ and $S_{1}(x) = x$. In our paper, it is necessary for us to extend the initial conditions to include negative indices, in which case $S_{-2}(x) = -1$ and $S_{-1}(x) = 0$. We also have the following result due to Bullock and Przytycki about the skein algebra of $\Sigma_{0,3}$. 

\begin{theorem}\cite{smquant}

$\mathcal{S}^{\mathit{alg}}(\Sigma_{0,3})$ is a commutative algebra and is isomorphic to $\mathbb{Z}[A^{\pm 1}][a_1, a_2, a_3]$.
    
\end{theorem}

\begin{figure}[ht]
\centering
\begin{subfigure}{.45\textwidth}
\centering
\begin{overpic}[unit=1mm, scale = 0.12]{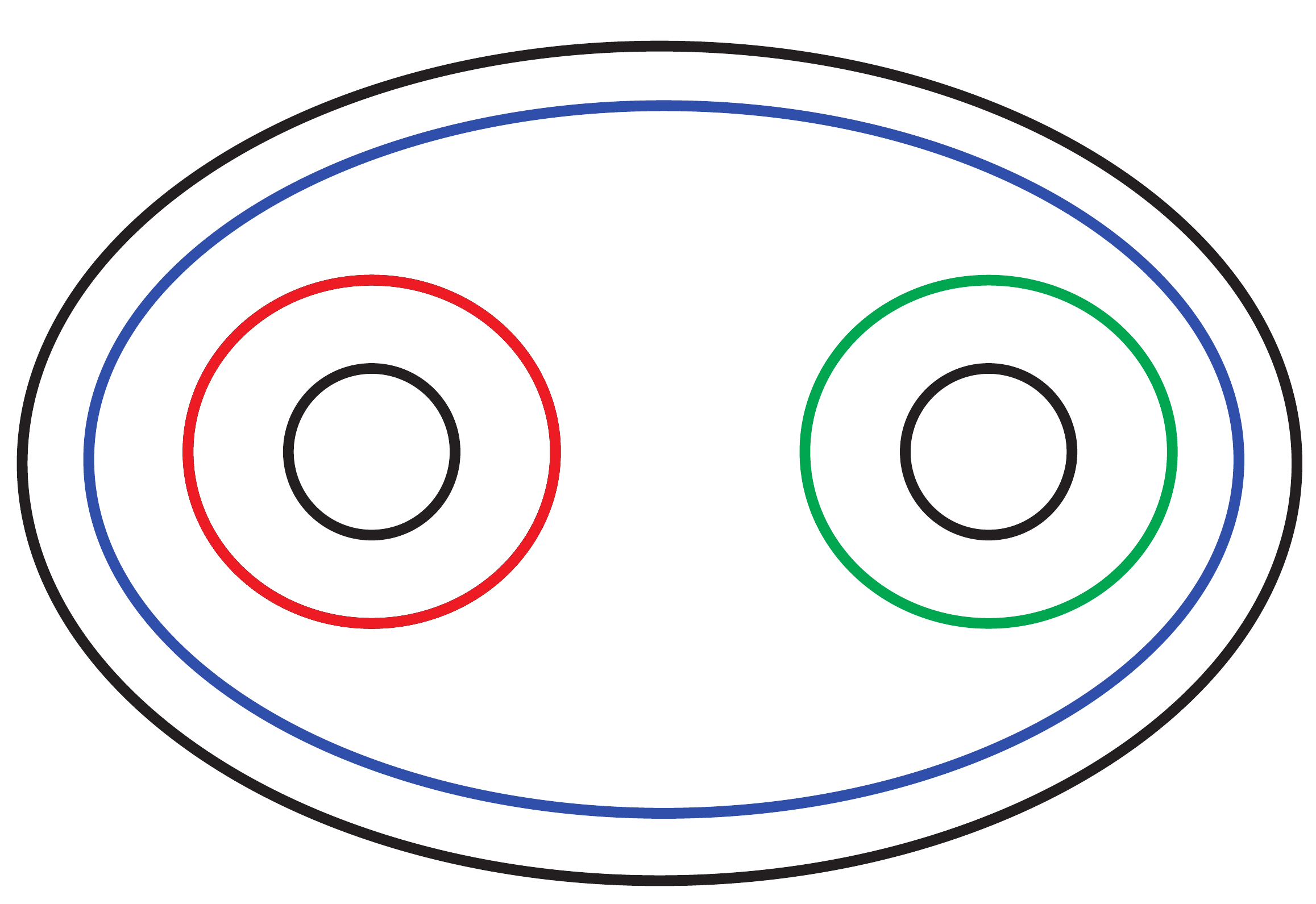}

\put(22,4.5){\footnotesize$a_3$}
\put(16,22.5){\footnotesize$a_1$}
\put(29,22.5){\footnotesize$a_2$}
\end{overpic}
\caption{The generators of $\mathcal S^{\mathit{alg}}\Sigma_{0,3}$.}
\label{pantscurves}
\end{subfigure} 
\centering
\begin{subfigure}{.45\textwidth}
\centering
\begin{overpic}[unit=1mm, scale = 0.12]{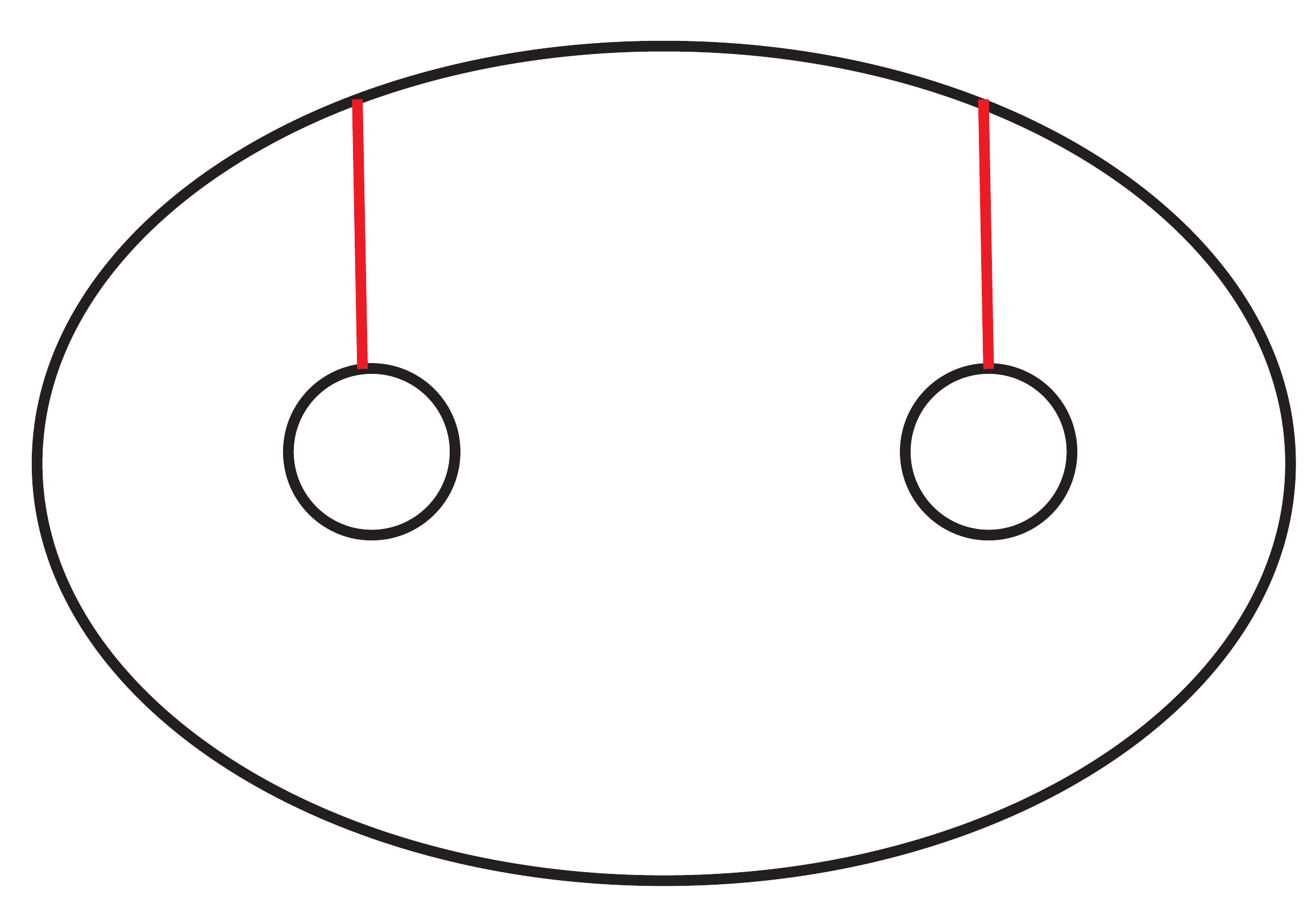}
\put(14,23){\footnotesize$\beta$}
\put(32.5,23){\footnotesize$\eta$}
\end{overpic}
\caption{Projection of $(S^1 \times S^2) \ \# \ (S^2 \times S^2)$ onto $\Sigma_{0,3}$.}
\label{s1s2connh1gamma}
\end{subfigure} 
\centering

\caption{{\color{white}.}}\label{pairofpants}
\end{figure}

Henceforth, we will use the notation $a_1$, $a_2$, and $a_3$ for the boundary parallel curves interchangeably with the boundary components they surround. We first compute the submodule $\mathcal J_1$ of $\mathcal S_{2,\infty}(H_2)$. The submodule $\mathcal J_2$ may be obtained symmetrically. 

\subsection{Handle sliding relations from relative skein modules}

We appeal to relative Kauffman bracket skein modules to compute the submodule $\mathcal J_1$ of handle sliding relations. Consider the two marked points $u$ and $v$, such that they lie on the simple closed curve $\beta$ in $\partial H_2$ and they divide the curve $\beta$ into two curves $\beta_1$ and $\beta_2$ (see Figure \ref{gamma12}). \\

\begin{figure}[ht]
    \centering
    \begin{overpic}[unit=1mm, scale = 0.17]{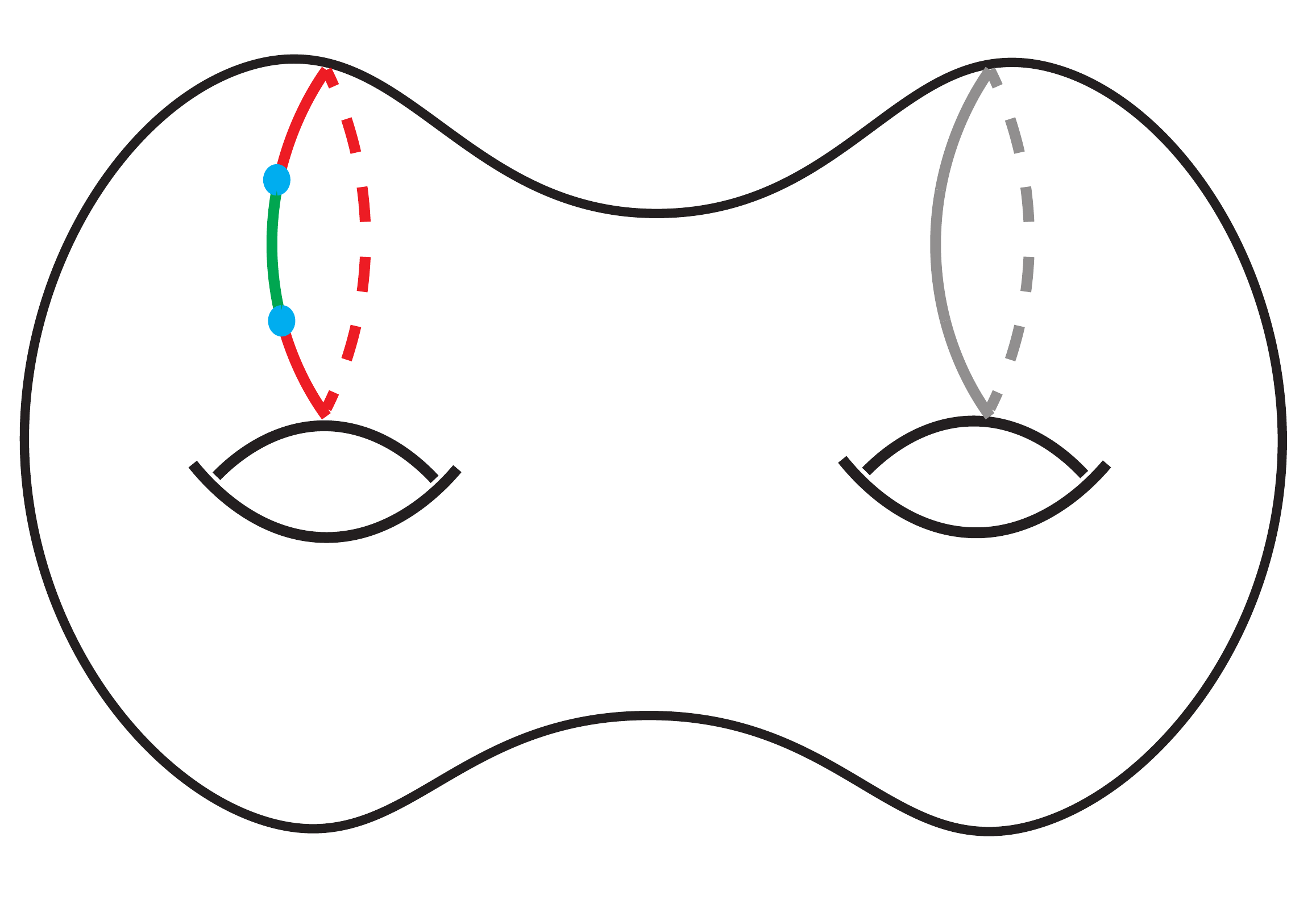}
\put(11.5,29){\tiny{$v$}}
\put(11.5,36){\tiny{$u$}}
\put(14.25,32.5){\tiny $\beta_2$}
\put(20,32.5){\tiny$\beta_1$}
\put(53,32.5){\tiny$\eta$}
\end{overpic}
\vspace{-5mm}
    \caption{Marked points $u$ and $v$ on the simple closed curve $\beta$ in $\partial H_2$ that divide it into curves $\beta_1$ and $\beta_2$.}
    \label{gamma12}
\end{figure}

Consider any relative curve $\alpha$ in $(H_2; u, v)$. Now, handle slidings in $(H_2)_{\beta}$ take place locally in the neighbourhood of the curve $\beta$. Consider fixed tangents at the points $u$ and $v$ and let the relative curve $\alpha$ approach these points along the tangents. For every relative curve $\alpha$, handle sliding in $(H_2)_{\beta}$ replaces the curve $\alpha \cup \beta_2$ with the curve $\alpha \cup \beta_1$. This gives the handle sliding relation, $\alpha \cup \beta_2 \equiv \alpha \cup \beta_1$. By introducing the $\mathbb Z[A^{\pm 1}]$-linear homomorphism $\omega: \mathcal S_{2,\infty}(H_2; u,v) \longrightarrow \mathcal S_{2,\infty}(H_2)$, defined by $\omega (\alpha) = \alpha \cup \beta_2 - \alpha \cup \beta_1$, we see that $\omega(\mathcal S_{2,\infty}(H_2; u,v)) = \mathcal J_1$. Hence, the image of any basis of $\mathcal S_{2,\infty}(H_2; u,v)$ generates $\mathcal J_1$. See Figure \ref{omegaalpha} for a visual explanation. We note that this method of describing $2$-handle sliding relations was pioneered by Bullock and Lo Faro in \cite{knotext}. See also \cite{blp}. \\ 

\begin{figure}[ht]
$\vcenter{\hbox{\begin{overpic}[scale=.13]{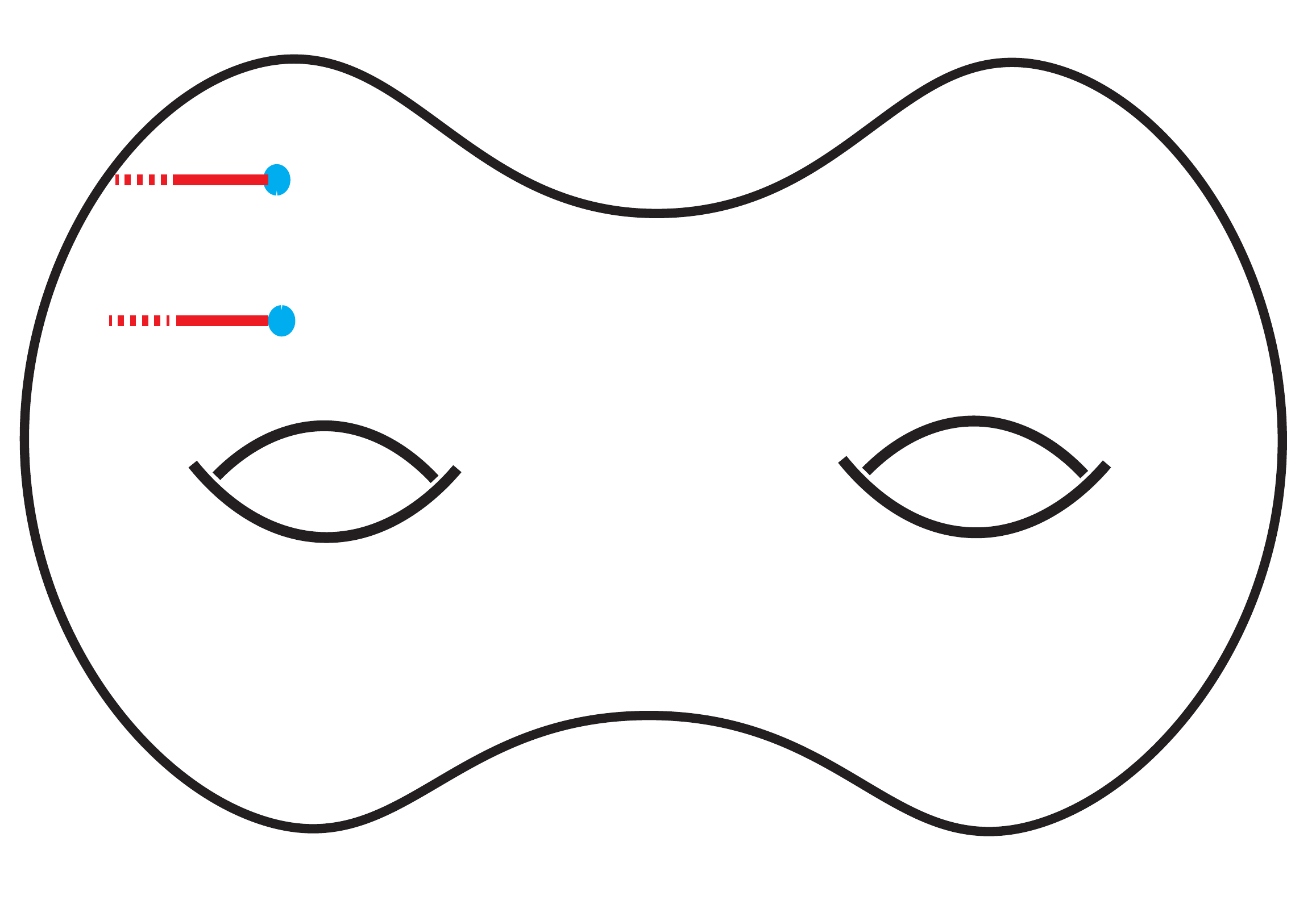}
\put(66, -5){$\alpha$}
\end{overpic}}}   \xrightarrow{\omega}
\vcenter{\hbox{\begin{overpic}[scale=.13]{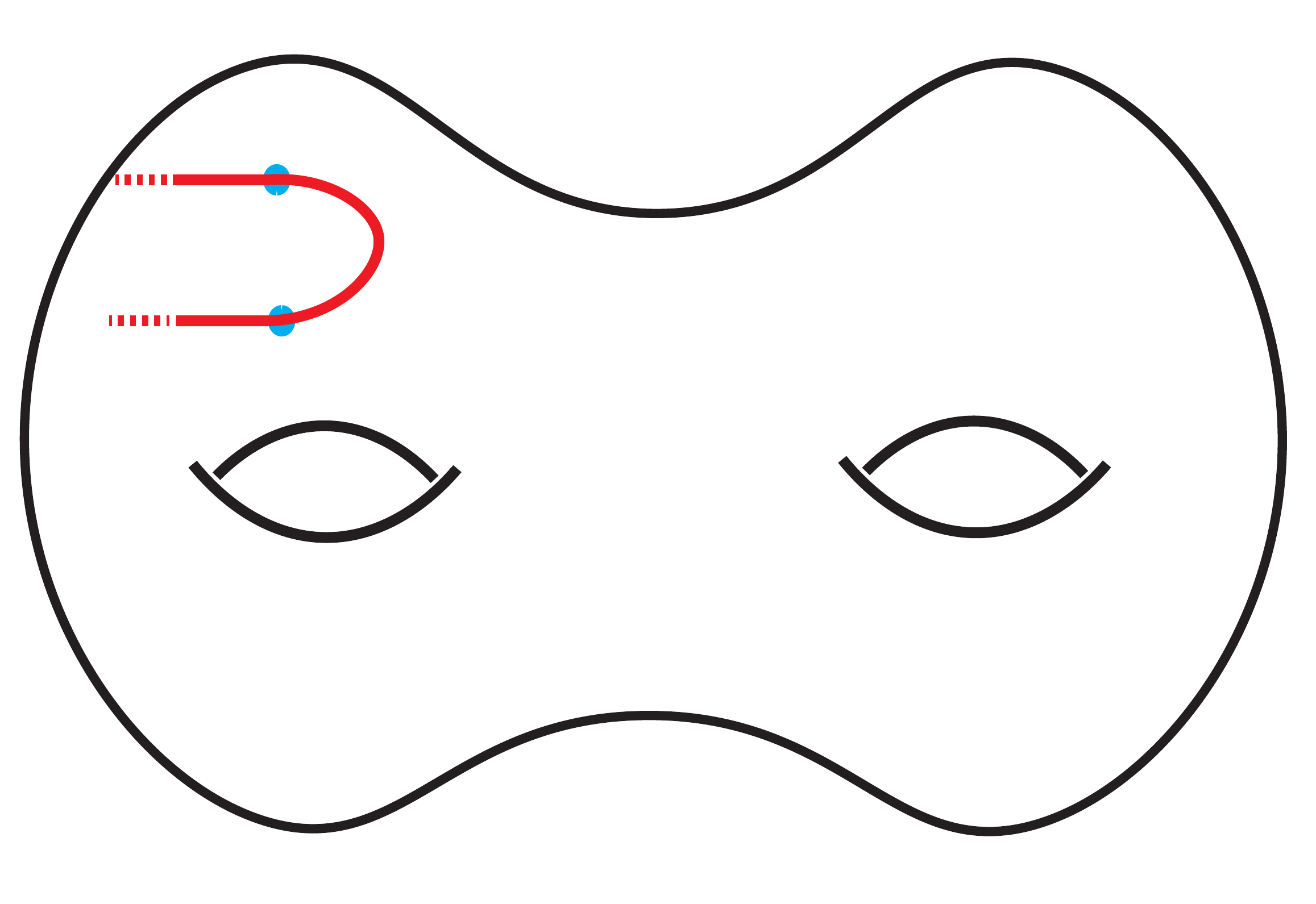}
\put(56, -5){$\alpha \cup \beta_2$}
\end{overpic}}} 
 -
\vcenter{\hbox{\begin{overpic}[scale=.13]{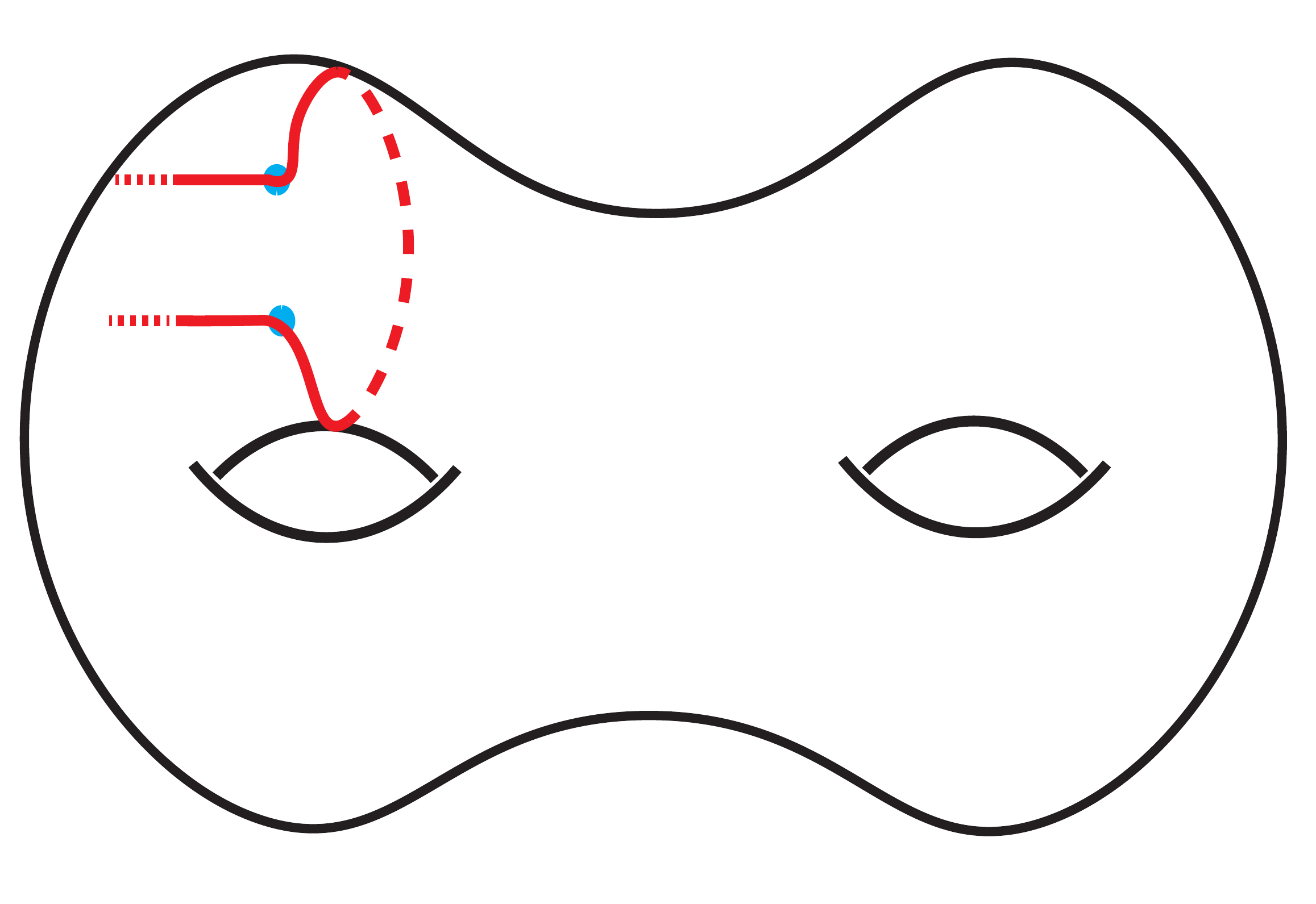}
\put(56, -5){$\alpha \cup \beta_1$}
\end{overpic}}}$
\vspace*{3mm}
\caption{Illustration of $\omega(\alpha)$.}
\label{omegaalpha}
\end{figure}

We emphasise that the submodule $\mathcal J_2$ of handle sliding relations that correspond to the $2$-handle glued to $\partial H_2$ along the curve $\eta$ may be obtained in a symmetric manner. We now discuss a basis of the relative skein module of $(H_2; u,v)$.

\subsection{Basis of the relative skein module of $(H_2; u,v)$}  Consider $\Sigma_{0,3}$ with marked points $u$ and $v$ on its boundary as illustrated in Figure \ref{c00}. 

\begin{figure}[ht]
    \centering
\begin{overpic}[unit=1mm, scale = 0.1]{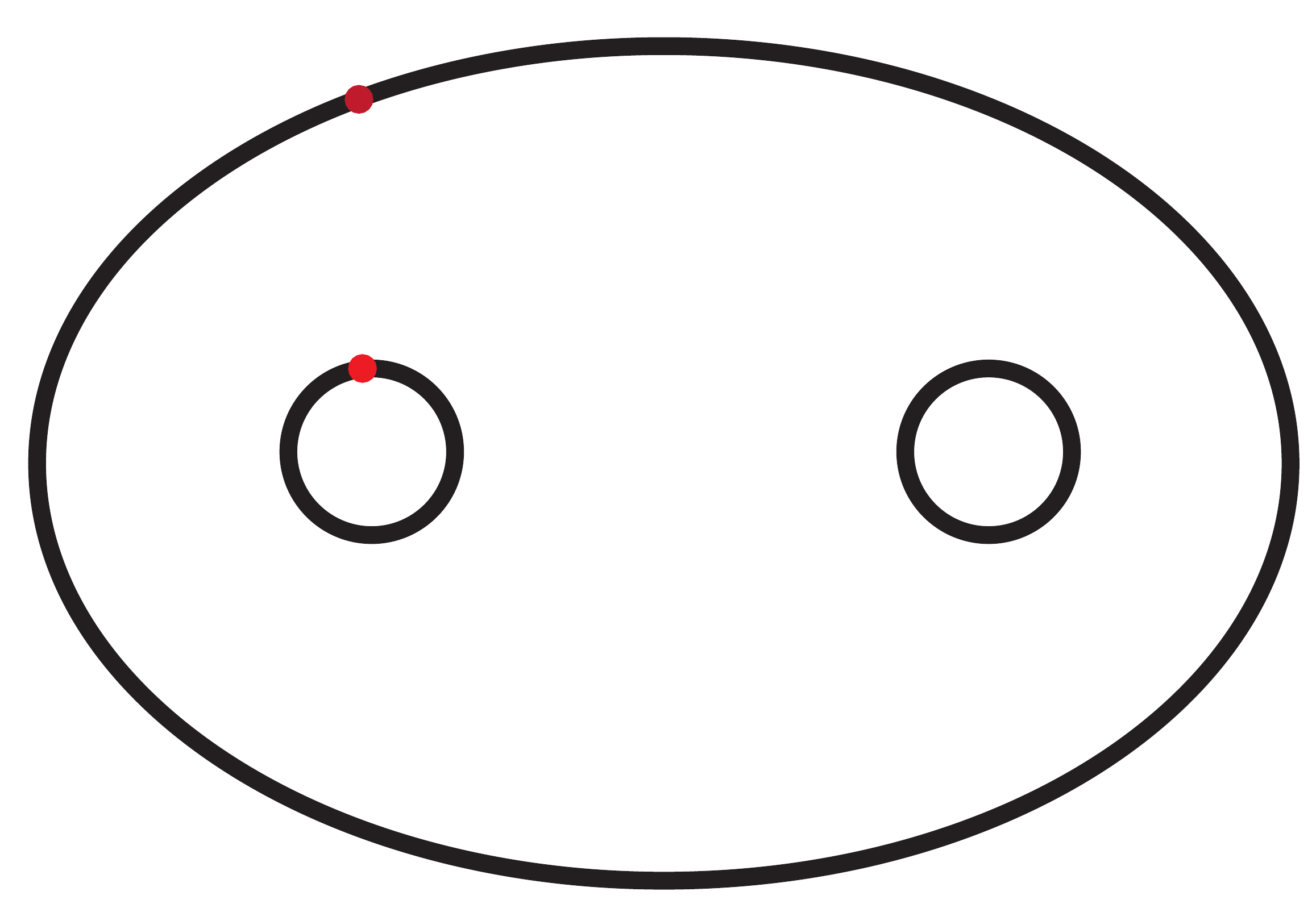}

\put(10,13.5){\footnotesize$v$}
\put(10,25){\footnotesize$u$}
\end{overpic}
    \caption{Marked points $u$ and $v$ on $\partial \Sigma_{0,3}$, where $u$ lies on the boundary component $a_3$ and $v$ lies on the boundary component $a_1$.}
    \label{c00}
\end{figure}

\begin{definition}

Let $c_{0,0}$ be the relative link represented by the line segment connecting $u$ and $v$, as illustrated in Figure \ref{rkbsmbasisf03uv}. Then we define the relative curves $c_{k,m} =
\tau^k_{a_1} \tau^m_{a_3}(c_{0,0})$, $k, m \in \mathbb Z$, where $\tau_c$ is the Dehn twist along $c$ for $c$ a simple closed curve.  The curves $c_{k,m}$ for small $k$ and $m$ are also illustrated in Figure \ref{rkbsmbasisf03uv}.
    
\end{definition}

\begin{figure}[ht]
\centering
$\hdots \vcenter{\hbox{\begin{overpic}[scale=.075]{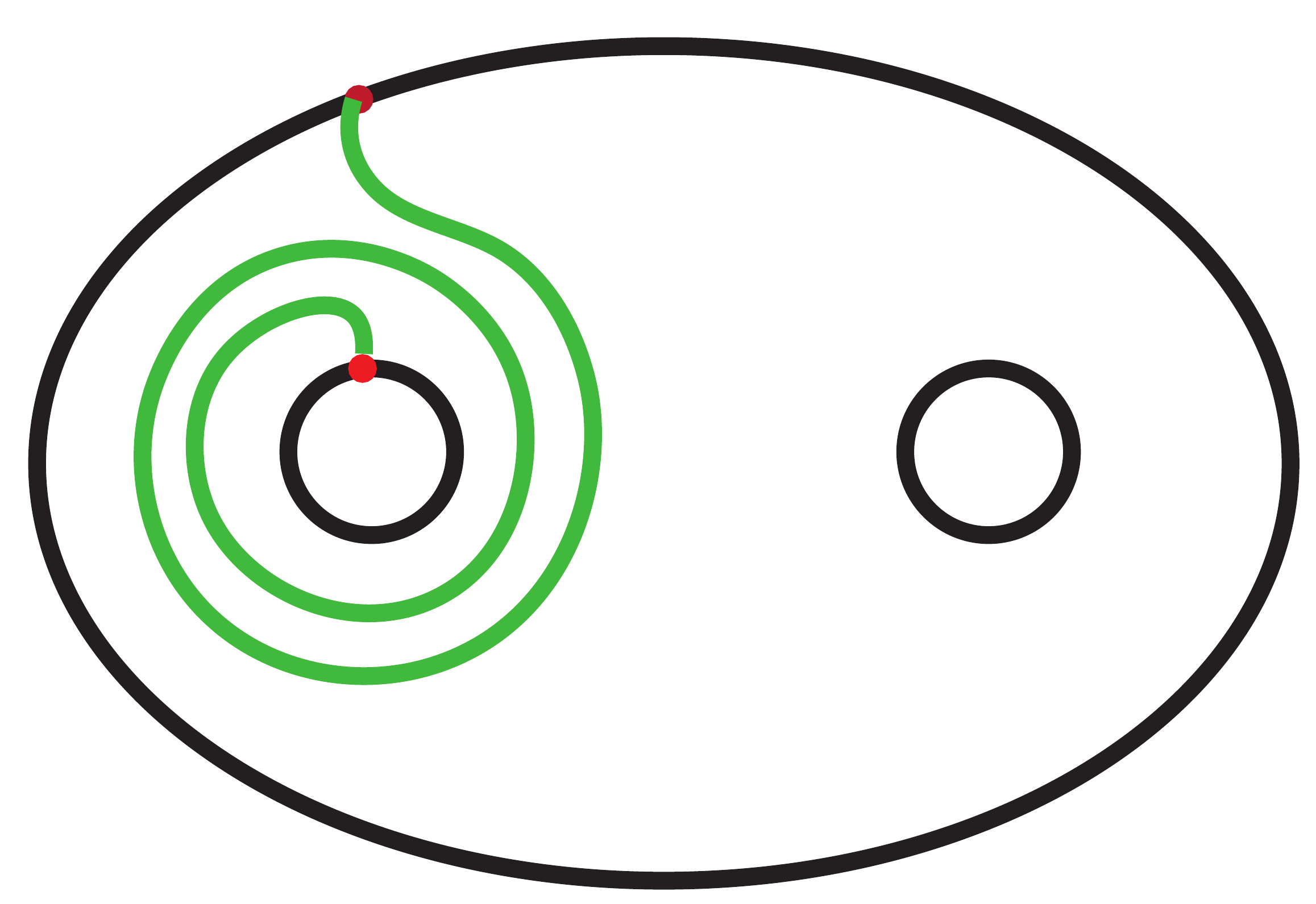}
\put(32, -7){$c_{-2,0}$}
\end{overpic}}}  \hspace{2mm}
\vcenter{\hbox{\begin{overpic}[scale=.075]{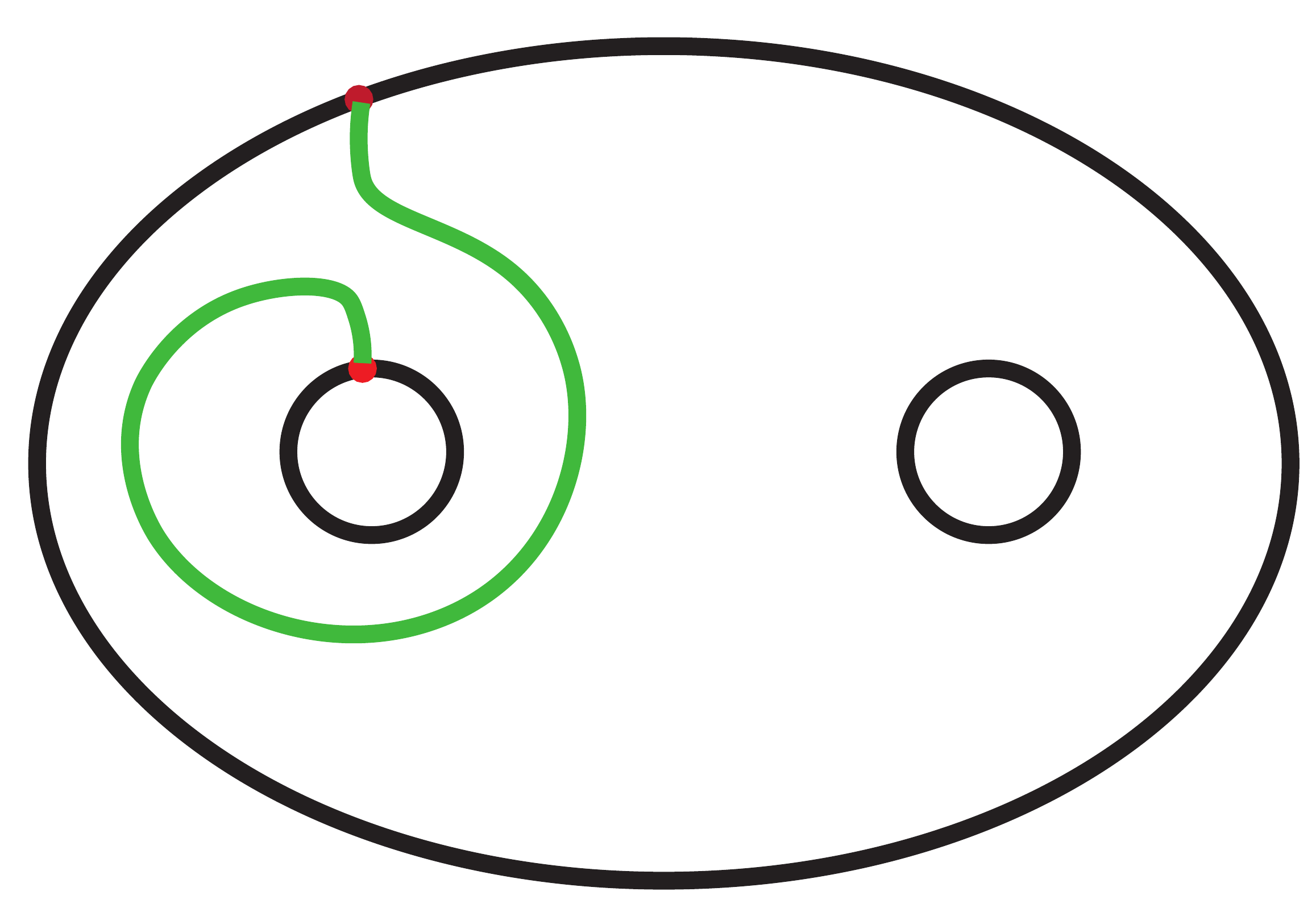}
\put(34, -7){$c_{-1,0}$}
\end{overpic}}} \hspace{2mm}
\vcenter{\hbox{\begin{overpic}[scale=.075]{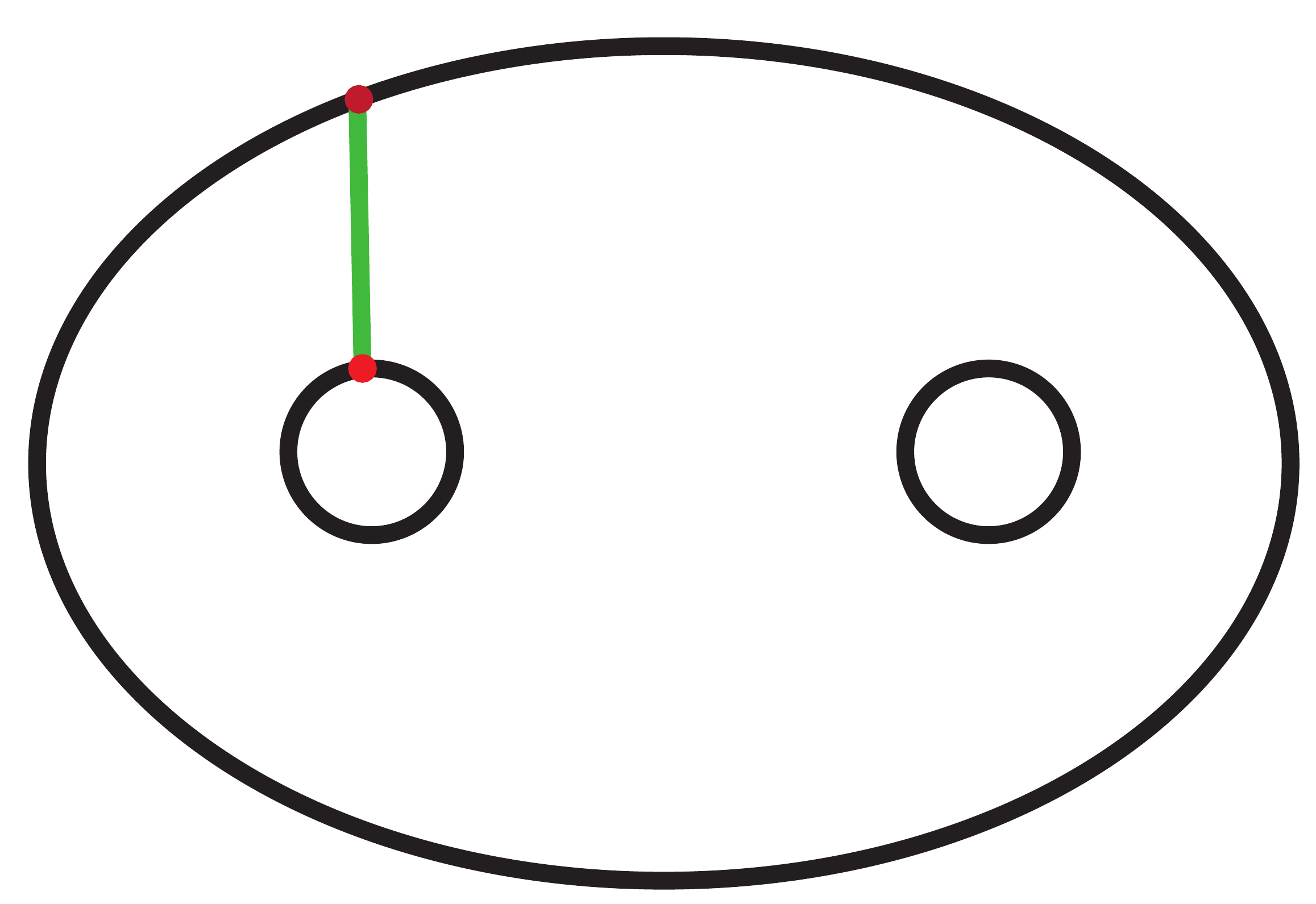}
\put(37, -7){$c_{0,0}$}
\put(20.5, 53.5){\footnotesize{$u$}}
\put(21, 28){\footnotesize{$v$}}
\end{overpic}}} \hspace{2mm}
\vcenter{\hbox{\begin{overpic}[scale=.075]{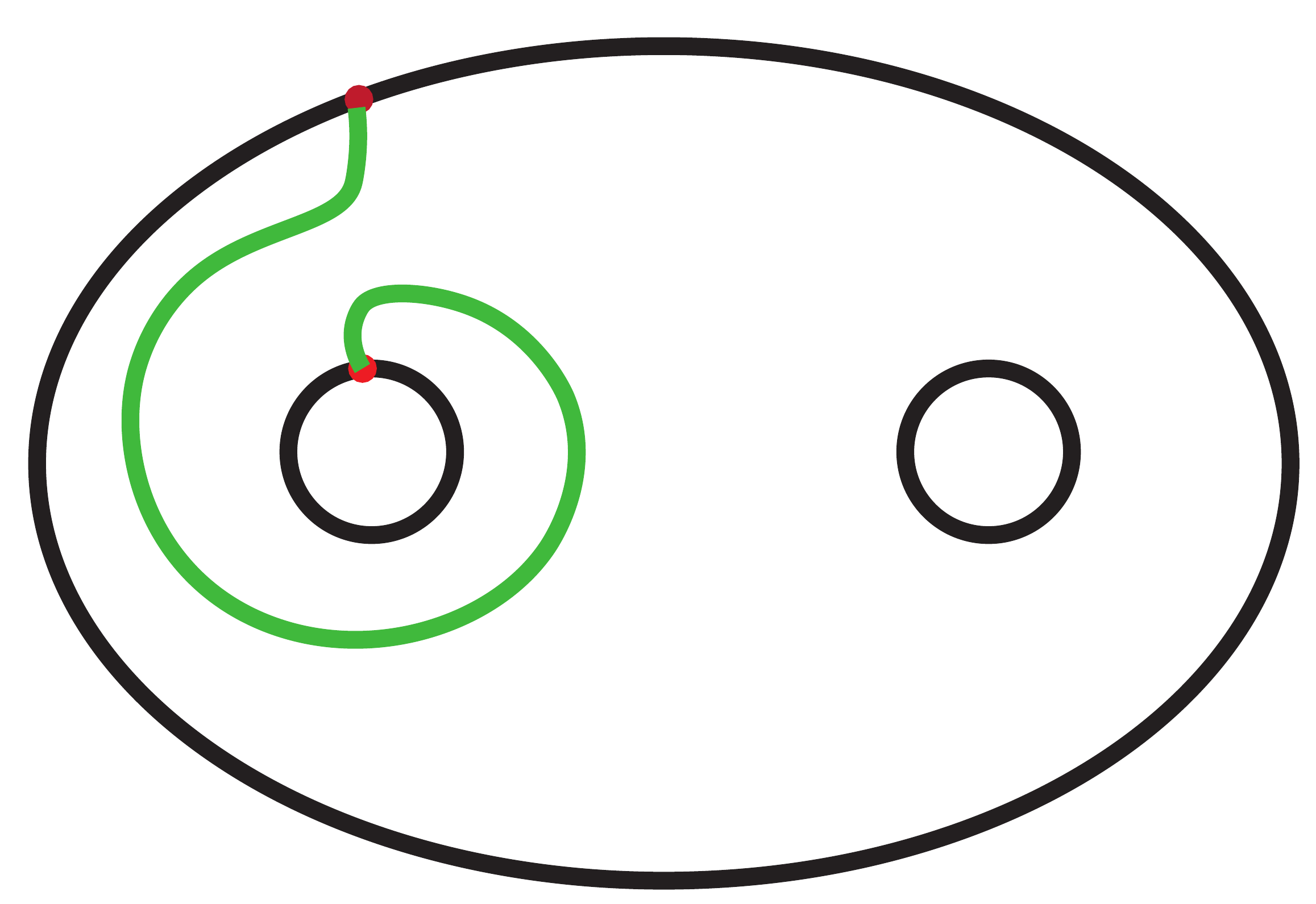}
\put(37, -7){$c_{1,0}$}
\end{overpic}}}  \hspace{2mm}
\vcenter{\hbox{\begin{overpic}[scale=.075]{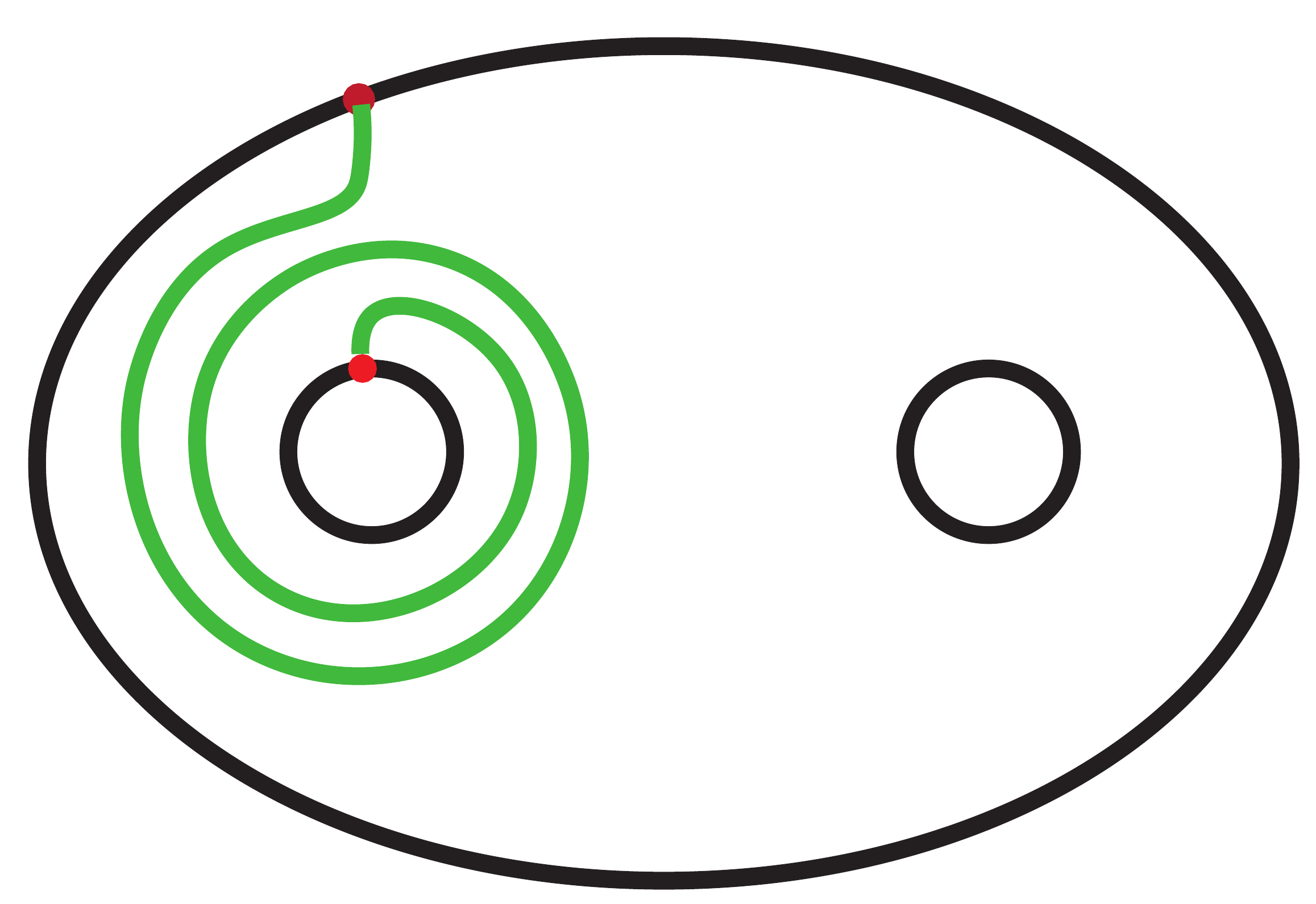}
\put(38, -7){$c_{2,0}$}
\end{overpic}}}  \hdots $ \\ \vspace*{5mm}
$\hdots \vcenter{\hbox{\begin{overpic}[scale=.075]{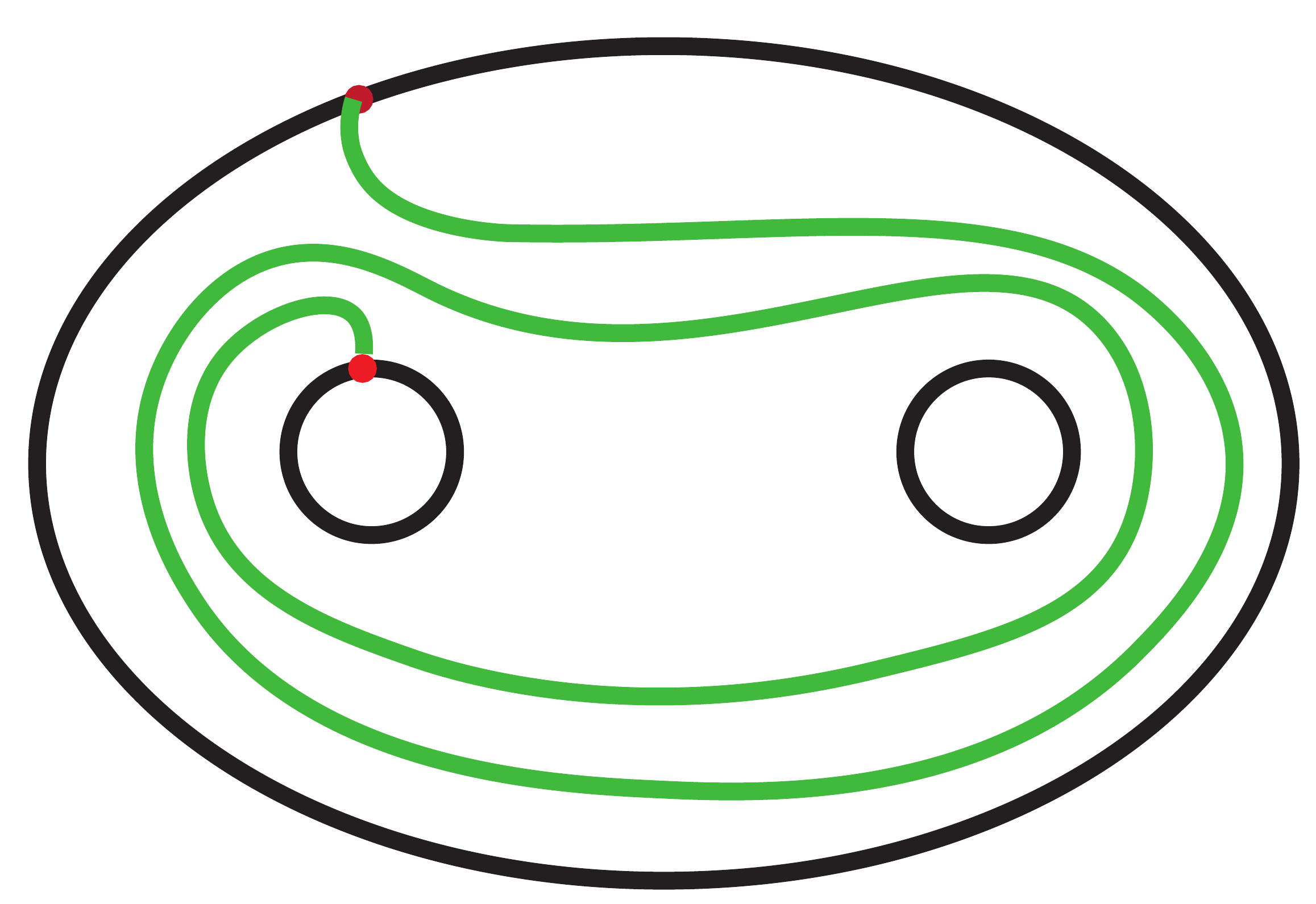}
\put(32, -7){$c_{0,-2}$}
\end{overpic}}}  \hspace{2mm}
\vcenter{\hbox{\begin{overpic}[scale=.075]{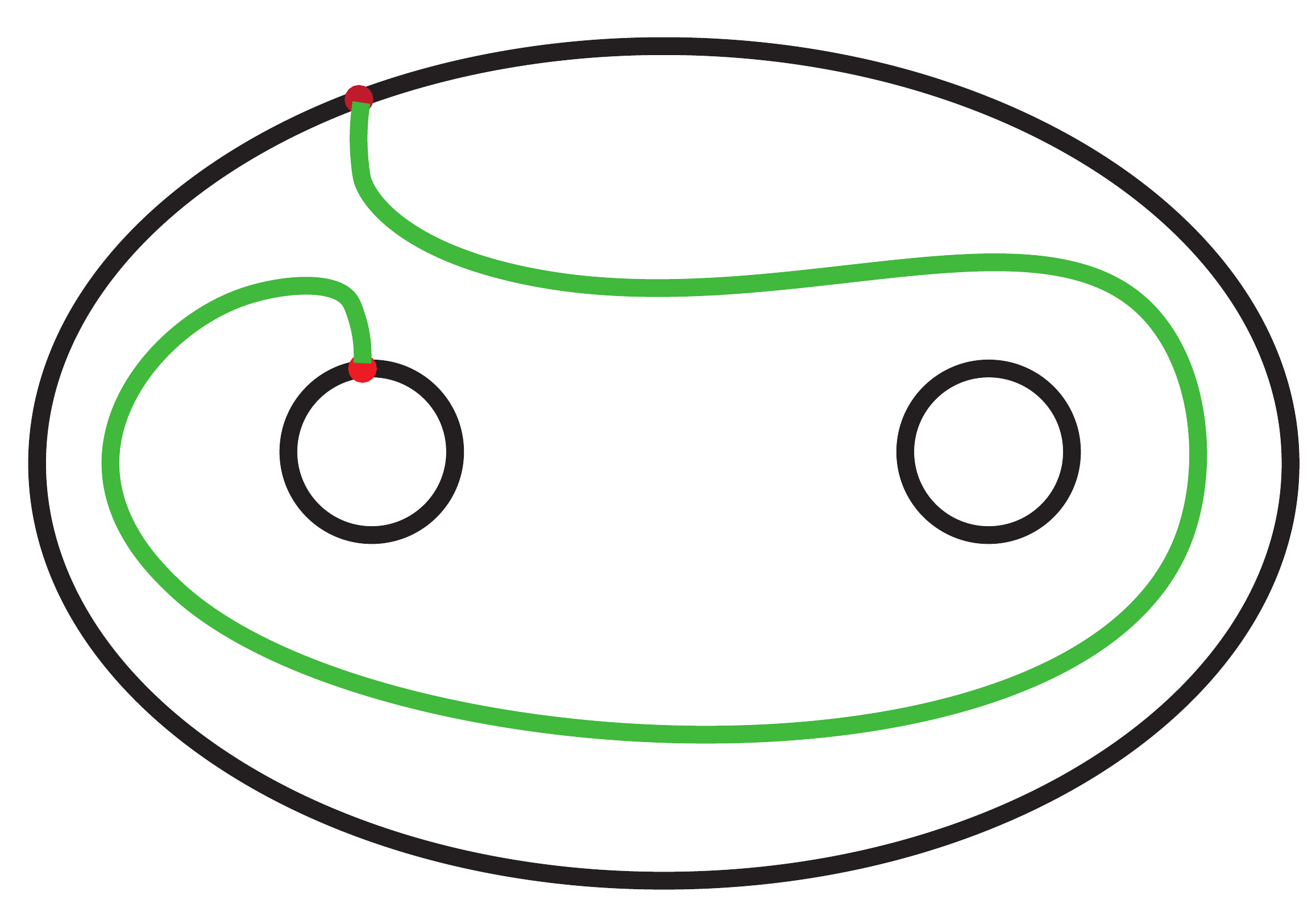}
\put(34, -7){$c_{0,-1}$}
\end{overpic}}} \hspace{2mm}
\vcenter{\hbox{\begin{overpic}[scale=.075]{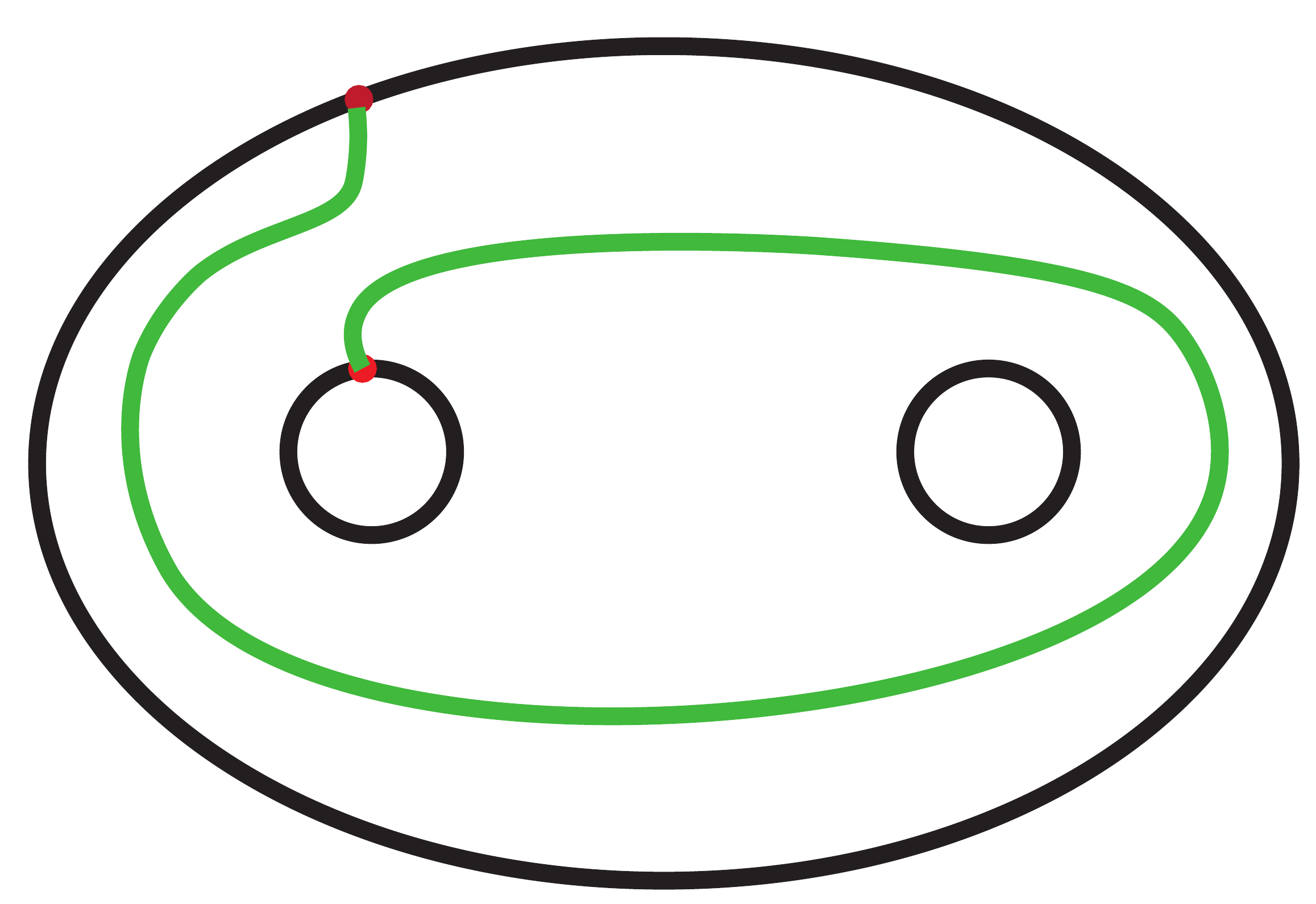}
\put(37, -7){$c_{0,1}$}
\end{overpic}}} \hspace{2mm}
\vcenter{\hbox{\begin{overpic}[scale=.075]{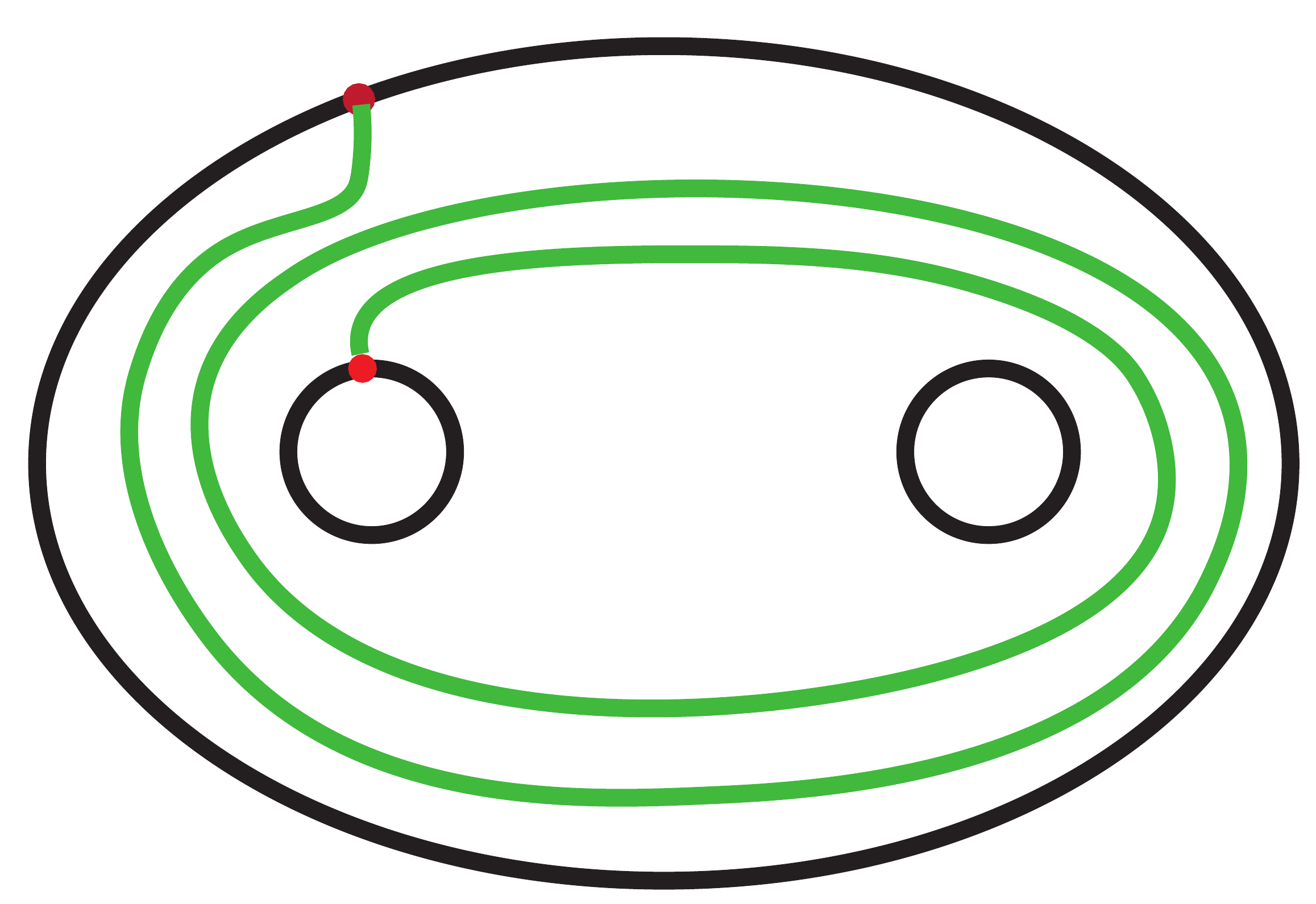}
\put(37, -7){$c_{0,2}$} 
\end{overpic}}} \hdots $ \\
\vspace*{5mm}
$\hdots \vcenter{\hbox{\begin{overpic}[scale=.075]{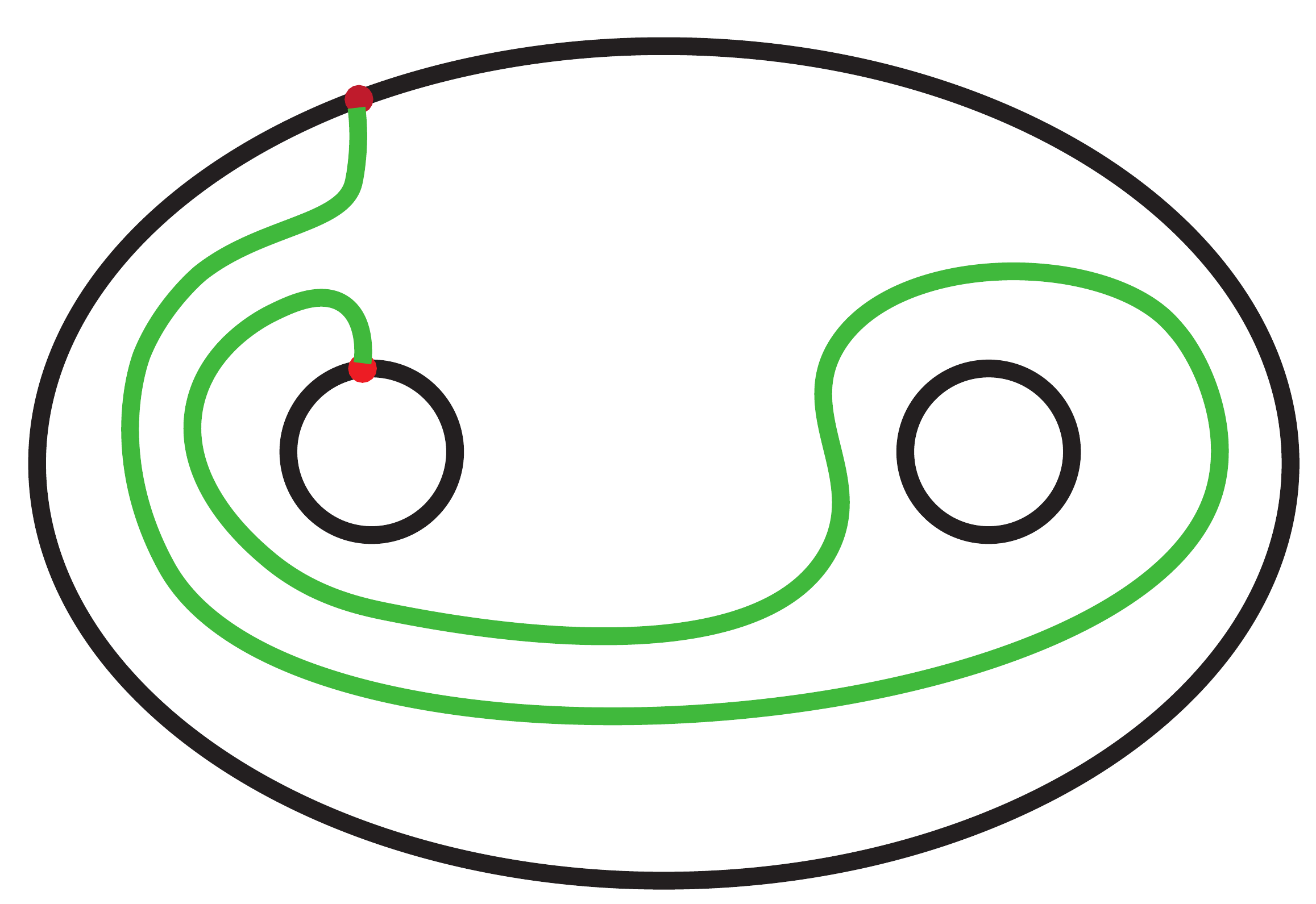}
\put(32, -7){$c_{-1,1}$}
\end{overpic}}}  \hspace{2mm}
\vcenter{\hbox{\begin{overpic}[scale=.075]{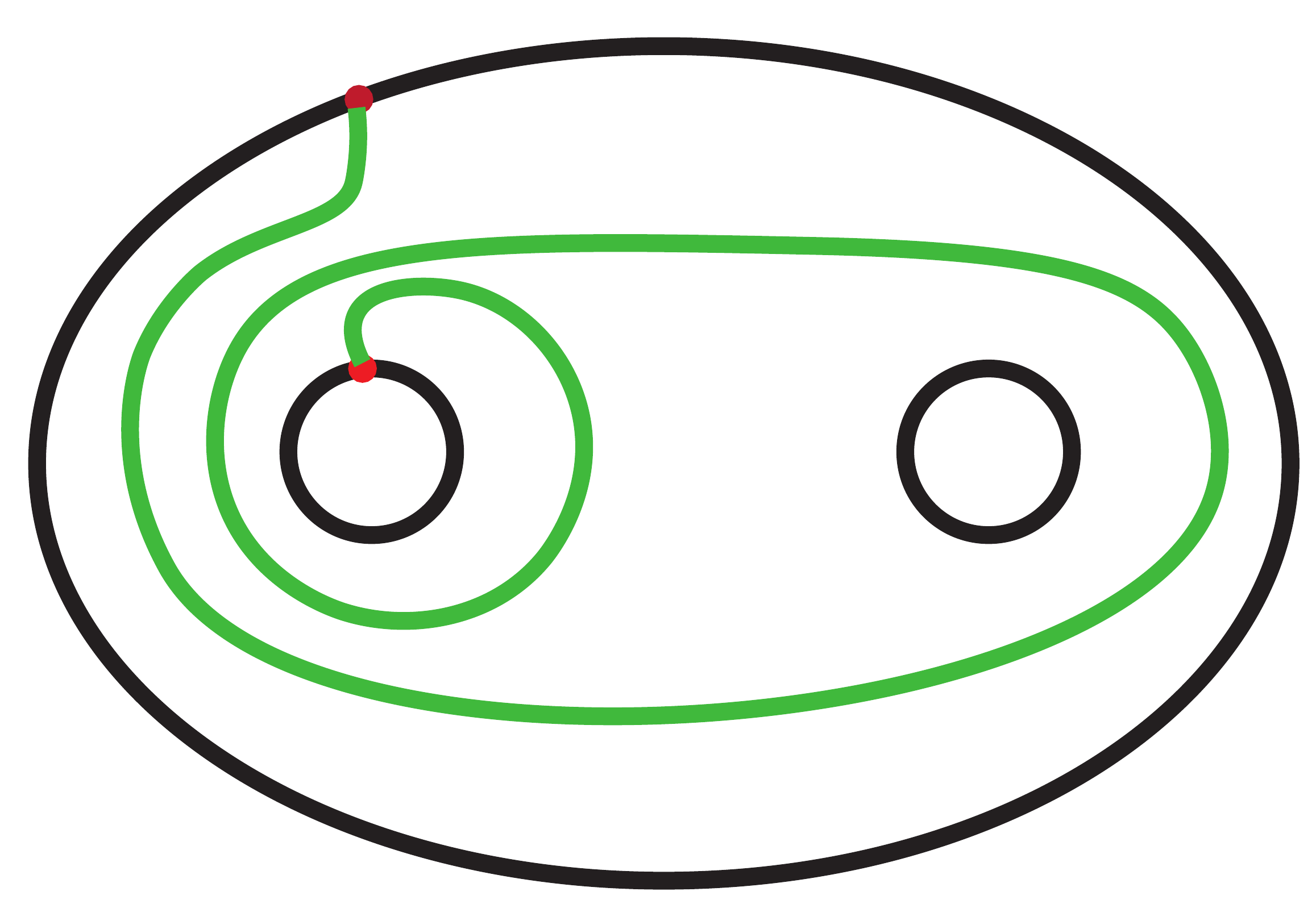}
\put(34, -7){$c_{1,1}$}
\end{overpic}}} \hspace{2mm}
\vcenter{\hbox{\begin{overpic}[scale=.075]{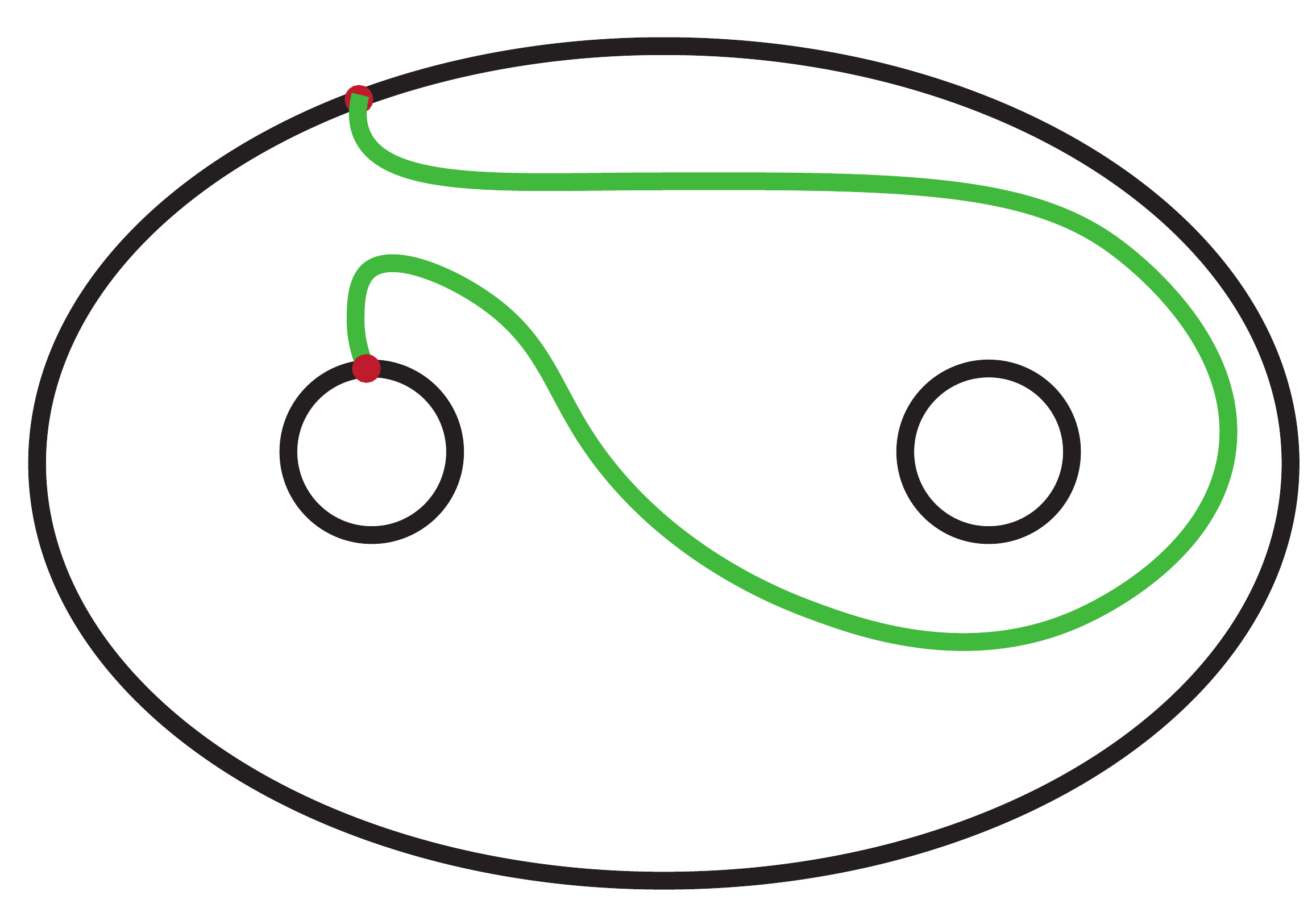}
\put(37, -7){$c_{1,-1}$}
\end{overpic}}} \hdots $
\vspace*{5mm}
\caption{Relative curves in $(\Sigma_{0,3} \times I; u,v)$.}
\label{rkbsmbasisf03uv}
\end{figure}

\begin{proposition}\label{relativecurvesf03}

Any relative curve connecting the points $u$ and $v$ is of the form $c_{k,m}$, $k, m \in \mathbb Z$.

\end{proposition}

\begin{proof}

The mapping class group $Mod^+(\Sigma_{0,3}) \cong PB_2 \times \mathbb Z \times \mathbb Z$, where $PB_2$ is the pure braid group on two strands, which is isomorphic to $\mathbb Z$. Hence, $Mod^+(\Sigma_{0,3}) \cong \mathbb Z \times \mathbb Z \times \mathbb Z$ and is generated by the $\tau_{a_i}$’s. Moreover, the set of relative arcs forms a single orbit under $Mod^+(\Sigma_{0,3})$-action and $\tau_{a_2}$ acts trivially on $c_{0,0}$.
    
\end{proof}

Note that the boundary curve $a_2$ multiplicatively commutes with $c_{k,m}$, for all $k$ and $m$, while the boundary curves $a_1$ and $a_3$ do not commute with any $c_{k,m}$. Thus, from Theorem \ref{rkbsmsib}, Remark \ref{relativetimesframed}, and Proposition \ref{relativecurvesf03} we get the following result. 


    

\begin{corollary}\label{rkbsmf03basis}

The elements $c_{k,m}S_q(a_2)$, $i \in \mathbb Z^+ \cup \{0\}$ form a basis for $\mathcal S_{2,\infty}(H_2; u,v)$. 
    
\end{corollary}


    

\subsection{Generators of the submodule $\mathcal J_1$} 

From Corollary~\ref{rkbsmf03basis} it follows that $\mathcal J_1$ is generated by $\omega(c_{m,n}a_{2}^k) =\omega(c_{m,n})a_{2}^k$. To find the generators of $\mathcal J_1$, let us consider $C(m,n) := \omega(c_{m,n})$ for all $m,n \in \mathbb{Z}$. We have the following three cases:
\begin{enumerate}
    \item $m,n \geq 0$,
    \item $m \geq 1$, $n \leq -1$,
    \item remaining cases.
\end{enumerate}


    
\noindent
{\bf Case I:} $C(m,n)$ for $m,n \geq 0$. \\
 For $m,n \geq 0$, $C(m,n) = \omega(c_{m,n}) = c_{m,n}\cup \beta_{2} - c_{m,n}\cup \beta_{1}.$ As an example, $C(2,1)$ is illustrated in Figure \ref{fig:C(2,1)}. 
\begin{figure}[ht]
    \centering
    \begin{eqnarray*}
    C(2,1) = \omega(c_{2,1}) & = &
\vcenter{\hbox{\begin{overpic}[scale=.1]{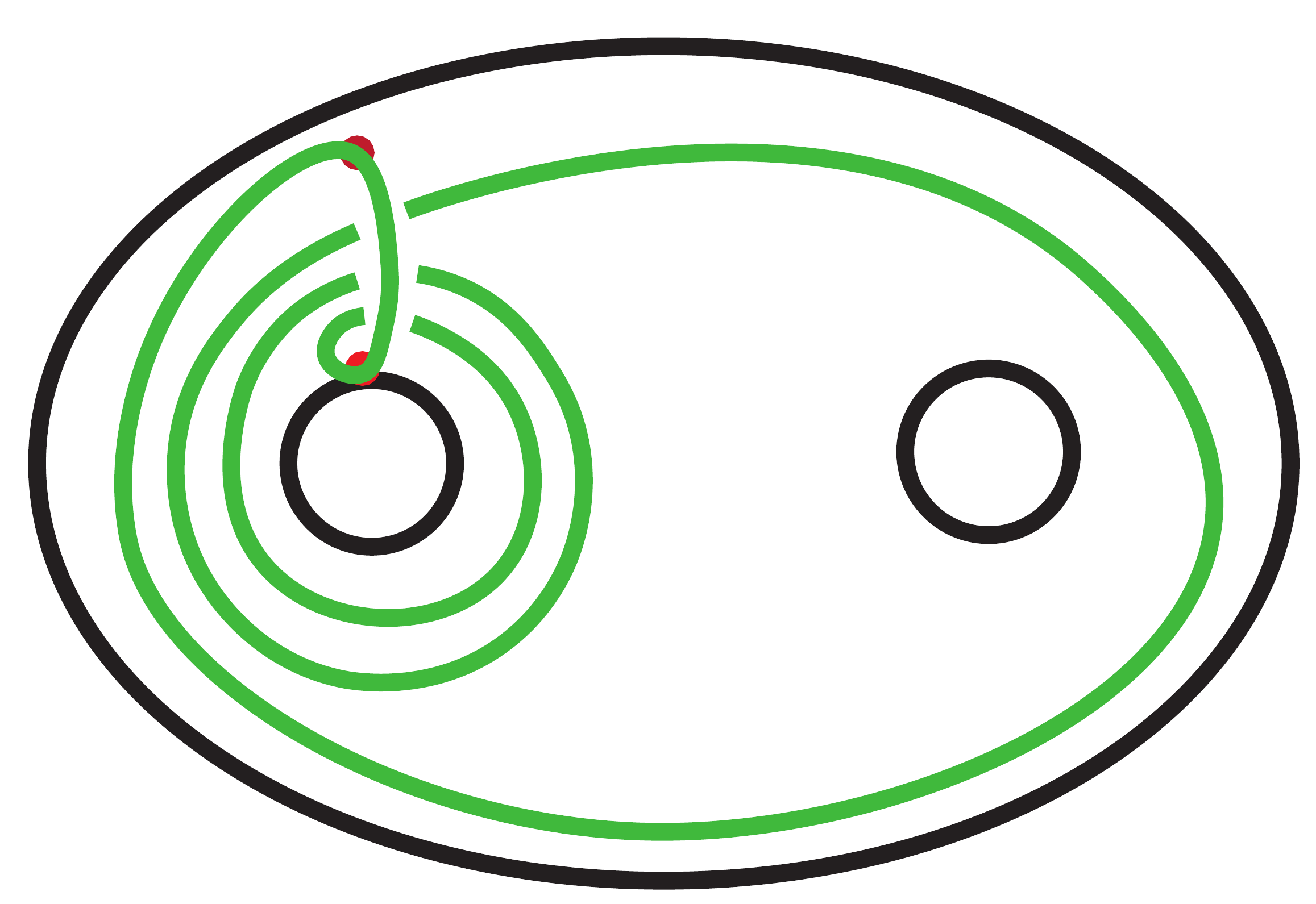}
\put(130, 35){$-$} 
\end{overpic}}} - 
\vcenter{\hbox{\begin{overpic}[scale=.1]{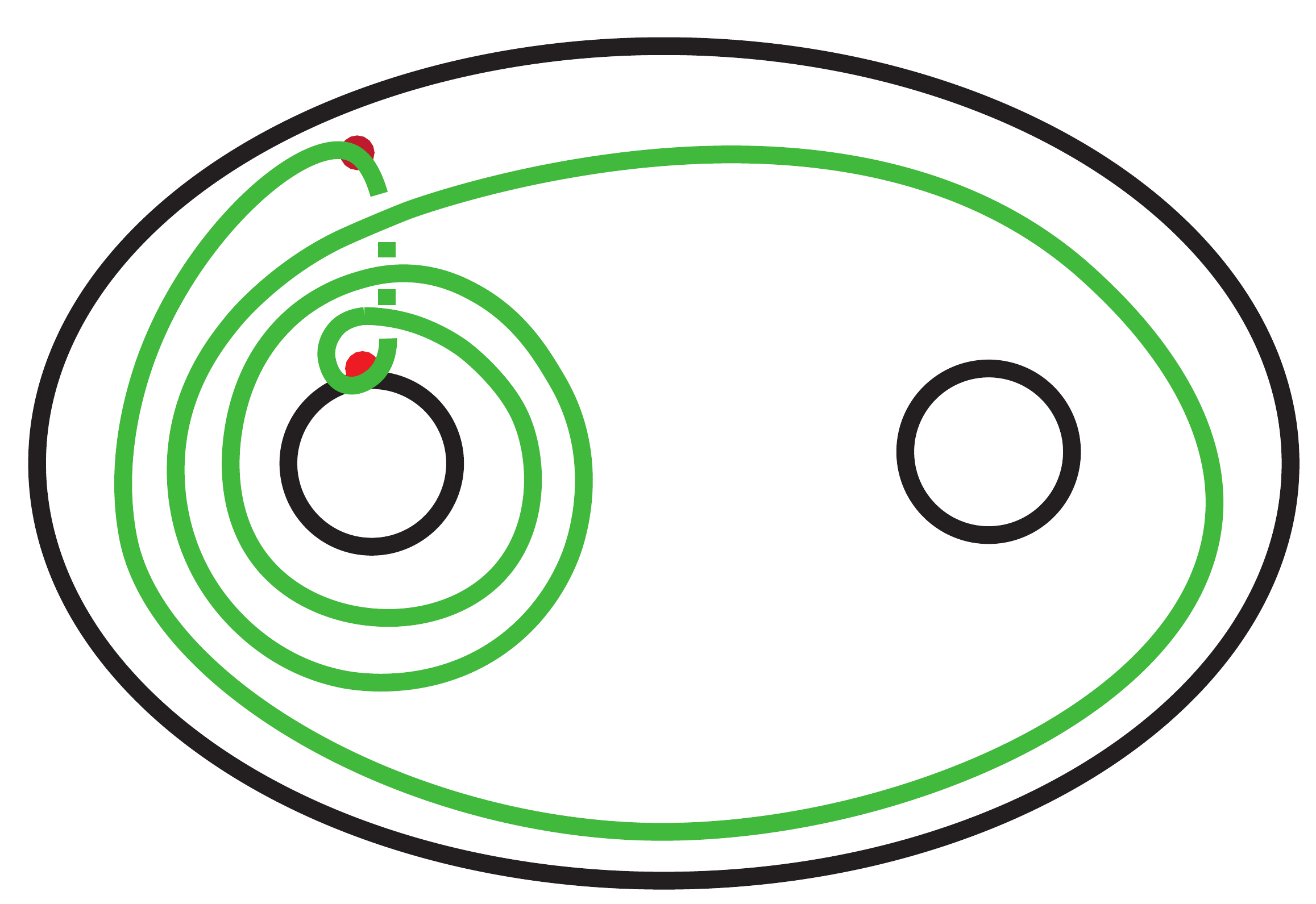}
\end{overpic}}} \\ 
& = & -A^3 \vcenter{\hbox{\begin{overpic}[scale=.1]{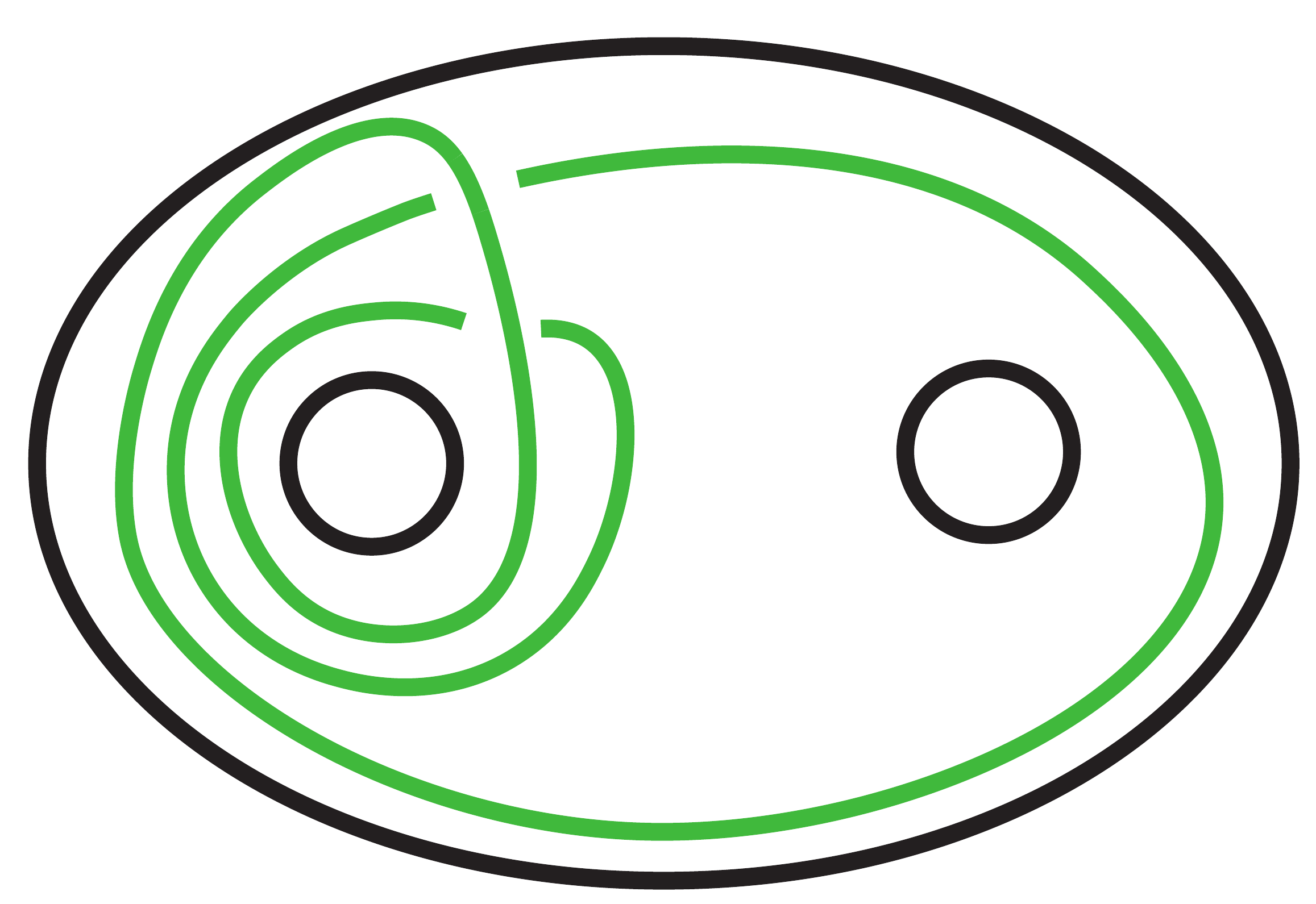}
\put(39, -10){$P(2,1)$}
\end{overpic}}} + A^{-3}
\vcenter{\hbox{\begin{overpic}[scale=.1]{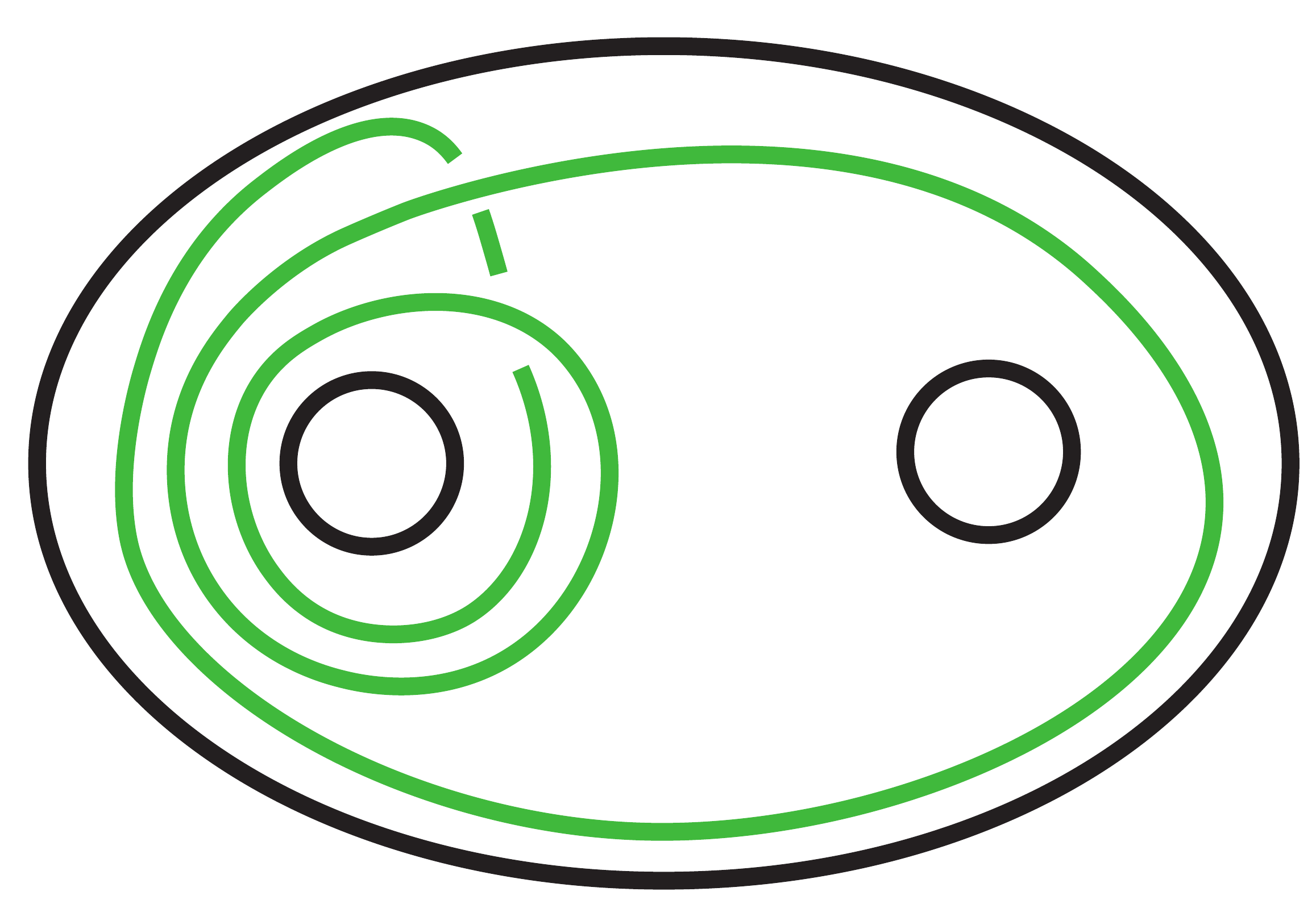}
\put(41, -10){$N(2,1)$} 
\end{overpic}}}
\vspace{7mm}
\end{eqnarray*}
    \caption{Illustration of $C(2,1)$.}
    \label{fig:C(2,1)}
\end{figure}
Since $c_{2,1}\cup \beta_{1}$ has a negative kink and $c_{2,1}\cup \beta_{2}$ has a positive kink, we obtain two diagrams without kinks with coefficients $-A^{-3}$ and $-A^{3}$, respectively. Let us denote the resultant diagrams without kinks by $N(2,1)$ and $P(2,1)$, respectively. In general, when $m,n \geq 0$ and $(m,n)\neq (0,0)$, we define $N(m,n)$ and $P(m,n)$ from $c_{m,n}\cup \beta_{1}$ and $c_{m,n}\cup \beta_{2}$, respectively, by removing kinks as described in Figure \ref{fig:C(m,n)}.
When $(m,n)=(0,0)$, by definition, $C(0,0)$=0, so we define $P(0,0)=-A^{-3}(-A^{2}-A^{-2})$, $N(0,0)=-A^{3}(-A^{2}-A^{-2})$. Now we have
$$C(m,n)= -A^{3}P(m,n)+A^{-3}N(m,n), m,n \geq 0.$$
\begin{figure}[ht]
    \centering
    \begin{eqnarray*}
    C(m,n) & = &
\omega\left(\vcenter{\hbox{\begin{overpic}[scale=.1]{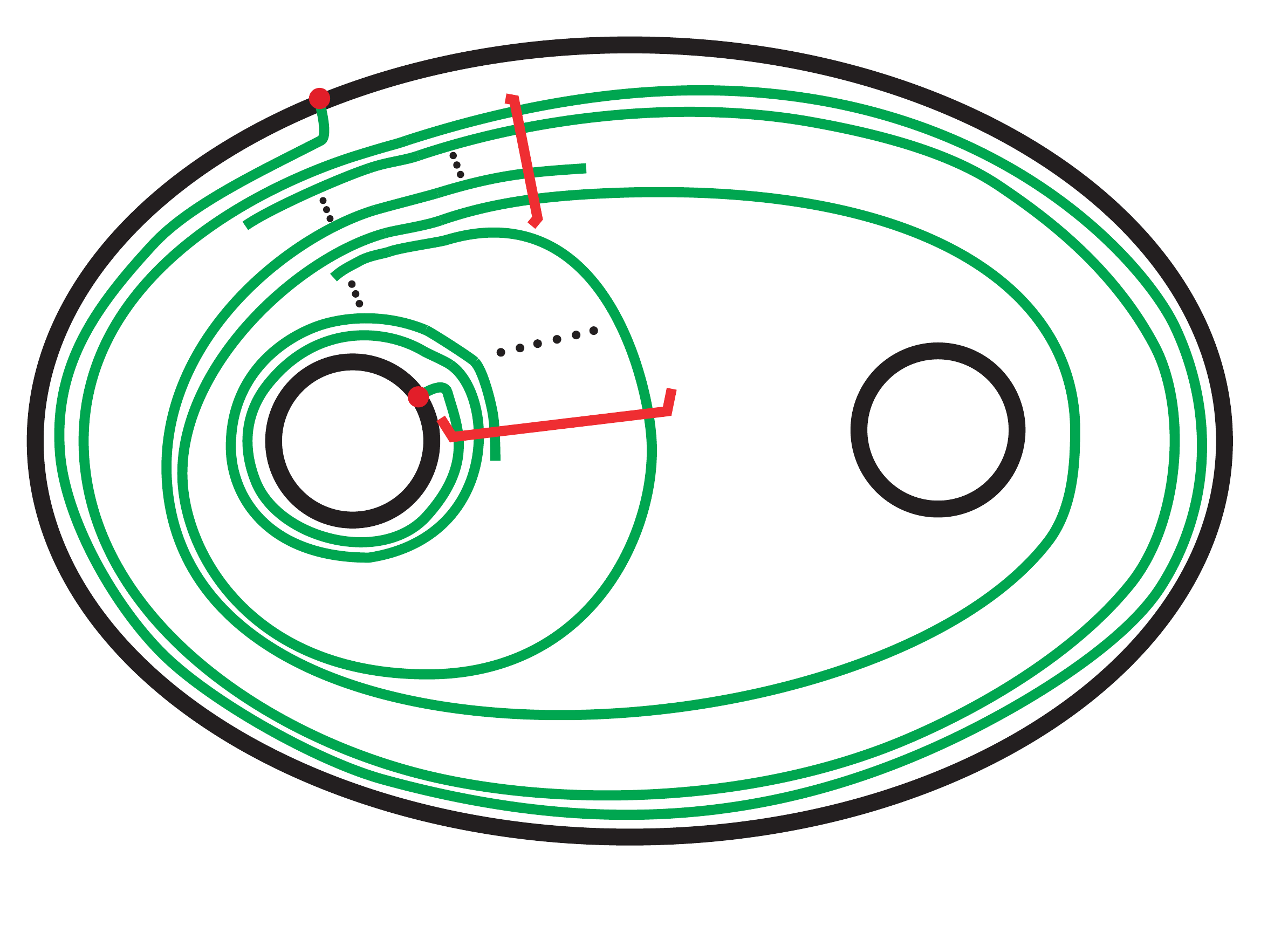} 
\put(48,42){\tiny{${m}$}}
\put(50,70){\tiny{${n}$}}
\end{overpic}}}\right) \\ 
& = & -A^3 \vcenter{\hbox{\begin{overpic}[scale=.1]{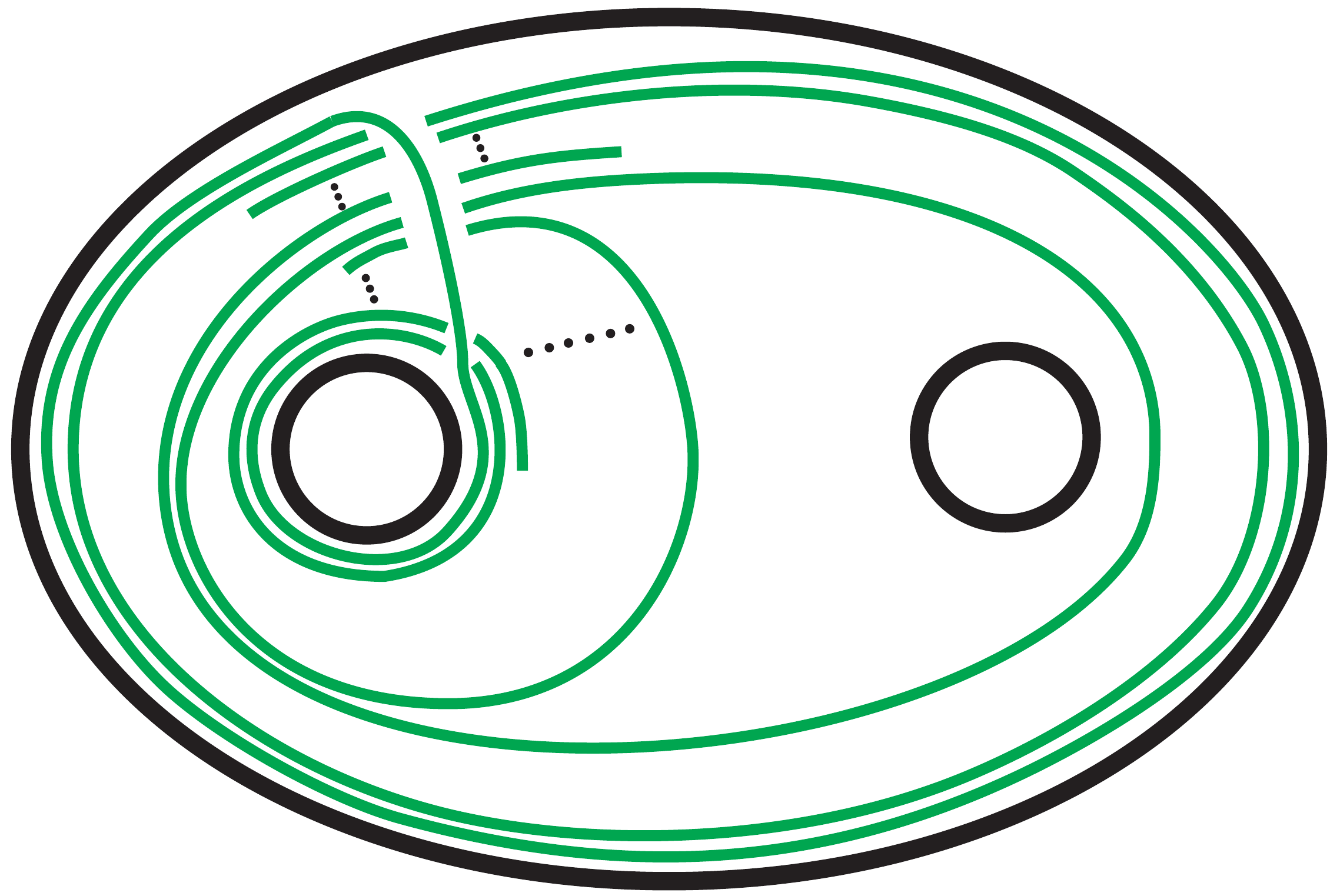}
\put(40,-10){$P(m,n)$}
\end{overpic}}} + A^{-3}
\vcenter{\hbox{\begin{overpic}[scale=.1]{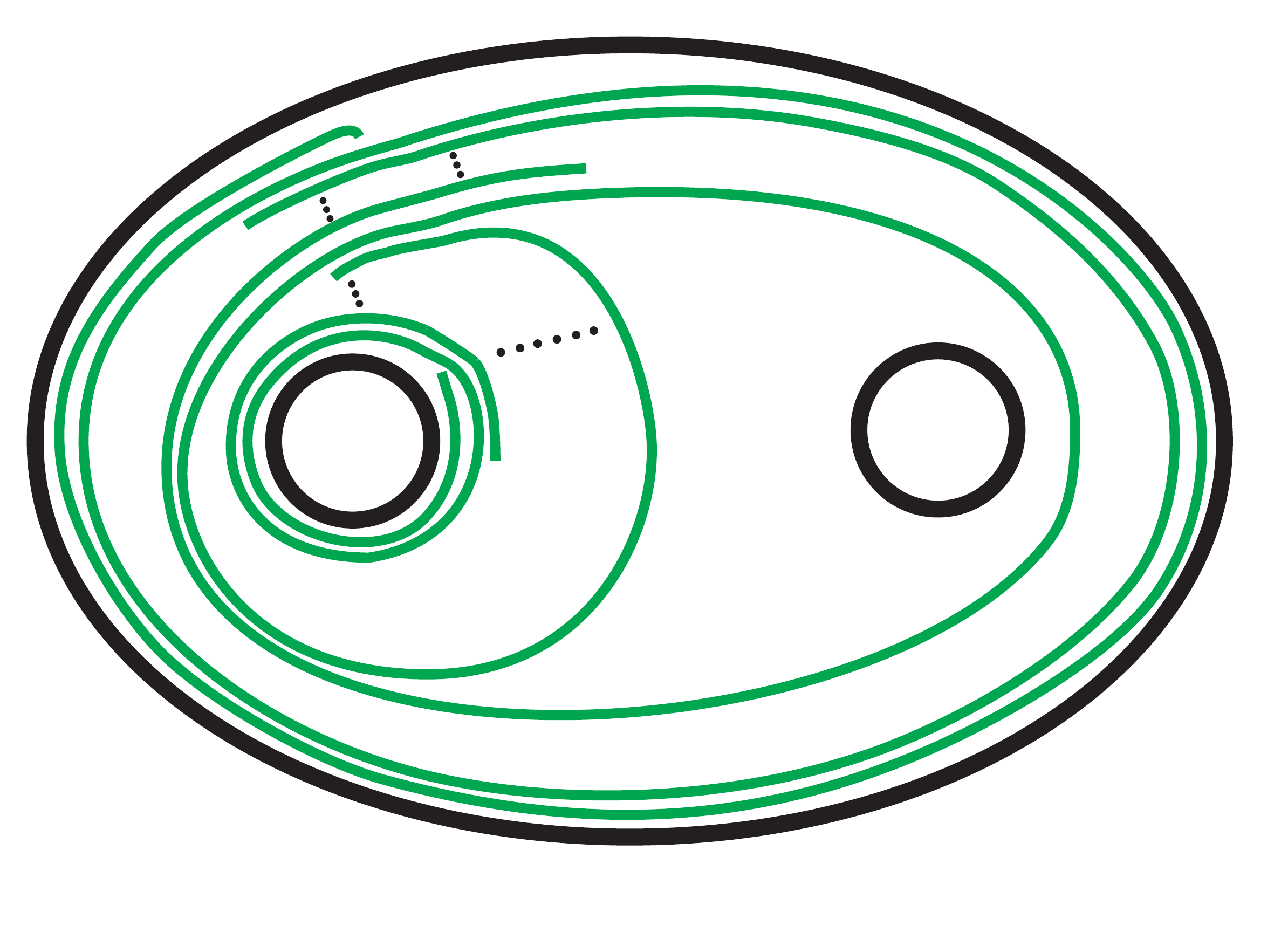}
\put(40,-2){$N(m,n)$}
\end{overpic}}}
\vspace{4mm}
\end{eqnarray*}
    \caption{Illustration of $C(m,n)$.}
    \label{fig:C(m,n)}
\end{figure}

Our goal is to find closed formulae for $P(m,n)$ and $N(m,n)$, which we achieve through the following series of lemmas and corollaries. Their proofs will be provided in the \hyperref[appendixs1s2s1s2]{Appendix}. Since the diagram for $N(m,n)$ is the mirror image of the diagram for
$P(m,n)$, it is sufficient to find the formula for $P(m,n)$. We obtain the following lemma for $P(m,n)$.

\begin{lemma}\label{lem:PP(m,n)}
    There exists a sequence $PP(m,n)$, $m,n \geq 0$ such that $$P(m,n) = A^{m+n-1}PP(m,n)-A^{m+n-5}PP(m-2,n),$$
    satisfying the following relations:
    \begin{eqnarray*}
PP(0,0)&=&1,~ PP(1,0) ~=~ a_{1},~ PP(0,1)~=~a_{3}, ~PP(1,1)~=~a_{1}a_{3},\\
PP(m,0) &=& PP(m-1,0)a_{1}-PP(m-2,0),~ m\geq 2, \\
PP(m,1) &=& PP(m,0)a_{3}+A^{-2}PP(m-1,0)a_{2}+A^{-4}PP(m-2,1),~ m\geq 2, \text{and}\\
PP(m,n) &=& PP(m,n-1)a_{3}-PP(m,n-2), m\geq 0, n\geq 2.
\end{eqnarray*}
\end{lemma}

Since $N(m,n)$ is the mirror image of $P(m,n)$, the equality $N(m,n)(A)= P(m,n)(A^{-1})$ follows and we obtain the following lemma.
\begin{lemma}\label{lem:NN(m,n)}
There exists a sequence $NN(m,n)$, $m,n \geq 0$ such that
$$N(m,n) = A^{-m-n+1}NN(m,n)-A^{-m-n+5}NN(m-2,n),$$
satisfying the following relations:
    \begin{eqnarray*}
NN(0,0)&=&1,~ NN(1,0) ~=~ a_{1},~ NN(0,1)~=~a_{3},~ NN(1,1)~=~a_{1}a_{3},\\
NN(m,0) &=& NN(m-1,0)a_{1}-NN(m-2,0), m\geq 2,\\
NN(m,1) &=& NN(m,0)a_{3}+A^{2}NN(m-1,0)a_{2}+A^{4}NN(m-2,1), m\geq 2, \text{and}\\
NN(m,n) &=& NN(m,n-1)a_{3}-NN(m,n-2), m\geq 0, n\geq 2.
\end{eqnarray*}
Therefore, we obtain 
$$C(m,n) = A^{m+n+2}PP(m,n)-A^{m+n-2}PP(m-2,n)-A^{-m-n-2}NN(m,n)+A^{-m-n+2}NN(m-2,n).$$
\end{lemma}

Notice that $PP(m,n)$ and $NN(m,n)$ satisfy the Chebyshev recurrence relation in the variable $a_{3}$ with respect to the variable $n$. From this observation we get the following lemma.
\begin{lemma}\label{lem:(m,n)-formula}
The sequence $PP(m,n)$ in Lemma~\ref{lem:PP(m,n)} satisfies
\begin{eqnarray*}
    PP(m,n) = PP(m,1)S_{n-1}(a_{3})-PP(m,0)S_{n-2}(a_{3}).
\end{eqnarray*}
    Analogously, the sequence $NN(m,n)$ in Lemma \ref{lem:NN(m,n)} satisfies
\begin{eqnarray*}
    NN(m,n) = NN(m,1)S_{n-1}(a_{3})-NN(m,0)S_{n-2}(a_{3}).
\end{eqnarray*}
\end{lemma}

Since $NN(m,0)$ and $PP(m,0)$ have the same recurrence relation and initial conditions, we obtain $PP(m,0)=NN(m,0)=S_{m}(a_{1})$ for $m \geq 0$, where $S_{m}(a_{1})$ is the Chebyshev polynomial of the second kind in the variable $a_{1}$. From the equalities 
\begin{eqnarray*}
&&PP(m,1)=PP(m,0)a_{3}+A^{-2}PP(m-1,0)a_{2}+A^{-4}PP(m-2,1),\\
\Leftrightarrow&& PP(m,1)-A^{-4}PP(m-2,1)=PP(m,0)a_{3}+A^{-2}PP(m-1,0)a_{2}, m\geq 2
\end{eqnarray*}
and 
\begin{eqnarray*}
&&NN(m,1) = NN(m,0)a_{3}+A^{2}NN(m-1,0)a_{2}+A^{4}NN(m-2,1),\\
\Leftrightarrow&& NN(m,1)-A^{4}NN(m-2,1) = NN(m,0)a_{3}+A^{2}NN(m-1,0)a_{2}, m\geq 2,
\end{eqnarray*}
we can prove the following statement.
\begin{lemma}\label{lem:C(m,n)-formula}
For $m,n \geq 0$,
    \begin{eqnarray*}
        C(m,n) &=& (-A^{m+n+2}+A^{-m-n-2})S_{m}(a_{1})S_{n}(a_{3})+\\
                &&+(-A^{m+n}+A^{-m-n})S_{m-1}(a_{1})S_{n-1}(a_{3})a_{2}+\\
                &&+(-A^{m+n-2}+A^{-m-n+2})S_{m-2}(a_{1})S_{n-2}(a_{3}).
    \end{eqnarray*}
\end{lemma}

\noindent
{\bf Case II:} $C(m,-n)$ for $m,n \geq 1$. \\
For $m,n \geq 1$, $C(m,-n) = \omega(c_{m,-n}) = c_{m,-n}\cup \beta_{2} - c_{m,-n}\cup \beta_{1}.$
Since $c_{m,-n}\cup \beta_{1}$ and $c_{m,-n}\cup \beta_{2}$ each have both a negative and a positive kink, we obtain two diagrams without any kinks and coefficients. We denote them by $N(m,-n)$ and
$P(m,-n)$, respectively. Hence, we obtain
$$C(m,-n)= P(m,-n)-N(m,-n), m, n \geq 1,$$
as illustrated in Figure ~\ref{fig:C(m,-n)}.
\begin{figure}[ht]
    \centering
    \begin{eqnarray*}
    C(m,-n) & = &
\omega\left(\vcenter{\hbox{\begin{overpic}[scale=.1]{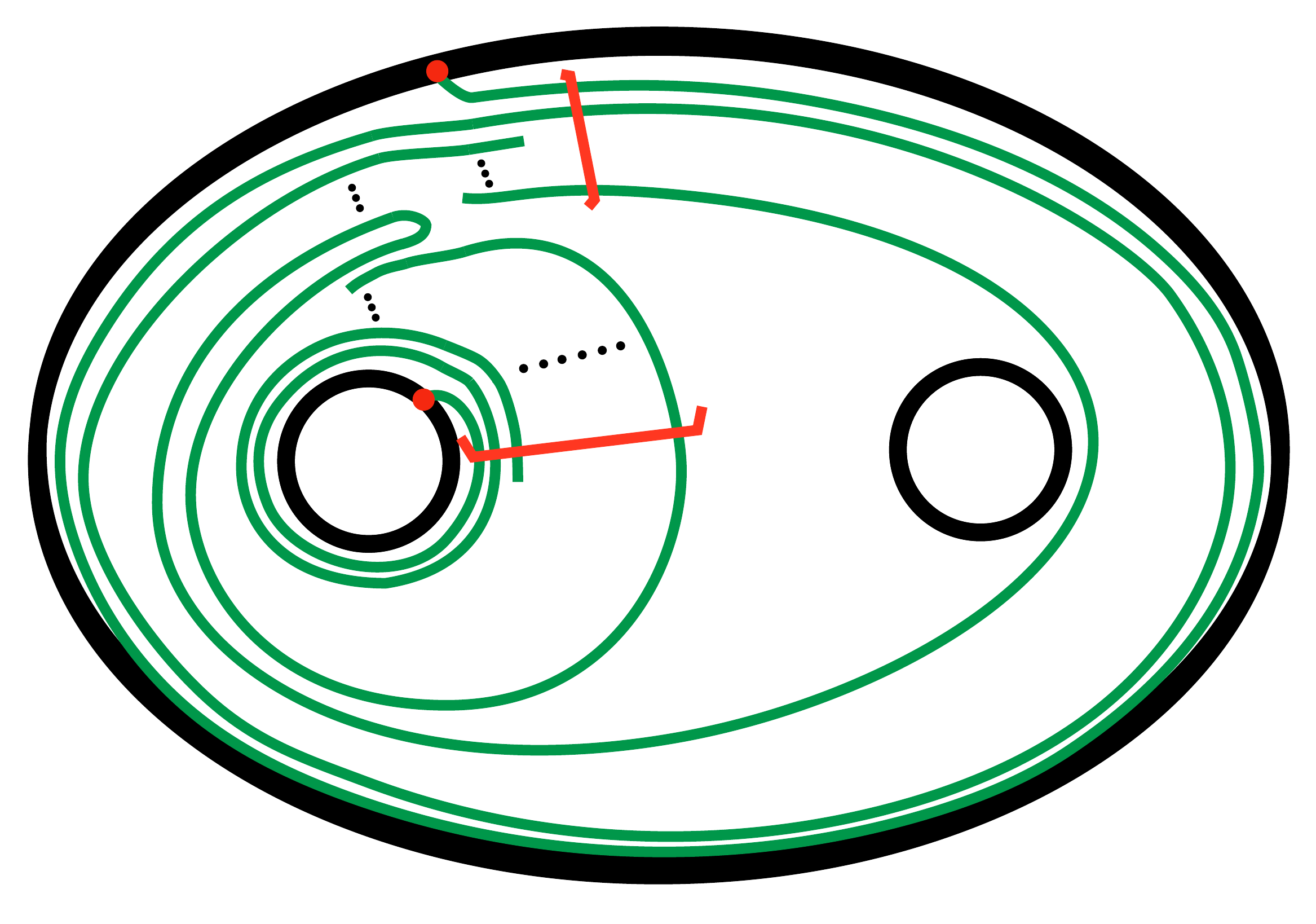} 
\put(48,36){\tiny{${m}$}}
\put(52,65){\tiny{${n}$}}
\end{overpic}}}\right) \\ 
& = & \vcenter{\hbox{\begin{overpic}[scale=.1]{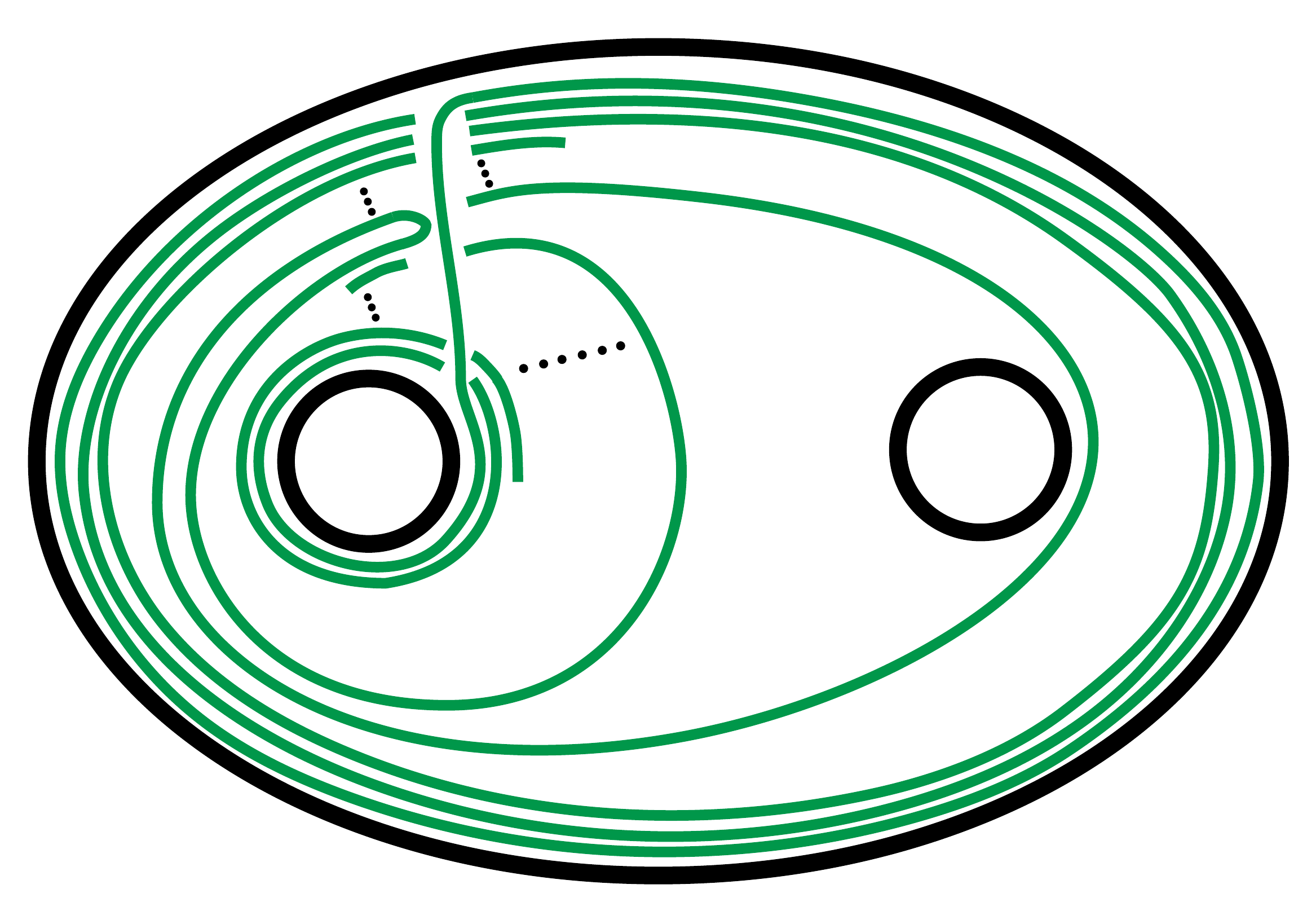}
\end{overpic}}} -
\vcenter{\hbox{\begin{overpic}[scale=.1]{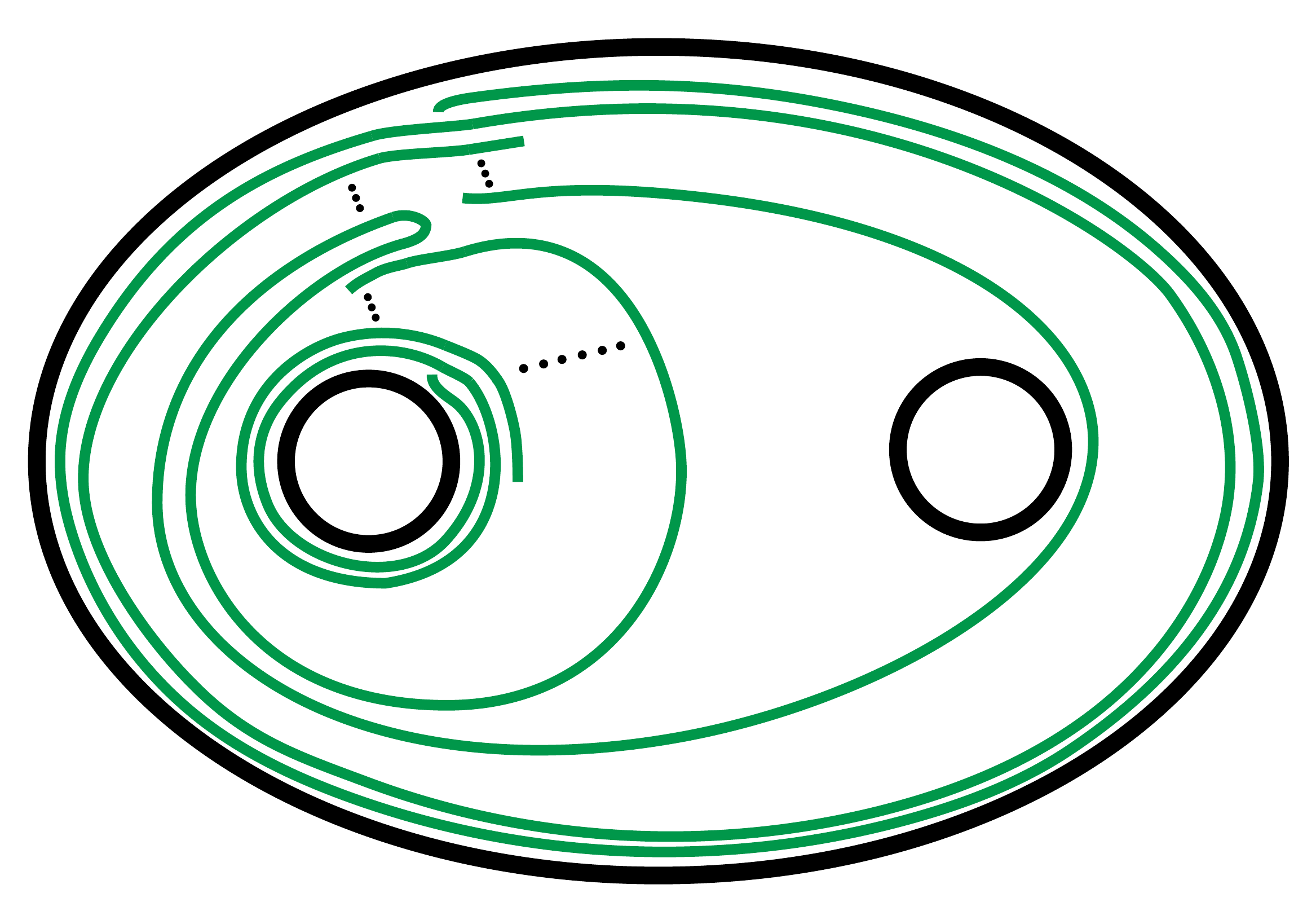}
\end{overpic}}}
\vspace{4mm}
\end{eqnarray*}
    \caption{Illustration of $C(m,-n)$.}
    \label{fig:C(m,-n)}
\end{figure}

Analogous to the previous case, we obtain the following series of lemmas.

\begin{lemma}\label{lem:PP(m,-n)}
There exists a sequence $\{PP(m,-n)\}_{m,n \geq 1}$ such that     
$$P(m,-n)=A^{m-n-1}PP(m,-n)-A^{m-n-5}PP(m-2,-n),$$

satisfying
\begin{eqnarray*}
PP(1,-1)&=&0,\\
PP(m,-1) &=& A^{3}PP(m,1)-A^{3}PP(m,0)a_{3}, m\geq 2,\\
PP(m,-2) &=& PP(m,-1)a_{3} +A^{3}PP(m,0), m \geq 1,  \text{and}\\ 
PP(m,-n)&=& PP(m,-n+1)a_{3}-PP(m,-n+2), m \geq 1,n\geq 3.
\end{eqnarray*}
\end{lemma}
Since $N(m,-n)$ is the mirror image of $P(m,-n)$ we obtain the following lemma.
\begin{lemma}\label{lem:NN(m,-n)}
There exists a sequence $\{NN(m,-n)\}_{m,n \geq 1}$ such that     
$$NN(m,-n)=A^{-m+n+1}NN(m,-n)-A^{-m+n+5}NN(m-2,-n),$$

satisfying
\begin{eqnarray*}
NN(1,-1)&=&0,\\
NN(m,-1) &=& A^{-3}NN(m,1)-A^{-3}NN(m,0)a_{3}, m\geq 2,\\
NN(m,-2) &=& NN(m,-1)a_{3} +A^{-3}NN(m,0), m \geq 1, \text{and}\\ 
NN(m,-n)&=& NN(m,-n+1)a_{3}-NN(m,-n+2), m \geq 1, n\geq 3.
\end{eqnarray*}
\end{lemma}

\begin{lemma}\label{lem:(m,-n)-formula}
The sequence $PP(m,-n)$ in Lemma~\ref{lem:PP(m,-n)} satisfies 
\begin{eqnarray*}
    PP(m,-n) &=& S_{n-2}(a_{3})PP(m,-2) - S_{n-3}(a_3)PP(m,-1)\\
    &=& A^{3}PP(m,1)S_{n-1}(a_{3})-A^{3}PP(m,0)S_{n}(a_{3}).
\end{eqnarray*}
    Analogously, the sequence $NN(m,-n)$ in Lemma~\ref{lem:NN(m,-n)} satisfies
\begin{eqnarray*}
    NN(m,n) &=& S_{n-2}(a_{3})NN(m,-2) - S_{n-3}(a_3)NN(m,-1)\\
    &=& A^{-3}NN(m,1)S_{n-1}(a_{3})-A^{-3}NN(m,0)S_{n}(a_{3}).
\end{eqnarray*}
\end{lemma}

\begin{lemma}\label{lem:C(m,-n)-formula}
    For all $m, n \geq 1$,
    \begin{eqnarray*}
        C(m,-n) &=& -(-A^{m-n+2}+A^{-m+n-2})S_{m}(a_{1})S_{n-2}(a_{3})\\
                &&-(-A^{m-n}+A^{-m+n})S_{m-1}(a_{1})S_{n-1}(a_{3})a_{2}\\
                &&-(-A^{m-n-2}+A^{-m+n+2})S_{m-2}(a_{1})S_{n}(a_{3}).
    \end{eqnarray*}
\end{lemma}

From Lemmas \ref{lem:C(m,n)-formula} and \ref{lem:C(m,-n)-formula} we obtain the following theorem.
\begin{theorem}\label{thm:C(m,n)-formula}
For $m,q \in \mathbb{N} \cup \{0\}$ and $n \in \mathbb{Z}$,
      \begin{eqnarray*}
        C(m,n)S_{q}(a_{2}) &=& (-A^{m+n+2}+A^{-m-n-2})S_{m}(a_{1})S_{n}(a_{3})S_{q}(a_{2})\\
                &+&(-A^{m+n}+A^{-m-n})S_{m-1}(a_{1})S_{n-1}(a_{3})S_{q+1}(a_{2})\\
                &+&(-A^{m+n}+A^{-m-n})S_{m-1}(a_{1})S_{n-1}(a_{3})S_{q-1}(a_{2})\\
                &+&(-A^{m+n-2}+A^{-m-n+2})S_{m-2}(a_{1})S_{n-2}(a_{3})S_{q}(a_{2}).
    \end{eqnarray*}
    where $S_{n}(a_{3})$ is the Chebyshev polynomial of the second kind extended to negative indices satisfying $S_{n}(a_{3}) = -S_{-n-2}(a_3)$ for $n\leq -2$.
\end{theorem}

\noindent
{\bf Case III:} $C(m,n)$ in general. \\
In the previous cases we calculated $C(m,n)$ for $m,n \in \mathbb{N} \cup \{0\}$ and $C(m,-n)$ for $m,n \in \mathbb{N}$. We now consider $C(-m,-n)=-A^{-3}P(-m,-n)+A^{3}N(-m,-n) $ for $m,n > 0$, illustrated in Figure~\ref{fig:C(-m,-n)}.
\begin{figure}[ht]
    \centering
    \begin{eqnarray*}
    C(-m,-n) & = &
\omega\left(\vcenter{\hbox{\begin{overpic}[scale=.1]{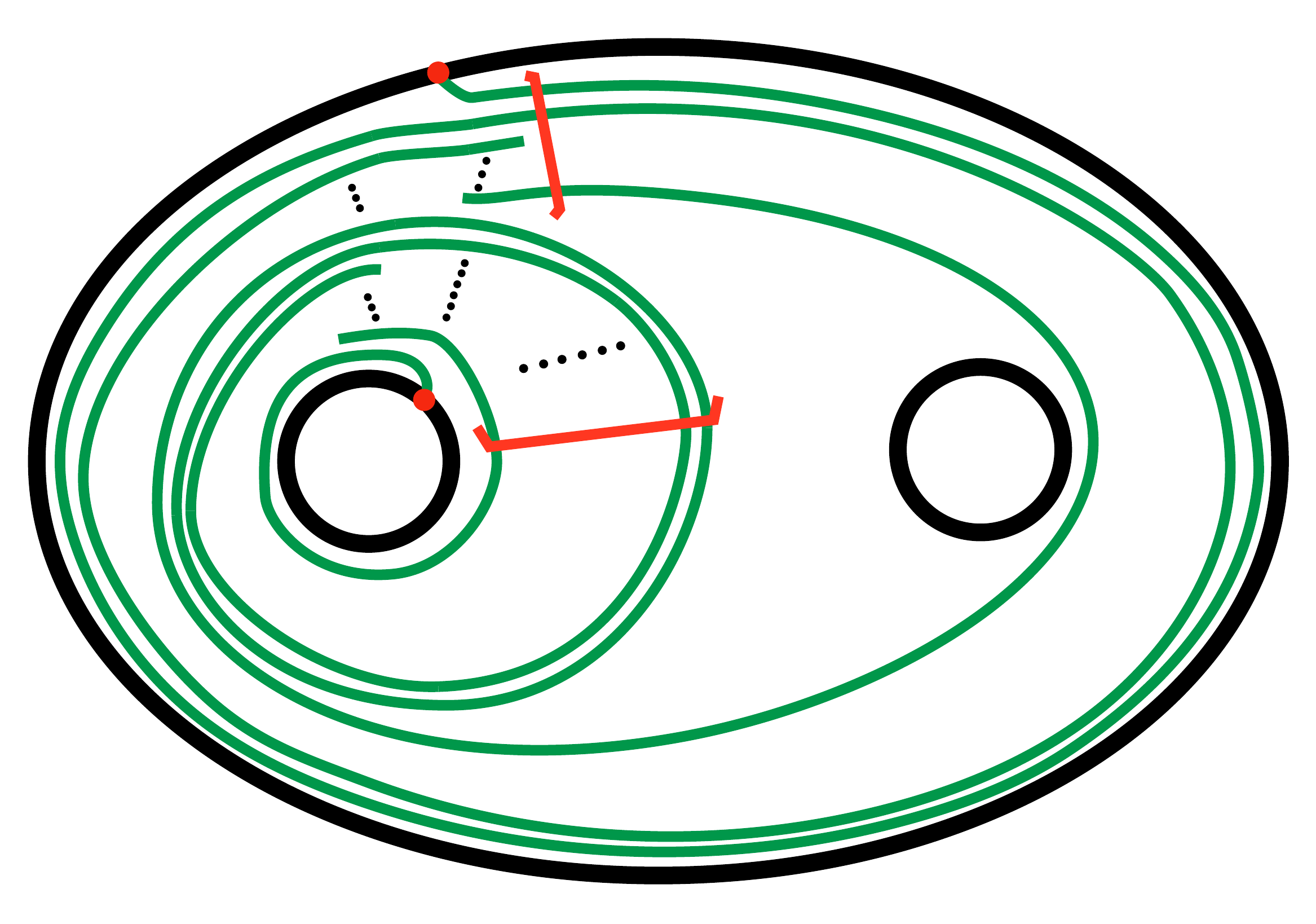} 
\put(48,36){\tiny{${m}$}}
\put(50,65){\tiny{${n}$}}
\end{overpic}}}\right)  \\ 
& = &  \vcenter{\hbox{\begin{overpic}[scale=.1]{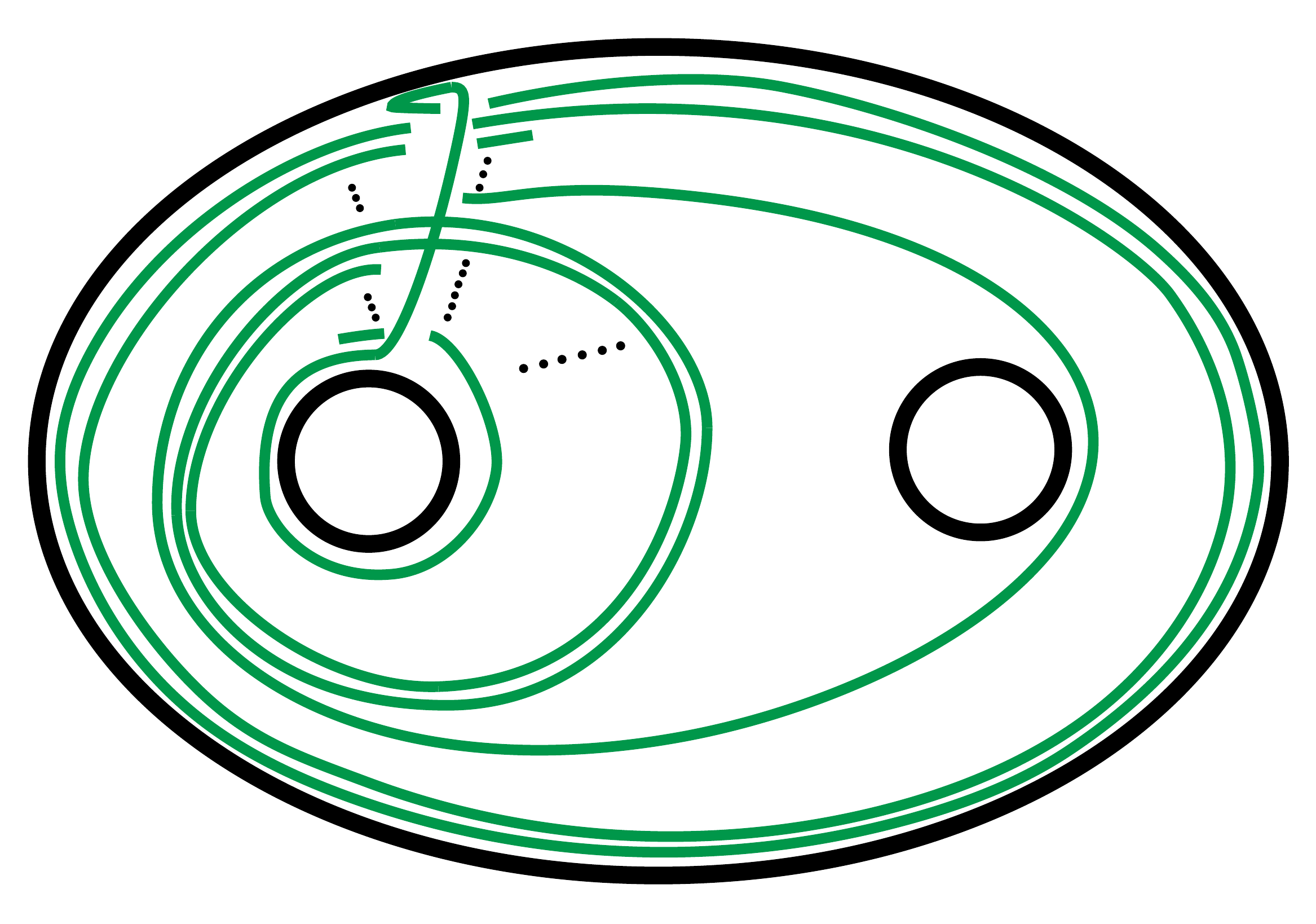}
\end{overpic}}}-
\vcenter{\hbox{\begin{overpic}[scale=.1]{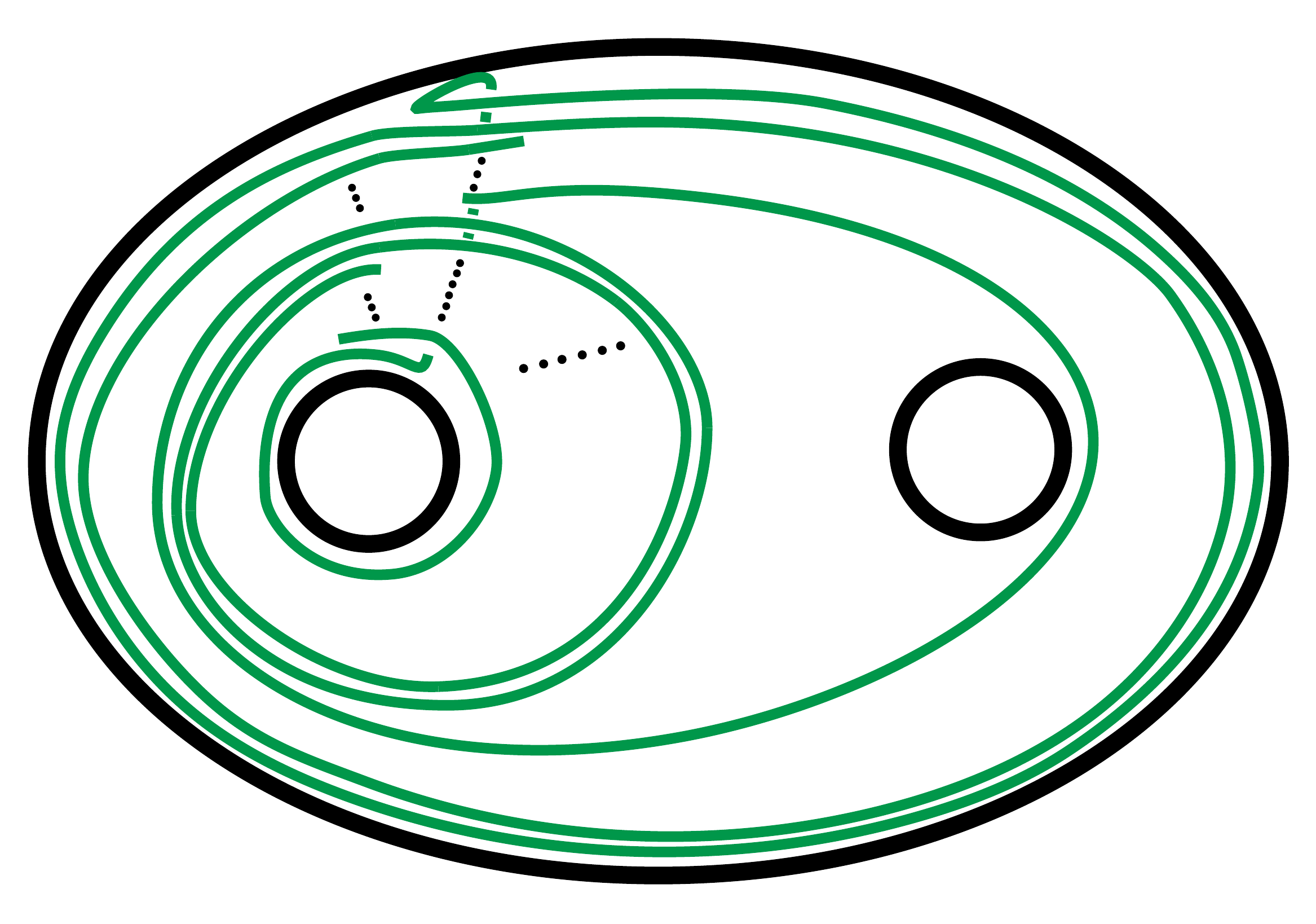}
\end{overpic}}}\\ 
& = & -A^{-3} \vcenter{\hbox{\begin{overpic}[scale=.1]{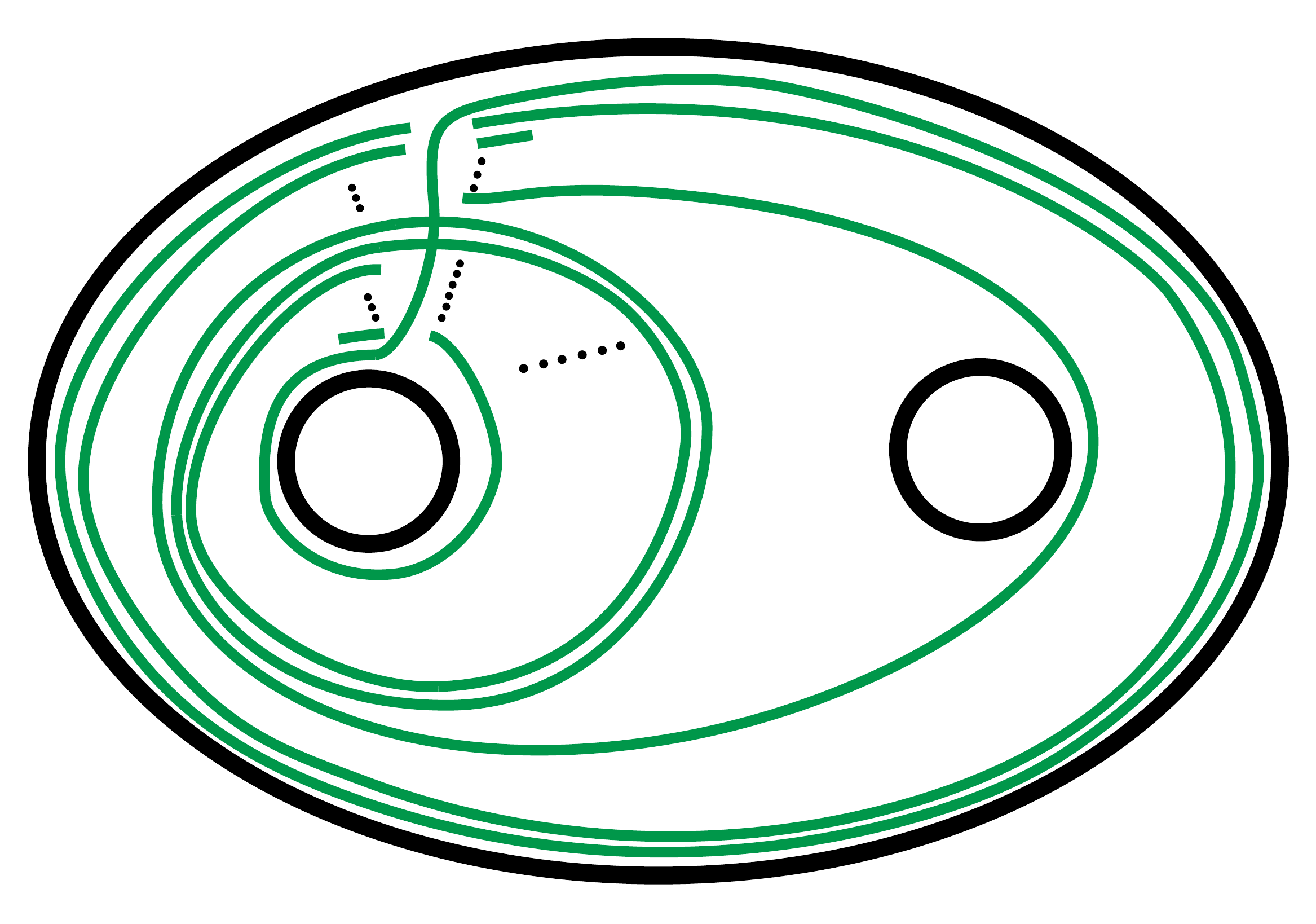}
\end{overpic}}} + A^{3}
\vcenter{\hbox{\begin{overpic}[scale=.1]{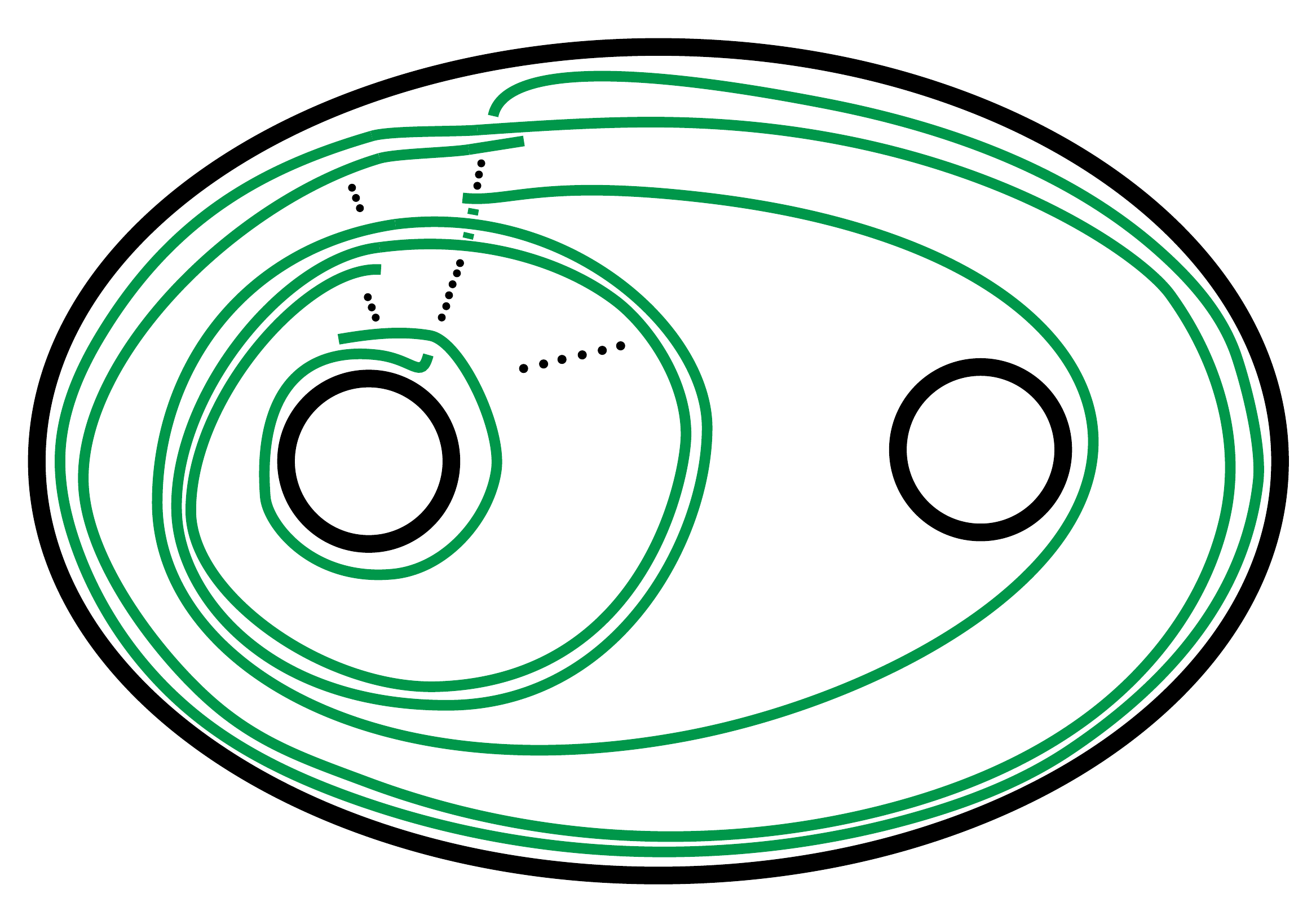}
\end{overpic}}}\\
& = & -A^{-3}P(-m,-n)+A^{3}N(-m,-n)
\vspace{4mm}
\end{eqnarray*}
    \caption{Illustration of $C(-m,-n)$.}
    \label{fig:C(-m,-n)}
\end{figure}
Note that $P(-m,-n)$ is a diagram of a link $L$ in $\Sigma_{0,3}\times [0,1]$ obtained by the projection of $L$ onto $\Sigma_{0,3} \times \{0\}$. We see that the projection of $L$ onto $\Sigma_{0,3} \times \{1\}$ is the mirror image of $P(-m,-n)$, which is $N(m,n)$. This is illustrated in Figure ~\ref{fig:P(-m,-n)=N(m,n)}.
\begin{figure}[ht]
\begin{tikzcd}[row sep=2em, column sep=2em]
\vcenter{\hbox{\begin{overpic}[scale=.1]{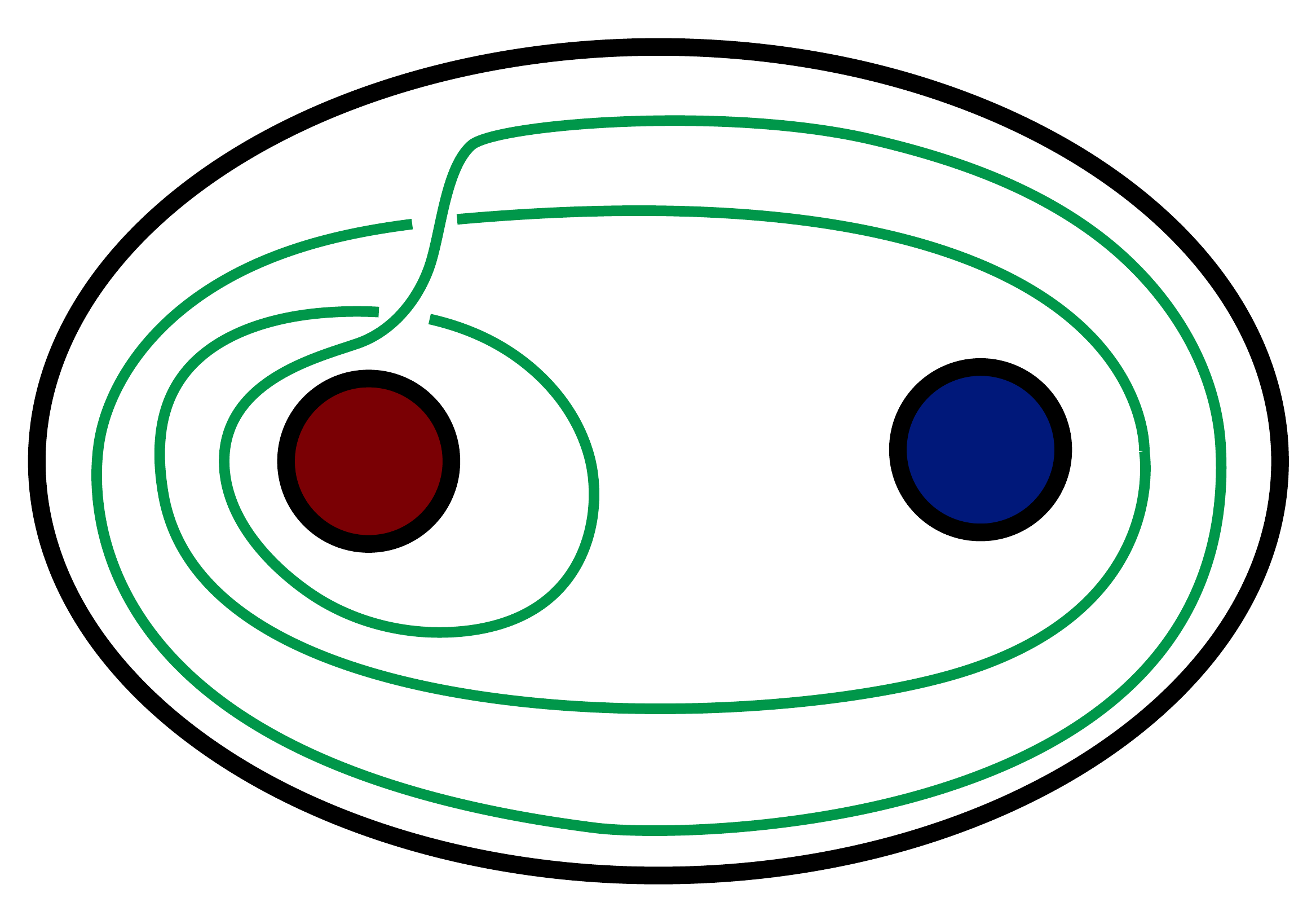}
\end{overpic}}}\rar & \vcenter{\hbox{\begin{overpic}[scale=.1]{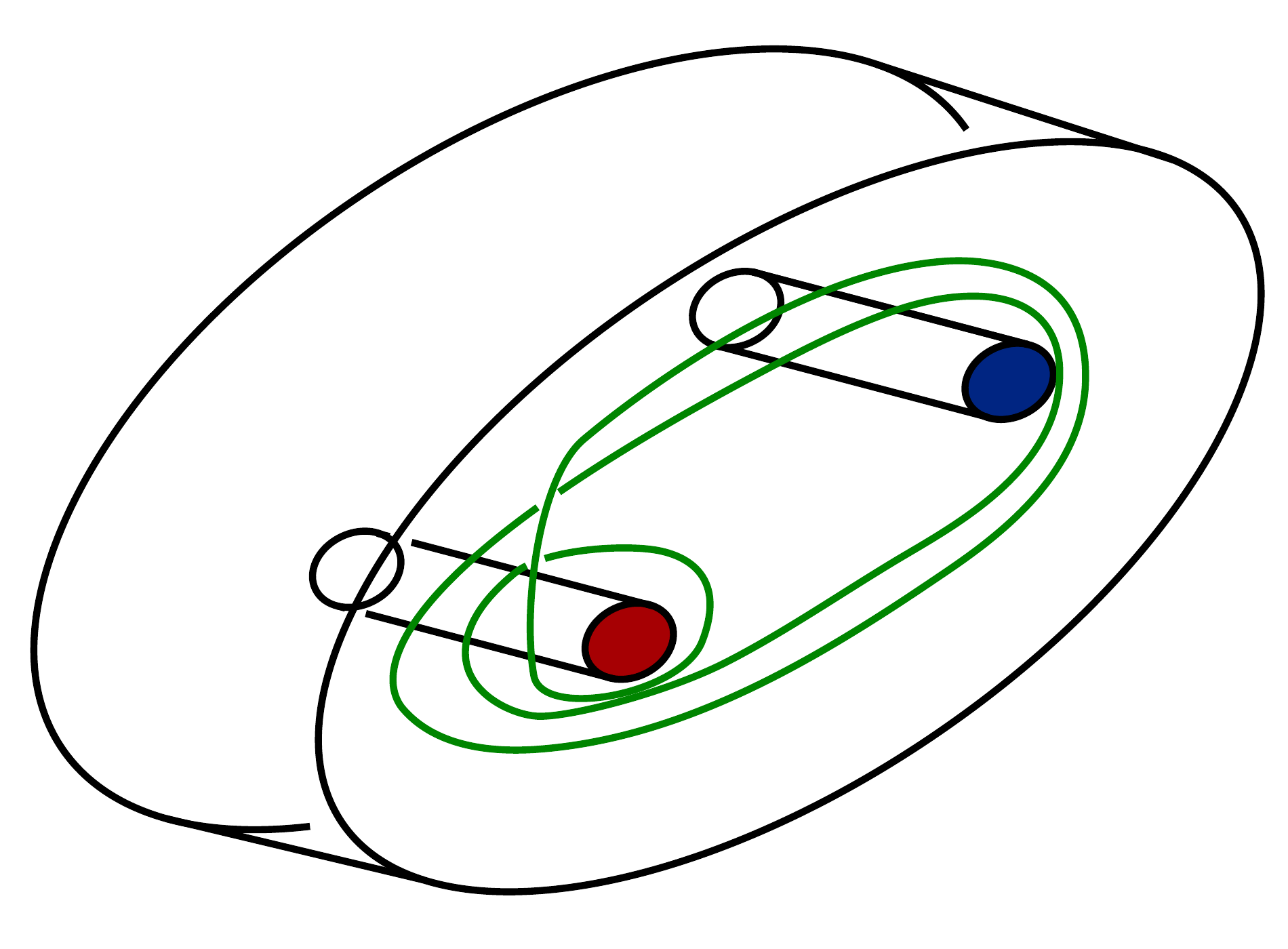}
\end{overpic}}} \rar[] & \vcenter{\hbox{\begin{overpic}[scale=.1]{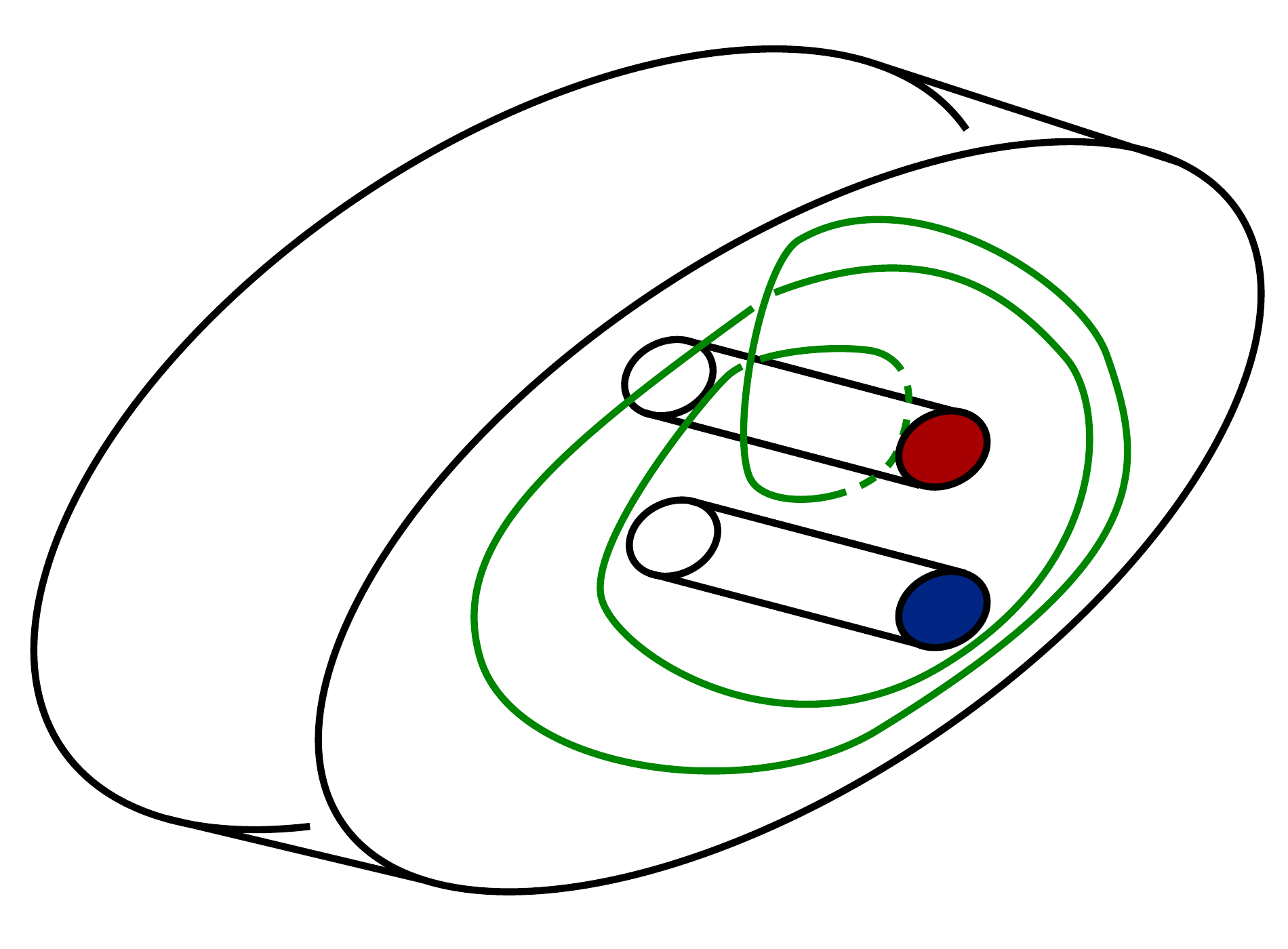}
\end{overpic}}} \dar[]  \\
&\vcenter{\hbox{\begin{overpic}[scale=.1]{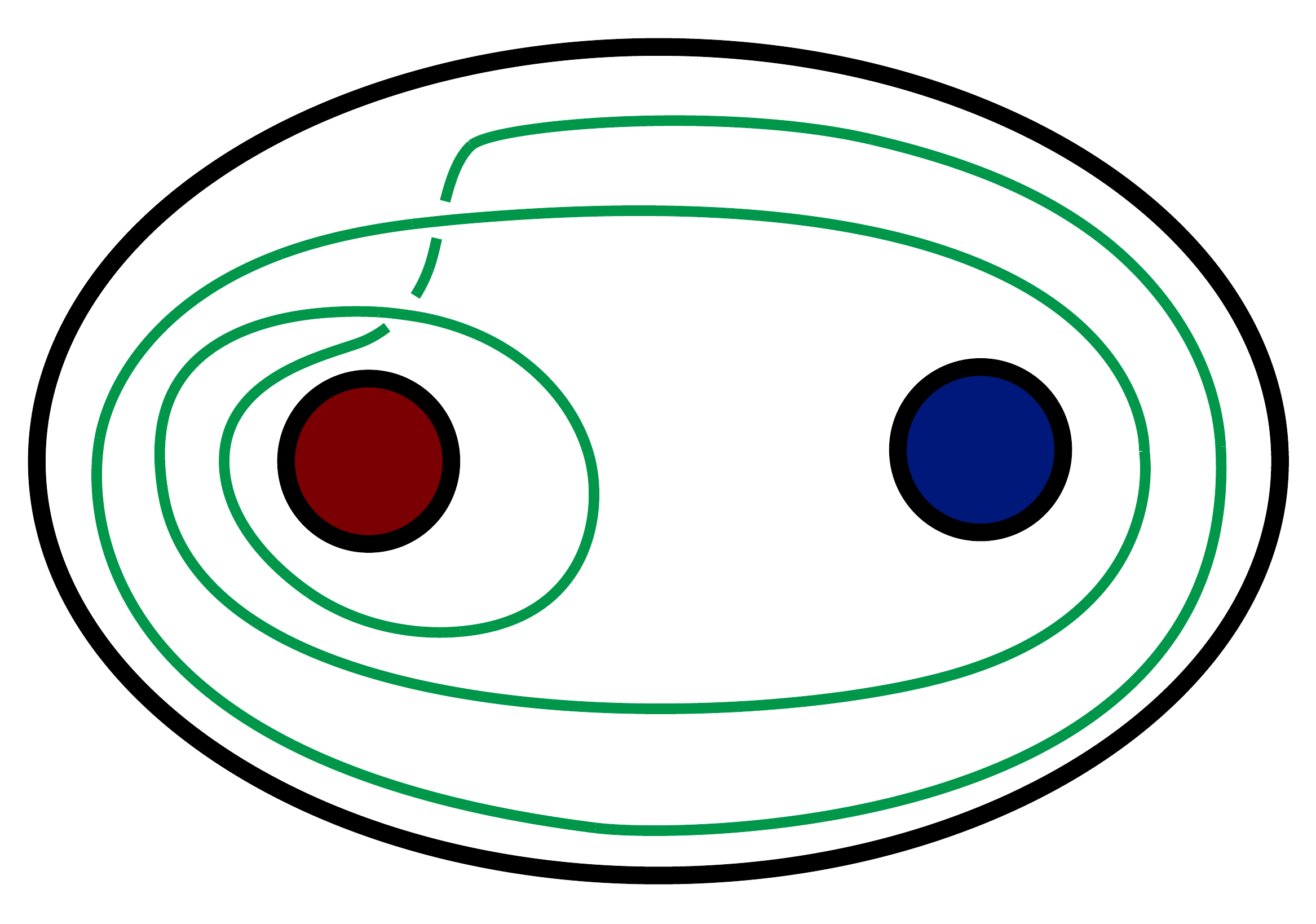}
\end{overpic}}} & \lar[]  \vcenter{\hbox{\begin{overpic}[scale=.1]{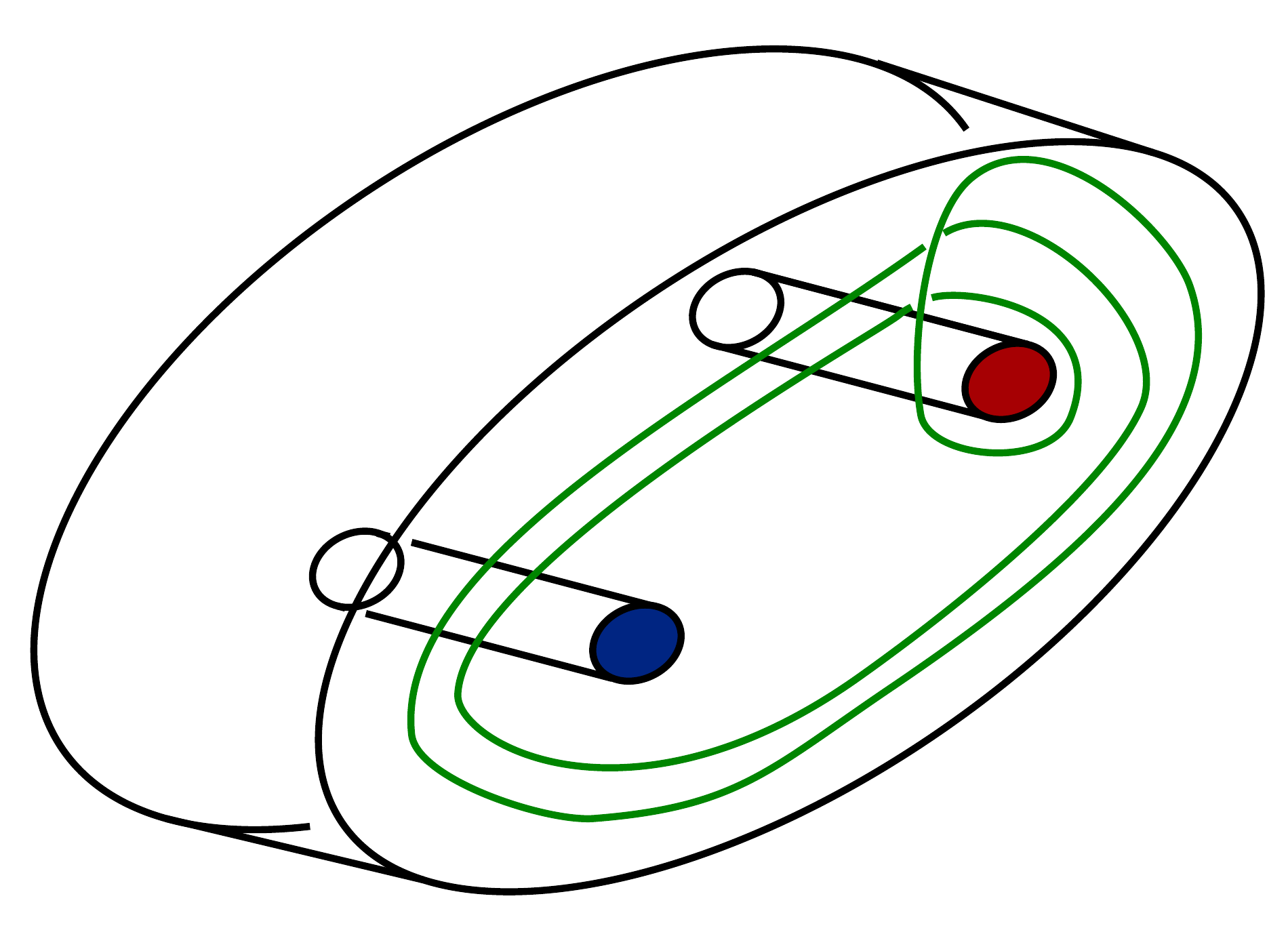}
\end{overpic}}}
\end{tikzcd}
   \caption{$P(-1,-2)$ and $N(1,2)$ are diagrams of isotopic links in $\Sigma_{0,3} \times [0,1]$.}
    \label{fig:P(-m,-n)=N(m,n)}  
\end{figure}
Similarly, the diagrams of $N(-m,-n)$ and $P(m,n)$ are isotopic in $\Sigma_{0,3} \times [0,1]$. Therefore, we obtain $$C(-m,-n) = -A^{-3}P(-m,-n)+A^{3}N(-m,-n) = -A^{-3}N(m,n)+A^{3}P(m,n) = -C(m,n).$$
Analogously, one can show that $C(-m,n) = -C(m,-n),$ for $m, n \in \mathbb{N} \cup \{0\}$. We now obtain the following lemma.
\begin{lemma}\label{lem:negativeCmn}
For $m,n \in \mathbb{Z}$,
    $$C(m,n) = -C(-m,-n).$$
\end{lemma}

The results of this section lead to the following description of $\mathcal J_1$.

\begin{corollary}\label{j1mainresult}

The relations described in Theorem \ref{thm:C(m,n)-formula} generate the submodule $\mathcal J_1$ of handle sliding relations of $\mathcal S_{2,\infty}(H_2)$.
    
\end{corollary}

We also remark that $\mathcal S_{2,\infty}(H_2)/\mathcal J_1$ is isomorphic to $\mathcal S_{2,\infty}((S^1 \times S^2) \ \# \ H_1)$.


 \subsection{Generators of the submodule $\mathcal J_2$}   
Analogous to the case of $\mathcal J_1$, we find the exact description of the generators of the submodule $\mathcal J_2$ by using the relative Kauffman bracket skein module of  $\mathcal S_{2,\infty}(H_2; u',v')$. The manifold $\Sigma_{0,3}$ with marked points $u'$ and $v'$ on its boundary is illustrated in Figure~\ref{bar-c00}.

 \begin{figure}[ht]
    \centering
\begin{overpic}[unit=1mm, scale = 0.1]{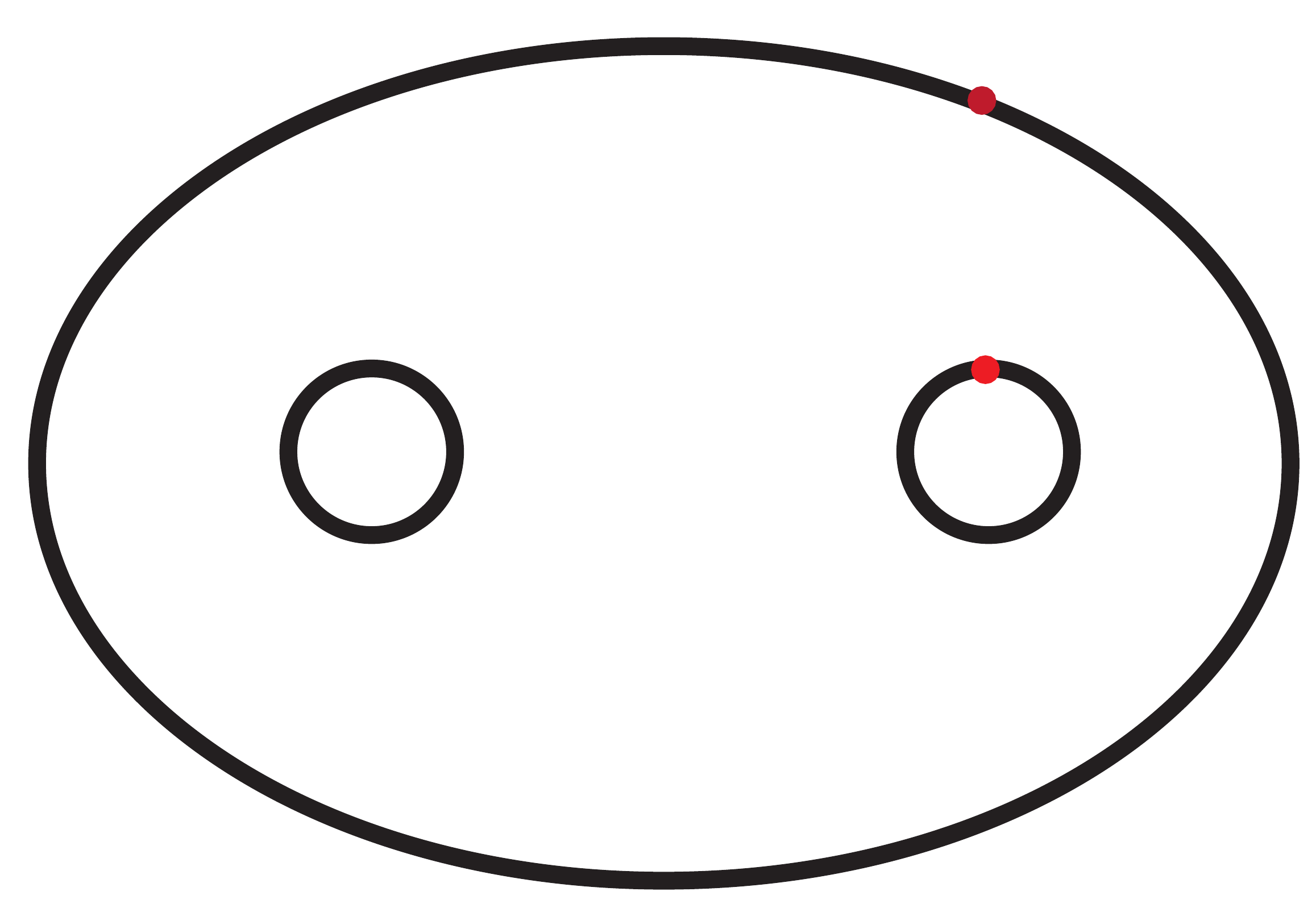}

\put(28.5,12.5){\footnotesize$v'$}
\put(29,25){\footnotesize$u'$}
\end{overpic}
    \caption{Marked points $u'$ and $v'$ on $\partial \Sigma_{0,3}$, where $u'$ lies on the boundary component $a_3$ and $v'$ lies on the boundary component $a_2$.}
    \label{bar-c00}
\end{figure}
 Let $\overline c_{q,n}$ be a relative curve connecting the points $u'$ and $v'$ such that $\overline c_{q,n}$ rotates $q$-times along $a_{2}$ and $n$-times along $a_{3}$ in either the counterclockwise or clockwise directions depending on the signs of $q$ and $n$, respectively. See Figure \ref{rkbsmbasisf03uvprime} for an illustration. \\

 \begin{figure}[ht]
\centering
$\hdots \vcenter{\hbox{\begin{overpic}[scale=.075]{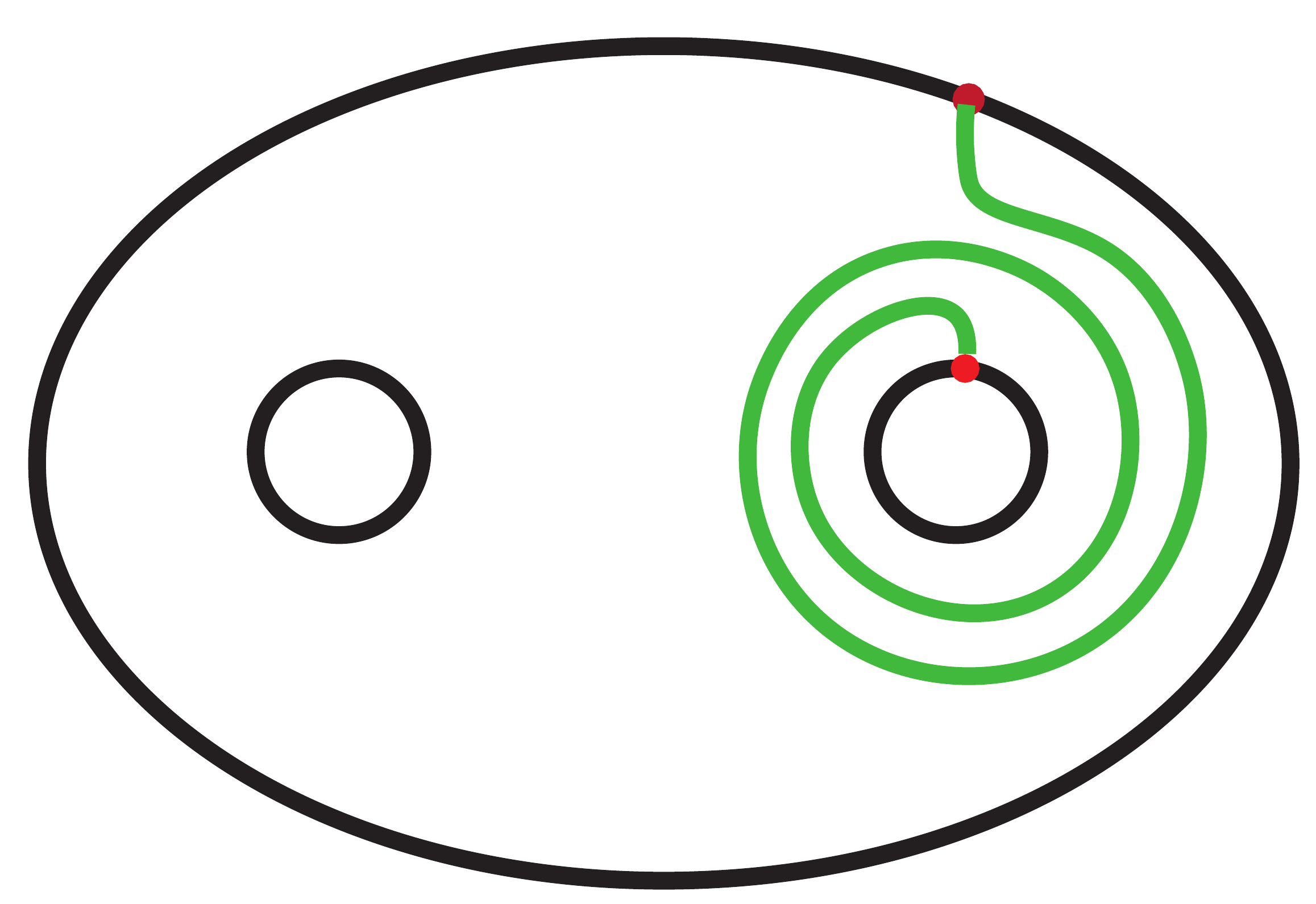}
\put(32, -8){$\overline c_{-2,0}$}
\end{overpic}}}  \hspace{2mm}
\vcenter{\hbox{\begin{overpic}[scale=.075]{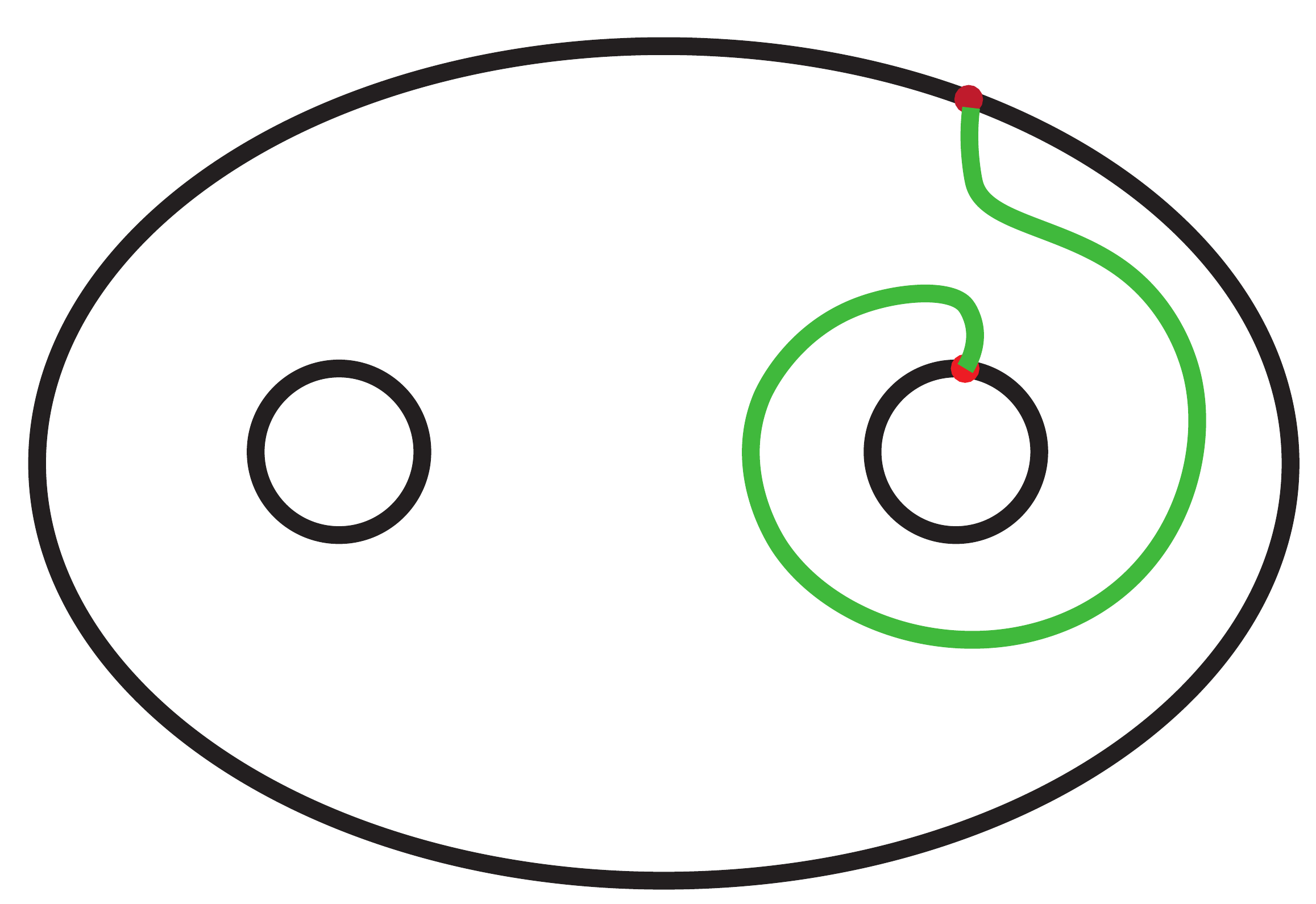}
\put(34, -8){$\overline c_{-1,0}$}
\end{overpic}}} \hspace{2mm}
\vcenter{\hbox{\begin{overpic}[scale=.075]{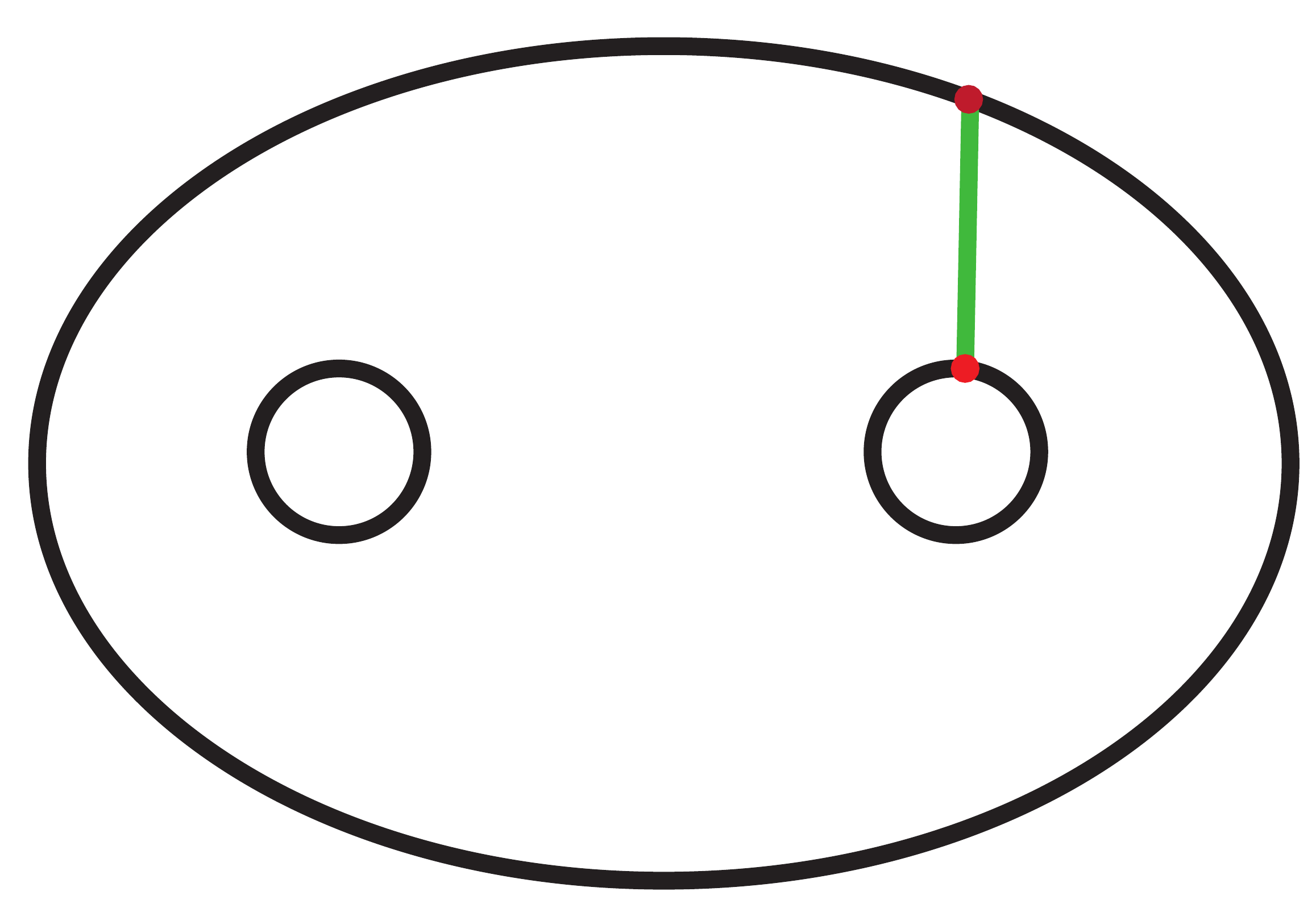}
\put(37, -8){$\overline c_{0,0}$}
\put(60.5, 53.5){\footnotesize{$u'$}}
\put(63, 34){\footnotesize{$v'$}}
\end{overpic}}} \hspace{2mm}
\vcenter{\hbox{\begin{overpic}[scale=.075]{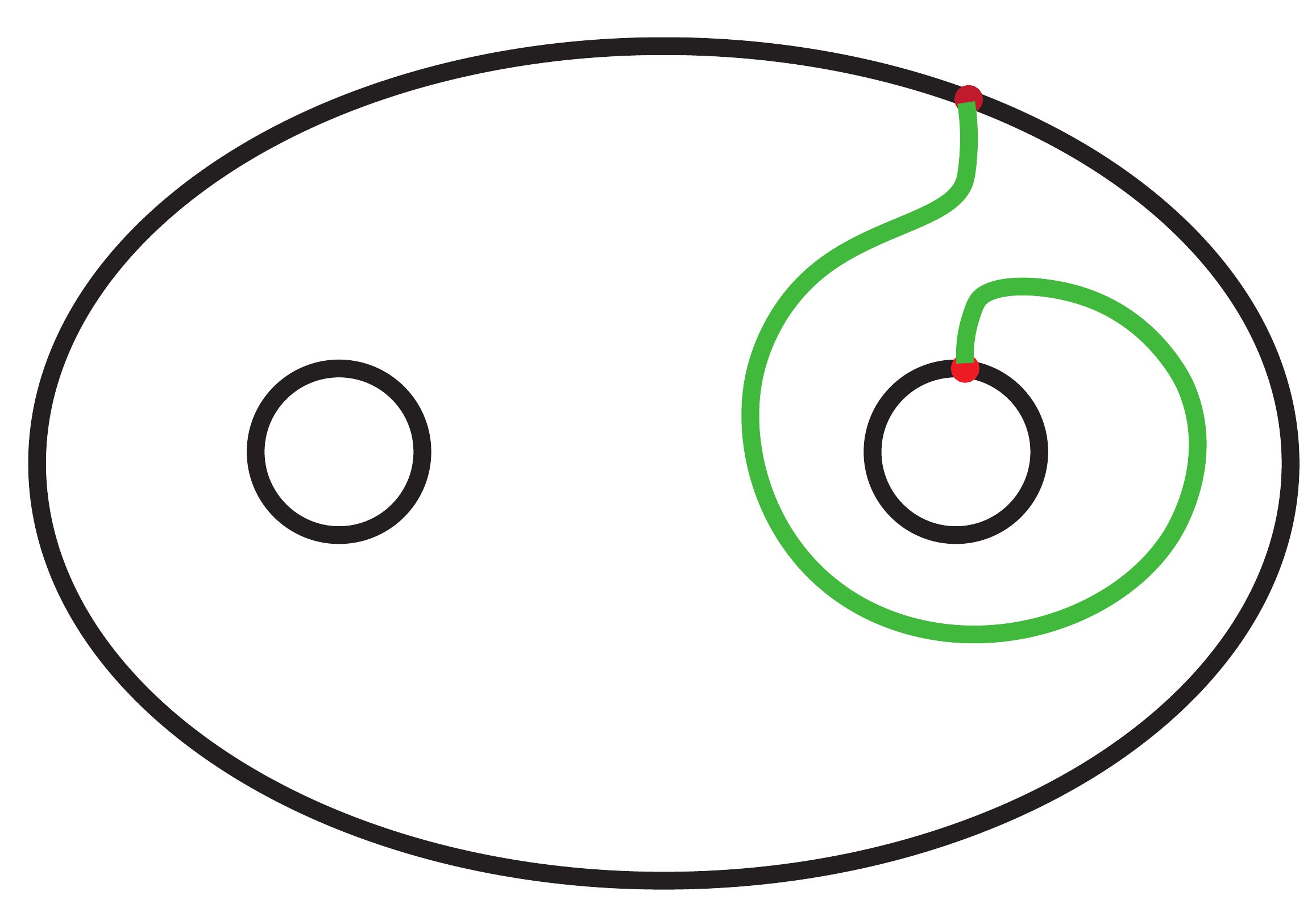}
\put(37, -8){$\overline c_{1,0}$}
\end{overpic}}}  \hspace{2mm}
\vcenter{\hbox{\begin{overpic}[scale=.075]{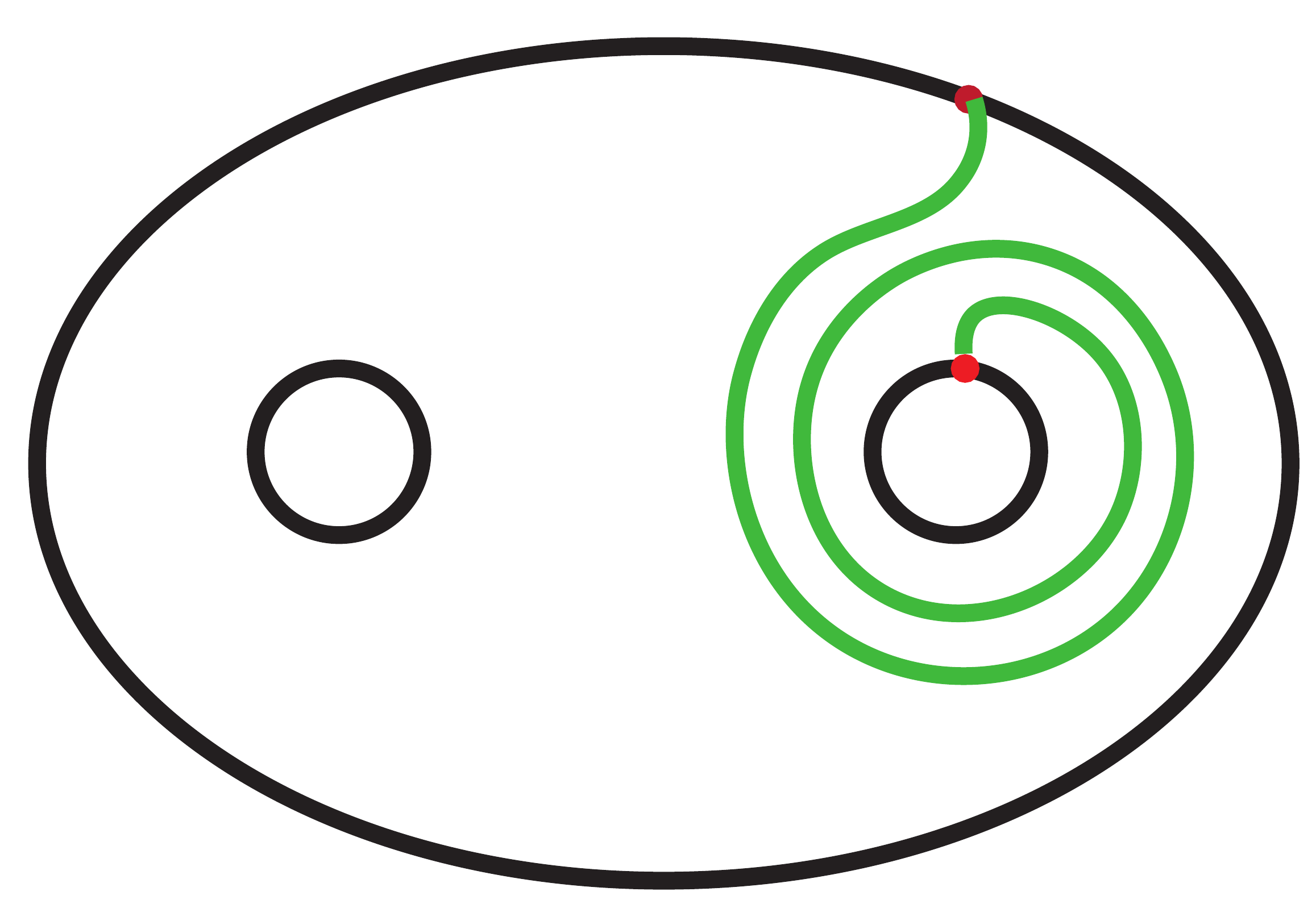}
\put(38, -8){$\overline c_{2,0}$}
\end{overpic}}}  \hdots $ \\ \vspace*{5mm}
$\hdots \vcenter{\hbox{\begin{overpic}[scale=.075]{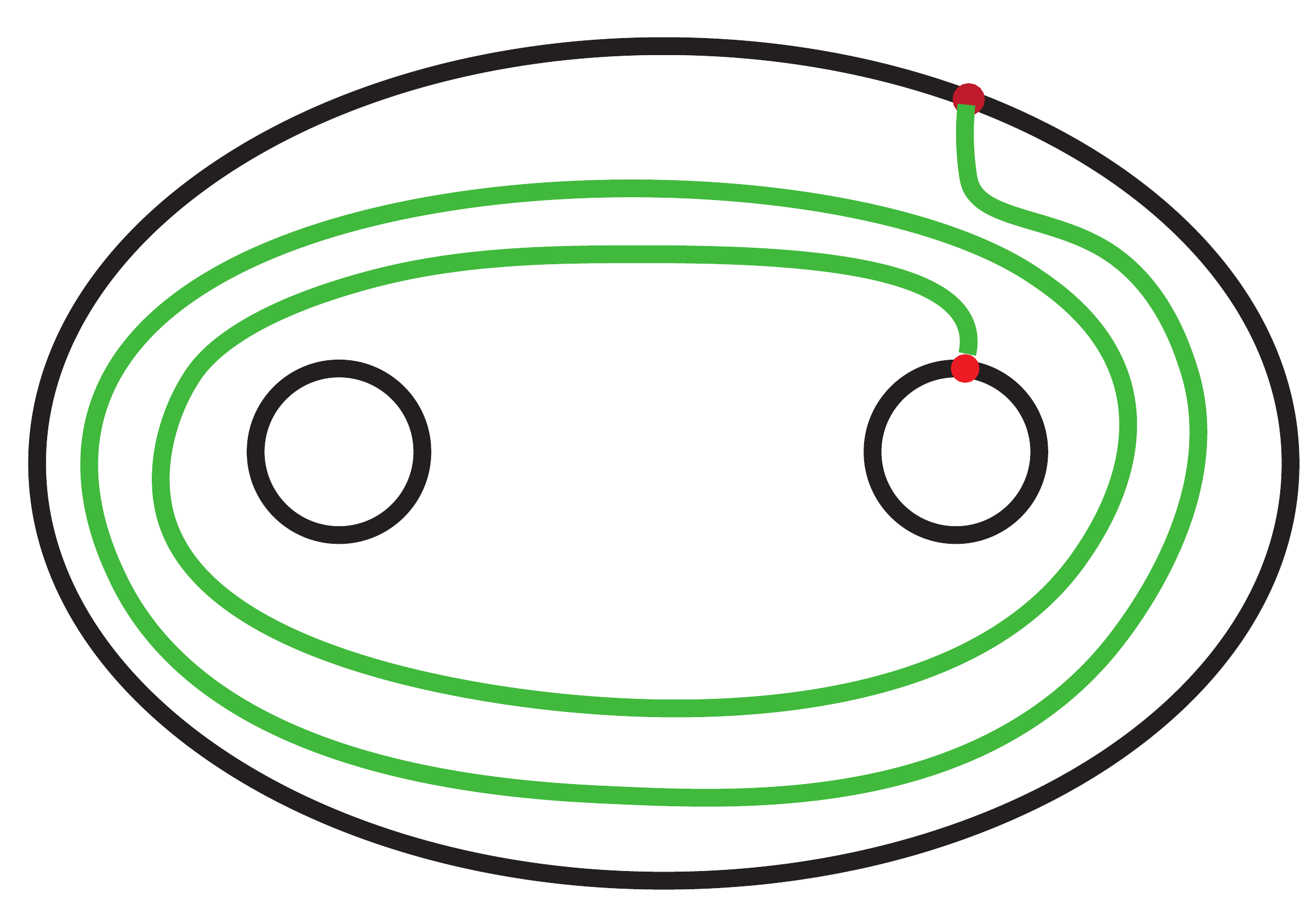}
\put(32, -8){$\overline c_{0,-2}$}
\end{overpic}}}  \hspace{2mm}
\vcenter{\hbox{\begin{overpic}[scale=.075]{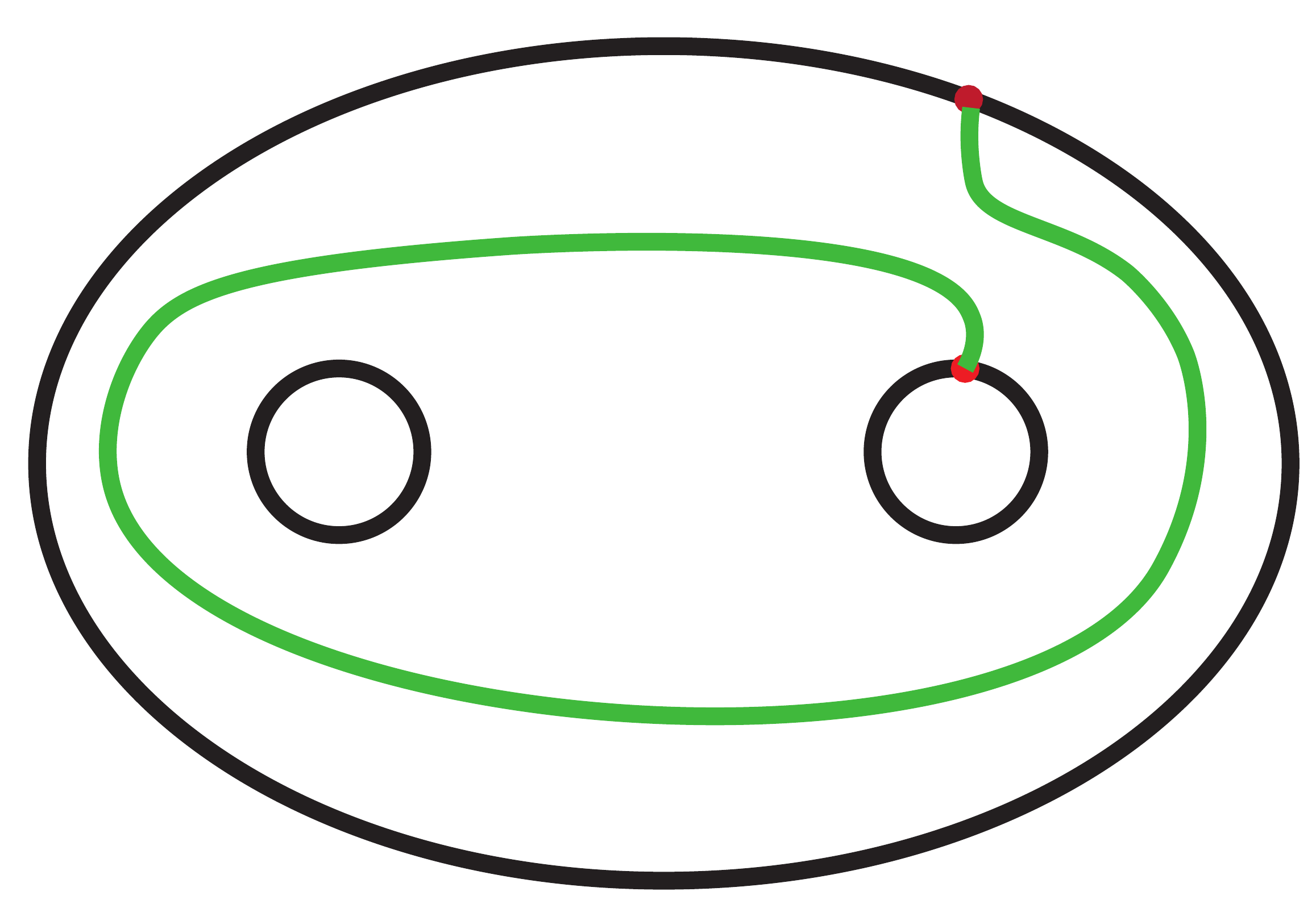}
\put(34, -8){$\overline c_{0,-1}$}
\end{overpic}}} \hspace{2mm}
\vcenter{\hbox{\begin{overpic}[scale=.075]{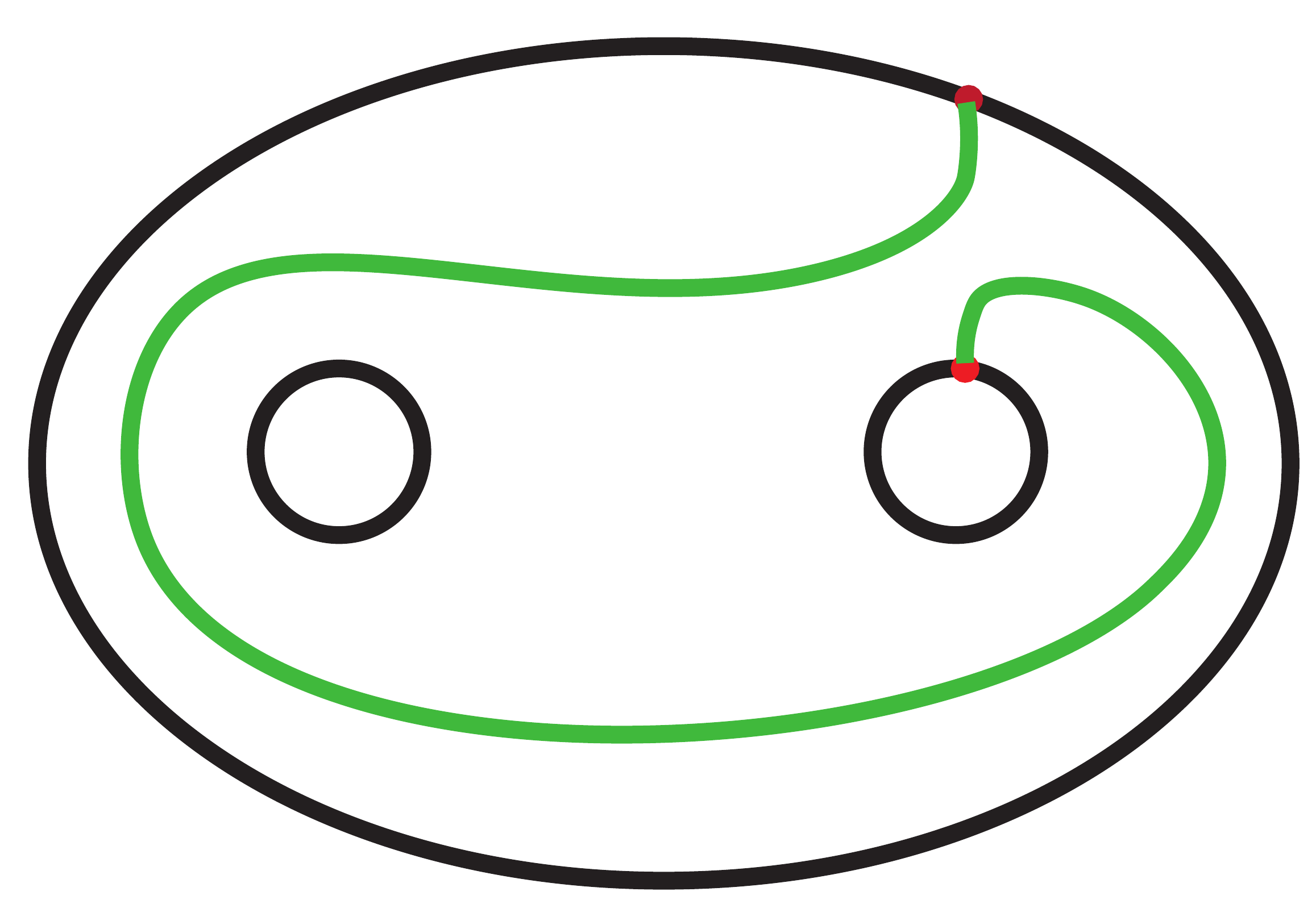}
\put(37, -8){$\overline c_{0,1}$}
\end{overpic}}} \hspace{2mm}
\vcenter{\hbox{\begin{overpic}[scale=.075]{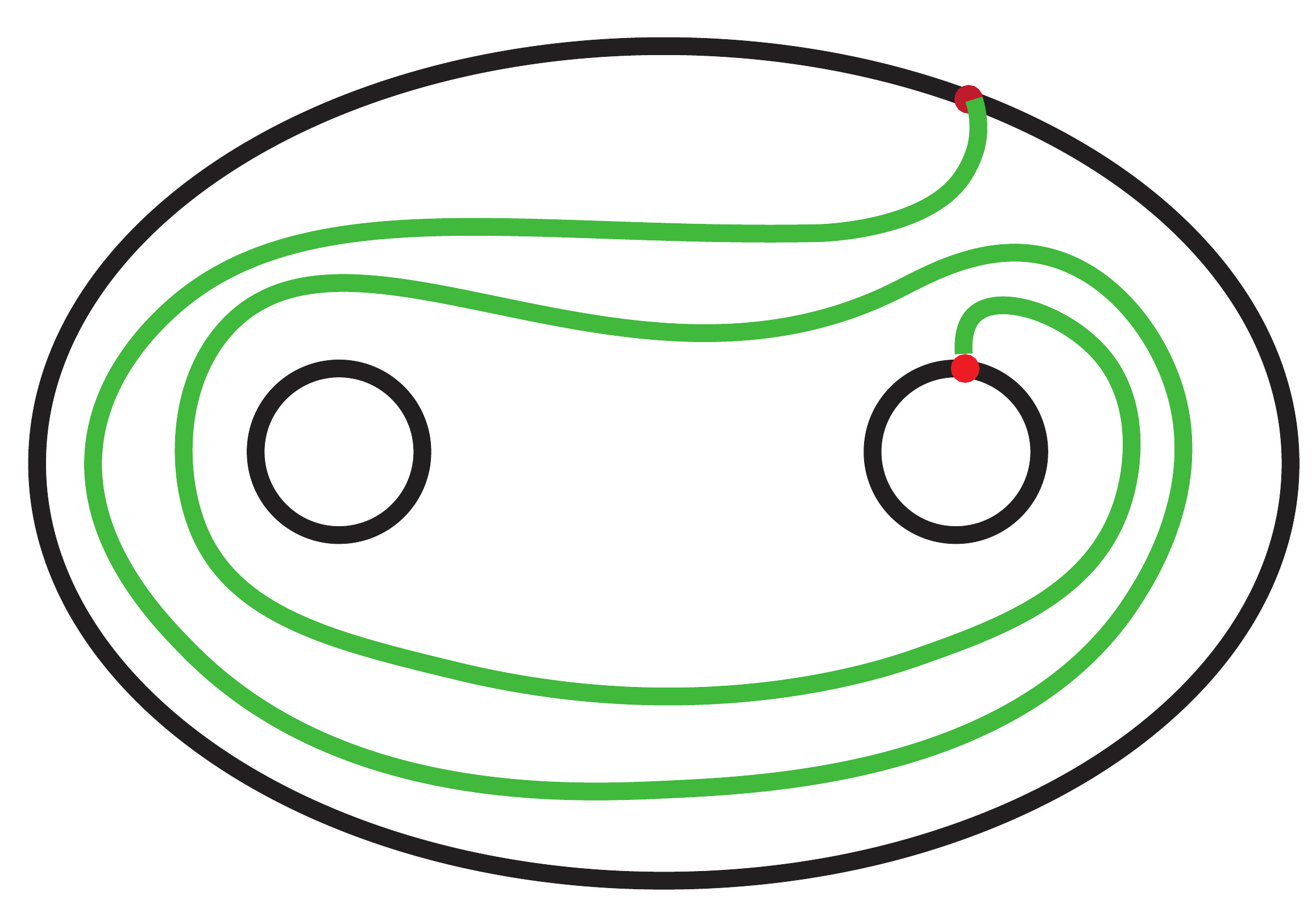}
\put(37, -8){$\overline c_{0,2}$} 
\end{overpic}}} \hdots $ \\
\vspace*{5mm}
$\hdots \vcenter{\hbox{\begin{overpic}[scale=.075]{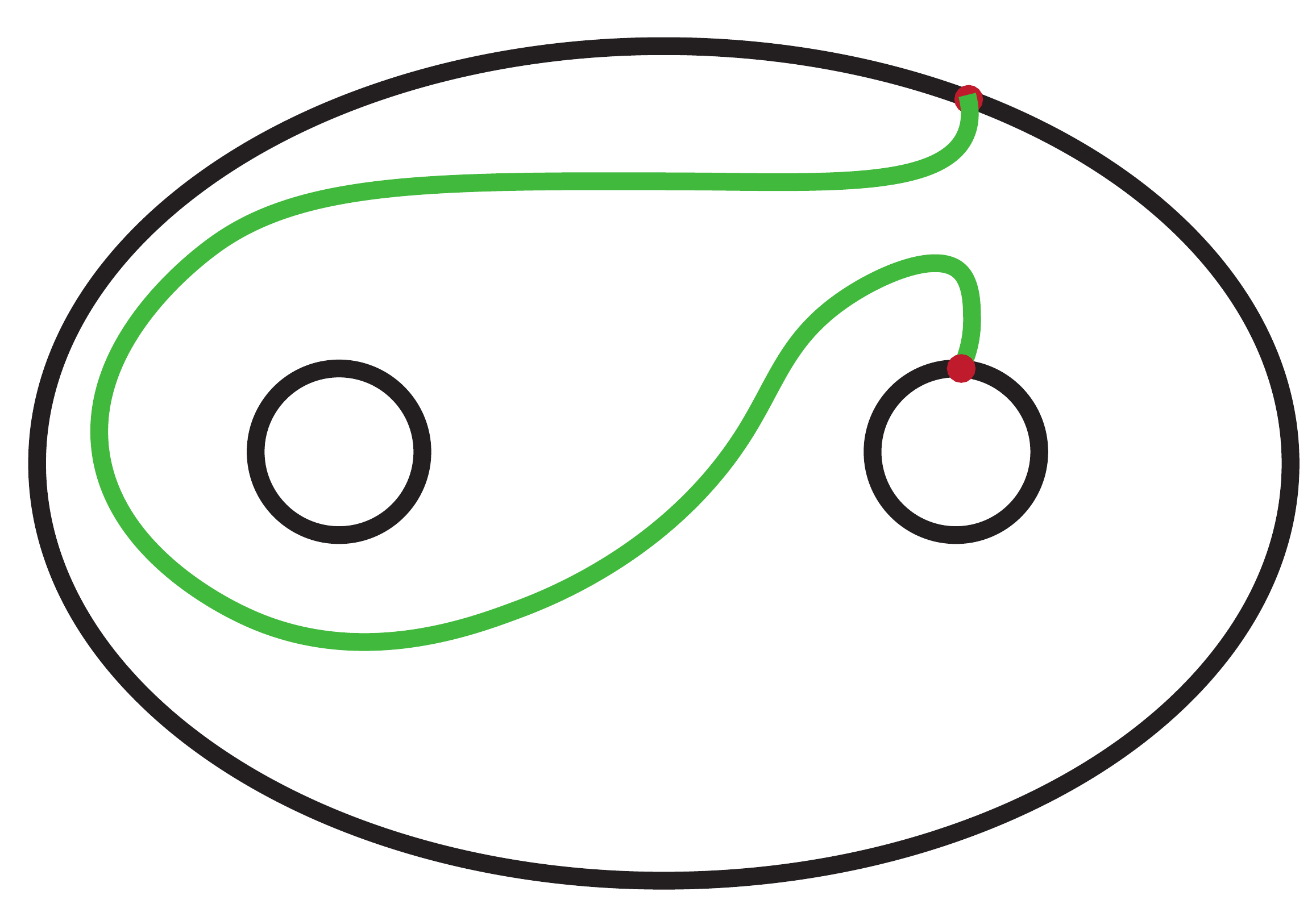}
\put(32, -8){$\overline c_{-1,1}$}
\end{overpic}}}  \hspace{2mm}
\vcenter{\hbox{\begin{overpic}[scale=.075]{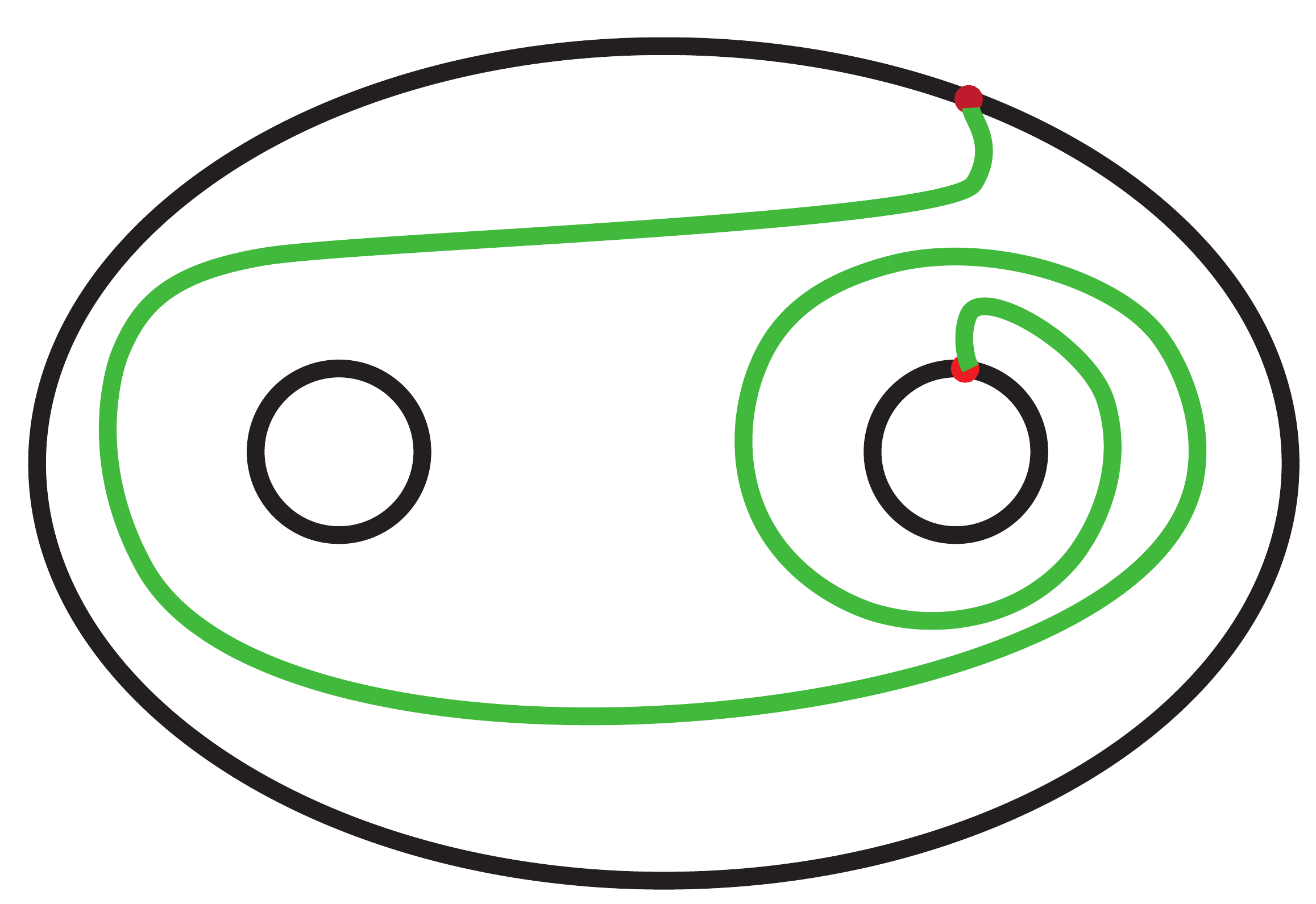}
\put(34, -8){$\overline c_{1,1}$}
\end{overpic}}} \hspace{2mm}
\vcenter{\hbox{\begin{overpic}[scale=.075]{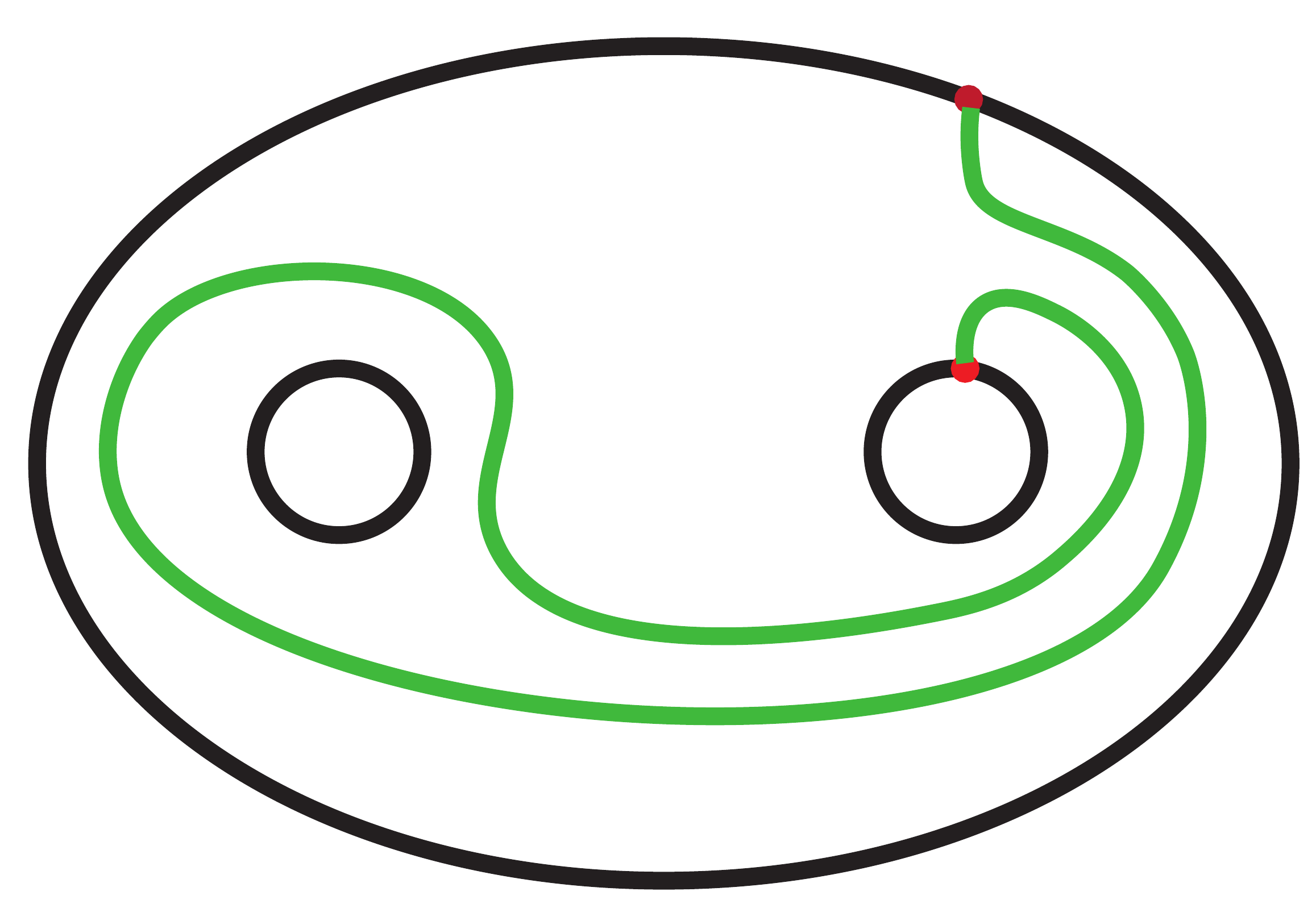}
\put(37, -8){$\overline c_{1,-1}$}
\end{overpic}}} \hdots $
\vspace*{5mm}
\caption{Relative curves in $(\Sigma_{0,3} \times I; u',v')$.}
\label{rkbsmbasisf03uvprime}
\end{figure}

 Since $S_{2,\infty}(H_{2},u',v')$ is generated by $\overline{c}_{q,n}a_{1}^m$, its submodule $\mathcal J_{2}$ is generated by $\overline{C}(q,n)a_{1}^{m} := \omega(\overline{c}_{q,n})a_{1}^{m}$. Analogous to the case of $\mathcal J_{1}$, we obtain the following two theorems.
 \begin{theorem}\label{thm:C(q,n)-formula}
For $q \in \mathbb{N} \cup \{0\}$ and $n \in \mathbb{Z}$,
      \begin{eqnarray*}
        \overline{C}(q,n)S_{m}(a_{1}) &=& (-A^{q+n+2}+A^{-q-n-2})S_{q}(a_{2})S_{n}(a_{3})S_{m}(a_{1})\\
                &+&(-A^{q+n}+A^{-q-n})S_{q-1}(a_{2})S_{n-1}(a_{3})S_{m+1}(a_{1})\\
                &+&(-A^{q+n}+A^{-q-n})S_{q-1}(a_{2})S_{n-1}(a_{3})S_{m-1}(a_{1})\\
                &+&(-A^{q+n-2}+A^{-q-n+2})S_{q-2}(a_{2})S_{n-2}(a_{3})S_{m}(a_{1})
    \end{eqnarray*}
     where $S_{n}(a_{3})$ is the Chebyshev polynomial of the second kind extended to negative indices satisfying $S_{n}(a_{3}) = -S_{-n-2}(a_3)$ for $n\leq -2$.
\end{theorem}

\begin{lemma}\label{lem:negative-Cmn}
For $q,n \in \mathbb{Z}$,
    $$\overline{C}(q,n) = -\overline{C}(-q,-n).$$
\end{lemma}

Combining Corollary \ref{j1mainresult} and Theorem \ref{thm:C(q,n)-formula} with Theorem \ref{kbsmheegaard} gives us the desired result for the structure of $\mathcal S_{2,\infty}((S^1\times S^2)\ \# \ (S^1 \times S^2))$. 

\begin{remark}

Note that we can also get $(S^1\times S^2)\ \# \ (S^1 \times S^2)$ from $H_1 \ \# \ H_1$ by adding a $2$-handle each to the curves $\beta$ and $\eta$. This is equivalent to starting from $H_2$ and adding three $2$-handles. The attaching curve of the $2$-handle attached in the middle can be made contractible after handle slidings. Therefore, the relations in \cite{blp} are contained in $\mathcal J_1 + \mathcal J_2$. 
    
\end{remark}

     \section{Another counterexample to Marché's conjecture}\label{props1s2s1s2}

     In \cite{basissurfacetimess1}, the following conjecture due to Marché was stated for the the skein module of closed oriented $3$-manifolds over the ring $\mathbb Z[A^{\pm 1}]$. 
   \begin{conjecture}\cite{basissurfacetimess1}\label{marcheconj}

Let $M$ be a closed oriented $3$-manifold. Then there exists an integer $d > 0$ and finitely generated $\mathbb Z[A^{\pm 1}]$-modules $N_k$ so that $$\mathcal S_{2,\infty} (M) \cong \mathbb Z[A^{\pm 1}]^d \oplus \bigoplus_{k \geq 1} N_k,$$ where, furthermore, the module $N_k$ is a $(A^k - A^{-k})$-torsion module, for each integer $k$.
         
     \end{conjecture}

 Note that the skein module of $S^1 \times S^2$, which was computed in \cite{s1s2}, satisfies Conjecture \ref{marcheconj}. We now show that $\mathcal S_{2,\infty}((S^1 \times S^2) \ \# \ (S^1 \times S^2))$ does not split into free and torsion submodules. In particular, this shows that the skein module of $(S^1 \times S^2) \ \# \ (S^1 \times S^2)$ provides a counterexample to Conjecture \ref{marcheconj}.
 

\begin{lemma}\label{emptylink}

 The empty link $\varnothing \in \mathcal{S}_{2,\infty}((S^{1}\times S^{2}) \ \# \ (S^{1}\times S^{2}))$ is not a torsion element.

\end{lemma}

\begin{proof}
Let $R$ be a ring with $(A^{k}-1)$ invertible for any $k$. Then from the main result of \cite{s1s2} we get that $\mathcal{S}_{2,\infty}(S^{1}\times S^{2}; R,A) \cong R$. Combining this with the main result of \cite{connsum} gives us that $\mathcal{S}_{2,\infty}((S^{1}\times S^{2}) \ \# \ (S^{1}\times S^{2}); R, A) \cong R$. Both these skein modules are all generated by empty links.
\end{proof}

\begin{theorem}\label{countermarches1s2s1s2}

 $\mathcal{S}_{2,\infty}((S^{1}\times S^{2}) \ \# \ (S^{1}\times S^{2}))$ does not split into the direct sum of free and torsion submodules. In particular, Conjecture \ref{marcheconj} does not hold.
 
\end{theorem}

\noindent Suppose we have such a decomposition for $\mathcal{S}_{2,\infty}((S^{1}\times S^{2}) \ \# \ (S^{1}\times S^{2}))$. Then by Lemma \ref{emptylink}, $d=1$. Thus, there exists an isomorphism $$\bar{\iota}:\mathcal{S}_{2,\infty}((S^{1}\times S^{2}) \ \# \ (S^{1}\times S^{2})) \longrightarrow \mathbb{Z}[A^{\pm1}]\oplus N,$$ where $N$ is the torsion submodule in the decomposition. Pre-composing this with the projection onto the first component of the direct sum, we obtain a module homomorphism $$\iota:\mathcal{S}_{2,\infty}((S^{1}\times S^{2}) \ \# \ (S^{1}\times S^{2})) \longrightarrow \mathbb{Z}[A^{\pm1}].$$ Denote $\iota(S_{n_{1}}(a_{1})S_{n_{3}}(a_{3})S_{n_{2}}(a_{2}))$ by $\alpha_{n_{1},n_{3},n_{2}}$. We will prove Theorem \ref{countermarches1s2s1s2} using the following sequence of lemmas.

\begin{lemma}\label{00n}

$\alpha_{0,0,0}\neq 0$ and $\alpha_{0,0,n}=0$, for all $n>0$.
    
\end{lemma}

\begin{proof}

Lemma \ref{emptylink} implies that $\alpha_{0,0,0}\neq 0$.  We now prove $\alpha_{0,0,n}=0$, for all $n>0$, by induction on $n$. When $n=1$, from the relator $  \overline{C}(1,0)S_{0}(a_{1})$, we get $(-A^3+A^{-3})S_{1}(a_{2})=0$. This implies that $(-A^3+A^{-3})\alpha_{0,0,1}=0$ and thus, $\alpha_{0,0,1}=0$. When $n=2$, from the relator $  \overline{C}(2,0)S_{0}(a_{1})$, we get $(-A^4+A^{-4})\alpha_{0,0,2}=0$, which implies $\alpha_{0,0,2}=0$. Now suppose the result holds for $n<k$. We will show that the result holds for $n=k$. 
From the relator $\overline{C}(k,0)S_{0}(a_{1})$, we get $(-A^{k+2}+A^{-k-2})\alpha_{0,0,k}-(-A^{k-2}+A^{-k+2})\alpha_{0,0,k-2}=0$, which implies $\alpha_{0,0,k}=0$ by induction hypothesis.

\end{proof}

\begin{lemma}\label{11n}

$\alpha_{1,1,n}=0$, for all $n\geq 2$.
    
\end{lemma}

\begin{proof}

From the relator $C(1,1)S_{n}(a_{2})$, we get $$(-A^{4}+A^{-4})\alpha_{1,1,n}+(-A^{2}+A^{-2})\alpha_{0,0,n+1}+(-A^{2}+A^{-2})\alpha_{0,0,n-1}=0.$$
An application of Lemma \ref{00n} gives us $\alpha_{1,1,n}=0$, for $n\geq 2$.
    
\end{proof}

\begin{lemma}\label{mmn}

$\alpha_{m,m,n}=0$, for all $n>m\geq 0$.
    
\end{lemma}

\begin{proof}

Let $t=n+2m$. We will prove the result by induction on $t$.
When $t=1$, we need $\alpha_{0,0,1}=0$, which is true by Lemma \ref{00n}. Suppose the result holds for $t<k$. We show the result holds for $t=k$, that is, $n+2m=k>1$.

\begin{enumerate}
    \item When $m=0$, we need to show $\alpha_{0,0,k}=0$. Notice now $k>1$, so this is true by Lemma \ref{00n}.
    \item When $m=1$, we need to show $\alpha_{1,1,k-2}=0$ with $k-2>1$. This is true by Lemma \ref{11n}.
    \item When $m\geq 2$, by relator $C(m,m)S_{n}(a_{2})$, we have, 
    \begin{eqnarray*}
     & & (-A^{2m+2}+A^{-2m-2})\alpha_{m,m,n} +(-A^{2m}+A^{-2m})\alpha_{m-1,m-1,n+1}\\ &+&(-A^{2m}+A^{-2m})\alpha_{m-1,m-1,n-1}  
     +(-A^{2m-2}+A^{-2m+2})\alpha_{m-2,m-2,n}=0.
     \end{eqnarray*}
By induction hypothesis, $\alpha_{m-1,m-1,n+1}=\alpha_{m-1,m-1,n-1}=\alpha_{m-2,m-2,n}=0$, which implies $\alpha_{m,m,n}=0$ , when $2m+n=k$ and $m\geq 2$.
\end{enumerate}
Combining the three cases, we complete the proof.  
\end{proof}

\begin{lemma}\label{nnn}

$(-A^{2n+2}+A^{-2n-2})\alpha_{n,n,n}+(-A^{2n}+A^{-2n})\alpha_{n-1,n-1,n-1}=0$ for all $n\geq 1$, and $\alpha_{n,n,n}\neq 0$ for all $n\geq 0$.

\end{lemma}

\begin{proof}

When $n\geq 1$, from the relator $C(n,n)S_{n}(a_{2})$, we get 
\begin{eqnarray*}
    & & (-A^{2n+2}+A^{-2n-2})\alpha_{n,n,n}+(-A^{2n}+A^{-2n})\alpha_{n-1,n-1,n+1} \\ &+&
    (-A^{2n}+A^{-2n})\alpha_{n-1,n-1,n-1}+(-A^{2n-2}+A^{-2n+2})\alpha_{n-2,n-2,n}=0.
\end{eqnarray*}

From Lemma \ref{mmn}, we get 
$(-A^{2n+2}+A^{-2n-2})\alpha_{n,n,n}+(-A^{2n}+A^{-2n})\alpha_{n-1,n-1,n-1}=0.$ \\ 

Next, we prove $\alpha_{n,n,n} \neq 0$, or all $n\geq 0$, by induction on $n$.
When $n=0$, $\alpha_{0,0,0}\neq 0$ by Lemma \ref{00n}. Suppose the result holds for $n<k$. We show the case for $n=k$. Since we have $$(-A^{2k+2}+A^{-2k-2})\alpha_{k,k,k}+(-A^{2k}+A^{-2k})\alpha_{k-1,k-1,k-1}=0,$$
the induction hypothesis gives us the desired result.

\end{proof}

We are now ready to prove Theorem \ref{countermarches1s2s1s2}.

\begin{proof}
From the equation
$(-A^{2n+2}+A^{-2n-2})\alpha_{n,n,n}+(-A^{2n}+A^{-2n})\alpha_{n-1,n-1,n-1}=0$ and Lemma \ref{nnn},
we get that $breadth(\alpha_{n,n,n})<breadth(\alpha_{n-1,n-1,n-1})$.\footnote{The breadth of a polynomial is defined as the difference between its highest and lowest degrees.} This implies that $$breadth(\alpha_{0,0,0})>breadth(\alpha_{1,1,1})>\cdots >breadth(\alpha_{n,n,n})>\cdots $$
However, since $breadth(\alpha_{n,n,n})\in \mathbb{N}\cup\{0\}$, we get a contradiction. Hence,  $\mathcal{S}_{2,\infty}((S^{1}\times S^{2}) \ \# \ (S^{1}\times S^{2}))$ does not split into direct sum of free and torsion submodules. 

\end{proof}

We conjecture the following.

\begin{conjecture}\hfill

     \item Let $M$ and $N$ be nontrivial closed oriented $3$-manifolds. Then $\mathcal S_{2,\infty}(M \ \# \ N)$ does not split into the sum of free and torsion modules. In particular, $\mathcal S_{2,\infty}(M \ \# \ N)$ does not split into the sum of free and $(A^k - A^{-k})$-torsion modules, for each $k\geq 1$.

\end{conjecture}


    



    

\section{Torsion elements in $\mathcal S_{2,\infty}((S^{1}\times S^{2}) \ \# \ (S^{1}\times S^{2}))$}\label{section: torsionelements}
In this section, we illustrate two families of torsion elements in $\mathcal S_{2,\infty}((S^{1}\times S^{2}) \ \# \ (S^{1}\times S^{2}))$. The following lemma is a key observation used in determining these two families of torsion elements.
\begin{lemma}\label{4.9}
    If $\sum\limits_{i=1}^{\infty}f_{i}(A)S_{n_{i,1}}(a_{1})S_{n_{i,2}}(a_{2})S_{n_{i,3}}(a_{3})=0$ in $\mathcal S_{2,\infty}((S^{1}\times S^{2}) \ \# \ (S^{1}\times S^{2}))$, with $f_{i}(A)\in  \mathbb Z[A^{\pm 1}]$ and $f_{i}(A)\neq 0$ for finitely many $i$, then $f_{i}(1)=0$ and $f_{i}(-1)=0$.
\end{lemma}

\begin{proof}
$$\sum\limits_{i=1}^{\infty}f_{i}(A)S_{n_{i,1}}(a_{1})S_{n_{i,2}}(a_{2})S_{n_{i,3}}(a_{3})=0$$ if and only if  $$\sum\limits_{i=1}^{\infty}f_{i}(A)S_{n_{i,1}}(a_{1})S_{n_{i,2}}(a_{2})S_{n_{i,3}}(a_{3})=\sum\limits_{m,n,q=1}^{\infty}g_{m,n,q}(A)C(m,n)S_{q}(a_{2})+\widetilde{g}_{m,n,q}(A)\overline{C}(q,n)S_{m}(a_{1}),$$
 where $g_{m,n,q}(A)$ and $\widetilde{g}_{m,n,q}(A)$ are in  $\mathbb Z[A^{\pm 1}]$ and are nonzero only for finitely many of them. Since the set $\{S_{m}(a_{1})S_{n}(a_{2})S_{q}(a_{3})\}_{m,n,q\geq 0}$ forms a basis of $\mathcal S_{2,\infty}(H_{2})$, we have $f_{i}(A)=\sum\limits_{i=1}^{N_{i}}r_{i}(A)(A^{k_{i}}-A^{-k_{i}})$ for some integer $N_{i}$ and $r_{i}(A)\in \mathbb Z[A^{\pm 1}]$. Thus, $f_{i}(1)=0$ and $f_{i}(-1)=0$.
\end{proof}

The first family of torsion elements come from the torsion elements in  $\mathcal S_{2,\infty}(S^1 \times S^2)$ that survive "in" $\mathcal S_{2,\infty}((S^{1}\times S^{2}) \ \# \ (S^{1}\times S^{2}))$.  Let us recall the torsion elements in $\mathcal S_{2,\infty}(S^1 \times S^2)$.

\begin{definition}

We consider the skein module of $S^{1}\times S^{2}$ as obtained from the skein module of $S^{1}\times D^{2}$ by attaching a two handle along a meridian curve $w$ and then a three handle.  Let $z$ represent the class of the longitude curve. We follow the notation in \cite{s1s2} that $e_{i}=S_{i}(z)$ and $e'_{1}=e_{1}$, $e'_{2}=e_{2}$, and $e'_{i}=e_{i}+e'_{i-2}$ for $i\geq 2$.
    
\end{definition}

\begin{remark}
The following direct sum decomposition is provided in \cite{s1s2}
$$\mathcal S_{2,\infty}(S^1 \times S^2)=\mathbb Z[A^{\pm 1}]\oplus \bigoplus_{i=1}^{\infty} \mathbb Z[A^{\pm 1}]/(1-A^{2i+4}),$$
$e'_{i}$ are the generators of the $\mathbb Z[A^{\pm 1}]/(1-A^{2i+4})$ direct summand.
    
\end{remark}

\begin{corollary}

Consider the embedding $j:(S^1 \times S^2)-D^{3} \rightarrow (S^{1}\times S^{2}) \ \# \ (S^{1}\times S^{2})$ that sends $z$ to $a_{1}$ and $w$ to $\beta$. Then $(1-A^{2i+4})j_{*}(e_{i}')=0$ and $j_{*}(e_{i}')\neq 0$ in $\mathcal S_{2,\infty}((S^{1}\times S^{2}) \ \# \ (S^{1}\times S^{2}))$ for all $i$.
    
\end{corollary}

\begin{proof}
$(1-A^{2i+4})j_{*}(e_{i}')=0$ is immediately true since $(1-A^{2i+4})e_{i}'=0$ in $\mathcal S_{2,\infty}(S^1 \times S^2)$. Under this embedding, $e_{i}$ is sent to $S_{i}(a_{1})S_{0}(a_{3})S_{0}(a_{2})=S_{i}(a_{1})$, then $j_{*}(e_{i}')$ is a linear combination of $S_{j}(a_{1})$, $j\leq i$, and the term $S_{i}(a_{1})$ is of coefficient $1$. By Lemma \ref{4.9}, $j_{*}(e_{i}')$ is not zero in the skein module of $(S^{1}\times S^{2}) \ \# \ (S^{1}\times S^{2})$.
\end{proof}

In the following theorem, we provide the second family of elements with $(1-A^{2})$-torsion. 

\begin{theorem}
The set of elements of the form
\begin{eqnarray*}
\tau_{m,n,q}&=&\Sigma_{i=0}^{m+n+1}(A^{-m-n-1+2i})S_{m}(a_{1})S_{n}(a_{3})S_{q}(a_{2})\\
&+&\Sigma_{i=0}^{m+n-1}(A^{-m-n+1+2i})S_{m-1}(a_{1})S_{n-1}(a_{3})S_{q+1}(a_{2})\\
&+&\Sigma_{i=0}^{m+n-1}(A^{-m-n+1+2i})S_{m-1}(a_{1})S_{n-1}(a_{3})S_{q-1}(a_{2})\\
&+&\Sigma_{i=0}^{m+n-3}(A^{-m-n+3+2i})S_{m-2}(a_{1})S_{n-2}(a_{3})S_{q}(a_{2})
\end{eqnarray*}
is a family of $(1-A^{2})$-torsion elements, for all $m,q \in \mathbb{N} \cup \{0\}$ and $n \in \mathbb{Z}$.
\end{theorem}
\begin{proof}

When $A=1$, the coefficients of the basis elements are $m+n+2$, $m+n$, $m+n$ and $m+n-2$, respectively. So by Lemma \ref{4.9}, we know $\tau_{m,n,q}$ is not zero in $\mathcal S_{2,\infty}((S^{1}\times S^{2}) \ \# \ (S^{1}\times S^{2}))$. On the other hand, we notice that $(-A+A^{-1})\tau_{m,n,q}=C(m,n)S_{q}(a_{2})=0$ in $\mathcal S_{2,\infty}((S^{1}\times S^{2}) \ \# \ (S^{1}\times S^{2}))$, which implies $\tau_{m,n,q}$ is annihilated by $(1-A^{2})$. Furthermore, for $(1+A)\tau_{m,n,q}$, the coefficients of the basis elements are $2(m+n+2)$, $2(m+n)$, $2(m+n)$ and $2(m+n-2)$, respectively, when $A=1$. For $(1-A)\tau_{m,n,q}$, the coefficients of the basis elements are $2(m+n+2)$, $2(m+n)$, $2(m+n)$ and $2(m+n-2)$, respectively, when $A=-1$. By Lemma \ref{4.9}, $\tau_{m,n,q}$ is neither annihilated by $(1+A)$ nor by $(1-A)$. Therefore, we have $\tau_{m,n,q}$ are $(1-A^{2})$-torsion elements.
    
\end{proof}

\section{Future Directions}\label{fds1s2s1s2}

Keeping our motivation of constructing traces on skein modules in mind, as a next step it would be beneficial to compute that of  $\#_k(S^1 \times S^2)$ over the ring $\mathbb Z[A^{\pm 1}]$. To understand the skein module of connected sums of arbitrary $3$-manifolds, knowing the skein module of $H_n \ \# \ H_m$ would also be beneficial. We note that a result to this end was published in \cite{connsum} but later proved to be false in \cite{counterhandle}. Characterising the complete set of handle sliding relations is the hardest problem for this manifold. Furthermore, computing the skein module of $\mathbb RP^3 \ \# \ L(p,q), (p,q)\neq (2,1)$ would be another interesting project because this is one of the few examples of connected sums of $3$-manifolds for which it is unknown whether the skein module has torsion or not. See Theorem 4.2 in \cite{fundamentals}. The result for the case $(p, q) = (2k, 1)$ is now known due to the work of Belletti and Detcherry \cite{bellettidetcherry} .

\section{Appendix}\label{appendixs1s2s1s2}
In this section we will provide the details for the proof of Theorem~\ref{thm:C(m,n)-formula}.
\subsection{Calculation of formulas for $C(m,n)$ for $m,n \geq 0$}
In this section we will prove Lemmas~\ref{lem:PP(m,n)}, \ref{lem:(m,n)-formula}, and \ref{lem:C(m,n)-formula}. As described in Figure~\ref{fig:C(m,n)}, we obtain $C(m,n)= -A^{3}P(m,n)+A^{-3}N(m,n)$. First, in the following lemma we describe the recurrence relations for $P(m,n)$. 
\begin{customlemma}{A}
The sequence $P(m,n)$ for $m,n \in \mathbb{N} \cup \{0\}$ satisfies the following relation:
\begin{eqnarray*}
P(0,0)&=&-A^{-3}(-A^{2}-A^{-2}),~ P(1,0) = a_{1},\\ P(0,1)&=&a_{3},~P(1,1)= Aa_{1}a_{3}+A^{-1}a_{2},\\
P(m,0) &=& AP(m-1,0)a_{1}-A^{2}P(m-2,0), m\geq 2,\\
P(m,1) &=& AP(m,0)a_{3}+P(m-1,0)a_{2}+A^{-2}P(m-2,1), m\geq 2,\\
P(m,n) &=& AP(m,n-1)a_{3}-A^{2}P(m,n-2), m \geq 0, n\geq 2.
\end{eqnarray*}
\end{customlemma}

\begin{proof}
The initial conditions for $P(1,0)$, $P(0,1)$, and $P(1,1)$ are determined from their diagrams. 
From a direct calculation on the diagram of $P(m,n)$ in $\Sigma_{0,3}$ we obtain the relations. See Figures
~\ref{fig:cal-P(m,0)}, \ref{fig:cal-P(m,n)}, and \ref{fig:cal-P(m,1)}.
    
\end{proof}
 Our strategy is to find another sequence $PP(m,n)$ with Chebyshev recurrence relations in the variables $a_{1}$ and $a_{3}$ so that the sequence $P(m,n)$ can be presented by a combination of $PP(m,n)$. We first define a sequence $\{Q(m,n)\}_{m,n \in \mathbb{N}\cup\{0\}}$ as follows: 
\begin{eqnarray*}
Q(0,0)&=&-A^{-5}, ~Q(1,0) = 0,\\ Q(0,1)&=&0,~Q(1,1)=-A^{-1}a_{2},\\
Q(m,0) &=& AQ(m-1,0)a_{1}-A^{2}Q(m-2,0), m\geq 2,\\
Q(m,1) &=& AQ(m,0)a_{3}+Q(m-1,0)a_{2}+A^{-2}Q(m-2,1), m\geq 2,\\
Q(m,n) &=& AQ(m,n-1)a_{3}-A^{2}Q(m,n-2), m \geq 0, n\geq 2,
\end{eqnarray*}
and
define $\{PP(m,n)\}_{m,n \in \mathbb{N}\cup\{0\}}$ by
$$PP(m,n) = A^{-m-n+1}(P(m,n)+Q(m,n)).$$

\begin{customlemma}{3.6.A}
The sequence $\{PP(m,n)\}_{m,n \in \mathbb{N}\cup\{0\}}$ satisfies
\begin{eqnarray*}
PP(0,0)&=&1, ~PP(1,0) = a_{1},\\ 
 PP(0,1)&=&a_{3},~  PP(1,1)=a_{1}a_{3},\\
 PP(m,0) &=& PP(m-1,0)a_{1}-PP(m-2,0), m\geq 2,\\
PP(m,1) &=& PP(m,0)a_{3}+A^{-2}PP(m-1,0)a_{2}+A^{-4}PP(m-2,1), m\geq 2,\\
 PP(m,n) &=& PP(m,n-1)a_{3}-PP(m,n-2), m\geq 0, n\geq 2.
\end{eqnarray*}
\end{customlemma}
\begin{proof}
We can prove the statement by the following direct calculations:

\begin{minipage}{0.45\textwidth}
    \begin{eqnarray*}
PP(0,0)&=&A(P(0,0)+Q(0,0))\\
                &=&A(A^{-1}+A^{-5}-A^{-5})\\&=&1.\\PP(1,0)&=&P(1,0)+Q(1,0)\\
                &=&a_{1}.
    \end{eqnarray*}
    \end{minipage}
    \begin{minipage}{0.45\textwidth}
    \begin{eqnarray*}
        PP(0,1)&=&P(0,1)+Q(0,1)\\
                &=&a_{3}.\\
        PP(1,1)&=&A^{-1}(P(1,1)+Q(1,1))\\
                &=&A^{-1}(Aa_{1}a_{3}+A^{-1}a_{2}-A^{-1}a_{2})\\&=&a_{1}a_{3}.
    \end{eqnarray*}
    \end{minipage}
    
    \begin{eqnarray*}
        PP(m,0)&=&A^{-m+1}(P(m,0)+Q(m,0))\\
                &=&A^{-m+1}(AP(m-1,0)a_{1}-A^{2}P(m-2,0)+AQ(m-1,0)a_{1}-A^{2}Q(m-2,0))\\
                &=&A^{-m+1}(A(P(m-1,0)+Q(m-1,0))a_{1}-A^{2}(P(m-2,0)+Q(m-2,0)))\\
                &=&A^{-m+1+1}(P(m-1,0)+Q(m-1,0))a_{1}-A^{-m+2+1}(P(m-2,0)+Q(m-2,0))\\
                &=&PP(m-1,0)a_{1}-PP(m-2,0).\\
        PP(m,1)&=&A^{-m}(P(m,1)+Q(m,1))\\
                &=&A^{-m}(AP(m,0)a_{3}+P(m-1,0)a_{2}+A^{-2}P(m-2,1)\\&&+AQ(m,0)a_{3}+Q(m-1,0)a_{2}+A^{-2}Q(m-2,1))\\
                &=&A^{-m+1}(P(m,0)+Q(m,0))a_{3}+A^{-m}(P(m-1,0)+Q(m-1,0))a_{2}\\&&+A^{-m-2}(P(m-2,1)+Q(m-2,1))\\
                &=&PP(m,0)a_{3}+A^{-2}A^{-m+2}(P(m-1,0)+Q(m-1,0))a_{2}\\&&+A^{-4}A^{-m+2}(P(m-2,1)+Q(m-2,1))\\
                &=& PP(m,0)a_{3}+A^{-2}PP(m-1,0)a_{2}+A^{-4}PP(m-2,1).\\
        PP(m,n)&=&A^{-m-n+1}(P(m,n)+Q(m,n))\\
                &=&A^{-m-n+1}(AP(m,n-1)a_{3}-A^{2}P(m,n-2)+AQ(m,n-1)a_{3}-A^{2}Q(m,n-2))\\
                &=&A^{-m-n+1}(A(P(m-1,0)+Q(m-1,0))a_{3}-A^{2}(P(m-2,0)+Q(m-2,0)))\\
                &=&A^{-m+1-n+1}(P(m-1,0)+Q(m-1,0))a_{3}-A^{-m+2-n+1}(P(m-2,0)+Q(m-2,0))\\
                &=&PP(m,n-1)a_{3}-PP(m,n-2).
    \end{eqnarray*}
\end{proof}
\begin{customlemma}{3.6.B} The sequence $Q(m,n)$ satisfies $Q(m,n) = A^{m+n-5}PP(m-2,n)$, for $m\geq 2, n\geq 0$. Hence, it follows that 
$$P(m,n) = A^{m+n-1}PP(m,n)-A^{m+n-5}PP(m-2,n),$$
for $m \geq 2,n \geq 0$.
\end{customlemma}

\begin{proof}
We will use mathematical induction on $m,n$. For $n=0$
    \begin{eqnarray*}
        Q(2,0)&=&A^{-3}=A^{-3}PP(0,0),\\
        Q(3,0)&=&AQ(2,0)a_{1}-A^{-2}Q(1,0)=A^{-2}a_{1} = A^{3+0-5}PP(1,0).\\
    \end{eqnarray*}
    Let us assume that the statement holds for $2\leq k<m$ and $n=0$.  Then we obtain 
    \begin{eqnarray*}
        Q(m,0)&=&AQ(m-1,0)a_{1}-A^{2}Q(m-2,0)\\
        &=& A\cdot A^{m+n-6}PP(m-1,0)a_{1}- A^{2}\cdot A^{m+n-7}PP(m-2,0)\\
        &=& A^{m+n-5}(PP(m-1,0)a_{1}-PP(m-2,0))\\
        &=& A^{m+n-5}PP(m,0).
    \end{eqnarray*}
For $n=1$
    \begin{eqnarray*}
        Q(2,1)&=&AQ(2,0)a_{3}+Q(1,0)a_{2}+Q(0,1)\\
        &=&A\cdot A^{-3}a_{3} = A^{2+1-5}PP(0,1),
    \end{eqnarray*}
    \begin{eqnarray*}
Q(3,1)&=&AQ(3,0)a_{3}+Q(2,0)a_{2}+A^{-2}Q(1,1)\\
        &=&A\cdot A^{-1}a_{1}a_{3} +A^{-3}a_{2}+A^{-2}(-A^{-1}a_{2}) \\
        &=& A^{-1}a_{1}a_{3} = A^{3+1-5}PP(1,1).
    \end{eqnarray*}
Let us assume that the statement holds for $2\leq k<m$ and $n=1$. Then it follows that
    \begin{eqnarray*}
        Q(m,1)&=&AQ(m,0)a_{3}+Q(m-1,0)a_{2}+A^{-2}Q(m-2,1)\\
        &=&A\cdot A^{m-5}PP(m-2,0)a_{3} +A^{m-6}PP(m-3,0)a_{2}+A^{-2} \cdot A^{m-2+1-5}PP(m-4,1)\\
        &=& A^{m-4}(PP(m-2,0)a_{3} +A^{-2}PP(m-3,0)a_{2}+A^{-4}PP(m-4,1))\\
        &=&A^{m+1-5}PP(m-2,1).
    \end{eqnarray*}
Now let us suppose that the statement holds for $\{(m,l)~|~m\geq 2, 0\leq l<n\}$. Finally we obtain
    \begin{eqnarray*}
       Q(m,n) &=& AQ(m,n-1)a_{3}-A^{2}Q(m,n-2)\\
       &=&A\cdot A^{m+n-1-5}PP(m-2,n-1)a_{3}-A^{2}\cdot A^{m+n-2-5}PP(m-2,n-2)\\
       &=&A^{m+n-5}(PP(m-2,n-1)a_{3}-PP(m-2,n-2))\\
       &=&A^{m+n-5}PP(m-2,n),
    \end{eqnarray*}
    and by mathematical induction the statement $Q(m,n)=A^{m+n-5}PP(m-2,n)$ is proved.
In addition, by the definition of $PP(m,n)$ we obtain
$$P(m,n) = A^{m+n-1}PP(m,n)-A^{m+n-5}PP(m-2,n).$$

\end{proof}


Since $PP(m,n)$ and $NN(m,n)$ satisfy the recurrence relation for Chebyshev polynomials, we obtain Lemma \ref{lem:(m,n)-formula}.

\begin{customlemma}{3.8}
For $m,n \in \mathbb{N} \cup \{0\}$
\begin{eqnarray*}
    PP(m,n) = PP(m,1)S_{n-1}(a_{3})-PP(m,0)S_{n-2}(a_{3}).
\end{eqnarray*}
    Analogously, 
\begin{eqnarray*}
    NN(m,n) = NN(m,1)S_{n-1}(a_{3})-NN(m,0)S_{n-2}(a_{3}).
\end{eqnarray*}
\end{customlemma}

\begin{proof} 
We will prove the lemma by mathematical induction on $n$. Since $S_{-2}(a_{3})=-1$, $S_{-1}(a_{3})=0$, and $S_{0}(a_{3})=1$, the lemma holds for $n=0$ and $1$: 
\begin{eqnarray*}
    PP(m,1)S_{-1}(a_{3})-PP(m,0)S_{-2}(a_{3}) &=& PP(m,0),\\
   PP(m,1)S_{0}(a_{3})-PP(m,0)S_{-1}(a_{3}) &=& PP(m,1).
\end{eqnarray*}
Let us assume that the result holds for $n=k$. Then
\begin{eqnarray*}
    PP(m,k+1) &=& PP(m,k)a_{3}-PP(m,k-1) \\
    &\stackrel{\mathit{Induction}}{=}& (PP(m,1)S_{k-1}(a_{3})-PP(m,0)S_{k-2}(a_{3}))a_{3}\\
    &&-(PP(m,1)S_{k-2}(a_{3})-PP(m,0)S_{k-3}(a_{3})) \\
    &=& (PP(m,1)S_{k-1}(a_{3})a_{3}-PP(m,1)S_{k-2}(a_{3}))\\
    &&-(PP(m,0)S_{k-2}(a_{3})a_{3}-PP(m,0)S_{k-3}(a_{3})) \\
    &=& PP(m,1)(S_{k-1}(a_{3})a_{3}-S_{k-2}(a_{3}))-PP(m,0)(S_{k-2}(a_{3})a_{3}-S_{k-3}(a_{3})) \\
    &=& PP(m,1)S_{k}(a_{3})-PP(m,0)S_{k-1}(a_{3}).
\end{eqnarray*}
\end{proof}
From Lemmas~\ref{lem:PP(m,n)} and \ref{lem:NN(m,n)} we obtain the following equality
\begin{equation}\label{eqn:equality-C(m,n)}
  C(m,n) = -A^{m+n+2}PP(m,n)+A^{m+n-2}PP(m-2,n)+A^{-m-n-2}NN(m,n)-A^{-m-n+2}NN(m-2,n).  
\end{equation}

\begin{customlemma}{3.9}
For all $m,n \in \mathbb{N}\cup \{0\}$,
    \begin{eqnarray*}
        C(m,n) &=& (-A^{m+n+2}+A^{-m-n-2})S_{m}(a_{1})S_{n}(a_{3})\\
                &&+(-A^{m+n}+A^{-m-n})S_{m-1}(a_{1})S_{n-1}(a_{3})a_{2}\\
                &&+(-A^{m+n-2}+A^{-m-n+2})S_{m-2}(a_{1})S_{n-2}(a_{3}).
    \end{eqnarray*}
\end{customlemma}
\begin{proof}
From the equality (\ref{eqn:equality-C(m,n)}) the following equalities are obtained
    \begin{eqnarray*}
        C(m,n) &=& -A^{m+n+2}PP(m,n)+A^{-m-n-2}NN(m,n)\\&&+A^{m+n-2}PP(m-2,n)-A^{-m-n+2}NN(m-2,n)\\
        &\stackrel{\mathit{Lemma~\ref{lem:(m,n)-formula}}}{=}&-A^{m+n+2}(S_{n-1}(a_{3})PP(m,1)-S_{n-2}(a_{3})PP(m,0))\\
        &&+A^{-m-n-2}(S_{n-1}(a_{3})NN(m,1)-S_{n-2}(a_{3})NN(m,0))\\
        &&+A^{m+n-2}(S_{n-1}(a_{3})PP(m-2,1)-S_{n-2}(a_{3})PP(m-2,0))\\
        &&-A^{-m-n+2}(S_{n-1}(a_{3})NN(m-2,1)-S_{n-2}(a_{3})NN(m-2,0))\\
        &=&-A^{m+n+2}S_{n-1}(a_{3})(PP(m,1)-A^{-4}PP(m-2,1))\\
        &&+A^{-m-n-2}S_{n-1}(a_{3})(NN(m,1)-A^{4}NN(m-2,1))\\
        &&-(-A^{m+n+2}+A^{-m-n-2})S_{n-2}(a_{3})S_{m}(a_{1})\\
        &&+(-A^{m+n-2}+A^{-m-n+2})S_{n-2}(a_{3})S_{m-2}(a_{1})\\
        &=&-A^{m+n+2}S_{n-1}(a_{3})(PP(m,0)a_{3}+A^{-2}PP(m-1,0)a_{2})\\
        &&+A^{-m-n-2}S_{n-1}(a_{3})(NN(m,0)a_{3}+A^{2}NN(m-1,0)a_{2})\\
        &&-(-A^{m+n+2}+A^{-m-n-2})S_{n-2}(a_{3})S_{m}(a_{1})\\
        &&+(-A^{m+n-2}+A^{-m-n+2})S_{n-2}(a_{3})S_{m-2}(a_{1})\\
        &=&(-A^{m+n+2}+A^{-m-n-2})S_{n-1}(a_{3})S_{m}(a_{1})a_{3}\\
        &&(-A^{m+n}+A^{-m-n})S_{n-1}(a_{3})S_{m-1}(a_{1})a_{2}\\
        &&-(-A^{m+n+2}+A^{-m-n-2})S_{n-2}(a_{3})S_{m}(a_{1})\\
        &&+(-A^{m+n-2}+A^{-m-n+2})S_{n-2}(a_{3})S_{m-2}(a_{1})\\
        &\substack{\text{Chebyshev}\\{=}\\ \text{relation}}&(-A^{m+n+2}+A^{-m-n-2})S_{m}(a_{1})S_{n}(a_{3})+\\
                &&+(-A^{m+n}+A^{-m-n})S_{m-1}(a_{1})S_{n-1}(a_{3})a_{2}+\\
                &&+(-A^{m+n-2}+A^{-m-n+2})S_{m-2}(a_{1})S_{n-2}(a_{3}).
    \end{eqnarray*}
     The fourth equality above is obtained from the following equalities:
    \begin{eqnarray*}
        &&PP(m,1) = PP(m,0)a_{3}+A^{-2}PP(m-1,0)a_{2}+A^{-4}PP(m-2,1)\\
        &\Leftrightarrow&PP(m,1)-A^{-4}PP(m-2,1) = PP(m,0)a_{3}+A^{-2}PP(m-1,0)a_{2}
    \end{eqnarray*}
    and
    \begin{eqnarray*}
    &&NN(m,1) = NN(m,0)a_{3}+A^{2}NN(m-1,0)a_{2}+A^{4}NN(m-2,1)\\
    &\Leftrightarrow& NN(m,1)-+A^{4}NN(m-2,1) = NN(m,0)a_{3}+A^{2}NN(m-1,0)a_{2}.
    \end{eqnarray*}
    Since $PP(m,0)=NN(m,0) = S_{m}(a_{1})$, the third and fifth equalities hold.
\end{proof}

\subsection{Calculation of formulas for $C(m,-n)$ for $m,n \geq 1$} 
In this section we will prove Lemmas~\ref{lem:PP(m,-n)}, \ref{lem:(m,-n)-formula}, and \ref{lem:C(m,-n)-formula}. As described in Figure~\ref{fig:C(m,-n)}, we obtain $C(m,-n) = P(m,-n) - N(m,-n)$. We first describe the recurrence relations for $P(m,-n)$. 

\begin{customlemma}{B}
For $m,n\geq 1$, $P(m,-n)$ satisfies the following:
\begin{eqnarray*}
    P(1,-1)&=&a_{2}, \\
    P(m,-1) &=& AP(m-1,0)a_{2} + A^{-1}P(m-1,1)=AP(m,1)-A^{2}P(m,0)a_{3}, m\geq 2,\\
    P(m,-2) &=& AP(m,0)+A^{-1}P(m,-1)a_{3},\\
    P(m,-n)&=&A^{-1}P(m,-n+1)a_{3}-A^{-2}P(m,-n+2),n\geq 3.
\end{eqnarray*}
\end{customlemma}

\begin{proof}
The initial condition $P(1,-1)$ is determined from the diagram of $P(1,-1)$. From a direct calculation on the diagram of $P(m,-n)$ on $\Sigma_{0,3}$, we obtain the relations. See Figures
~\ref{fig:cal-P(m,-1)}, \ref{fig:cal-P(m,-2)}, and \ref{fig:cal-P(m,-n)}.

\end{proof}
Our strategy is the same as that in the previous subsection. Let us define $\{Q(m,-n)\}_{m,n \in \mathbb{N}}$ by 
\begin{eqnarray*}
 Q(1,-1)&=&-a_{2}, \\
 Q(m,-1) &=& AQ(m-1,0)a_{2} + A^{-1}Q(m-1,1)=AQ(m,1)-A^{2}Q(m,0)a_{3}, m\geq 2,\\
    Q(m,-2) &=& AQ(m,0)+A^{-1}Q(m,-1)a_{3},\\
    Q(m,-n)&=&A^{-1}Q(m,-n+1)a_{3}-A^{-2}Q(m,-n+2),n\geq 3,
\end{eqnarray*}
and
define $\{PP(m,-n)\}_{m,n\geq 1}$ by
$$PP(m,n) = A^{-m+n+1}(P(m,n)+Q(m,n)).$$

\begin{customlemma}{3.10.A}
$\{PP(m,-n)\}_{m,n\geq 1}$ satisfies
\begin{eqnarray*}
PP(1,-1)&=&0,\\
PP(m,-1) &=& A^{3}PP(m,1)-A^{3}PP(m,0)a_{3}, m\geq 2,\\
PP(m,-2) &=& PP(m,-1)a_{3} +A^{3}PP(m,0),\\ PP(m,-n)&=& PP(m,-n+1)a_{3}-PP(m,-n+2), n\geq 3.
\end{eqnarray*}
\end{customlemma}

\begin{proof}
We prove the statement from the following direct calculations:
    \begin{eqnarray*}
PP(1,-1)&=& A^{-1+1+1}(P(1,-1)+Q(1,-1)) = A(a_{2}-a_{2})=0.\\
PP(m,-1) &=& A^{-m+2}(P(m,-1)+Q(m,-1))\\
&=&A^{-m+2}(AP(m,1)-A^{2}P(m,0)a_{3} + AQ(m,1)-A^{2}Q(m,0)a_{3})\\
&=&A^{-m+3}(P(m,1)+Q(m,1))+A^{-m+4}(P(m,0)+Q(m,0))a_{3}\\
&=&A^{3}\cdot A^{-m}(P(m,1)+Q(m,1))+A^{3}\cdot A^{-m+1}(P(m,0)+Q(m,0))a_{3}\\
&=&A^{3}PP(m,1)+A^{3}PP(m,0).
\end{eqnarray*}
\begin{eqnarray*}
PP(m,-2) &=& A^{-m+3}(P(m,-2)+Q(m,-2))\\ 
&=&A^{-m+3}(AP(m,1)+A^{-1}P(m,-1)a_{3} +AQ(m,1)+A^{-1}Q(m,-1)a_{3})\\
&=& A^{-m+4}(P(m,1)+Q(m,1))+A^{-m+2}(P(m,-1)+Q(m,-1))a_{3} \\
&=&A^{3}\cdot A^{-m+1}(P(m,1)+Q(m,1))+PP(m,-1)a_{3}\\
&=&A^{3}PP(m,1)+PP(m,-1)a_{3}.
\end{eqnarray*}
\begin{eqnarray*}
    PP(m,-n)&=& A^{-m+n+1}(P(m,-n)+Q(m,-n))\\
    &=&A^{-m+n+1}(A^{-1}P(m,-n+1)a_{3}-A^{-2}P(m,-n+2)\\ &&+ A^{-1}Q(m,-n+1)a_{3}-A^{-2}Q(m,-n+2))\\
    &=& A^{-m+n}(P(m,-n+1)+Q(m,-n+1))a_{3}+A^{-m+n-1}(P(m,-n+2)+Q(m,-n+2))\\
    &=& PP(m,-n+1)a_{3}-PP(m,-n+2).
\end{eqnarray*}
\end{proof}

\begin{customlemma}{3.10.B}
  The sequence $Q(m,-n)$ satisfies $Q(m,-n) = A^{m-n-5}PP(m-2,-n)$, for $m\geq 2$. It follows that
    $$P(m,-n)=A^{m-n-1}PP(m,-n)-A^{m-n-5}PP(m-2,-n),$$
    for $m,n\geq 1$.
\end{customlemma}

\begin{proof}
When $n=1$,
    \begin{eqnarray*}
        Q(2,-1)&=&AQ(2,1)-A^{2}Q(2,0)a_{3}\\
        &=&A^{-1}PP(0,1)-A^{-1}PP(0,0)a_{3}=0=A^{2+1-2}PP(0,-1).\\
        Q(m,-1)&=&AQ(m,1)A^{2}Q(m,0)a_{3}\\
        &=&A\cdot A^{m-4}PP(m-2,1)-A^{2}\cdot A^{m-5}PP(m-2,0)a_{3}\\
        &=&A^{m-3}(PP(m-2,1)-PP(m-2,0)a_{3})\\
        &=&A^{m-6}(A^{3}PP(m-2,1)-A^{3}PP(m-2,0)a_{3}).\\
        &=&A^{m-6}PP(m,-1).
        \end{eqnarray*}
        \begin{eqnarray*}
        Q(m,-2)&=&AQ(m,0)+A^{-1}Q(m,-1)a_{3}\\
        &=&A\cdot A^{m-5}PP(m-2,0)+A^{-1}\cdot A^{m-6}PP(m-2,-1)a_{3}\\
        &=&A^{m-4}PP(m-2,0)+A^{m-7}PP(m-2,-1)a_{3}\\
        &=&A^{m-7}(A^{3}PP(m-2,0)+PP(m-2,-1)a_{3})=A^{m-7}PP(m,-2).\\
        Q(m,-n)&=&A^{-1}Q(m,-n+1)a_{3}-A^{-2}Q(m,-n+2)\\
        &=&A^{-1}\cdot A^{m-n+1-5}PP(m-2,-n+1)a_{3}-A^{-2}A^{m-n+2-5}PP(m-2,-n+2)\\
        &=&A^{m-n-5}(PP(m-2,-n+1)a_{3}-PP(m-2,-n+2))\\
        &=&A^{m-n-5}PP(m-2,-n).
    \end{eqnarray*}
    By definition of $PP(m,-n)$ we obtain
    $$P(m,-n)=A^{m-n-1}PP(m,-n)-A^{m-n-5}PP(m-2,-n),$$
    for $m,n\geq 1$.
\end{proof}
Since $PP(m,-n)$ and $NN(m,-n)$ satisfy the recurrence relation for Chebyshev polynomials, we obtain Lemma~\ref{lem:(m,-n)-formula}.



\begin{customlemma}{3.12}
For $m,n \geq 1$
\begin{eqnarray*}
    PP(m,-n) &=& S_{n-2}(a_{3})PP(m,-2) - S_{n-3}(a_{3})PP(m,-1)\\
    &=& A^{3}PP(m,1)S_{n-1}(a_{3})-A^{3}PP(m,0)S_{n}(a_{3}).
\end{eqnarray*}
    Analogously, 
\begin{eqnarray*}
    NN(m,n) &=& S_{n-2}(a_{3})NN(m,-2) - S_{n-3}(a_{3})NN(m,-1)\\
    &=& A^{-3}NN(m,1)S_{n-1}(a_{3})-A^{-3}NN(m,0)S_{n}(a_{3}).
\end{eqnarray*}
\end{customlemma}

\begin{proof}
    The proof of the first equality is analogous to the proof of Lemma~\ref{lem:(m,n)-formula}. The second equality can be obtained as follows:
    \begin{eqnarray*}
        PP(m,-n) &=& S_{n-2}(a_{3})PP(m,-2) - S_{n-3}(a_{3})PP(m,-1)\\
        &=& S_{n-2}(a_{3})(A^{3}PP(m,0)+PP(m,-1)a_{3})-S_{n-3}(a_{3})PP(m,-1)\\
        &=&A^{3}PP(m,0)S_{n-2}(a_{3})+PP(m,-1)S_{n-1}(a_{3})\\
        &=& A^{3}PP(m,0)S_{n-2}(a_{3})+ (A^{3}PP(m,1)-A^{3}PP(m,0))S_{n-1}(a_{3})a_{3}\\
        &=& A^{3}PP(m,0)S_{n-2}(a_{3})+ A^{3}PP(m,1)S_{n-1}(a_{3})-A^{3}PP(m,0)S_{n-1}(a_{3})a_{3}\\
        &=&A^{3}PP(m,1)S_{n-1}(a_{3})-A^{3}PP(m,0)S_{n}(a_{3}).
    \end{eqnarray*}
\end{proof}
From Lemmas~\ref{lem:PP(m,-n)} and \ref{lem:NN(m,-n)} we obtain the following equality
\begin{eqnarray}\label{eqn:C(m,-n)}
    C(m,-n) &=& A^{m-n-1}PP(m,-n)-A^{m-n-5}PP(m-2,-n) \nonumber\\ &&-A^{-m+n+1}NN(m,-n)+A^{-m+n+5}NN(m-2,-n).
\end{eqnarray}
\begin{customthm}{3.14}
    \begin{eqnarray*}
        C(m,-n) &=& -(-A^{m-n+2}+A^{-m+n-2})S_{m}(a_{1})S_{n-2}(a_{3})\\
                &&-(-A^{m-n}+A^{-m+n})S_{m-1}(a_{1})S_{n-1}(a_{3})a_{2}\\
                &&-(-A^{m-n-2}+A^{-m+n+2})S_{m-2}(a_{1})S_{n}(a_{3}),
    \end{eqnarray*}
    for all $m,n\geq 1$.
\end{customthm}

\begin{proof}
From equality (\ref{eqn:C(m,-n)}), we obtain the following equalities.
\begin{eqnarray*}
    C(m,-n) &=& A^{m-n-1}PP(m,-n)-A^{m-n-5}PP(m-2,-n)\\
    &&-A^{-m+n+1}NN(m,-n)+A^{-m+n+5}NN(m-2,-n)\\
    &\stackrel{\mathit{Lemma~\ref{lem:(m,-n)-formula}}}{=}&A^{m-n-1}(A^{3}PP(m,1)S_{n-1}(a_{3})-A^{3}PP(m,0)S_{n}(a_{3}))\\&&-A^{m-n-5}(A^{3}PP(m-2,1)S_{n-1}(a_{3})-A^{3}PP(m-2,0)S_{n}(a_{3}))\\&&-A^{-m+n+1}(A^{-3}NN(m,1)S_{n-1}(a_{3})-A^{-3}NN(m,0)S_{n}(a_{3}))\\&&+A^{-m+n+5}(A^{-3}NN(m-2,1)S_{n-1}(a_{3})-A^{-3}NN(m,0)S_{n}(a_{3}))\\
    &=&A^{m-n+2}(PP(m,1)-A^{-4}PP(m-2,1))S_{n-1}(a_{3})\\
    &&-A^{-m+n-2}(NN(m,1)-A^{4}NN(m-2,1))S_{n-1}(a_{3})\\
    &&+(-A^{m-n+2}+A^{-m+n-2})PP(m,0)S_{n}(a_{3})\\
    &&-(-A^{m-n-2}+A^{-m+n+2})PP(m-2,0)S_{n}(a_{3})\\
    &=&A^{m-n+2}(PP(m,0)a_{3}+A^{-2}PP(m-1,0)a_{2})S_{n-1}(a_{3})\\
    &&-A^{-m+n-2}(NN(m,0)a_{3}+A^{2}NN(m-1,0))S_{n-1}(a_{3})\\
    &&+(-A^{m-n+2}+A^{-m+n-2})PP(m,0)S_{n}(a_{3})\\
    &&-(-A^{m-n-2}+A^{-m+n+2})PP(m-2,0)S_{n}(a_{3})\\
    &=&-(-A^{m-n+2}+A^{m+n-2})PP(m,0)S_{n-1}(a_{3})a_{3}\\
    &&-(-A^{m-n}+A^{-m+n})PP(m-1,0)S_{n-1}(a_{3})a_{2}\\
     &&+(-A^{m-n+2}+A^{-m+n-2})PP(m,0)S_{n}(a_{3})\\
    &&-(-A^{m-n-2}+A^{-m+n+2})PP(m-2,0)S_{n}(a_{3})\\
    &=&-(-A^{m-n+2}+A^{-m+n-2})PP(m,0)S_{n-2}(a_{3})\\
    &&-(-A^{m-n}+A^{-m+n})PP(m-1,0)S_{n-1}(a_{3})a_{2}\\
    &&-(-A^{m-n-2}+A^{-m+n+2})PP(m-2,0)S_{n}(a_{3}),\\
    &=& -(-A^{m-n+2}+A^{-m+n-2})S_{m}(a_{1})S_{n-2}(a_{3})\\
    &&-(-A^{m-n}+A^{-m+n})S_{m-1}(a_{1})S_{n-1}(a_{3})a_{2}\\
    &&-(-A^{m-n-2}+A^{-m+n+2})S_{m-2}(a_{1})S_{n}(a_{3}).
\end{eqnarray*}
The fourth equality is obtained from the following equalites:
    \begin{eqnarray*}
        &&PP(m,1) = PP(m,0)a_{3}+A^{-2}PP(m-1,0)a_{2}+A^{-4}PP(m-2,1)\\
        &\Leftrightarrow&PP(m,1)-A^{-4}PP(m-2,1) = PP(m,0)a_{3}+A^{-2}PP(m-1,0)a_{2}
    \end{eqnarray*}
    and
    \begin{eqnarray*}
    &&NN(m,1) = NN(m,0)a_{3}+A^{2}NN(m-1,0)a_{2}+A^{4}NN(m-2,1)\\
    &\Leftrightarrow& NN(m,1)-+A^{4}NN(m-2,1) = NN(m,0)a_{3}+A^{2}NN(m-1,0)a_{2}.
    \end{eqnarray*}
    Since $PP(m,0) = S_{m}(a_{1})$, the last equality follows.
\end{proof}

\newpage
    \begin{figure}[H]
    \centering
    \begin{eqnarray*}
    P(m,0) & = &
\vcenter{\hbox{\begin{overpic}[scale=.1]{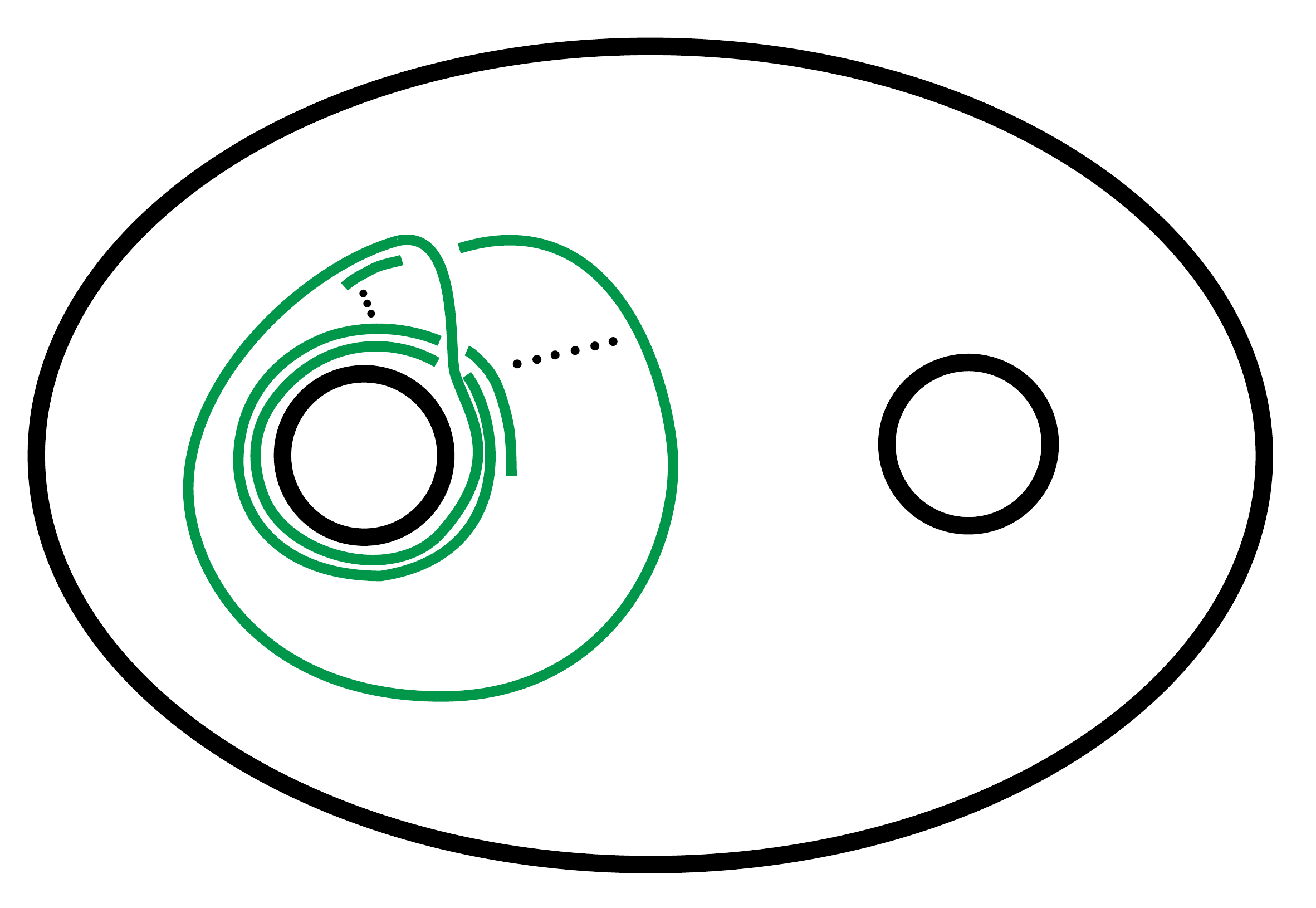} 
\put(48,42){\tiny{${m}$}}
\end{overpic}}} 
 = A \vcenter{\hbox{\begin{overpic}[scale=.1]{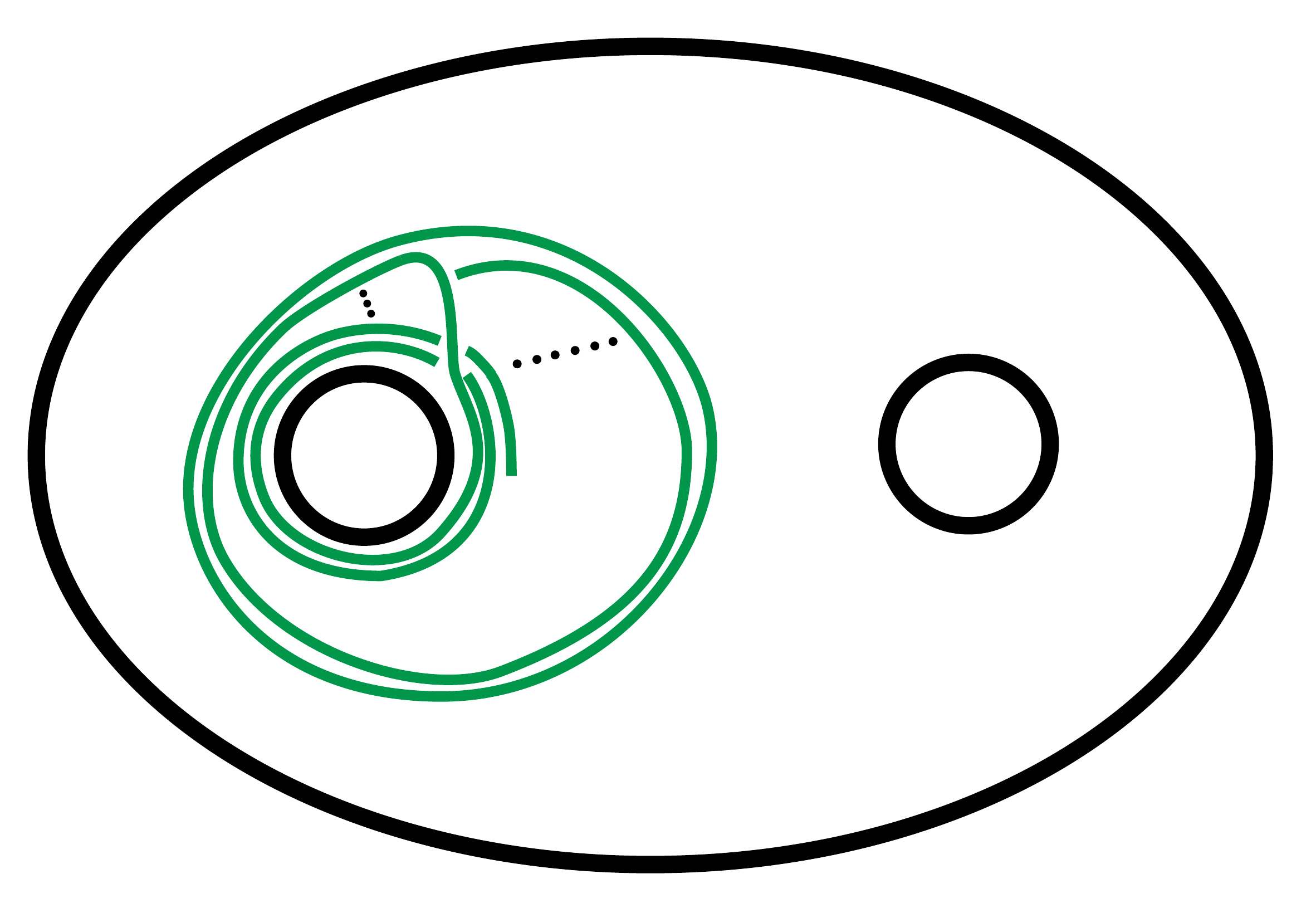}
\end{overpic}}} + A^{-1}
\vcenter{\hbox{\begin{overpic}[scale=.1]{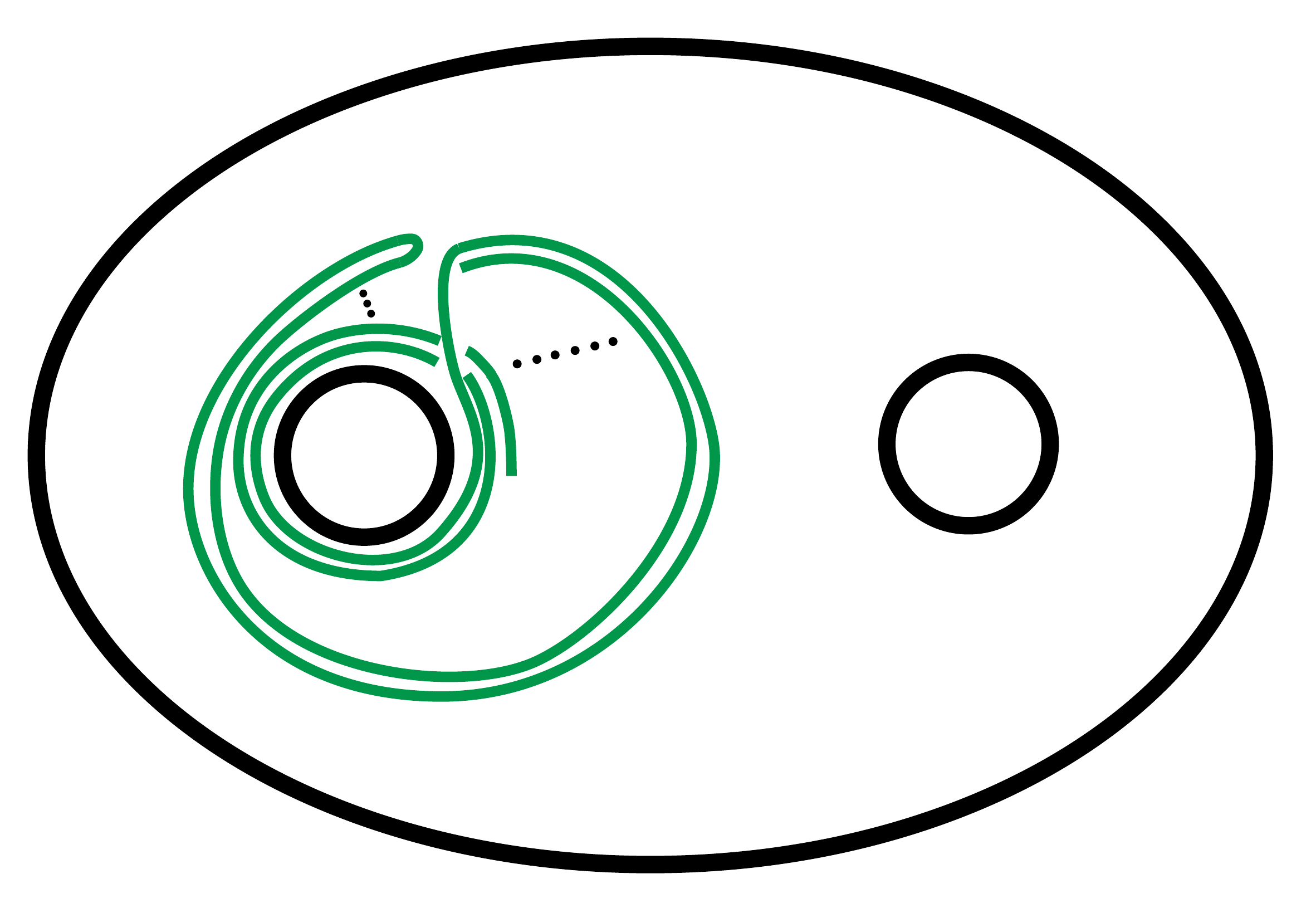}
\end{overpic}}}\\ 
& = & A \vcenter{\hbox{\begin{overpic}[scale=.1]{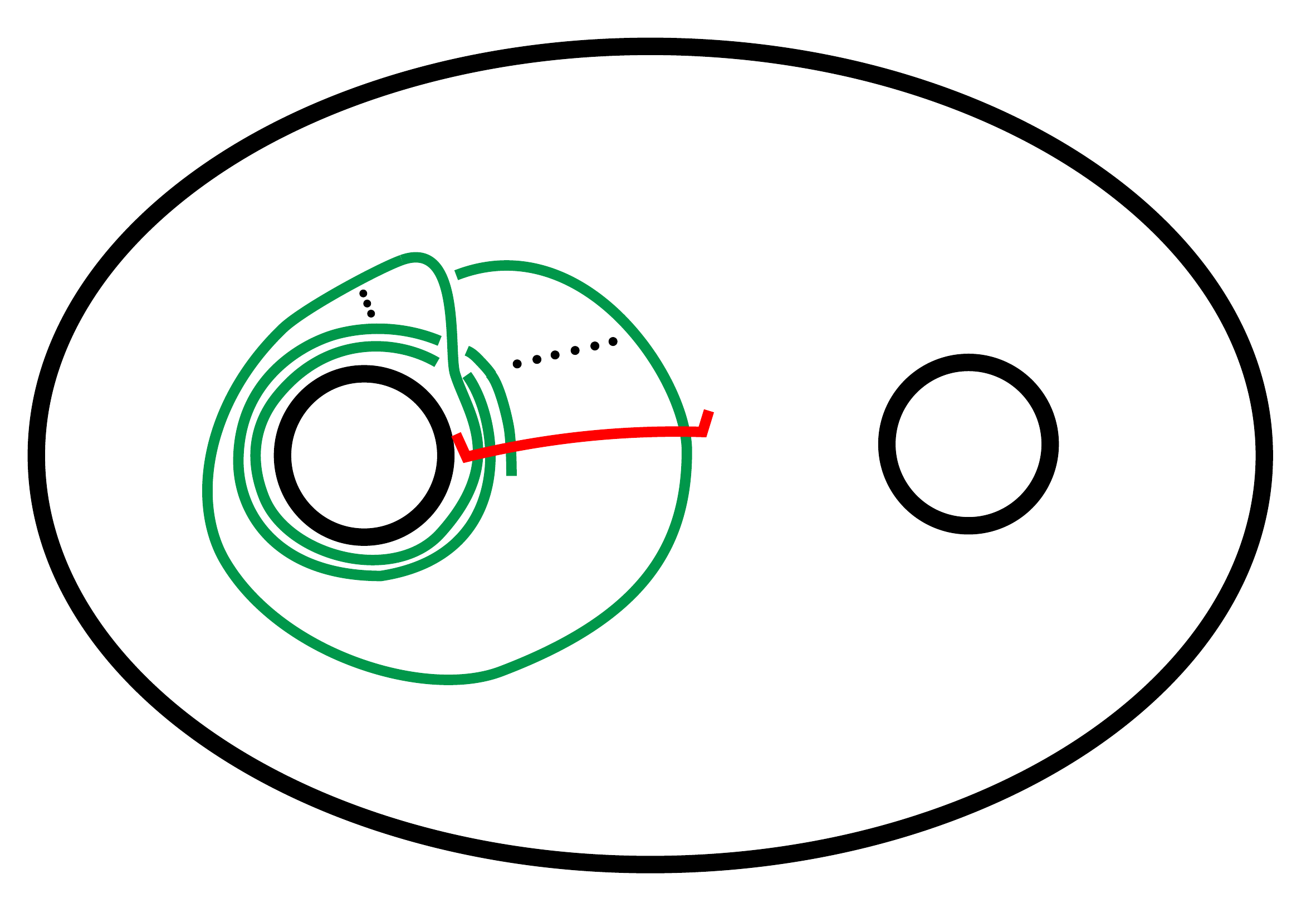}
\put(42,35){\tiny{${m-1}$}}
\end{overpic}}} + A^{-1}
\vcenter{\hbox{\begin{overpic}[scale=.1]{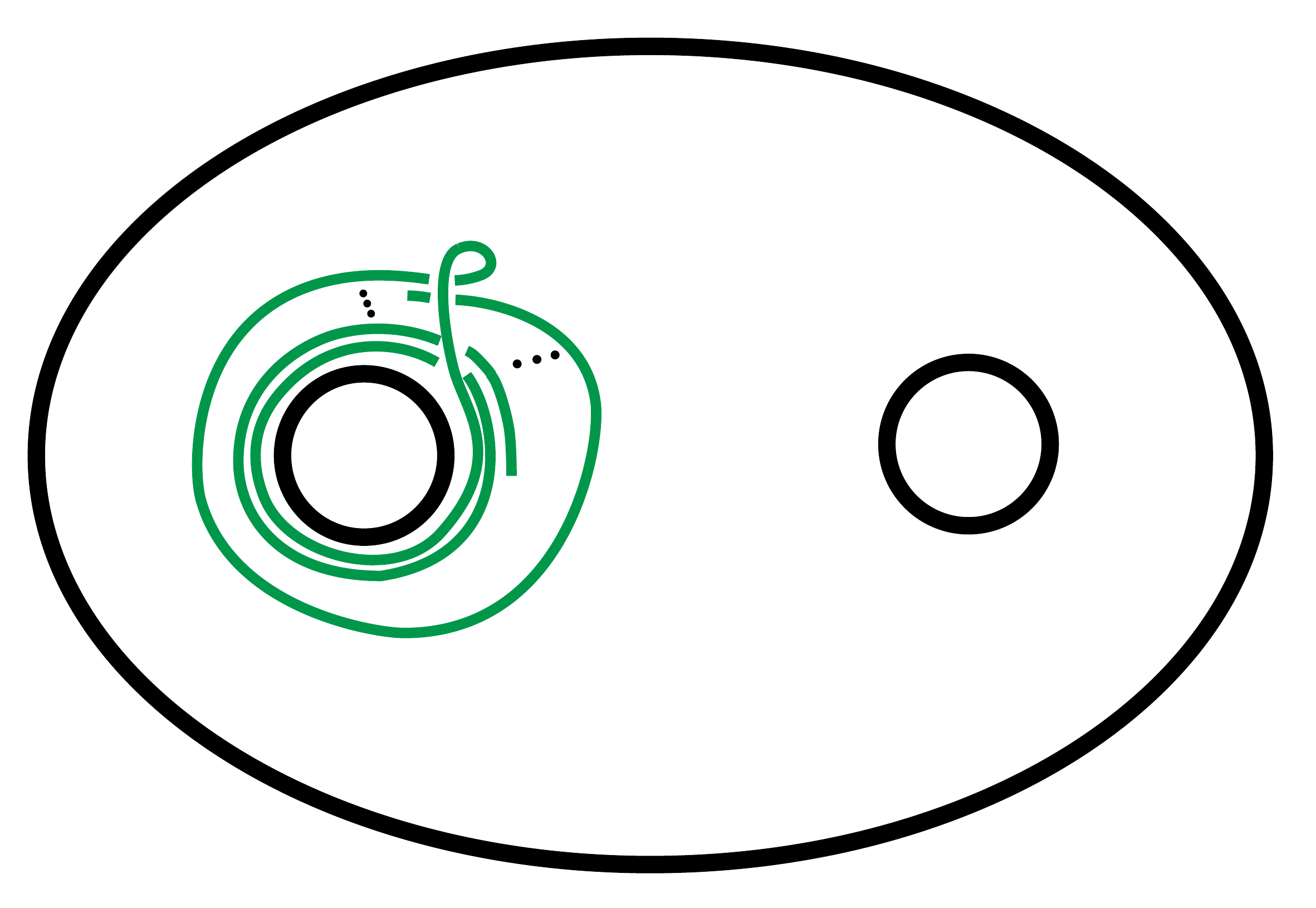}
\end{overpic}}}\\ 
& = & A \vcenter{\hbox{\begin{overpic}[scale=.1]{Pm,0-A-2.pdf}
\put(42,35){\tiny{${m-1}$}}
\end{overpic}}} - A^{2}
\vcenter{\hbox{\begin{overpic}[scale=.1]{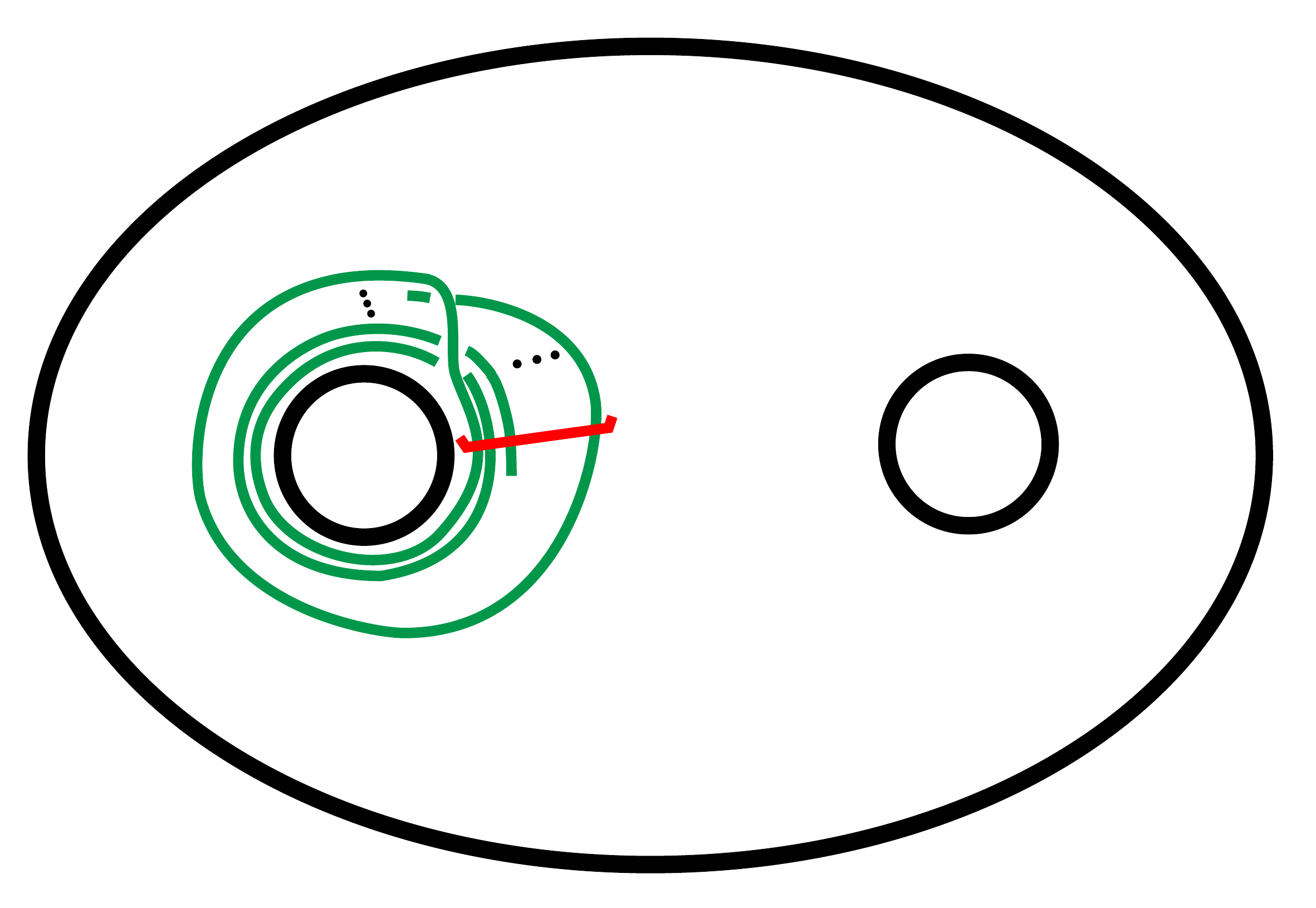}
\put(42,35){\tiny{${m-2}$}}
\end{overpic}}}\\
&=& AP(m-1,0)-A^{2}P(m-2,0).
\end{eqnarray*}
    \caption{$P(m,0)=AP(m-1,0)-A^{2}P(m-2,0)$ for $m\geq 2$. \\ \ \\ \ \\}
    \label{fig:cal-P(m,0)}
\end{figure}
\begin{figure}[H]
    \centering
    \begin{eqnarray*}
    P(m,n) & = &
\vcenter{\hbox{\begin{overpic}[scale=.1]{Pm,n.pdf} 
\end{overpic}}} 
 =  A \vcenter{\hbox{\begin{overpic}[scale=.1]{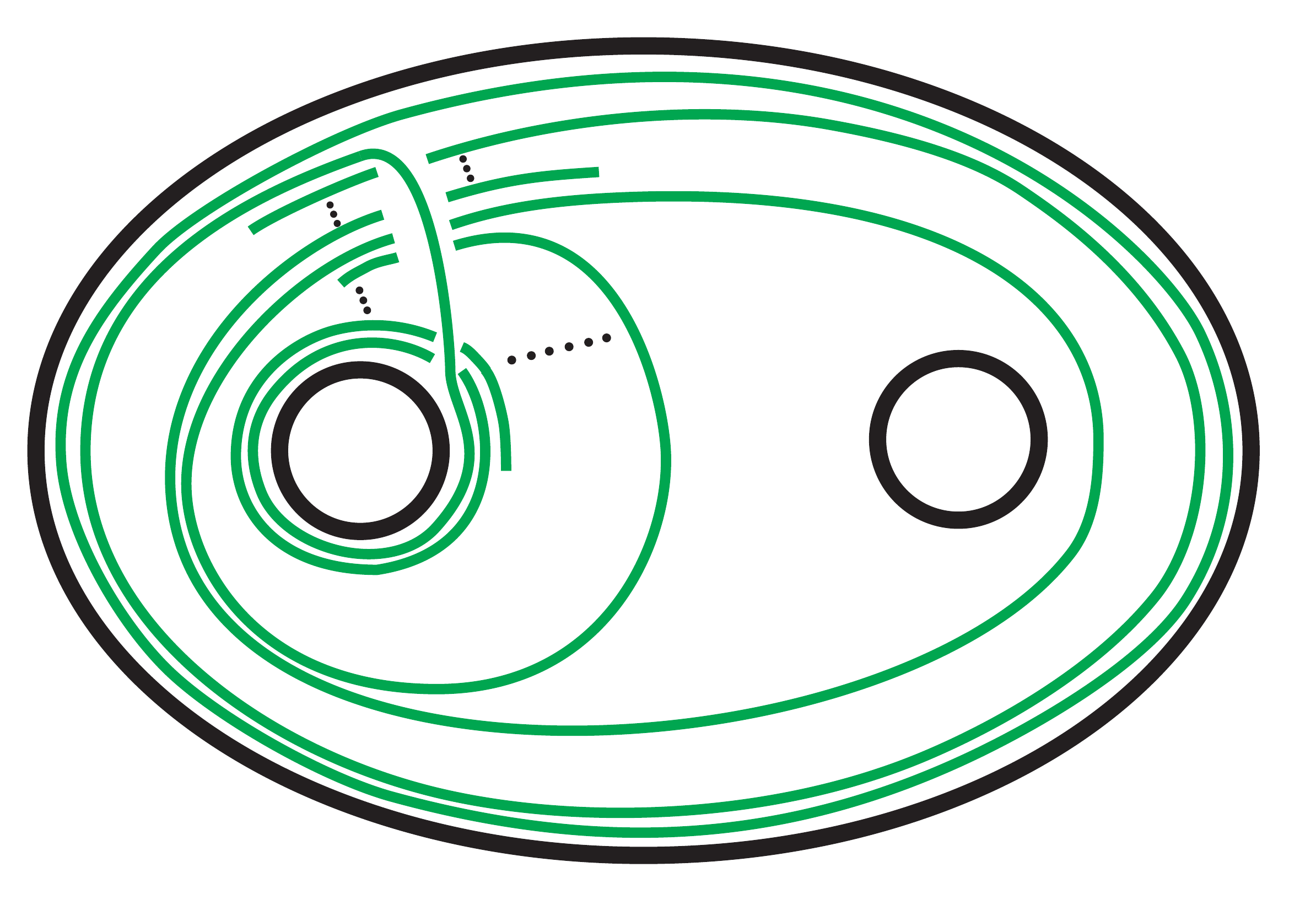}
\end{overpic}}} + A^{-1}
\vcenter{\hbox{\begin{overpic}[scale=.1]{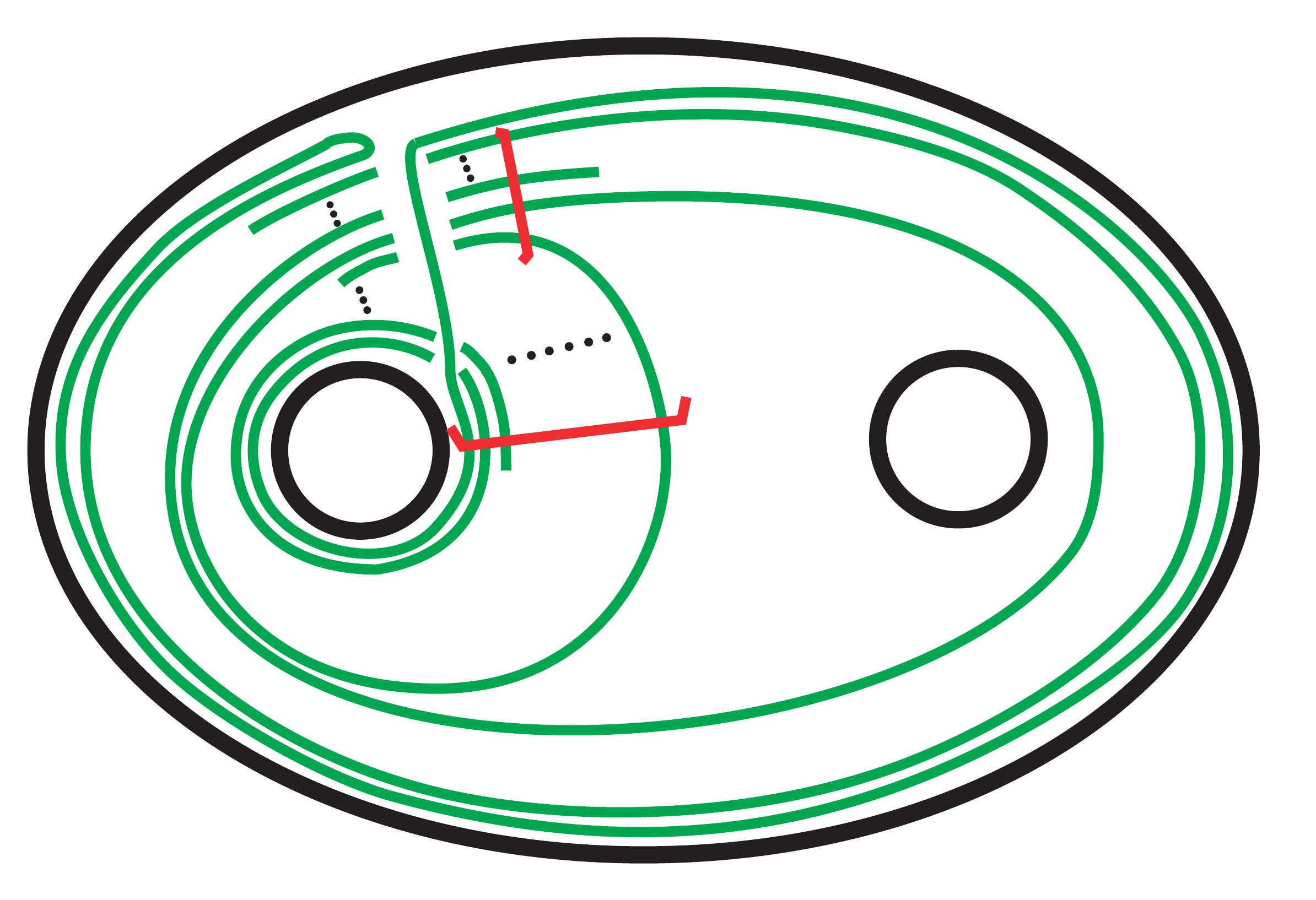}
\put(42,35){\tiny{${m}$}}
\put(50,63){\tiny{${n-2}$}}
\end{overpic}}}\\ 
& = & A \vcenter{\hbox{\begin{overpic}[scale=.1]{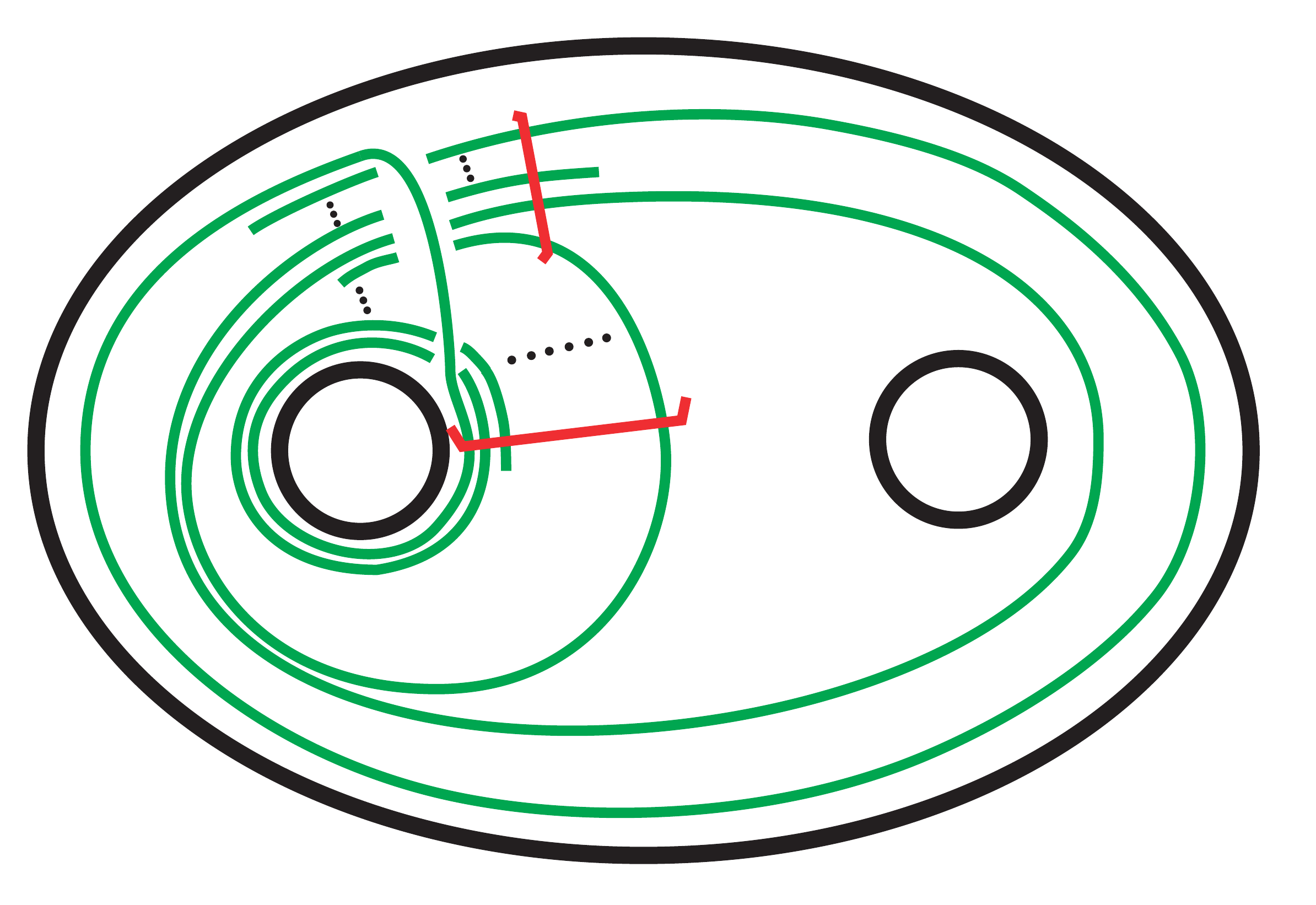}
\put(42,35){\tiny{${m}$}}
\put(50,63){\tiny{${n-1}$}}
\end{overpic}}}a_{3} + A^{-1}
\vcenter{\hbox{\begin{overpic}[scale=.1]{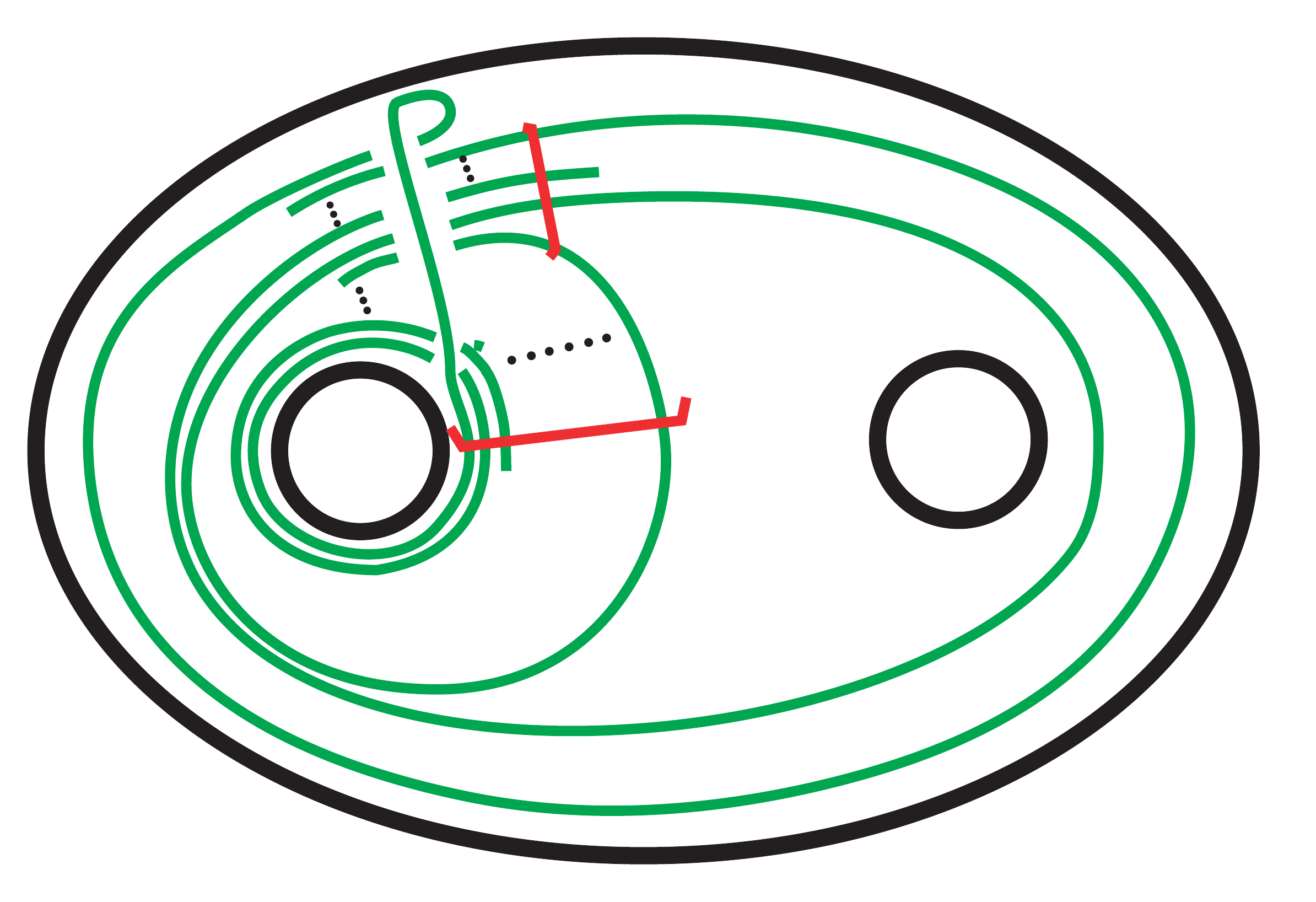}
\put(42,35){\tiny{${m}$}}
\put(50,63){\tiny{${n-2}$}}
\end{overpic}}}\\ 
& = & A \vcenter{\hbox{\begin{overpic}[scale=.1]{Pm,n-A-2.pdf}
\put(42,35){\tiny{${m}$}}
\put(50,63){\tiny{${n-1}$}}
\end{overpic}}}a_{3} - A^{2}
\vcenter{\hbox{\begin{overpic}[scale=.1]{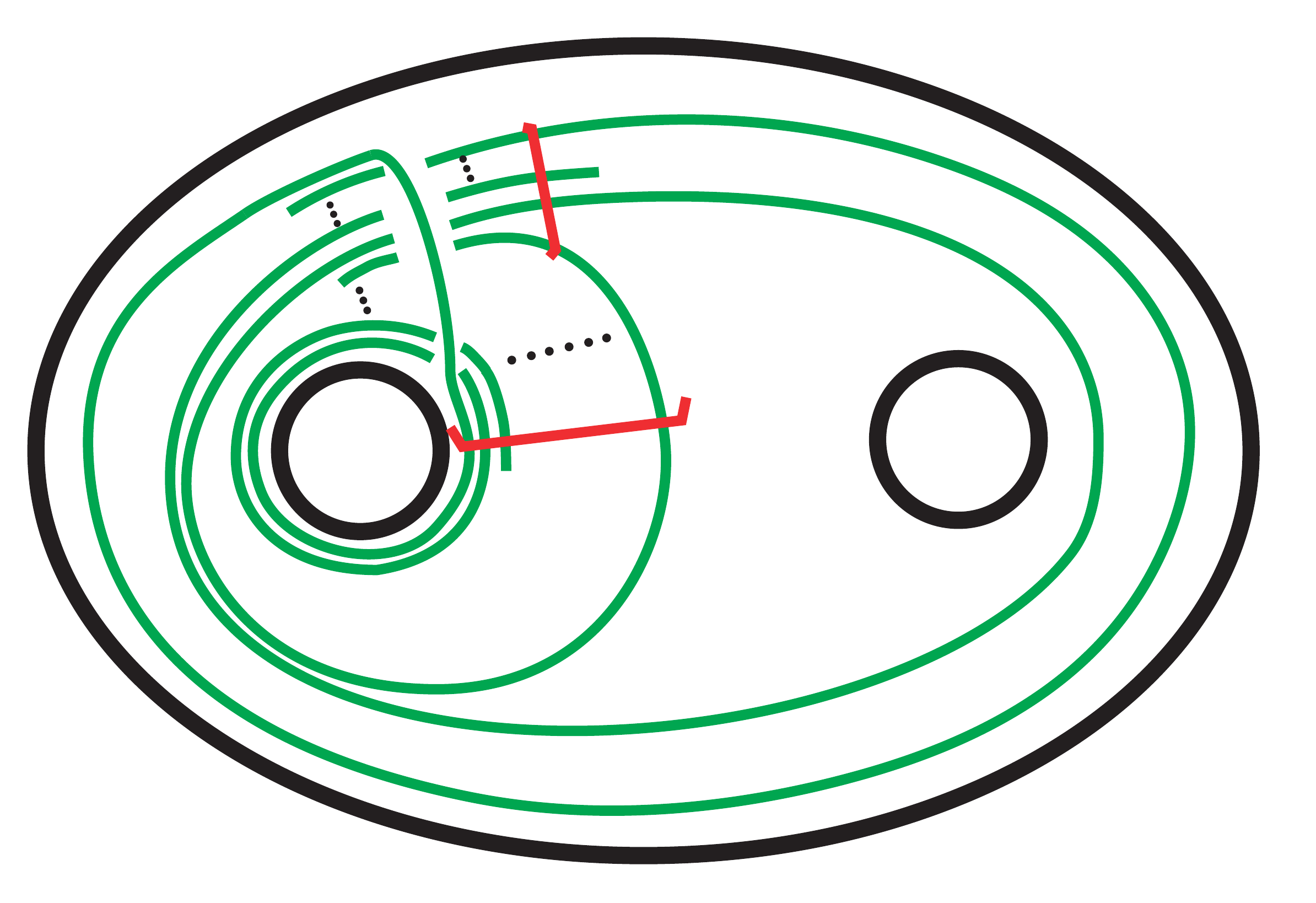}
\put(42,35){\tiny{${m}$}}
\put(50,63){\tiny{${n-2}$}}
\end{overpic}}}\\
&=& AP(m,n-1)a_{3}-A^{2}P(m,n-2).
\end{eqnarray*}
\caption{$P(m,n)=AP(m,n-1)a_{3}-A^{2}P(m,n-2)$ for $n\geq 2$.}\label{fig:cal-P(m,n)}
\end{figure}

\begin{center}
\begin{figure}[H]
\vspace{2.8cm}
    \centering
    \begin{eqnarray*}
    P(m,1) & = &
\vcenter{\hbox{\begin{overpic}[scale=.1]{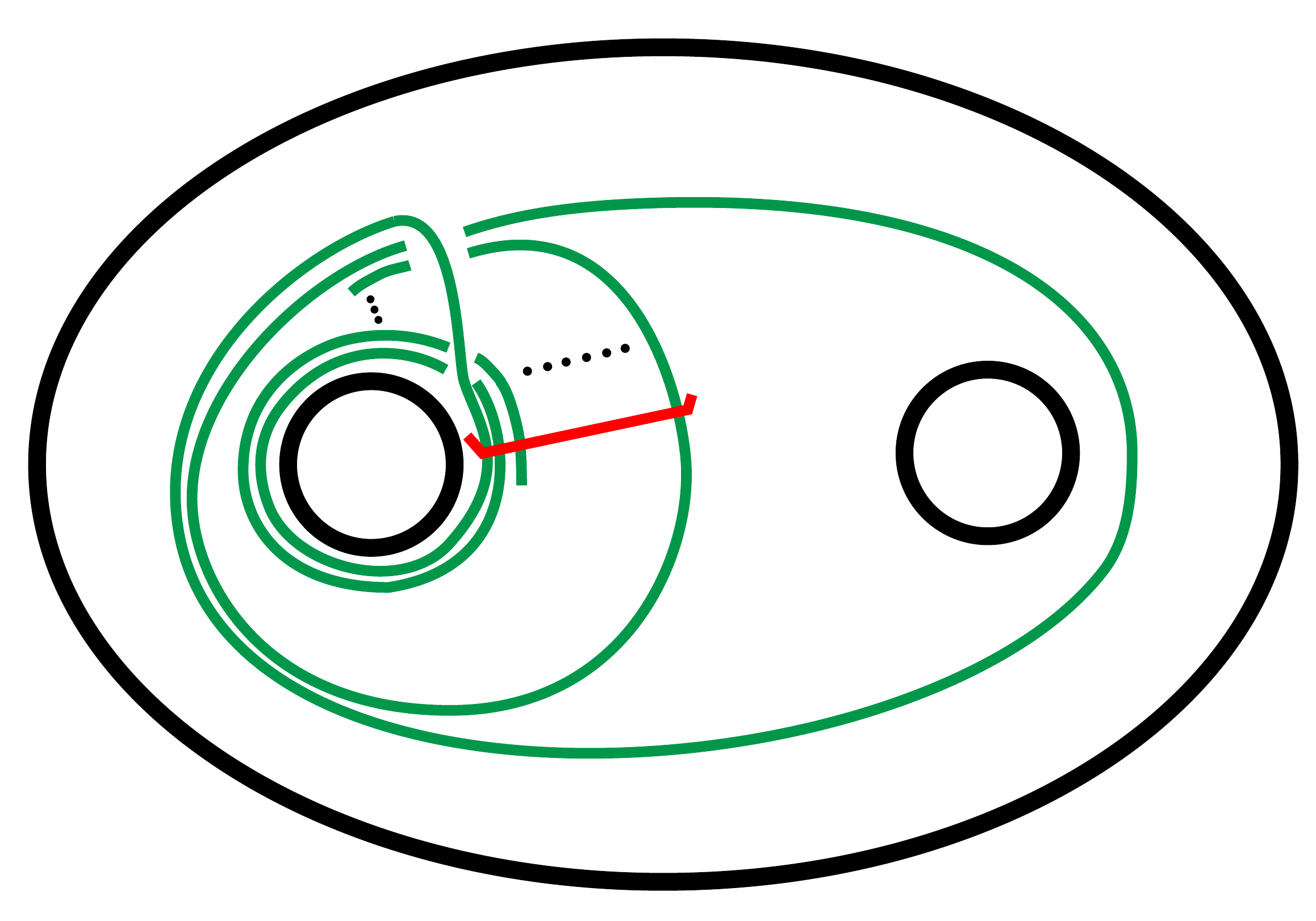} 
\put(48,42){\tiny{${m}$}}
\end{overpic}}} \\
 &=& A \vcenter{\hbox{\begin{overpic}[scale=.1]{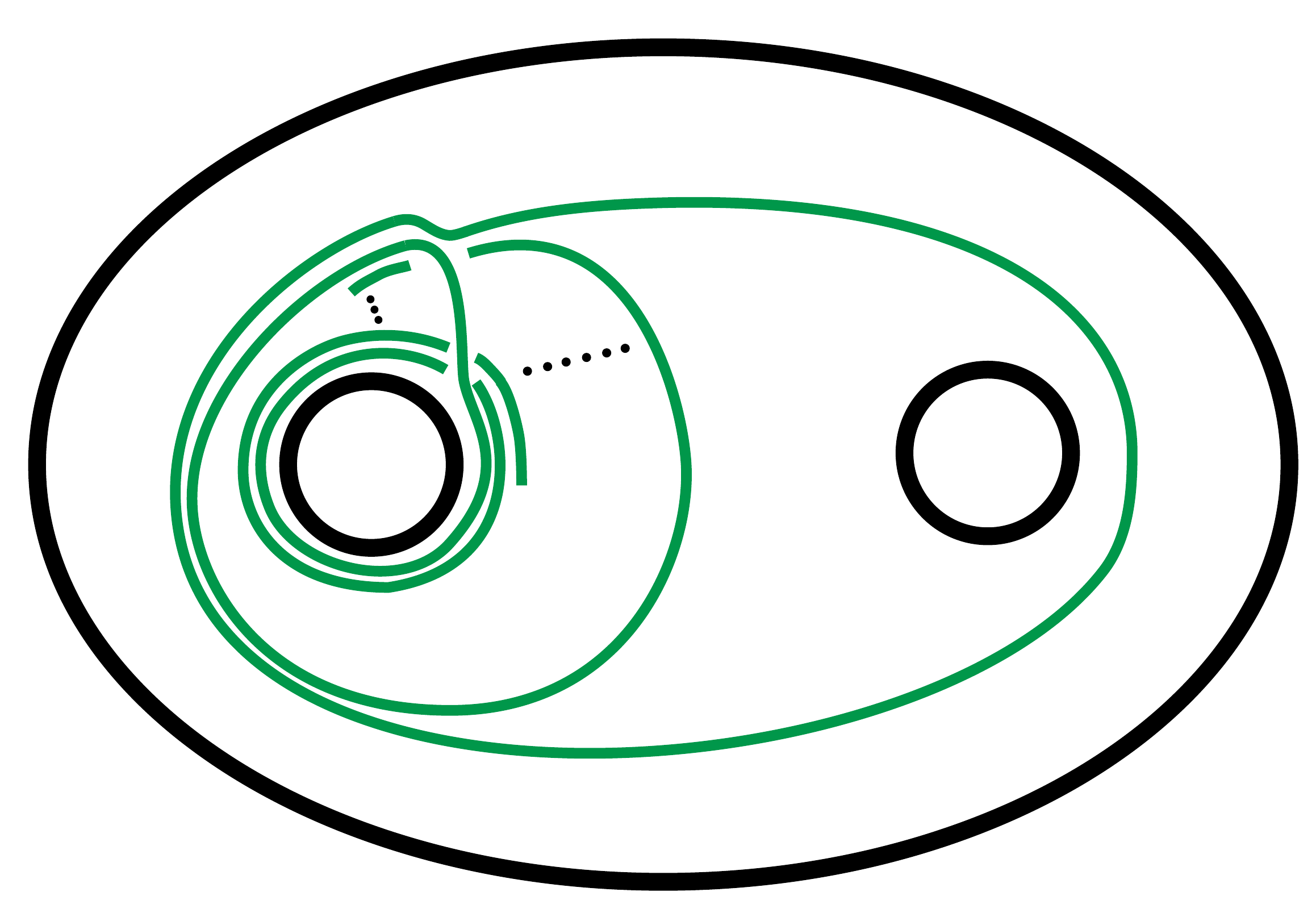}
\end{overpic}}} + A^{-1}
\vcenter{\hbox{\begin{overpic}[scale=.1]{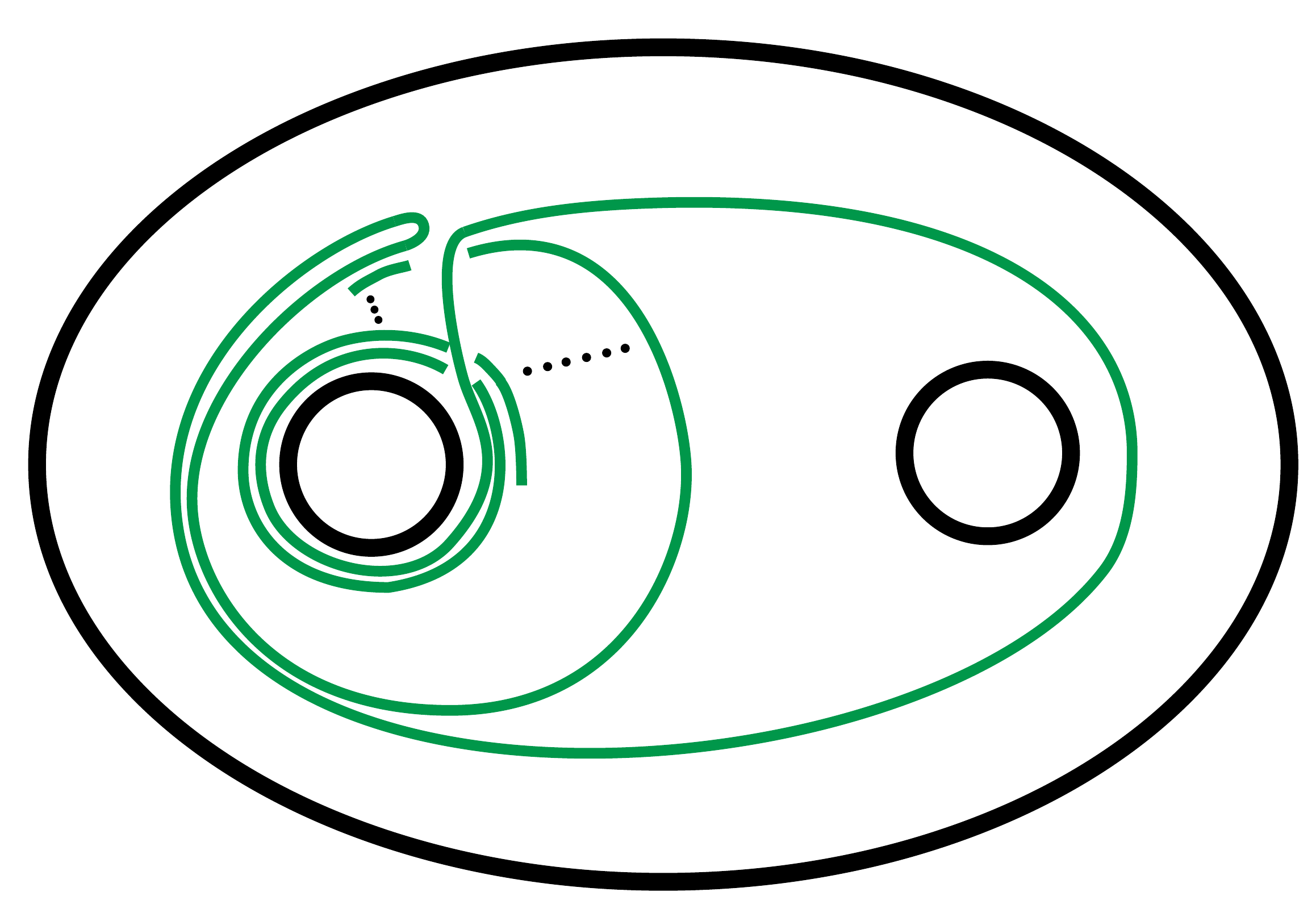}
\end{overpic}}}\\ 
& = & A \vcenter{\hbox{\begin{overpic}[scale=.1]{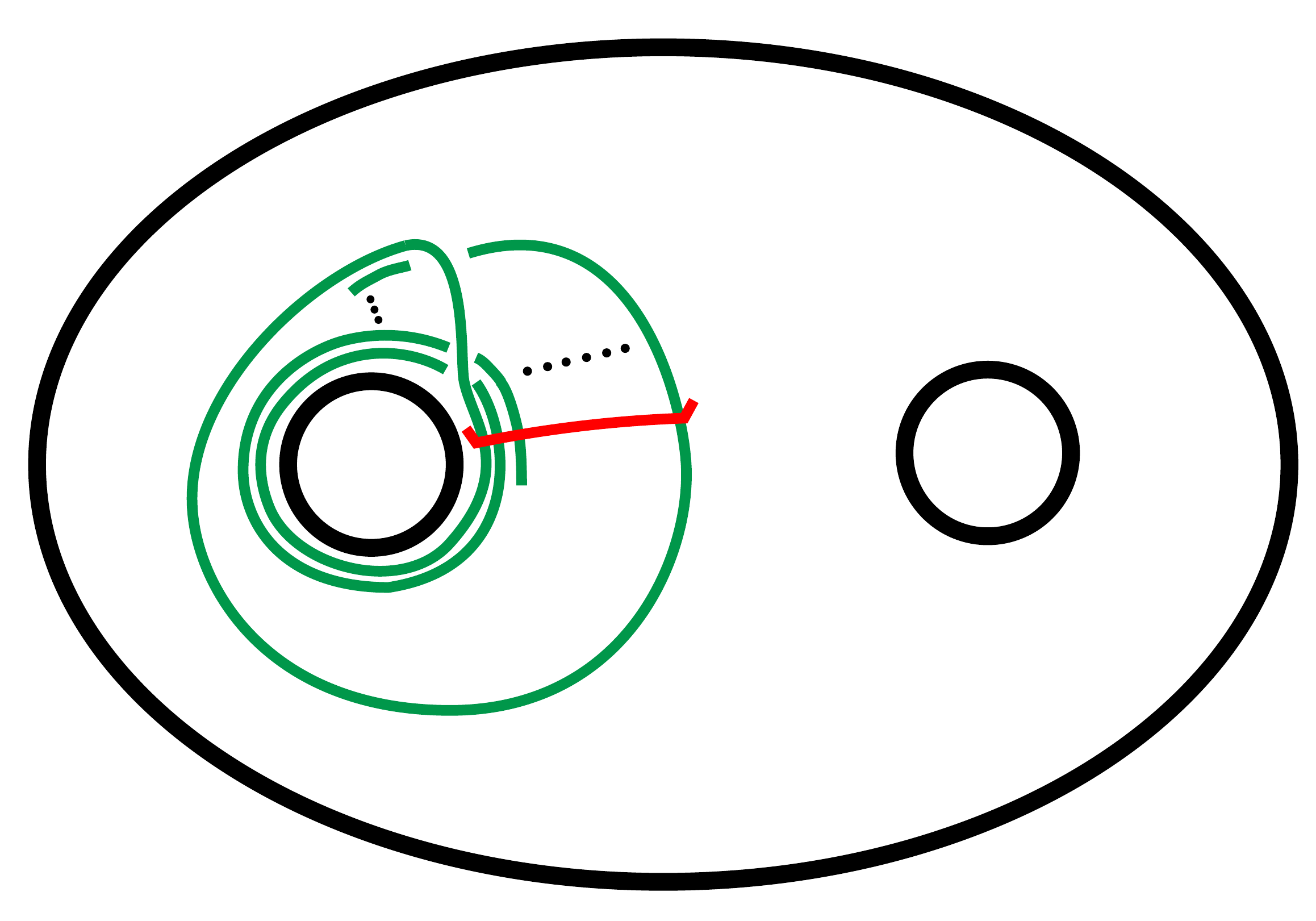}
\put(44,35){\tiny{${m}$}}
\end{overpic}}}a_{3} + A^{-1}
\vcenter{\hbox{\begin{overpic}[scale=.1]{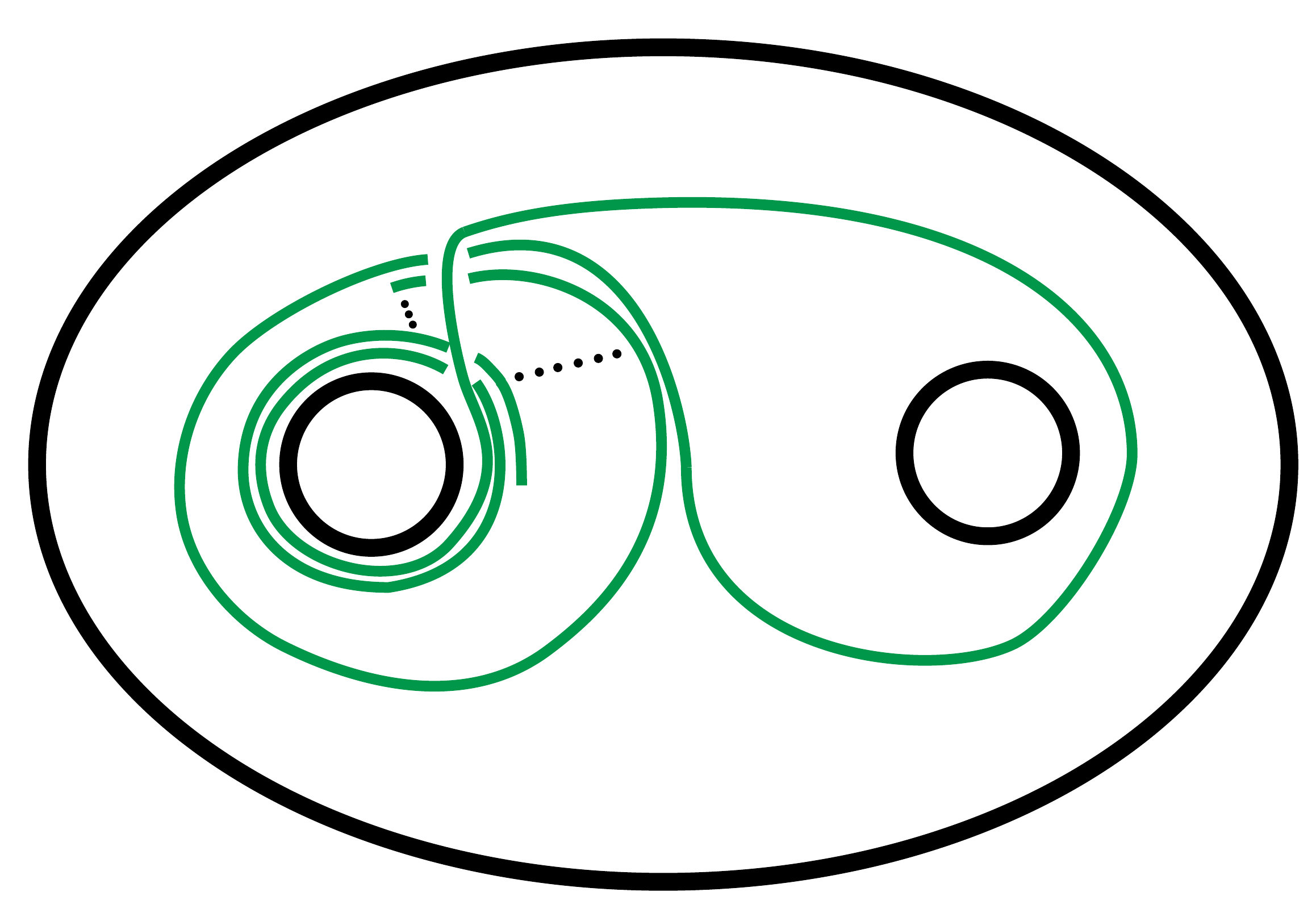}
\end{overpic}}}\\ 
& = & A\vcenter{\hbox{\begin{overpic}[scale=.1]{Pm,1-A-2.pdf}
\put(44,35){\tiny{${m}$}}
\end{overpic}}}a_{3} + 
\vcenter{\hbox{\begin{overpic}[scale=.1]{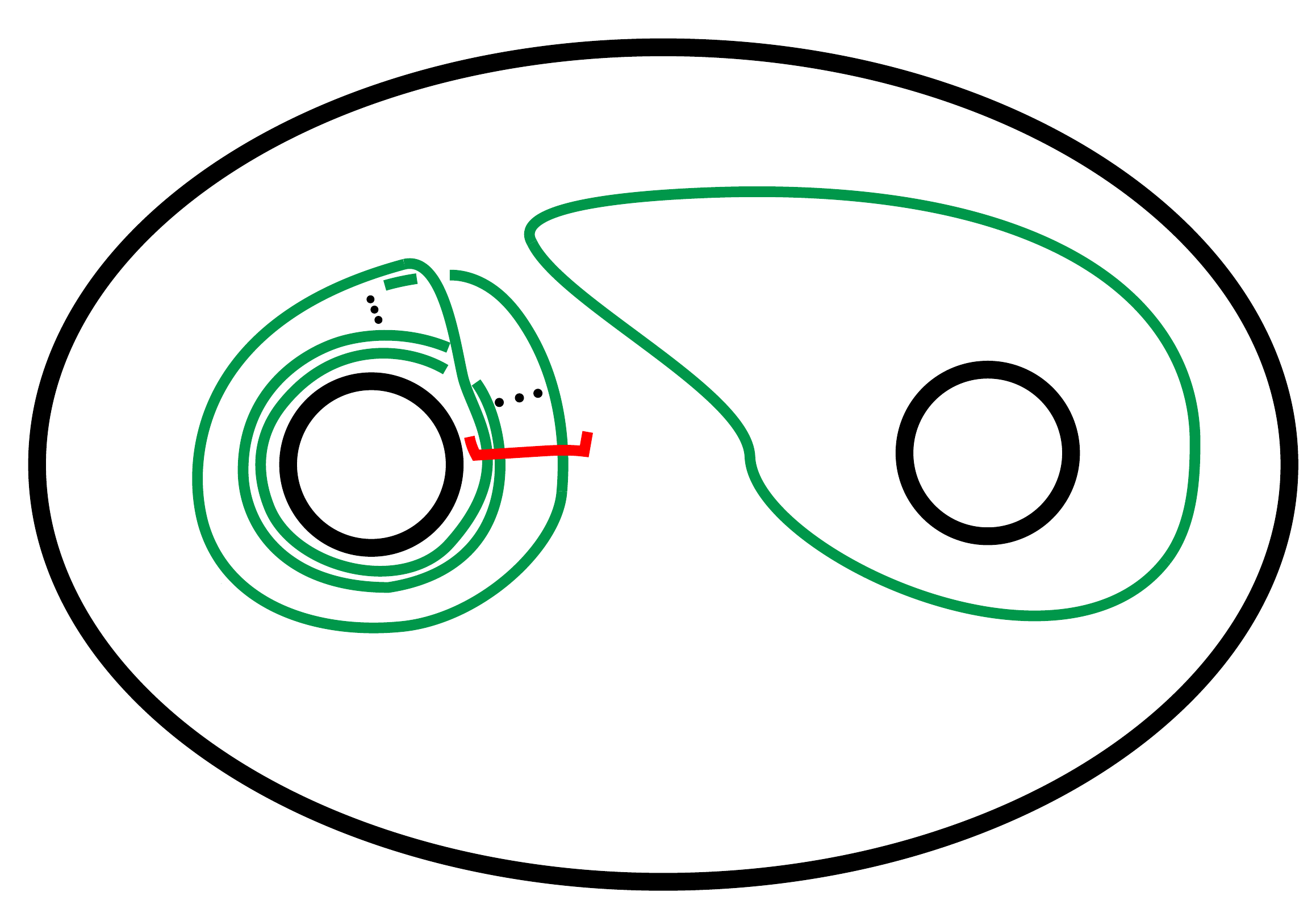}
\put(42,35){\tiny{${m-1}$}}
\end{overpic}}}+A^{-2}\vcenter{\hbox{\begin{overpic}[scale=.1]{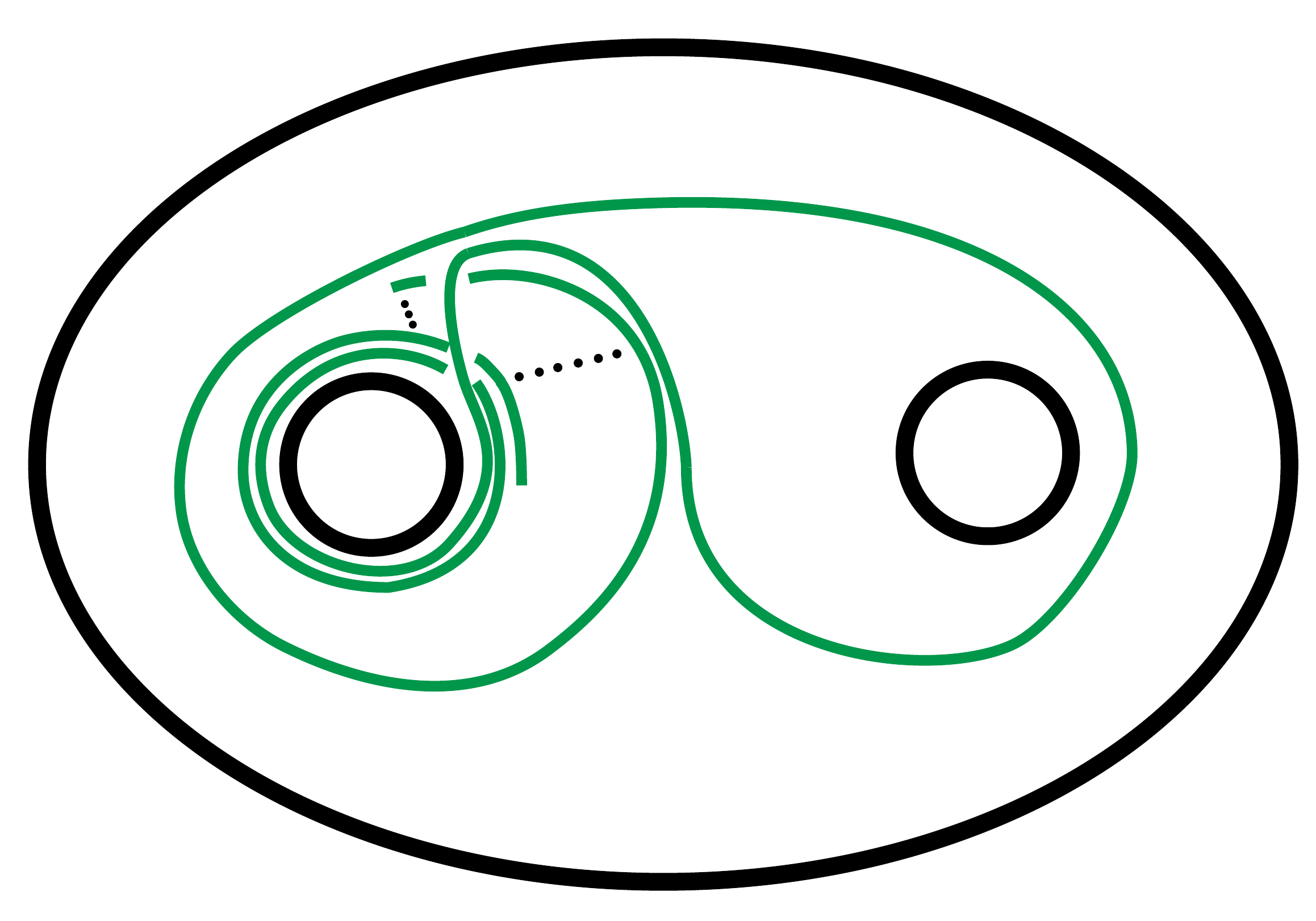}
\end{overpic}}}\\
& = & A\vcenter{\hbox{\begin{overpic}[scale=.1]{Pm,1-A-2.pdf}
\put(44,35){\tiny{${m}$}}
\end{overpic}}}a_{3} +
\vcenter{\hbox{\begin{overpic}[scale=.1]{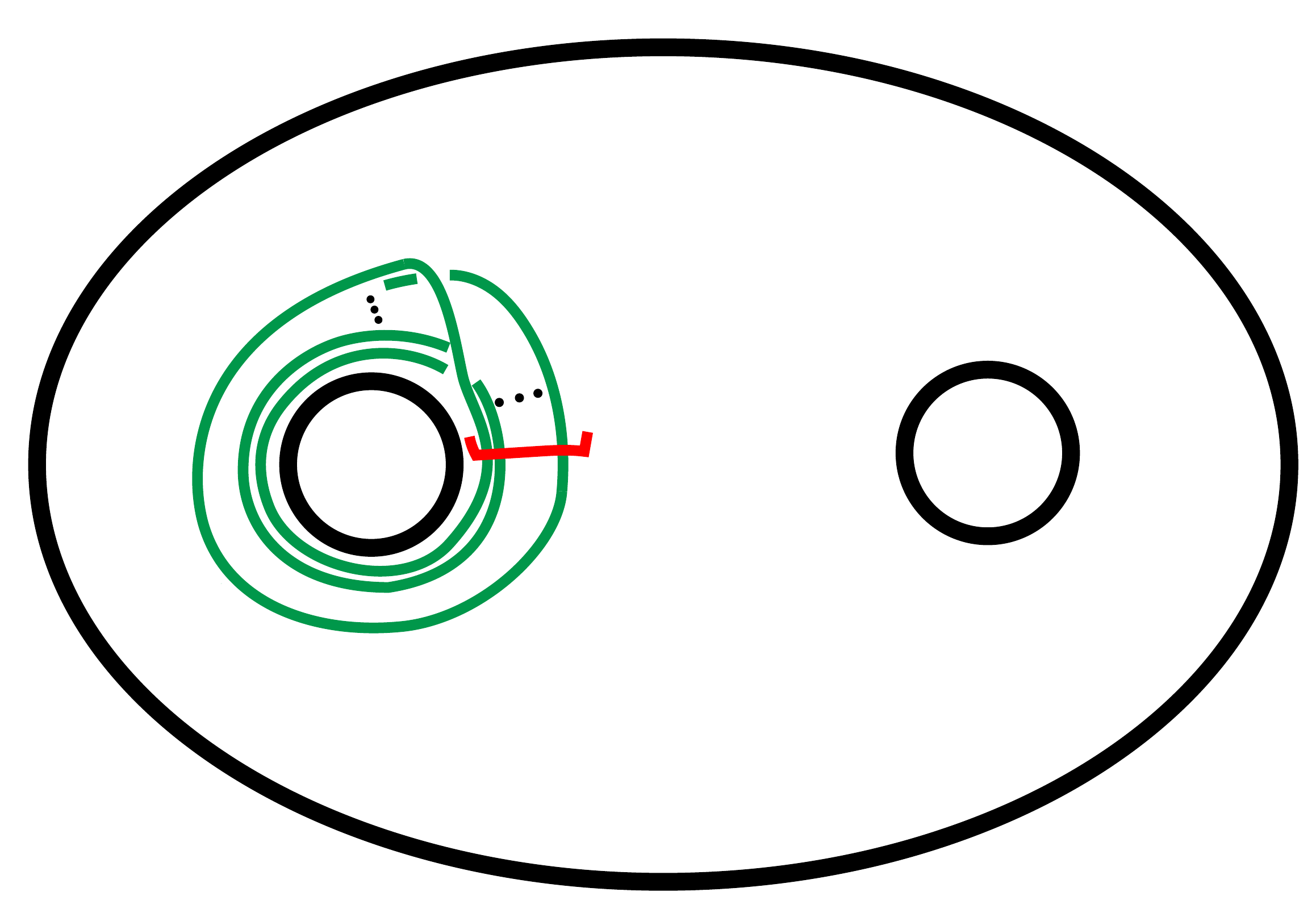}
\put(42,35){\tiny{${m-1}$}}
\end{overpic}}}a_{2}+A^{-2}\vcenter{\hbox{\begin{overpic}[scale=.1]{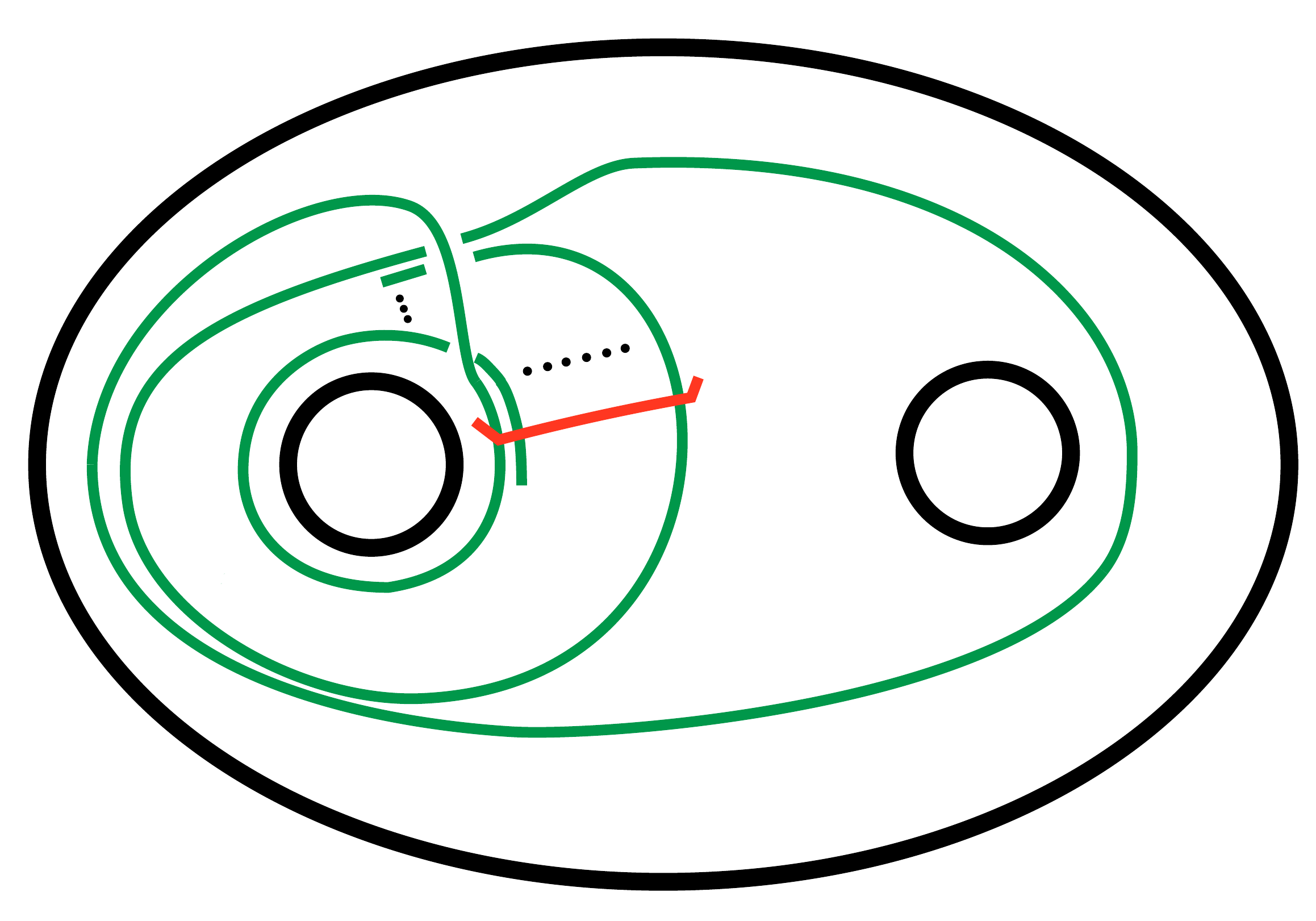}
\put(42,35){\tiny{${m-2}$}}
\end{overpic}}}\\
&=&AP(m,0)a_{3}+P(m-1)a_{2}+A^{-2}P(m-2,1).
\end{eqnarray*}
    \caption{$P(m,1)=AP(m,0)+P(m-1,0)a_{2}+A^{-2}P(m-2,1)$ for $m\geq 2$.}
    \label{fig:cal-P(m,1)}
\end{figure}
\end{center}

\begin{figure}[H]
    \centering
    \begin{eqnarray*}
    P(m,-1) & = &
\vcenter{\hbox{\begin{overpic}[scale=.1]{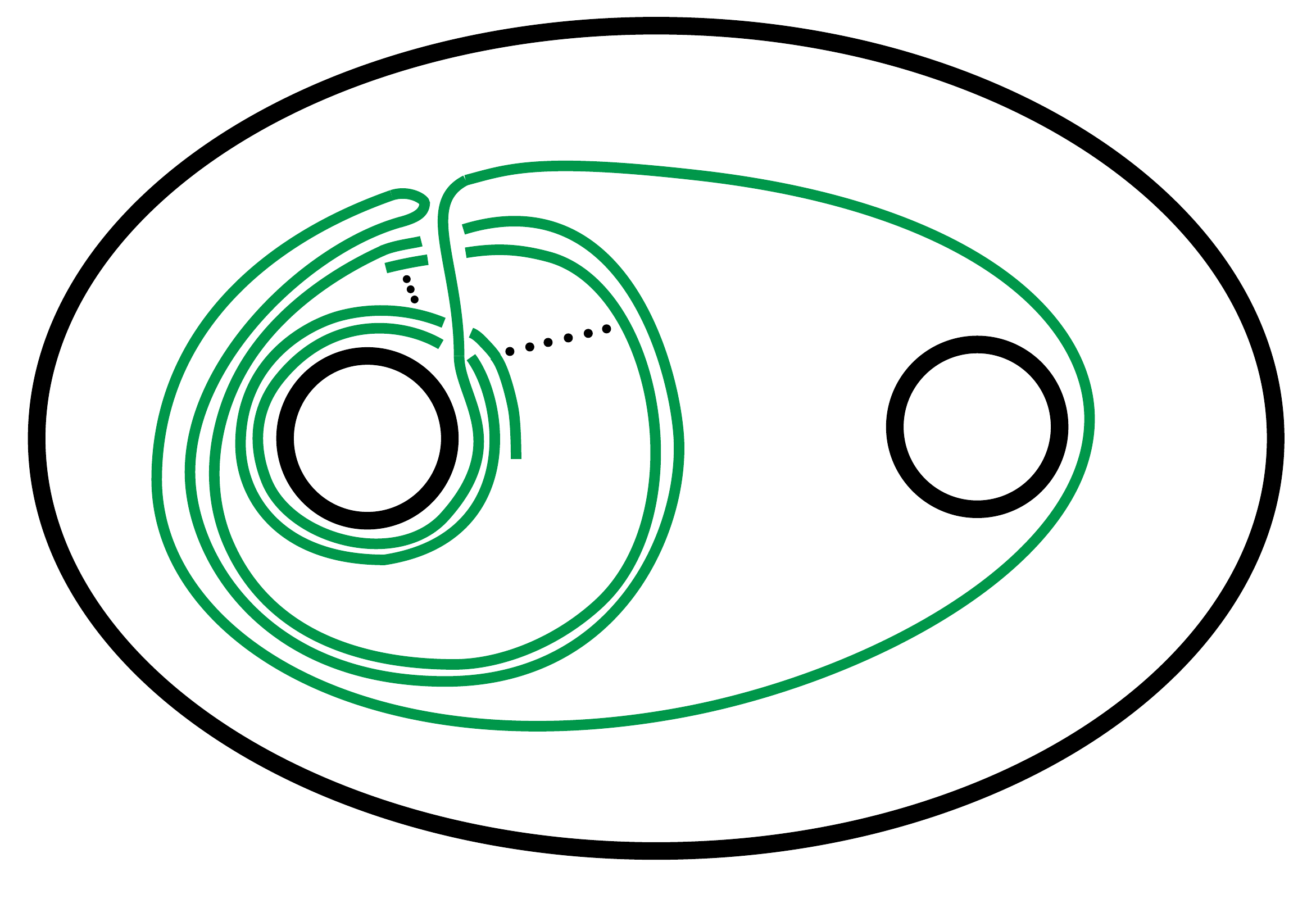} 
\end{overpic}}}=\vcenter{\hbox{\begin{overpic}[scale=.1]{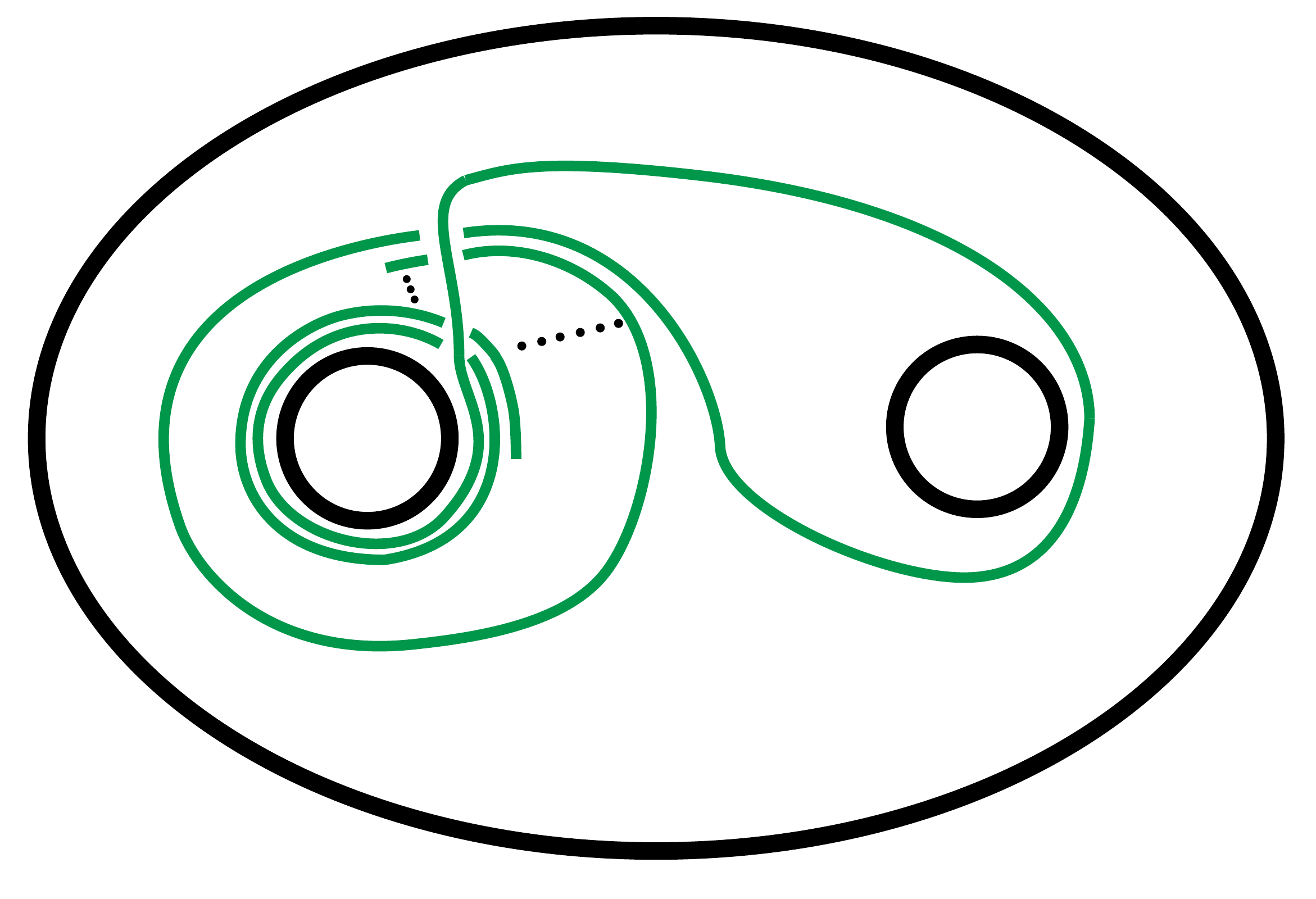} 
\end{overpic}}}\\ 
& = & A \vcenter{\hbox{\begin{overpic}[scale=.1]{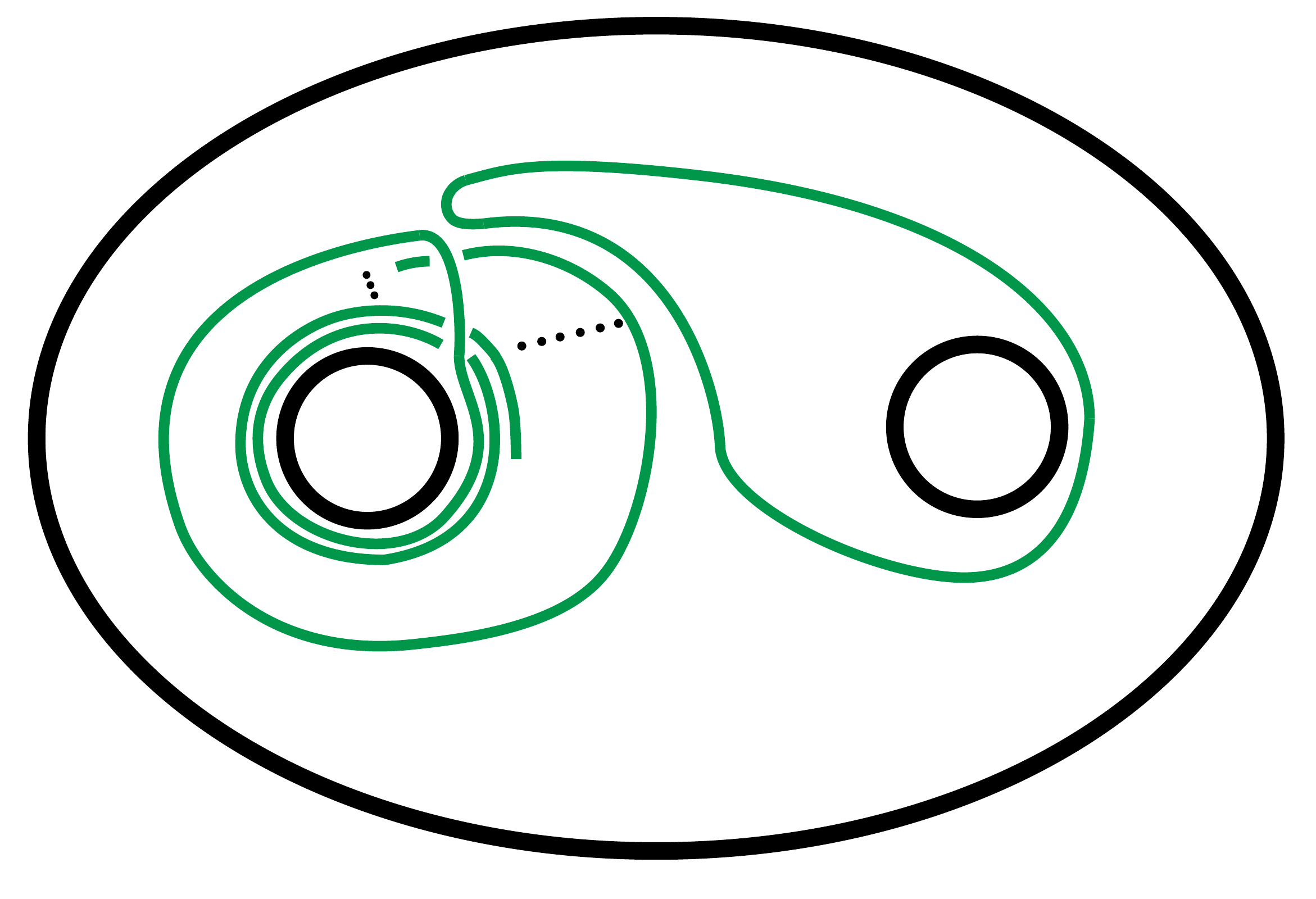}
\end{overpic}}} + A^{-1}
\vcenter{\hbox{\begin{overpic}[scale=.1]{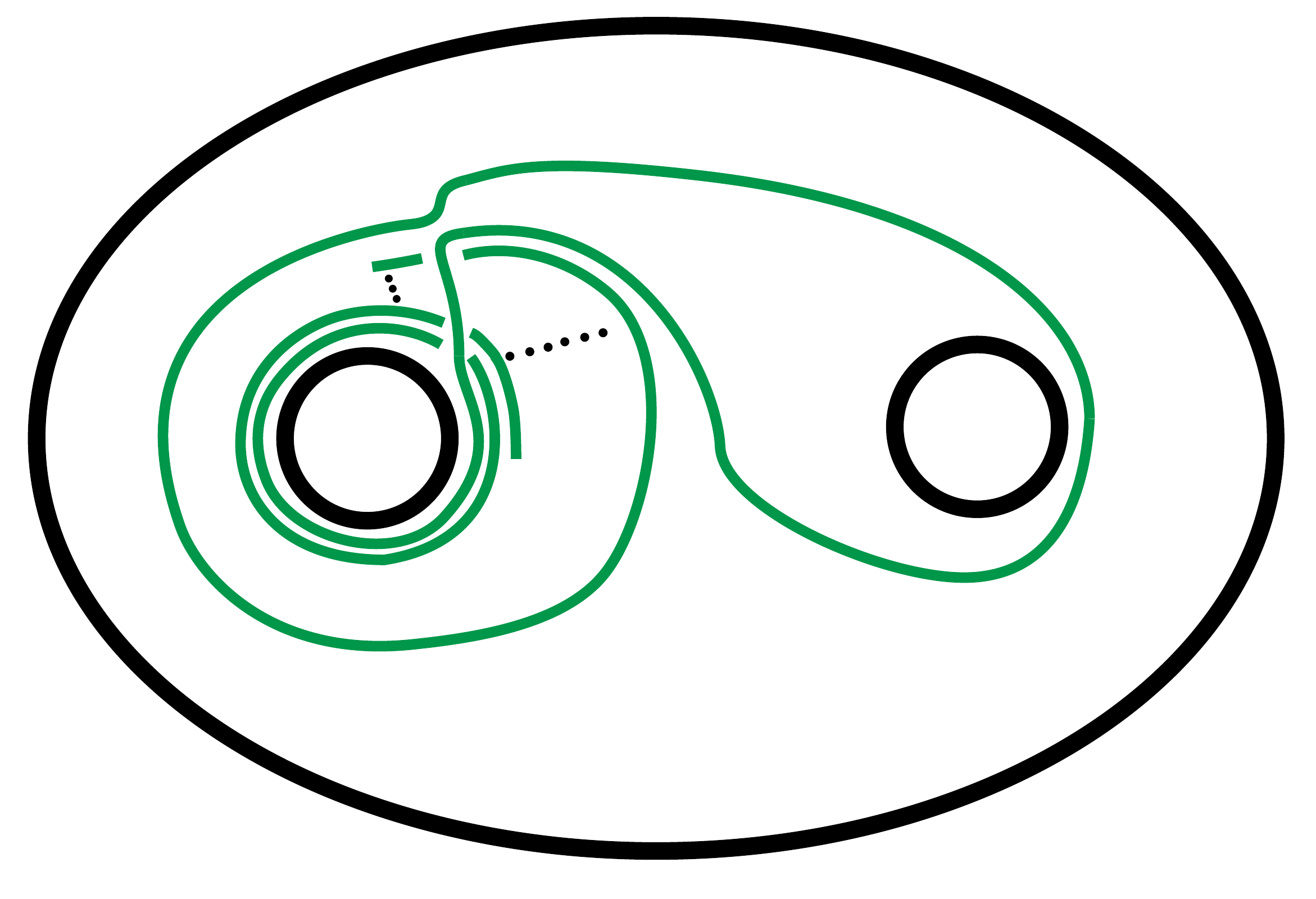}
\end{overpic}}}\\ 
& = & A \vcenter{\hbox{\begin{overpic}[scale=.1]{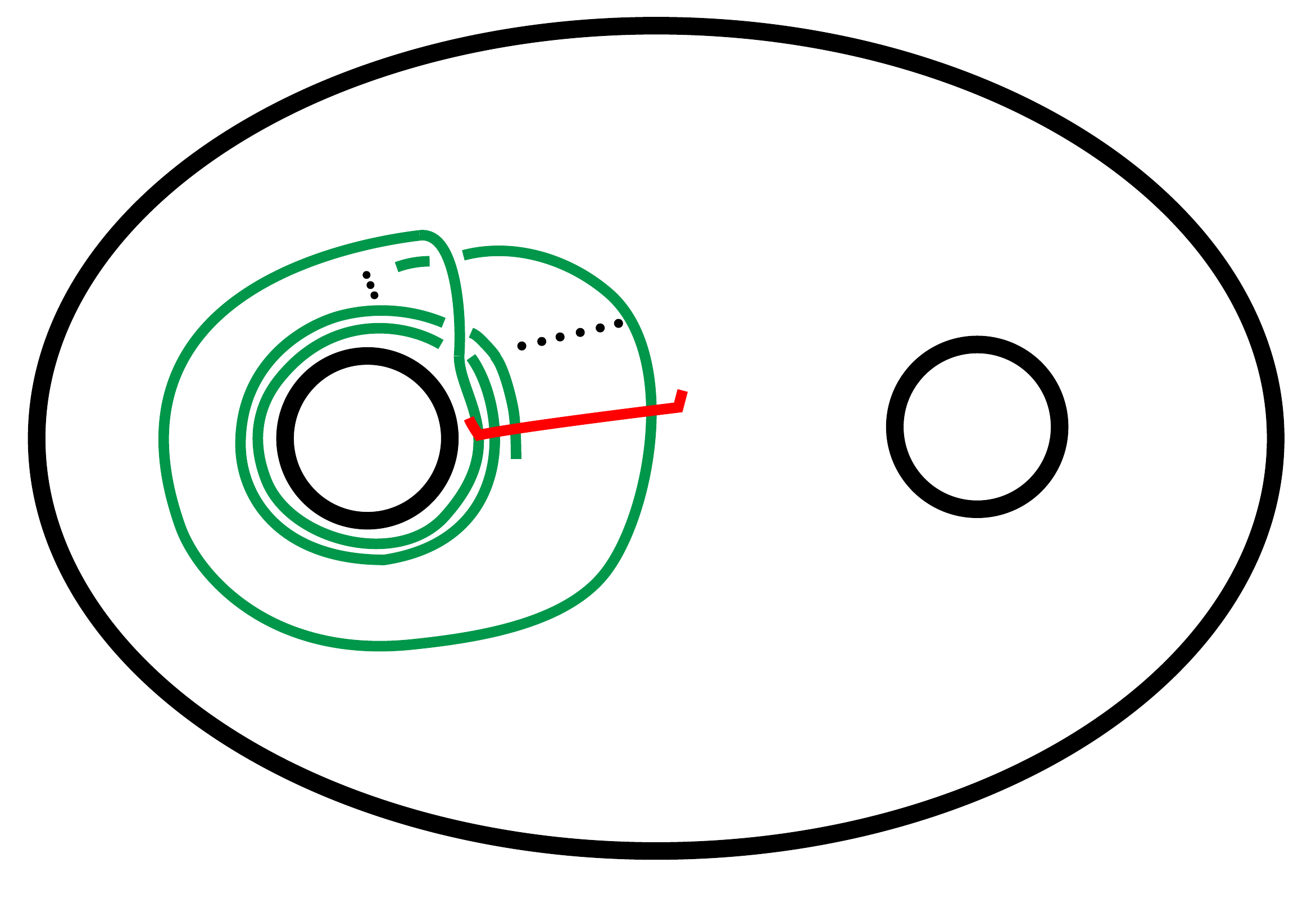}
\put(44,35){\tiny{${m}$}}
\end{overpic}}}a_{2} + A^{-1}
\vcenter{\hbox{\begin{overpic}[scale=.1]{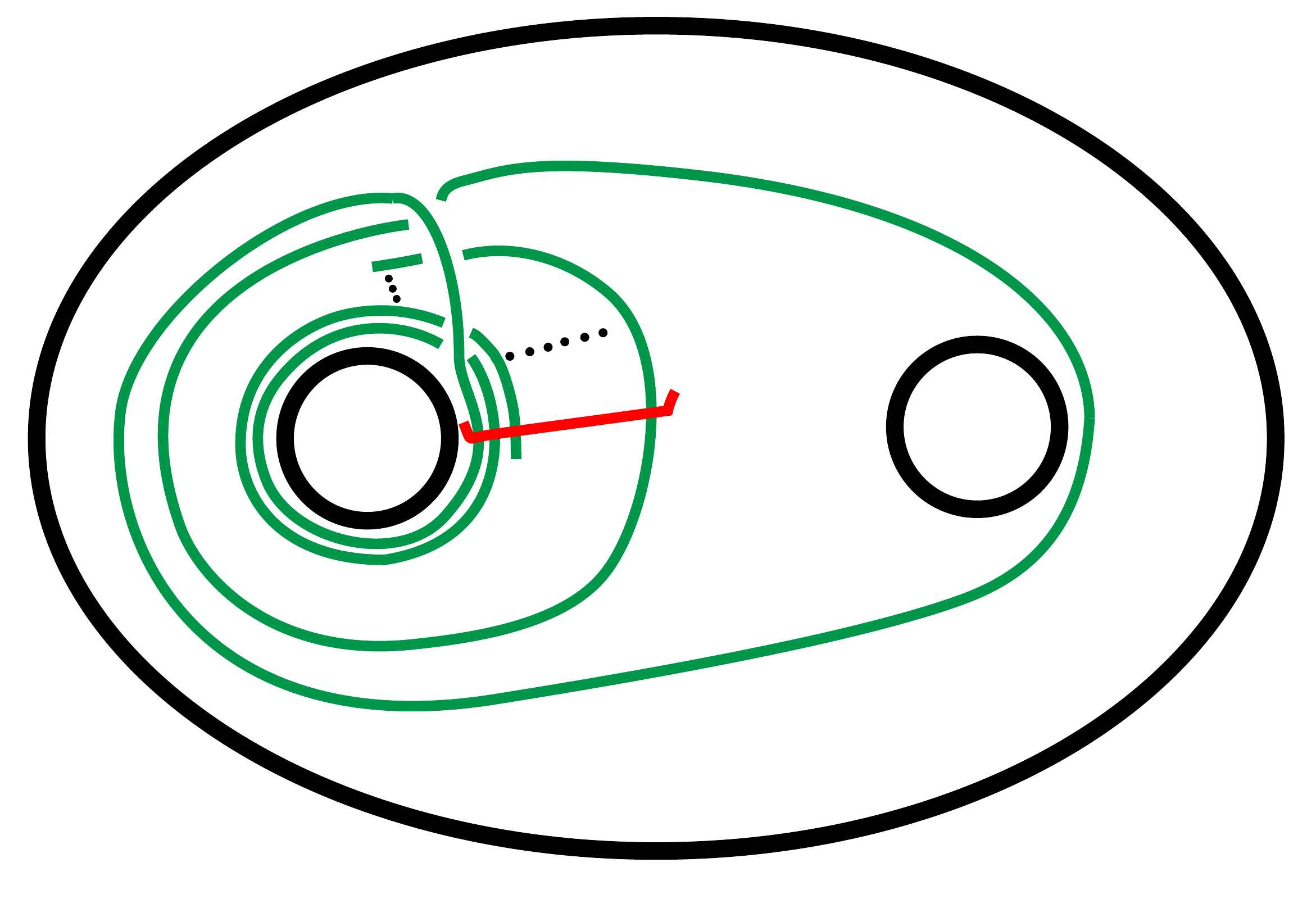}
\put(44,35){\tiny{${m}$}}
\end{overpic}}}\\
&=& AP(m,0)a_{2}+A^{-1}P(m,1).
\end{eqnarray*}
\caption{$AP(m,0)a_{2}+A^{-1}P(m,1)$ for $m>0$.}\label{fig:cal-P(m,-1)}
\end{figure}

\begin{figure}[H]
    \centering
    \begin{eqnarray*}
    P(m,-2) & = &
\vcenter{\hbox{\begin{overpic}[scale=.1]{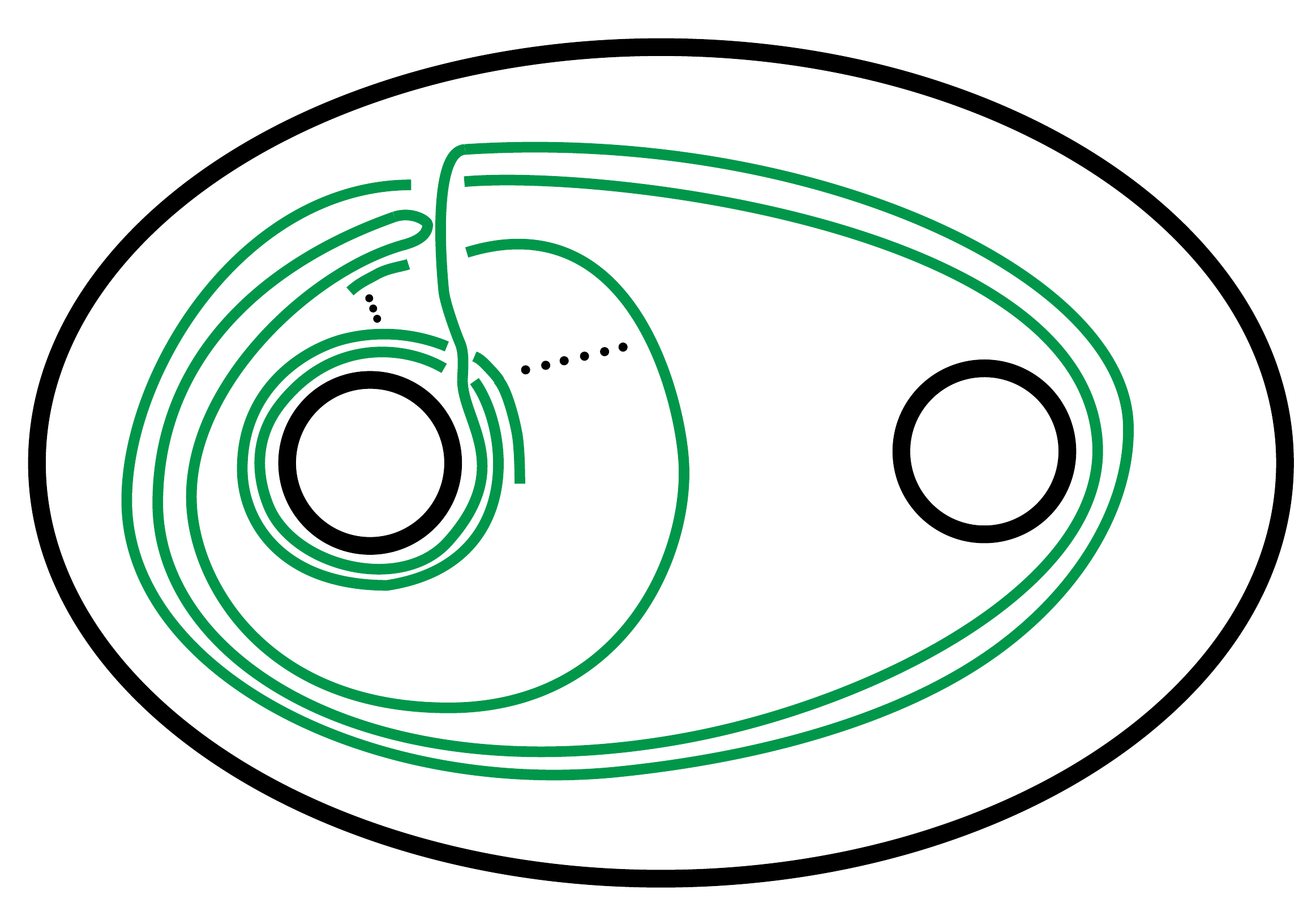} 
\end{overpic}}}=\vcenter{\hbox{\begin{overpic}[scale=.1]{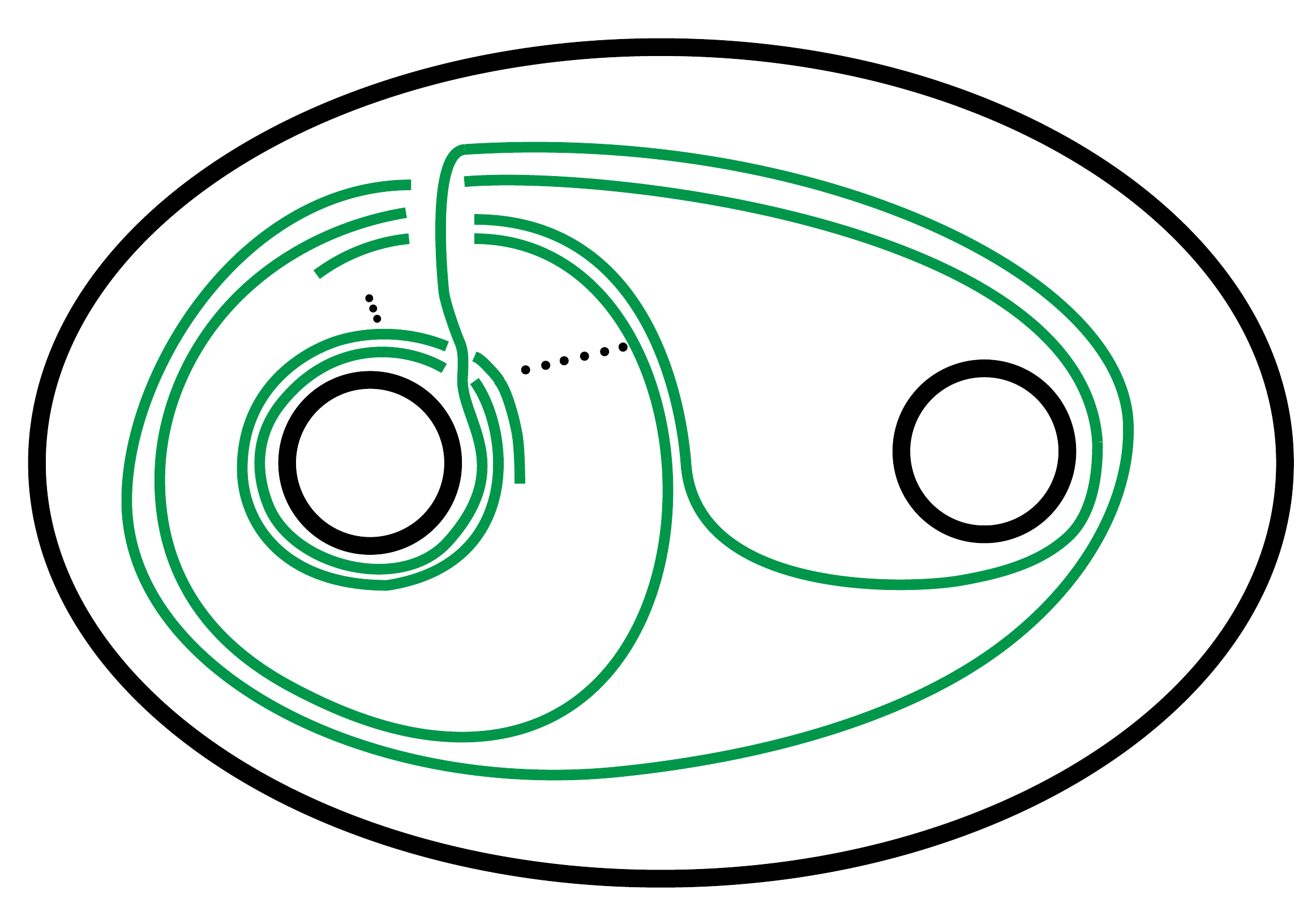} 
\end{overpic}}}\\ 
& = & A \vcenter{\hbox{\begin{overpic}[scale=.1]{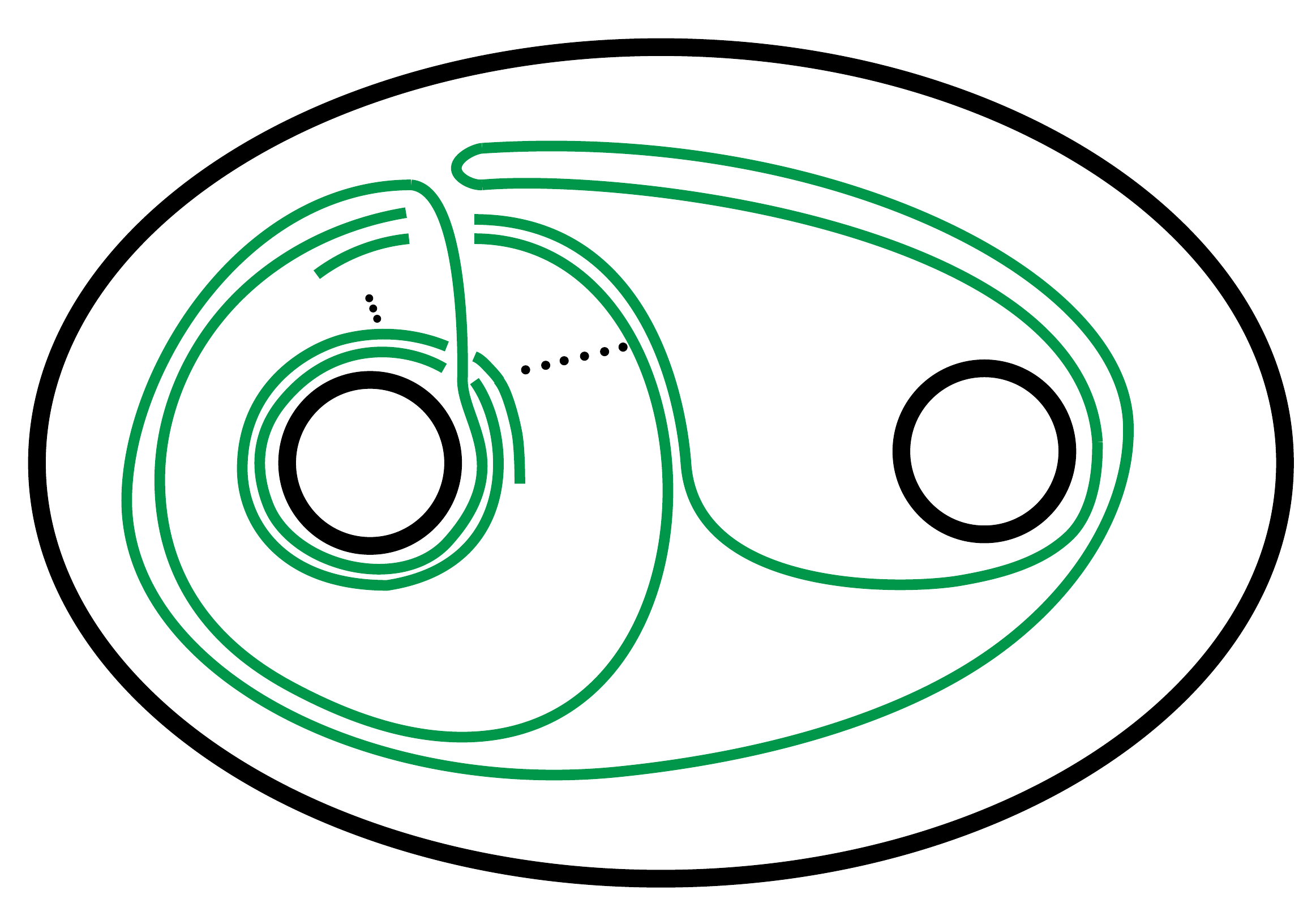}
\end{overpic}}} + A^{-1}
\vcenter{\hbox{\begin{overpic}[scale=.1]{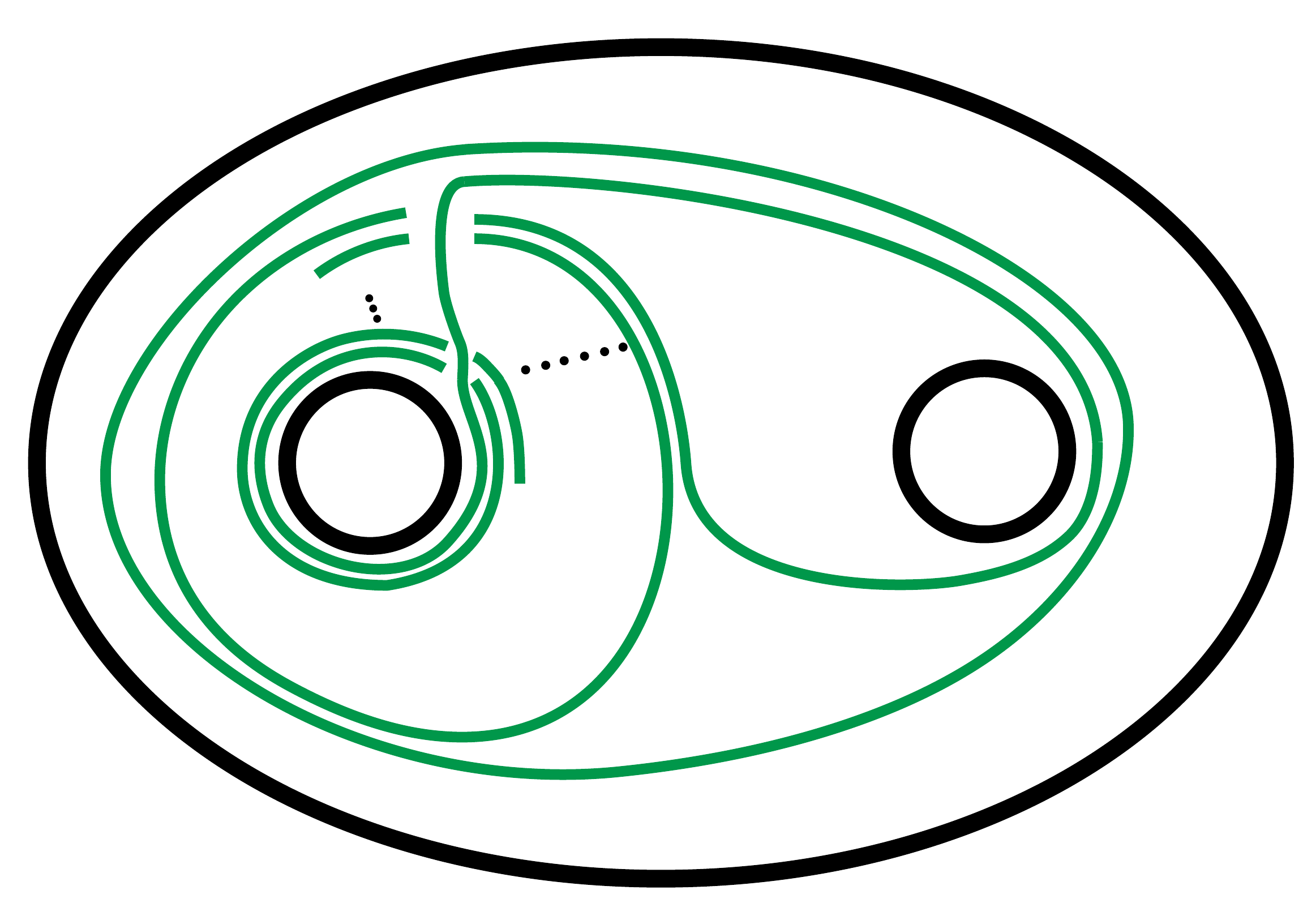}
\end{overpic}}}\\ 
& = & A \vcenter{\hbox{\begin{overpic}[scale=.1]{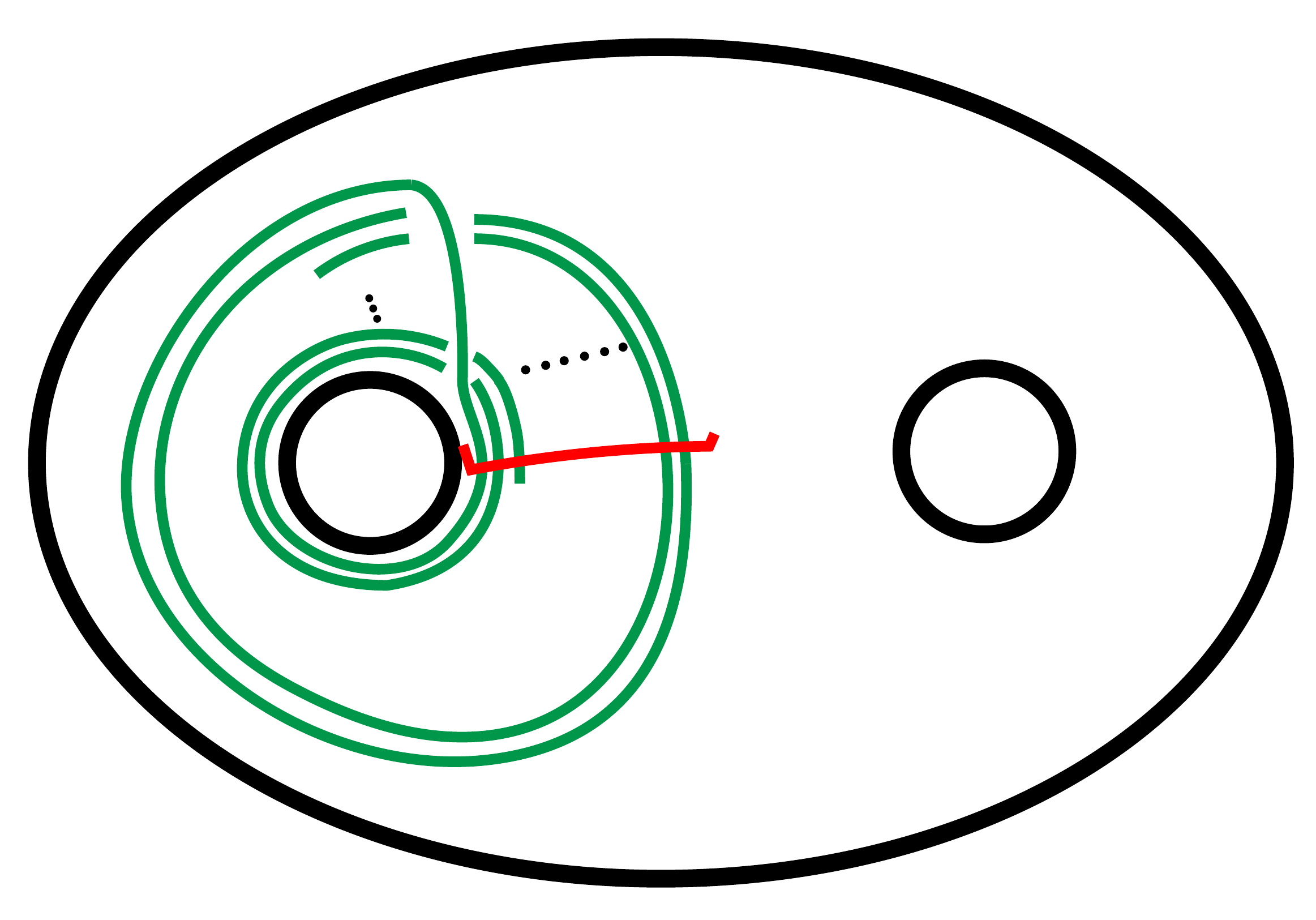}
\put(44,35){\tiny{${m}$}}
\end{overpic}}}a_{2} + A^{-1}
\vcenter{\hbox{\begin{overpic}[scale=.1]{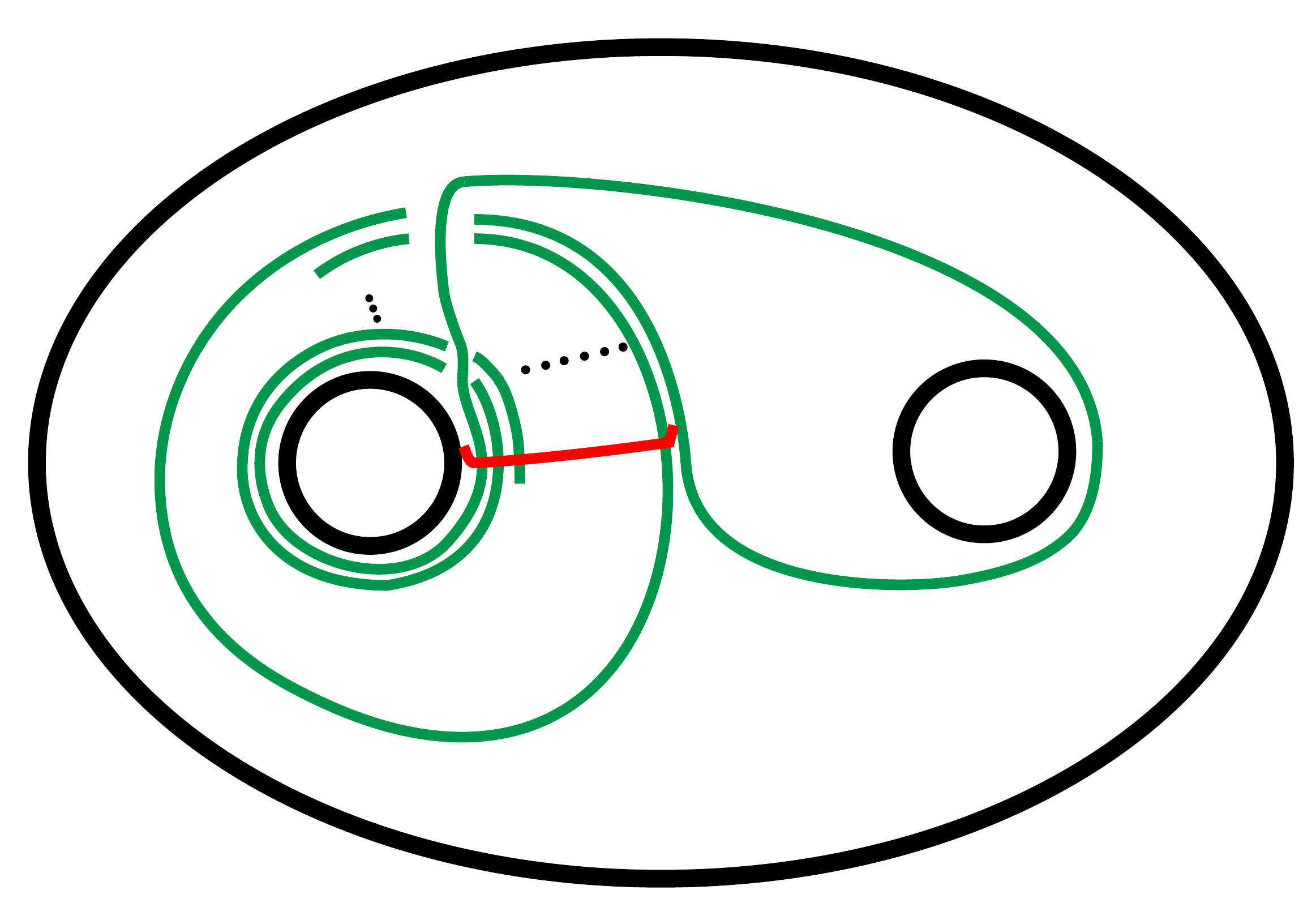}
\put(44,35){\tiny{${m}$}}
\end{overpic}}}\\
&=& AP(m,0)+A^{-1}P(m,-1)a_{3}.
\end{eqnarray*}
\caption{$P(m,-2)=AP(m,0)+A^{-1}P(m,-1)a_{3}$ for $m>0$. \\ \ \\ \ \\}\label{fig:cal-P(m,-2)}
\end{figure}

\begin{figure}[H]
    \centering
    \begin{eqnarray*}
    P(m,-n) & = &
\vcenter{\hbox{\begin{overpic}[scale=.1]{Pm,-n.pdf} 
\end{overpic}}}
 =  A \vcenter{\hbox{\begin{overpic}[scale=.1]{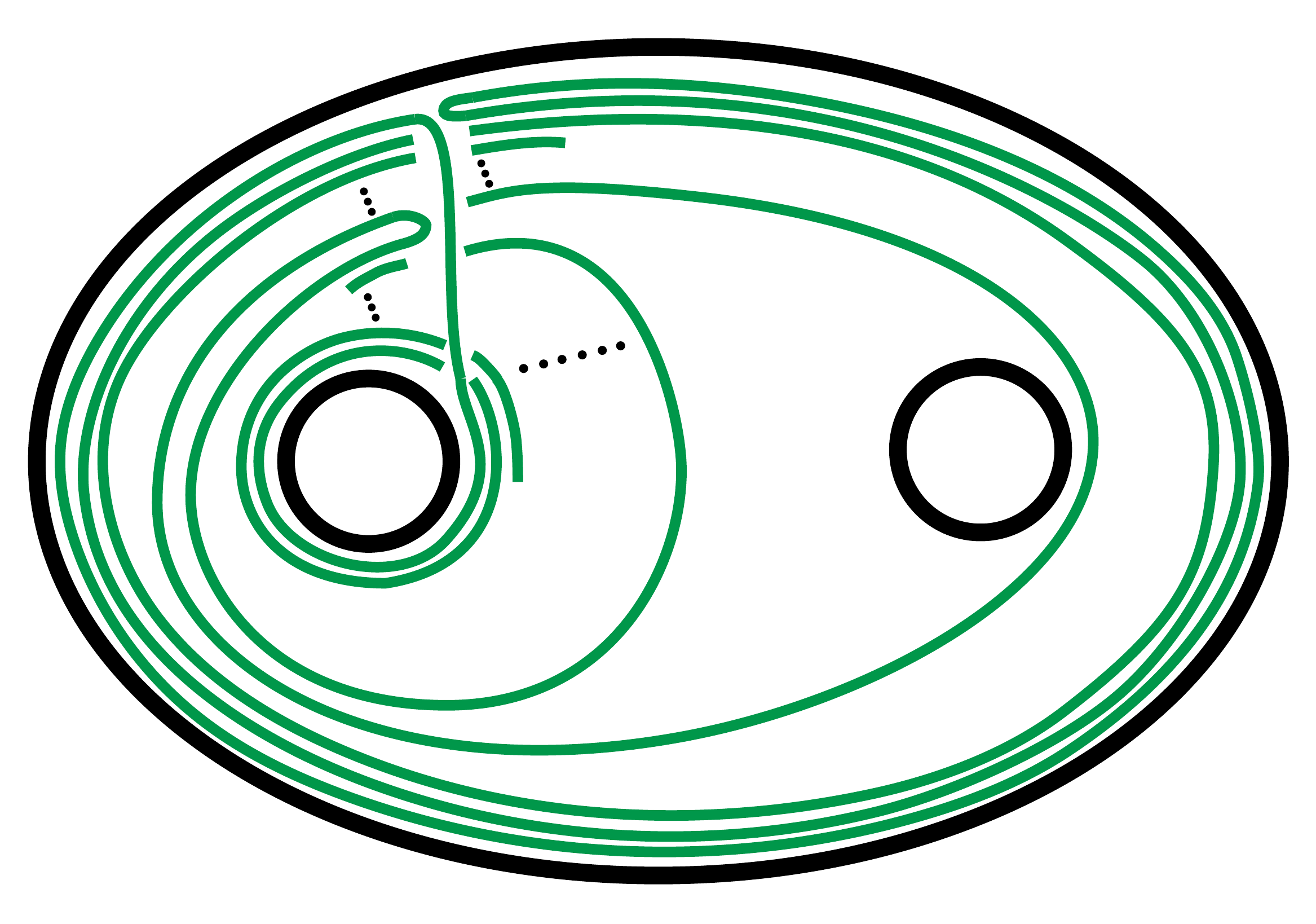}
\end{overpic}}} + A^{-1}
\vcenter{\hbox{\begin{overpic}[scale=.1]{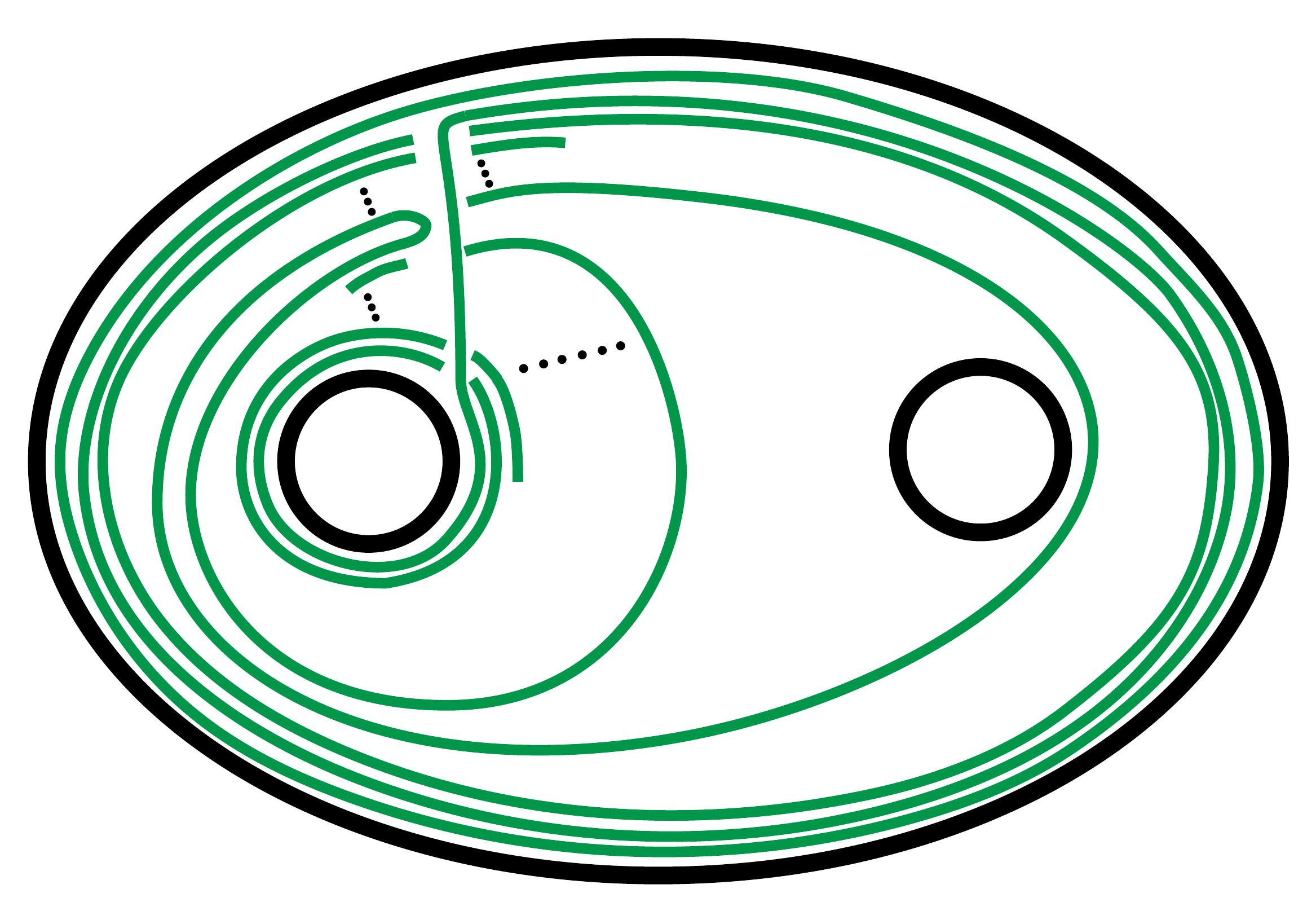}
\end{overpic}}}\\ 
& = & A\vcenter{\hbox{\begin{overpic}[scale=.1]{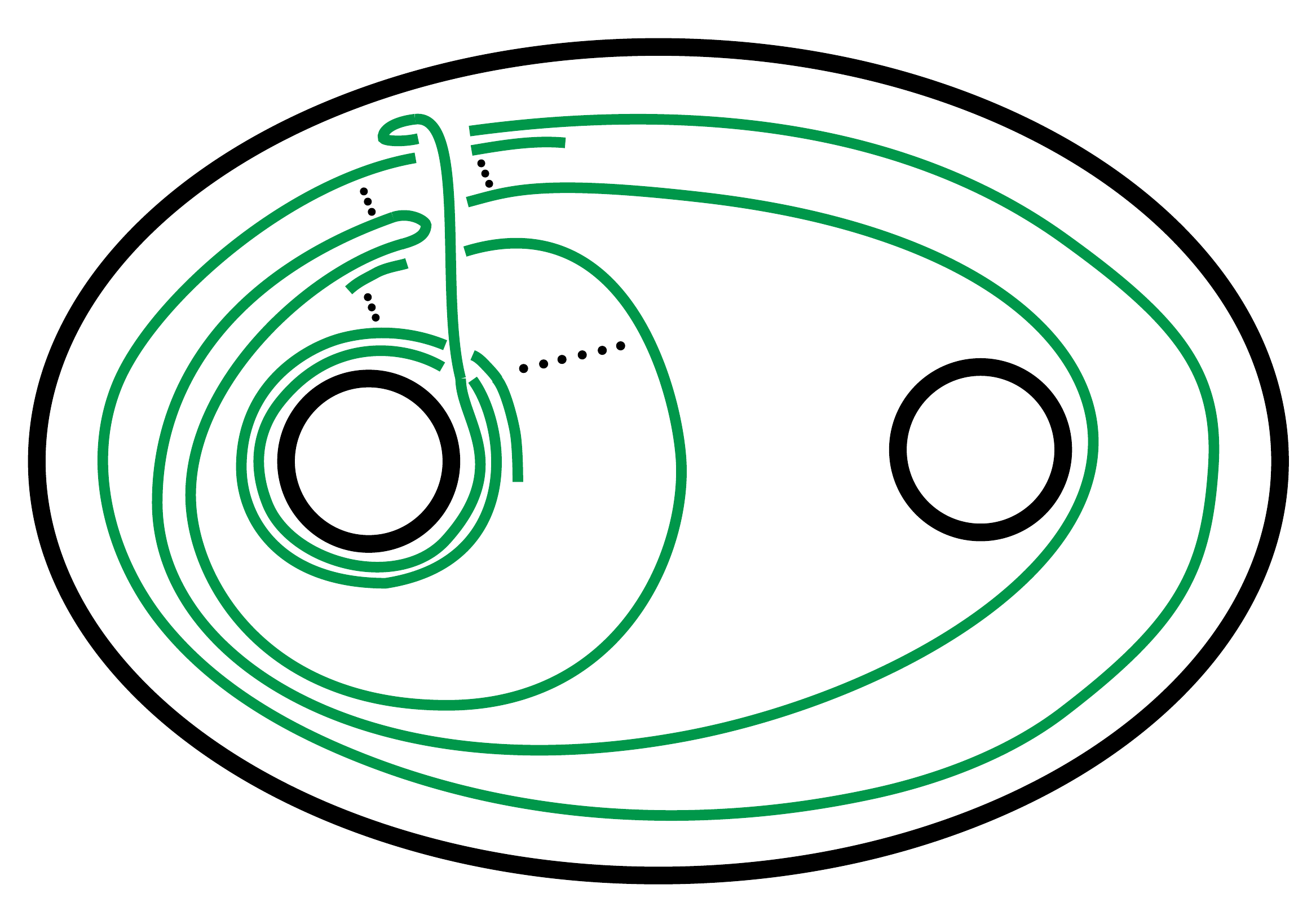}
\end{overpic}}}a_{3} + A^{-1}
\vcenter{\hbox{\begin{overpic}[scale=.1]{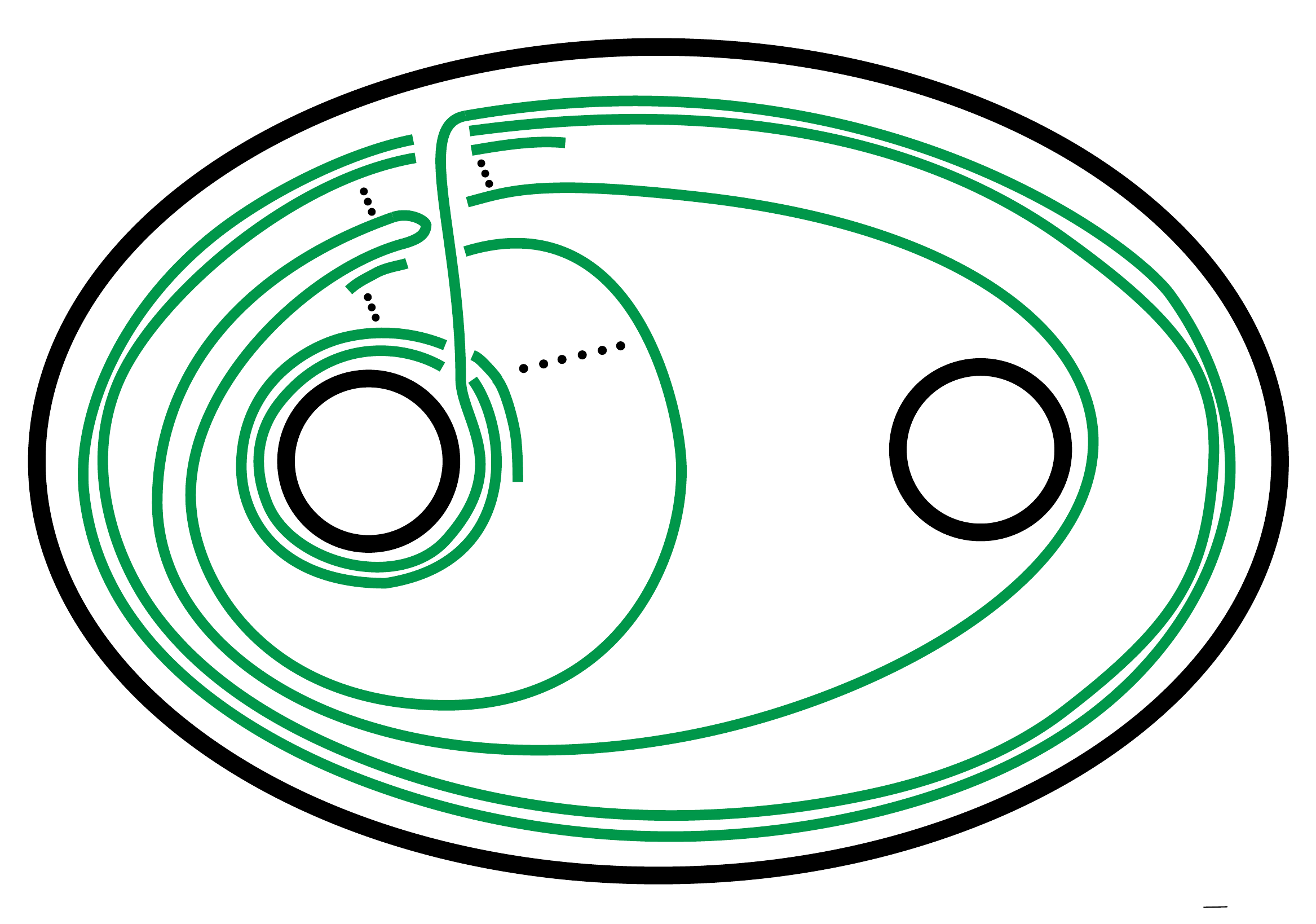}
\end{overpic}}}a_{3}\\ 
& = & -A^{-2} \vcenter{\hbox{\begin{overpic}[scale=.1]{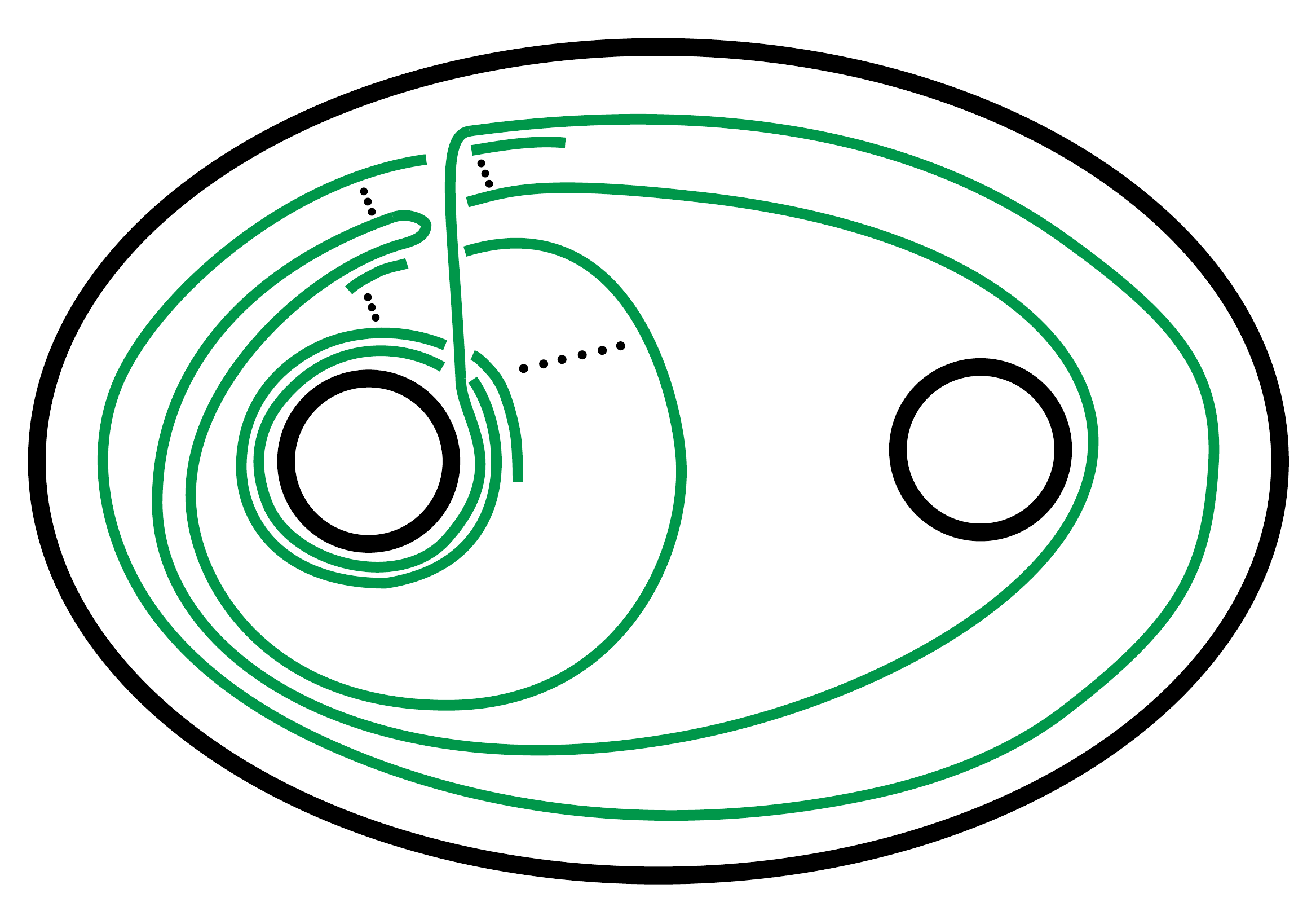}
\end{overpic}}}a_{3} +A^{-1}
\vcenter{\hbox{\begin{overpic}[scale=.1]{Pm,-n-B-2.pdf}
\end{overpic}}}\\
&=& -A^{-2}P(m,-n+2)+A^{-1}P(m,-n+1)a_{3}.
\end{eqnarray*}
\caption{$P(m,-n)=-A^{-2}P(m,-n+2)+A^{-1}P(m,-n+1)a_{3}$ for $m>0, n\geq 3$.}\label{fig:cal-P(m,-n)}
\end{figure}

\end{document}